\documentclass[10pt,a4paper]{article}
\usepackage{graphicx} 
\usepackage{fullpage}
\usepackage[utf8]{inputenc}
\usepackage{amsmath,amsfonts,amsthm,amssymb}
\usepackage{graphicx}
\usepackage{graphics}
\usepackage{url}
\usepackage{hyperref}
\usepackage{color}
\usepackage{enumitem}
\usepackage{lineno}
\usepackage{subfiles}
\usepackage{mathrsfs}
\usepackage{caption}
\usepackage[subsection]{algorithm}
\usepackage{algpseudocode}
\usepackage{varwidth}
\usepackage{mathabx}
\usepackage{svg}
\usepackage{booktabs}
\usepackage{multirow}
\usepackage{multicol}
\usepackage{subcaption}
\usepackage{here}
\usepackage{xspace}
\usepackage[export]{adjustbox} 
\def\eps{\varepsilon}
\usepackage{tikz}
\usepackage{tkz-graph}
\usetikzlibrary{shapes.geometric}
\usetikzlibrary{positioning}
\usetikzlibrary {decorations.markings}
\usetikzlibrary{arrows.meta}
\usetikzlibrary{calc}

\newtheorem{thm}{Theorem}[section]
\newtheorem{lem}[thm]{Lemma}

\newtheorem{claim}[thm]{Claim}
\newtheorem{coro}[thm]{Corollary}
\newtheorem{dfn}[thm]{Definition}

\newtheorem{fact}[thm]{Fact}

\newtheorem{task}[thm]{Task}
\renewcommand{\bar}{\widebar}
\renewcommand{\hat}{\widehat}

\def\nil{\textit{nil}}
\def\head{\textit{head}}
\def\tail{\textit{tail}}
\def\reverse{\textit{rev}}
\def\successor{\textit{succ}}
\def\predecessor{\textit{pred}}
\def\Cpp{\textnormal C\raisebox{0.5ex}{\tiny++}}

\algnewcommand{\Break}{\State \textbf{break}}
\algnewcommand{\Continue}{\State \textbf{continue}}
\algnewcommand{\True}{\textbf{true}}
\algnewcommand{\False}{\textbf{false}}

\newcommand{\bojan}[1]{\textbf{\textcolor{blue}{~{B: } #1}}}

\newcommand{\mikkel}[1]{\textbf{\textcolor{red}{{MT: } #1}}}

\newcommand{\KK}[1]{\textbf{\textcolor{blue}{{KK: } #1}}}

\definecolor{orange}{rgb}{1,0.5,0}
\newcommand{\carsten}[1]{\textbf{\textcolor{orange}{{C: } #1}}}
\def\cR{{\mathcal R}}
\def\cD{{\mathcal D}}
\def\cZ{{\mathcal Z}}

\def\cC{{\mathcal C}}
\def\Tseven{{T_{7^3}}}

\newcommand{\yuta}[1]{\textbf{\textcolor{violet}{{Y: }#1}}}

\newcommand\drop[1]{}

\def\showlabel#1{}

\renewcommand{\thepage}{\roman{page}}

\title{The Four Color Theorem with Linearly Many Reducible Configurations and Near-Linear Time Coloring\footnote{A preliminary version of this paper will appear in FOCS'26.}}

\begin{document}

\author{
    {\em Yuta Inoue}\thanks{The University of Tokyo, Tokyo, Japan, \texttt{yutainoue@is.s.u-tokyo.ac.jp}.\\  
    Supported by JSPS Kakenhi 26K21777 and JP25K24465, by JST ASPIRE JPMJAP2302, and by Hirose Foundation Scholarship.}
    \and
    {\em Ken-ichi Kawarabayashi}\thanks{National Institute of Informatics \& The University of Tokyo, Tokyo, Japan, \texttt{k\_keniti@nii.ac.jp}.\\ 
    Supported by Alexander von Humboldt Foundation (Humboldt Research Award), JSPS Kakenhi 26K21777 and JP25K24465, and by JST ASPIRE JPMJAP2302.}
    \and
    {\em Atsuyuki Miyashita}\thanks{The University of Tokyo, Tokyo, Japan, \texttt{miyashita-atsuyuki869@is.s.u-tokyo.ac.jp}.\\  
    Supported by JSPS Kakenhi 26K21777 and JP25K24465, by JST ASPIRE JPMJAP2302, and by JSPS KAKENHI 26KJ0817.}
    \and
    {\em Bojan Mohar}\thanks{Simon Fraser University, Burnaby, BC, Canada \& FMF, University of Ljubljana, Slovenia, \texttt{mohar@sfu.ca}.\\ 
    Supported in part by NSERC Discovery Grant R832714 (Canada),
    by the ERC Synergy grant (European Union, ERC, KARST, project number 101071836),
    and by the Research Project N1-0218 of ARIS (Slovenia).}
    \and
    {\em Carsten Thomassen}\thanks{Technical University of Denmark, \texttt{ctho@dtu.dk}}
    \and
     {\em Mikkel Thorup}\thanks{ University of Copenhagen, Denmark, 
 \texttt{mikkel2thorup@gmail.com}.\\ 
 Supported by VILLUM Foundation grant 54451, Basic Algorithms Research Copenhagen (BARC).}}

\date{\today\vspace{-1cm}\\}

\maketitle

\begin{abstract}

We give a near-linear time 4-coloring algorithm for planar graphs, improving on the previous quadratic time algorithm by Robertson et al. from 1996. Such an algorithm cannot be achieved by the known proofs of the Four Color Theorem (4CT). Technically speaking, we show the following significant generalization of the 4CT: every planar triangulation contains linearly many pairwise non-touching reducible configurations or obstructing cycles of length at most 5 (which all allow for making effective 4-coloring reductions). 

The known proofs of the 4CT only show the existence of a single reducible configuration or obstructing cycle in the above statement. The existence is proved using the discharging method based on combinatorial curvature. It identifies reducible configurations in parts where the local neighborhood has positive combinatorial curvature. Our result significantly strengthens the known proofs of 4CT, showing that we can also find reductions in large ``flat" parts where the curvature is zero, and moreover, we can make reductions almost anywhere in a given planar graph. This also opens possibilities for extensions to higher surfaces since we can find such flat parts in any large-width triangulation of any fixed surface.

From a computational perspective, the old proofs allowed us to apply induction on a problem that is smaller by some additive constant. The inductive step took linear time, resulting in a quadratic total time.
With our linear number of reducible configurations or obstructing cycles, we can reduce the problem size by a constant factor. Our inductive step takes $O(n\log n)$ time, yielding a 4-coloring in $O(n\log n)$ total time.

In order to efficiently handle a linear number of reducible configurations, we need them to have certain 
robustness that could also be useful in other applications. All our reducible configurations are what is known as D-reducible.
\end{abstract}

\newpage
\setcounter{page}1
\renewcommand{\thepage}{\arabic{page}}

\section{Introduction}

The celebrated Four Color Theorem (4CT) 
was conjectured by Francis Guthrie in 1852 (see \cite{madison1897,tait1880dual} for some of the early history), and remained open for more than 100 years, until Appel and Haken found a proof \cite{4ct1, 4ct2} in 1976. The theorem states that every planar map or graph with no loops can be colored with four colors. So 4-coloring is always possible, but how do we 4-color a concrete planar graph? There is no easy answer. In fact, historically, it is so complicated that it has led to many proposed counter-examples to 4CT.

Robertson, Sanders, Seymour, and Thomas \cite{RSST, RSST-STOC} 
found a simpler proof and
proved in 1997 that a 4-coloring can be found in quadratic time. Here we show how to 4-color planar graphs in near-linear time:
\begin{thm}\label{thm:linear-time}
    We can 4-color a planar graph of order $n$ in $O(n\log n)$ time.
\end{thm}

All known proofs of the 4CT are based on finding a so-called \emph{reducible configuration} or an \emph{obstructing cycle}\footnote{Obstructing cycles and reducible configurations will be defined in Section \ref{sect:classic4ct}. They cannot be in a minimal counterexample to the 4CT.}.
These both allow us to reduce the planar graph (by an additive constant size) so that it can be colored inductively.

The proofs of the 4CT only show the existence of a single reducible configuration or obstructing cycle in the above statement. The existence proof is based on the notion of combinatorial curvature, which is defined at each vertex $v$ as the value $\kappa(v)=6-d(v)$. Euler's formula implies that the average curvature is strictly positive. The so-called \emph{discharging method} ``smooths out'' the curvature by redistributing the curvature locally, and it is designed so that a vicinity of each vertex whose curvature is positive after the redistribution contains a reducible configuration.

In this paper, we show that every planar triangulation contains not only one, but linearly many pairwise non-touching reducible configurations or obstructing cycles.
The main insight is that reducible configurations or obstructing cycles also appear in large ``flat" parts where the smoothed curvature is identically zero. This is the most significant novel contribution of this paper, with several nontrivial implications. An interesting aspect of this is that such large flat parts can be found on large-width triangulations of any fixed surface, which provides a way to use our results much beyond planar graphs.
\drop{
\mikkel{I am not sure if this is the right place to say it (perhaps better later), but people keep thinking that reducible configurations are everywhere, but I think the best we have is: The linear number of reducible configurations implies that there is a constant $\eps$ so that if there is a region of our planar graph with at most an $\eps$-fraction of nodes on the boundary, then the region contains a reducible configuration not intersecting the boundary.}}

From a computational perspective, as described in \cite{RSST-STOC}, the old proofs allowed us to apply induction on a problem that is smaller by some additive constant. The inductive step took linear time, using so-called Kempe changes \cite{kempe1879}, resulting in quadratic total time.
With our linear number of reducible configurations or obstructing cycles, we can reduce the problem size by a constant factor. With a joint use of Kempe chains, for graphs of order $n$, our inductive step takes $O(n\log n)$ time, yielding 4-coloring in $O(n\log n)$ total time.

In order to efficiently handle a linear number of reducible configurations, we need them to be of a special, simple form called D-reducible. 
To find a linear number of D-reducible configurations in any planar graph, we use more configurations than previous papers that only needed to find \emph{one} reducible configuration.
Appel and Haken \cite{AppelHaken89} used 1482 reducible configurations. Robertson, Seymour, Sanders, and Thomas \cite{RSST} got down to 
633. These papers used arbitrary reducibility, but Steinberger \cite{steinberger2010unavoidable} proved that it suffices to consider D-reducible configurations; however, he then used a list of 2822 D-reducible configurations. Here we want something much stronger. We do not just want to find \emph{one}, but rather a \emph{linear} number of D-reducible configurations, which forces us to handle the flat case. For this, we use a list of more than 8200 D-reducible configurations, including those from Steinbeger's list.
As we shall discuss later, only computers will perform the local case studies involving all these configurations.

Since we are simultaneously working on a linear number of reducible configurations and obstructing cycles, one might wonder whether our $O(n\log n)$ 4-coloring could be made faster on a parallel computer. However, Chechik and Mukhtar \cite{CM19}
have shown that we need $\Omega(n)$ rounds in the \emph{local model} to 4-color
planar graphs, so the scope for improvement using parallelism is limited. This implies that the Kempe chains used in the inductive step cannot be efficiently parallelized in this model.

{
\paragraph{Computer checks in proofs.}
All (correct\footnote{Incorrect proofs have been proposed and even published regularly for more than a century, e.g.,
a human proof was proposed in 1879 by Kempe \cite{kempe1879}, but a 
mistake was found 11 years later by Heawood in 1890 \cite{heawood1890}.}) 
proofs of the 4CT use computers to check concrete, well-defined, finite cases that would be unrealistic for humans to check reliably. We have also used computers. Our proof that there is a linear number of D-reducible configurations relies on lemmas with statements of the form: if we run a certain $\Cpp$ program with a given input, e.g., the D-reducible configurations, then it yields a certain output. $\Cpp$ is a classic imperative programming language used since
1986 \cite{Stroustrup86}.
The $\Cpp$ code and the inputs are made available on GitHub (see \url{https://github.com/near-linear-4ct})
so that any (skeptical) reader can run the code on their own computer.

Our use of computer checks is pragmatic. We use a computer when the code performing a case analysis of a check is simpler to understand than a case analysis completed by hand.
Using computers, we can thus focus on high-level design of the analysis (its code) rather than on tedious low-level details.

More specifically, we use computers, not only to check that our file of 8000+ proposed D-reducible configurations are indeed D-reducible (as previous papers did), but also more systematically in the analysis showing that local 
neighborhoods with certain charging always have one of these configurations (so no human needs to look at our configurations; it is only the computer programs that need access to a file containing them). For a computer check, we need to prove that it suffices to consider certain neighborhoods of limited size and limited degrees, but this still leaves us with more than $2^{100}$
neighborhoods to consider. To enable the computer to handle them, we design the analysis to adaptively group the neighborhoods. The grouping has to be so fine that each group can be analyzed efficiently. At the same time, it has to be so coarse that there are not too many groups for the computer to handle.
The algorithms behind this adaptive analysis are described in the body of the paper. We also provide both pseudocode and $\Cpp$ code. Using a powerful computer with 256 cores, we could run all the computer checks within a couple of hours.

The same code was also used many times in order to identify our sufficient set of D-reducible configurations.
Given a file of configurations as well as the discharging rules, the code reported local neighborhoods not containing a configuration, hence  where a new configuration needs to be found. We stopped when no more such neighborhoods were reported.

We now make some notes on how to verify our proof.
The body of the paper
provides an algorithmic description of the computer checks proving that they
provide the claimed checks. The
pseudocode just clarifies the actual steps of the algorithms.
This is all verified the way we normally verify mathematical descriptions of algorithms.

The next step is to verify that the $\Cpp$ code implements the pseudocode. Since the pseudocode is sufficiently detailed, this is a simple translation.
The implementation needs
to declare the variables, procedures and templates and include the right libraries, and 
then it has to  translate the steps of the pseudocode into $\Cpp$ statements. Note that verifying the translation between pseudocode and $\Cpp$ code can be done independently of the verification of the algorithm with the psuedocode. Also note that we only use $\Cpp$ to implement combinatorial algorithms, we have no issues with precision and rounding.


One might ask for a computer checked formal proof, but this would require a proof system that already had an accepted formal understanding of $\Cpp$ including the libraries involved, and we need an efficient programming language like $\Cpp$, simply because the computational task is so large (much bigger than if we just need to prove the existence of a reducible configuration as done by Gonthier \cite{gonthier08}).
However, our proof  aims to explain why planar graphs have a linear number of D-reducible configurations or obstructing cycles. It is thus intended to be read, understood, and verified by humans, save the execution of $\Cpp$ programs \footnote{While our proof, including pseudocode and $\Cpp$ code, is intended to be read, understood, and verified by humans, we also asked Gemini to check it. Using Gemini is no guarantee of correctness, but Gemini suggested we write the following explanation of its role: ``The formal pseudocodes (Appendix A) and the corresponding $\Cpp$ implementation were cross-verified using advanced reasoning AI models (such as Gemini and DeepThink). This AI-assisted verification acted as an advanced static analyzer—auditing loop invariants, pointer safety within the dart representations, and logical isomorphism—ensuring that the software architecture flawlessly executes the deductive mathematics.''}.

A curious aspect 
of using computers to model all possible neighborhoods is that it is convenient to modify the usual dart representation of combinatorial embeddings \cite{MT} to include a terminal object from category theory. This aligns well with the common use of category theory in formal program specification \cite{BarrWels95}.
}

\paragraph{Preliminaries.}
We shall use the following standard definitions and notation. For a vertex-set $U\subseteq V(G)$, we denote by $G[U]$ the subgraph of $G$ induced by $U$. By $d_G(v)$ or simply $d(v)$ if $G$ is clear from the context, we denote the degree of the vertex $v$ in $G$. A \emph{plane graph} is a graph with a fixed embedding in the plane. 
In a 2-connected plane graph, every face boundary is a cycle. The \emph{outer face}, that is the unbounded face, is bounded by the \emph{outer cycle}. Every other face is called an \emph{inner face} and is bounded by an \emph{inner facial cycle.}
For further reading on topological graph theory, we refer the reader to the monograph \cite{MT}.

We also assume the following algorithmic understanding. Given a planar graph, we can embed it in the plane in linear time, as proved by Hopcroft and Tarjan
\cite{HopcroftT74}. This orders the edges around each vertex.
Moreover, since the graph is planar, we can order
the vertices so that each vertex has at most 5 succeeding neighbors (recursively remove a vertex of degree at most 5, as guaranteed by Euler's formula). If each vertex remembers its succeeding neighbors (those removed later), we can test if two vertices are neighbors in constant time. We can then also triangulate a plane graph, ensuring that no loops or parallel edges are introduced. The triangulation can only make coloring harder, so we may assume that our input is a planar triangulation supporting constant time neighbor queries.

\section{Classic 4CT (high level description)}\label{sect:classic4ct}

By a \emph{planar triangulation} (or just \emph{triangulation}) we mean a (simple) graph embedded in the 2-sphere, all of whose faces are triangles. The 4CT says that any planar graph can be 4-colored, but since triangulation can only make coloring harder, it suffices to consider planar triangulations. 

We say a cycle $C$ is \emph{obstructing} if it is of length at most 4 with at least one vertex on each side, or of length 5, with at least 2 vertices on each side.
We say a triangulation $G$ is \emph{internally $6$-connected} if it has no obstructing cycles.
This implies that $G$ is a 5-connected triangulation and every separating set of cardinality 5 induces a 5-cycle surrounding a vertex of degree 5. 
Birkhoff proved in \cite{birkhoff} that a minimal counterexample to 4CT is an internally $6$-connected triangulation.

\subsection{Reducible configurations}
\label{sec:reducibility}

A key concept from the classic proof of the 4CT \cite{AppelHaken89,RSST} is that of a reducible configuration. A reducible configuration is a locally defined configuration for which we can say that if it is present in a triangulation, then it cannot be a minimal counterexample to 4CT.
To define reducible configurations formally, we first need to define configurations.

Sometimes we view a triangulation as embedded in the plane, and we may or may not treat one of the faces as the outer face. A \emph{near-triangulation} is a connected graph embedded in the plane such that all facial cycles, except the outer cycle, are triangles. This includes the \emph{trivial near-triangulations} consisting of a single edge or vertex.

By a \emph{$k$-coloring} of a graph $G$ we mean a coloring of $V(G)$ with colors
$1,2, \ldots, k$ such that each pair of neighbors gets different colors. Not all colors need to be present. If we just say \emph{coloring}, we mean $4$-coloring, unless otherwise stated.
Given a 4-coloring of $G$, a \emph{Kempe chain} is a 
connected component of the subgraph induced by the vertices of colors $a,b$ where $a,b$ are distinct colors in  $\{1,2,3,4\}$.
Two color pairs $\{a,b\}$ and $\{c,d\}$  are \emph{complementary} if $\{a,b,c,d\}=\{1,2,3,4\}$. A \emph{Kempe change} is a series of such color changes that result in new 4-colorings, switching colors $a$ and $b$ on some $ab$-colored Kempe chains, followed by switching of two complementary colors $c$ and $d$ at some of the $cd$-colored Kempe chains.  With this definition, when we consider sequences of Kempe changes, no two consecutive Kempe changes use the same or complementary pair of colors. If we only switch colors in one Kempe chain, we call that Kempe change a \emph{single Kempe change}. This approach was pioneered by Kempe in \cite{kempe1879}, and has been used in all later correct proofs of 4CT.

Let us first recall the basic definitions of configurations and reducibility, as formally introduced in \cite{AllaireSwart78}. 
We want to have the same definitions as in \cite{AllaireSwart78,RSST,steinberger2010unavoidable} since we will use some of their reducibility results. By a \emph{configuration} we mean a pair $(Z,\delta)$, where $Z$ is a near-triangulation and $\delta: V(Z)\to {\mathbb N}$ is the \emph{degree function} satisfying the following conditions:
\begin{enumerate}
  \item[\rm (Z1)] If $v\in V(Z)$ is a vertex that is not on the outer face of $Z$, then $\delta(v)=d_Z(v)$.
  \item[\rm (Z2)] If $v\in V(Z)$ is on the outer face of $Z$, then $\delta(v) - d_Z(v) >0$. 
  \item[\rm (Z3)] If $v\in V(Z)$ is on the outer face of $Z$ and is a cut-vertex of $Z$, then
  $v$ is contained in precisely two blocks of $Z$, and
  $\delta(v)=d_Z(v)+2$. We only allow one cut-vertex in a configuration. 
\end{enumerate}
We will usually assume that $\delta$ is given implicitly, and we will just speak about a configuration $Z$.

We say that a plane triangulation $G$ \emph{contains the configuration $(Z,\delta)$} if the following holds:
\begin{enumerate}
  \item[\rm (Z4)] $Z$ 
  is (isomorphic to) a subgraph of $G$, every triangular {(non-outer)} face of $Z$ is also a face in $G$, and all triangular faces of $Z$ have the same orientation in $G$.

\item[\rm (Z5)] The degree of each vertex $v\in V(Z)$ in $G$ is equal to $\delta(v)$, that is $d_G(v)=\delta(v)$. Moreover, if $v$ is a cut-vertex, then the two edges in $G$ but not in $Z$ are non-consecutive in the clockwise ordering. (Informally, they go to different sides of the cut-vertex.)
\end{enumerate}
We say that $Z$, or the configuration $(Z,\delta)$, is \emph{induced} in $G$ if $Z$ is an induced subgraph of $G$, that is, two vertices in $Z$ are $Z$-neighbors if and only if they are $G$-neighbors. We note that if $Z$ is not induced, it must be because some boundary vertices of $Z$ are joined by an edge in $G$ that is not in $Z$.

 The condition in (Z4) that
inner facial triangles in $Z$ are facial triangles in $G$ is not included in previous proofs of the 4CT, as those proofs 
only consider internally 6-connected triangulations where all triangles are facial. Here we consider arbitrary triangulations which could have separating triangles, but we do not want inner facial triangles of $Z$ to map to these.

We will sometimes say that a configuration $(Z,\delta)$ is contained in a certain subgraph $H$ of $G$. By this we mean that $(Z,\delta)$ is contained in $G$ with $Z$
embedded in  $H$. Likewise, if we say that the subgraph $H$ 
has an obstructing cycle, we mean that a cycle in $H$ is obstructing in $G$.



If $(Z,\delta)$ is a configuration, the \emph{free completion of $(Z,\delta)$} is the near-triangulation $\widehat Z$ obtained from $Z$ by adding a cycle $R$ disjoint from $Z$ and joining vertices on the outer face of $Z$ to $R$ in such a way that $R$ becomes the outer face of the near-triangulation $\widehat Z$ and such that $\widehat Z$ contains $(Z,\delta)$. The length of the cycle $R$ is precisely
$$
   |R| = \sum_{v\in V(Z)} (\delta(v) - d_Z(v)) - t = \sum_{v\in V(Z)} \delta(v) - 2|E(Z)| - t,
$$ 
where $t$ is the length of the outer facial walk of $Z$. The cycle $R = R(Z)$ is called the \emph{ring} of the free completion. We will sometimes refer to the free completion $\widehat Z$ as the \emph{extended configuration} of $Z$. By the remark made after (Z5), there is no ambiguity about how to construct $\widehat Z$ even if $Z$ contains cut-vertices. See Figure~\ref{fig:extended configuration} for an example.
\begin{figure}[htb]
    \centering
    \includegraphics[width=0.62\textwidth]{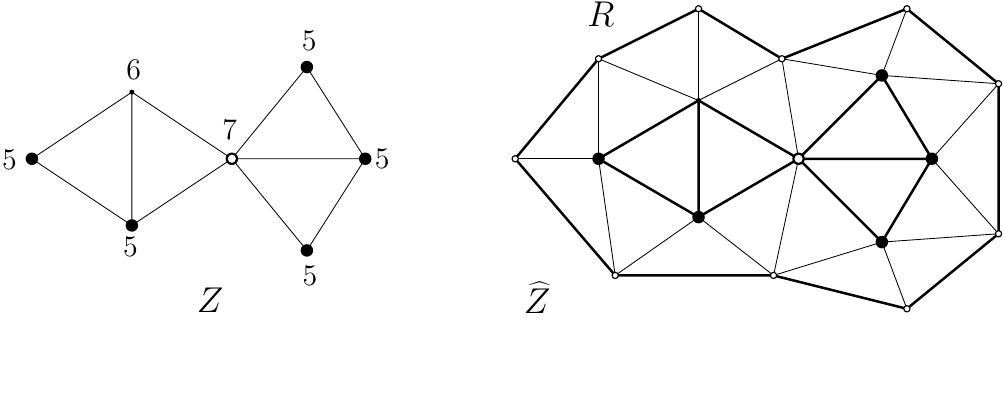}
    \caption{A D-reducible configuration $(Z,\delta)$ with a cut-vertex and its free completion.}
    \label{fig:extended configuration}
\end{figure}

When presenting configurations, as in Figure \ref{fig:extended configuration}, we use Heesch's notation \cite{Heesch69} shown in Figure~\ref{fig:shapes}.

\begin{figure}[htbp]
  \centering
  \scalebox{0.95}{%
  \includegraphics[width=0.68\columnwidth]{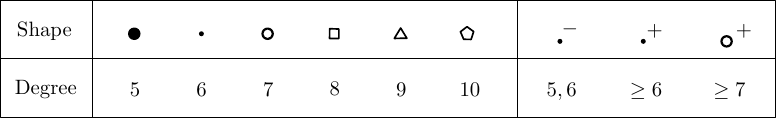}
  }
  \caption{Shapes used to designate vertices of specific degrees.}
  \label{fig:shapes}
\end{figure}

Suppose that $G$ contains an induced configuration $(Z,\delta)$. Since $Z$ is induced in $G$, the triangulation $G$ contains the free completion of $Z$, possibly with some vertices and edges on $R$ that may be identified in $G$. We note that there are no further requirements for the ring and the graph outside the ring. For example, the ring could be self-intersecting, corresponding to cut-vertices in $G-V(Z)$.
Every 4-coloring $\phi_0$ of $G-V(Z)$ induces a 4-coloring 
of the ring $R$ in the free completion. If this 4-coloring can be extended to a 4-coloring $\phi_1$ of $\widehat Z$, then the combination of $\phi_0$ and $\phi_1$ gives a 4-coloring of $G$. In the proof of 4CT, special configurations are shown to be reducible. This happens if any 4-coloring $\phi_0$ can be changed (by using a sequence of Kempe changes in $G-V(Z)$) so that the resulting coloring extends from $R$ to $\widehat Z$. The precise definition and the process related to Kempe changes will be described next.

We say that a coloring $\phi$ of $R$ is \emph{$0$-extendible} (or just \emph{extendible}) to $Z$ if it can be extended to a coloring of $\widehat Z$. For each integer $i>0$ we say that $\phi$ is \emph{$i$-extendible} to $Z$ if it is either $(i-1)$-extendible or the following holds: If $G'$ is an arbitrary plane graph containing the ring $R$ as a facial cycle and $\phi_0$ is any 4-coloring of $G'$ extending $\phi$, then there is a Kempe change of $\phi_0$ such that the resulting
4-coloring on $R$ after making the Kempe change is an $(i-1)$-extendible 4-coloring of $R$. We define the \emph{reducibility level} (or just \emph{level}) of $\phi$ as the minimum $i$ such that $\phi$ is $i$-extendible.
A Kempe change in $G-V(Z)$ is called an \emph{improving Kempe change} if it changes the coloring of $R$ to a lower level.

We say that the configuration $(Z,\delta)$ is D$^i$-\emph{reducible} if every 4-coloring of its ring is $i$-extendible, and D-\emph{reducible} means D$^i$-reducible for some finite $i$.
It is easy to see \cite{AllaireSwart78,Franklin1922} that a minimum counterexample to 4CT cannot contain D-reducible configurations.

Exercising the definition, we consider the cases of a single vertex of degree 3 and 4. 

\begin{lem}\label{lem:3-4}
    A single vertex of degree 3 or 4 is D-reducible. It is 0-extendible if the degree is 3 and 1-extendible if the degree is 4.
\end{lem}

\begin{proof}
    If the vertex has degree 3, we can remove it, color the remaining graph, and then put it back, assigning a color not used in its ring.
    If the vertex $v$ has degree 4, we also remove it and color the rest. If the ring has at most three
    colors, we are done as above. Otherwise, we need a Kempe change. Let $v_0,\ldots, v_3$ be the ring vertices in the order they appear 
    around $v$. Look at the Kempe chain $A$ of $v_0$ induced by the colors of $v_0$ and $v_2$. If $A$ does not include $v_2$, we simply perform the Kempe change on $A$. Otherwise, $A$ connects $v_0$ and $v_2$, but this implies that we can do a Kempe change on $v_1$ so that it gets the same color as $v_3$. In each case, the modified coloring can be extended to $v$. The proof for the degree-4 case goes back to Kempe \cite{kempe1879}.
\end{proof}

As mentioned earlier, if the triangulation is internally 6-connected, then we cannot have vertices of degree 3 and 4. However, the example is illustrative and our fast algorithm, which works on any triangulation, will use the fact that these cases are D-reducible.

We now discuss a more general notion of reducibility, called C-reducibility. 
It was introduced by Franklin \cite{Franklin1922} and used in \cite{AppelHaken89} and \cite{RSST}; see also \cite{AllaireSwart78}. 
Suppose that $X\subseteq E(\widehat Z)\setminus E(R)$ is a fixed non-empty edge-set such that the 
graph $\widehat Z/X$ obtained from $\widehat Z$  by contracting all edges of $X$ 
is loopless. If every 4-coloring of $R$ that arises from some 4-coloring of $\widehat Z / X$ is $i$-extendible, for some finite $i$, then we say that $Z$ is \emph{{\rm C}$^i$-reducible} with respect to the \emph{reducer} $X$. We note that when we say that the coloring of $R$ is $i$-extendible, it still means that by using at most 
$i$ Kempe changes in $G-V(Z)$ (not in $G/X$), we can get a coloring that is extendible to $Z$.  Some
subtleties exist. In \cite{RSST}, it is assumed that no triangle in $\widehat Z$ contains more than one edge of $X$ and also that $|X|\le 4$. However, we only need that $G/X$ cannot introduce any loop unless there is an obstructing cycle in $G$ with all vertices in $\widehat Z$, and that any coloring of $\widehat Z / X$ induces a coloring of $R$.  
\drop{
If such an $X$ exists and any induced coloring of $R$ is $i$-extendible, then we 
say that $(Z,\delta)$ is \emph{{\rm C}$^{\,i}$-reducible} (with respect to \emph{reducer} $X$).\footnote{Let us observe that although the reducer for Franklin's C$^0$-reducible configuration $F_6$ (see Figure \ref{fig:BirkhoffFranklin}) contains 18 edges, it satisfies our extended definition, see Lemma \ref{lem:BirkhoffFranklin}.}
}

Let us consider three examples illustrated in Figure \ref{fig:BirkhoffFranklin}.
Birkhoff \cite{birkhoff} proved that the configuration consisting of two adjacent triangles with all four vertices of degree 5 is D-reducible (see Figure \ref{fig:BirkhoffFranklin}(a)).  Based on this result, this configuration is known as the \emph{Birkhoff diamond}. More generally, if we increase the degrees of the middle vertices to 6, the resulting configuration is C-reducible. In this case, the reducer $X$ has four edges that are indicated in Figure \ref{fig:BirkhoffFranklin}(b) as bold half-edges leaving the configuration.
Franklin \cite{Franklin1922} proved that the configuration consisting of a vertex of degree 6 surrounded by six vertices of degree 6 is C-reducible (see Figure \ref{fig:BirkhoffFranklin}(c)). This case uses a reducer that identifies all 6 vertices of degree 4 on the ring into a single vertex. 

\begin{figure}[htb]
    \centering
    \includegraphics[width=0.65\textwidth]{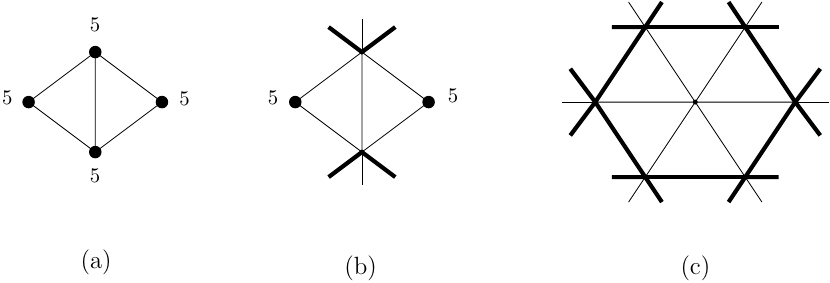}
    \caption{(a) Birkhoff diamond, (b) diamond with increased degrees, (c) Franklin's 6-regular configuration.}
    \label{fig:BirkhoffFranklin}
\end{figure}

\paragraph{Computer check of D-reducibility.}
We note that checking the reducibility of a configuration $Z$ is typically done by a computer that will compute the extendibility of every possible 4-coloring of the ring $R$
(there are nearly $3^{|R|-1}$ options). Our program to check for reducibility is very similar to the one from \cite{RSST}.

First, the program identifies which 4-colorings of $R$ are 0-extendible. Next, for rounds $i=1,2,\ldots$, we go through all unclassified 4-colorings. For each such 4-coloring, we check if there must
exist an \emph{improving Kempe change} resulting in an $(i-1)$-extendible 4-coloring of the ring. The computer proof that such an improving Kempe change exists uses arguments similar to those in the proof in Lemma \ref{lem:3-4} that a degree-4 vertex is D$^1$-reducible. The only reasoning we need is that if $v_1,\ldots, v_\ell$ is the walk along the ring, whose interior
has been removed, and if $\{a,b\}$ and $\{c,d\}$ are pairs of complementary colors and there is an $ab$-Kempe chain from $v_1$ to $v_j$, then there cannot be a $cd$-Kempe chain from a vertex in $\{v_2,\ldots,v_{j-1}\}$ to a vertex in $\{v_{j+1},\ldots,v_\ell\}$.
We thus have to consider all possible configurations and 4-colorings of $G\setminus Z$, including cases where a vertex is repeated in $R$, and hence a cut-vertex in $G-V(Z)$.

We terminate after round $i$ if it completes the classification of all 4-colorings; in that case, $Z$ is D-reducible. We also terminate after round $i$ if it finds no new $i$-extendible coloring. In this case, we know that $Z$ is not D-reducible. 

The pseudocode for checking $D$-reducibility is found in Algorithm \ref{alg:check-red_a}.

\subsection{4CT in quadratic time}

All known proofs of the 4CT show
that any internally 6-connected triangulation has a reducible configuration, contradicting the minimality of a hypothetical counterexample. The reducible configurations are all of constant size, since they all belong to
some specific finite set. Steinberger \cite{steinberger2010unavoidable} has even proved that we can always find a D-reducible configuration \cite{steinberger2010unavoidable}, hence the C-reducible configurations can be avoided.

\paragraph{A quadratic algorithm.}
In \cite{RSST-STOC}, the authors derive a quadratic-time 4-coloring algorithm. 
Recall that a planar triangulation $G$ is internally 6-connected if and only if it has no obstructing cycles, that is, a separating 3- or 4-cycle, or a 5-cycle with at least two vertices on both sides.

The above-mentioned proofs of 4CT imply that for an arbitrary triangulation, we can find either an obstructing cycle or a reducible configuration, and this can be done in linear time.

If we find a reducible configuration $Z$, we reduce it and color the rest by induction. This gives a coloring of the ring $R$. If it is an $i$-extendible configuration, then we may have to perform up to $i$ improving Kempe changes before we can extend the coloring to the original $Z$. 

Note that to tell if a Kempe chain is improving, we need to know the level of all colorings of the ring $R$. This is only a constant amount of information, since we only have a constant number of reducible configurations of constant size.

Also note that we only know that an improving Kempe change exists, but which one works depends on the graph and the coloring outside the ring. Recall that a Kempe change is characterized by two complementary pairs $\{a,b\}$ and $\{c,d\}$ of colors and a subset of $ab$- and $cd$-colored Kempe chains on which the two colors are swapped. The swapping only affects the coloring of the ring if the Kempe chain
intersects the ring. This means that there are fewer than $3\cdot 2^{|R|}=O(1)$ relevant Kempe changes to check, and we apply the first improving one we find. All this can be implemented in linear total time.

\cite{RSST-STOC} also shows how we can benefit from obstructing cycles. The general conclusion is that we can reduce the problem size by an additive constant in linear time, so the overall algorithm is quadratic. 

The main bottleneck in the above approach is the Kempe changes. We note that the interaction between different Kempe chains can be very complicated. For example, if we first perform a Kempe change using colors 1 and 2 and their complementary pair 3 and 4, then it affects the 13, 14, 23, and 24-colored Kempe chains. So, the result of multiple Kempe changes can be quite chaotic. It seems hopeless to update the information about Kempe chains in sublinear time.

\section{Our results (high level description)}

Previous proofs and algorithms for 4CT only guarantee a single reducible configuration. We will show here that we can always find a linear number of reducible configurations if the triangulation is internally 6-connected. In general, we find either a linear number of reducible configurations or a linear number of obstructing cycles (or both). To make algorithmic use of the many reducible configurations, we will need them to be D-reducible. 

To get a linear number of reducible configurations, we do need some 6-regular reducible configuration since the graph could have all vertices of degree 6 except for 12 of degree 5. The classic 6-regular configuration is Franklin's from Figure~\ref
{fig:BirkhoffFranklin} (c), but it is not D-reducible. However, we discovered the larger D-reducible 6-regular configuration in Figure \ref{fig:flatn}, which is the only 6-regular configuration we use.

\begin{figure}
    \centering
    \includegraphics[width=0.3\linewidth]{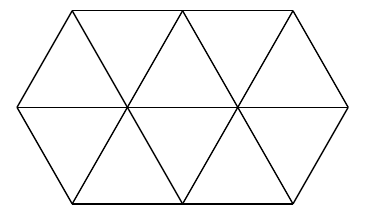}
    \caption{The configuration in $\mathcal D$ for the 6-regular case.} \label{fig:flatn}
\end{figure}


We will need a special property of configurations that will enable us to use the following fact. 

\begin{lem}\label{lem:induced-config}\showlabel{lem:induced-config}
Let $(Z,\delta)$ be a configuration. Suppose it satisfies the following conditions:
\begin{itemize}
    \item[(i)] The diameter of $Z$ is at most 4.
    \item[(ii)] If $u$ and $w$ are vertices at $Z$-distance 4, then there is a length-4 path $P$ joining them in $Z$ such that either $P$ is not contained in the outer face boundary of $Z$, or $P=(u,v_1,v_2,v_3,w)$ is  contained in the outer face boundary of $Z$,
    and then one of the three vertices $v_i$ is not a cut-vertex of $Z$ and has external degree $\delta(v_i)-d_Z(v_i)\geq 2$.
\end{itemize}
If $(Z,\delta)$ is contained in $G$ and $G$ does not contain an obstructing cycle whose vertices are all in $Z$, then $(Z,\delta)$ is induced in $G$.
\end{lem}

\begin{proof}    If $(Z,\delta)$ is contained in $G$ but not induced, it is because $G$ has an edge not in $Z$ but joining vertices $u$ and $w$ from $Z$. Then $u$ and $w$ must be on the boundary of $Z$ since only boundary vertices can have incident edges not in $Z$. By (i), we have a shortest path $P$ of length at most 4 between $u$ and $w$, and if it is of length 4, we can assume it satisfies (ii). Now $P$ and the edge $uw$ form a cycle $C$, which we want to prove is obstructing.

   We now walk clockwise from $u$ to $w$ along the outer face boundary of $Z$. Since the edge $uw$
   is not in $Z$, we must meet a vertex $v$ between $u$ and $w$. By (Z2) and (Z5), $v$ must have a neighbor $v'$ to the left and outside $Z$ (because the edges leaving to the left cannot all be chords). Clearly,
   $v$ also has a neighbor $v''$ to the right and outside $Z$.
   So, if $C$ has length at most $4$, then it is obstructing.

   Now suppose $C$ is of length 5. Then we need to find two vertices on either side of $C$. Again, we walk clockwise around the boundary of $Z$ from $u$ to $w$. If this walk is not just following $P$, then
   it is easy to find a vertex $v'$ in the interior of $C$ and a vertex $v''$ in the exterior of $C$. Condition (ii) implies that $v'$, respectively $v''$, is not the only such vertex. If the walk follows $P$, then again, condition (ii) implies that the interior, respectively exterior, of $C$ has at least two vertices.
   Therefore, $C$ is an obstructing 5-cycle.
\end{proof}

Note that the condition in (ii) is needed for the conclusion that $Z$ is induced. Indeed, if for all $i=1,2,3$, $v_i$ is a cut-vertex or $\delta(v_i)-d_Z(v_i)= 1$, then $u,v$ may be neighbors such that $P+uv$ is the cycle of neighbors around a single degree-5 vertex outside $Z$.


All reducible configurations used in this paper also satisfy (i) and (ii) of Lemma \ref{lem:induced-config} (see item (D0) in Lemma \ref{lem:reduc}). This property, which is needed for 
our proof, has not been addressed explicitly in previous proofs of the 4CT.
We have verified by computer that the conditions of Lemma \ref{lem:induced-config} are satisfied not only for all reducible configurations used in this paper, but also for all reducible configurations used by Robertson et al.\ in \cite{RSST} and by Steinberger in \cite{steinberger2010unavoidable}.

The D-reducible configurations that we shall use are from a set $\cD$ that has 8202 elements. Two of them are the single vertices of degree 3 or 4 from Lemma \ref{lem:3-4} that are relevant when we consider arbitrary triangulations. All other configurations in $\cD$ use only vertices of degree at least 5. These 8200 configurations are collected in our GitHub repository \footnote{
\url{https://github.com/near-linear-4ct/reducible-configurations}
}. The set $\cD$ contains the set of all D-reducible configurations used by Steinberger in \cite{steinberger2010unavoidable}. We will use $\cD_0$ to denote this subset of configurations. 

The important characteristics of the configurations in $\cD$ are summarized in the following lemma. 

\begin{lem}\label{lem:reduc}\showlabel{lem:reduc}
The set $\mathcal{D}$ consists of 8202 D-reducible configurations. It includes the D-reducible configurations consisting of a single vertex of degree 3 or 4 (cf. Lemma \ref{lem:3-4}), the ``flat'' case and includes Steinberger's set $\cD_0\subset \cD$ of configurations from \cite{steinberger2010unavoidable}. Each configuration $D\in \cD$ has the following properties:
\begin{itemize}
    \item[\rm (D0)] $D$ satisfies conditions (i) and (ii) of Lemma \ref{lem:induced-config}; in particular, it has diameter at most 4.
    \item[\rm (D1)] With a single exception, $D$ has 
    radius at most 2. We distinguish a vertex of eccentricity 2, which we call the
    \emph{center vertex} or just \emph{center}. The center is the only vertex in $D$ that can be of degree greater than 8, in which case it is between $9$ and $12$. The exceptional configuration in $\mathcal{D}$ which has radius $3$
    and degrees 5 and 7
    is shown in Figure~\ref{fig:radius3}. It is one of our configurations to handle the flat case.
    \item[\rm (D2)] $D$ has at most 19 vertices and its ring size is at most 18.
    \item[\rm (D3)] Every $4$-coloring of the ring of the free completion of $D$ is $25$-extendible.
\end{itemize}
\end{lem} 
 
\begin{figure}
    \centering
    \includegraphics[width=0.3\linewidth]{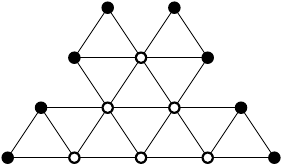}
    \caption{The configuration in $\mathcal D$ whose radius is 3 (but the diameter is still 4).} \label{fig:radius3}
\end{figure}

In order to make use of multiple reducible configurations in a triangulation, we want them to be \emph{non-touching} in the sense that they are disjoint and the graph has no edge joining them.\footnote{The definition of two subgraphs of $G$ to \emph{touch}, meaning that they either have a vertex in common or $G$ contains an edge joining them, is frequently used in graph minors theory, see \cite{Diestel_book,DvorakBook} or \cite{reed1997tree}.} This implies that their rings cannot cross.

Our main theorem is the following.

\begin{thm}\label{thm:main-linear}\showlabel{thm:main-linear}
    There is a linear-time algorithm that, given any triangulation, finds a linear number of non-touching either separating 3- or 4-cycles, or reducible configurations from $\cD$.
    All vertices involved have degree at most 12 in the triangulation.
\end{thm}
As a corollary using Lemma \ref{lem:induced-config}, we get 
\begin{coro}\label{cor:main-linear}\showlabel{cor:main-linear}
    There is a linear-time algorithm that, given any triangulation, finds one of the following: \begin{itemize}
    \item A linear number of pairwise non-touching induced reducible configurations in $\cD$.
    \item A linear number of non-touching obstructing cycles.
    \end{itemize}
        All vertices in the above-mentioned reducible configurations or obstructing cycles have degree at most 12 in the triangulation.
\end{coro}

\subsection{Sketch of $O(n\log n)$ recursion}
\label{sec:algorithm-sketch}\showlabel{sec:algorithm-sketch}

To obtain an $O(n\log n)$ 4-coloring algorithm for a planar triangulation $G$ of order $n$, it suffices that we can construct in $O(n\log n)$ time a planar triangulation $H$ that is smaller by a constant factor, color $H$ recursively, and use this coloring to color $G$ in $O(n\log n)$ additional time.

Given $G$, we first apply Theorem \ref{thm:main-linear}, and then we make a recursion depending on which of the two outcomes we get. The most interesting case is the first one.

\paragraph{D-reducible configurations.}
Suppose we get $\Omega(n)$ non-touching D-reducible configurations.
In the quadratic-time algorithm, we would just handle one of them in linear time, but
here we will handle all of them in $O(n\log n)$ time. 
Below, we first show this using randomization. 

More specifically, our first obvious step is to remove all the reducible configurations from the graph. This removes $\Omega(n)$ vertices. Next, using new auxiliary edges, we triangulate all the faces inside their now empty rings. Then, recursively, we color the resulting reduced triangulation, which has a constant factor fewer vertices.

Having colored the reduced
triangulation recursively, we restore the original triangulation exactly as it was, removing the auxiliary edges and putting back all the reducible configurations. However, we keep the recursive coloring, which is now a 4-coloring of all vertices outside the reducible configurations. We note that the reducible configurations, while non-colored, have all their rings colored. This follows because they are non-touching, implying that the ring of one reducible configuration does not intersect any other reducible configuration.

We will now start doing Kempe changes on the colored vertices. If the ring $R$ of a non-colored configuration $Z$ gets a 0-extendible coloring, then we extend this 4-coloring of $R$ to $Z$, and then $Z$ is no longer among the non-colored configurations (strictly speaking, it is $(Z,d_G)$ that is the reducible configuration).

In general, if the coloring of a ring is  $i$-extendible, then we may have to
do $i$ improving Kempe changes to make it $0$-extendible. We know from Lemma \ref{lem:reduc} that all colorings of the rings of our reducible configurations are 25-extendible. Each Kempe change may affect most of the graph, so we cannot afford to do this for one configuration at a time. 
However, we claim that random Kempe changes are expected to be improving for a constant fraction of 
the non-colored configurations, and this leads to the following very simple algorithm.

While there is a non-colored configuration, we repeat the following two
steps:
\begin{enumerate}
    \item Make
    a random Kempe change.
    First, we randomly
    decide which two complementary pairs $\{a,b\}$ and $\{c,d\}$ of colors will be used by the Kempe chains, that is, the Kempe chains 
    considered are the $ab$- and $cd$-colored 
    components. These partition all colored vertices. Next, for every $ab$-colored and every $cd$-colored component, we toss a coin to decide whether to swap the colors on it in the random Kempe change.
    \item For all non-colored components $Z$ whose ring coloring became $0$-extendible, we extend the coloring to $Z$.
\end{enumerate}
We claim that this 
algorithm terminates after $O(\log n)$ iterations with high probability. To prove 
this, we consider an arbitrary non-colored configuration $Z$, and argue that its ring coloring will become $0$-extendible within $O(\log n)$ rounds with high probability. The result then follows by a union bound.

Consider the iterations from
the perspective of any given reducible configuration $Z$ with ring $R$. If $Z$ is non-colored, it wants a sequence of at most 25 improving Kempe changes. As discussed in the quadratic algorithm, an improving Kempe change requires the right choice of complementary color pairs $\{a,b\}$ and $\{c,d\}$.
There is a $1/3$ chance that a random choice is right since the pairs are determined just by the pairing of color 1. Now the improvement only depends on the color change (color swap or no-swap) 
of the disjoint $ab$ and $cd$-colored components (Kempe chains) intersecting the ring $R$, of which there are at most
$|R|\leq 18$. A random Kempe change is thus improving for $Z$ with probability at least $\tfrac{1}{3}\cdot 2^{-18}$. We need at most 25 such consecutive improvements for $Z$ for the ring to become $0$-extendible. The probability of this success for $Z$ is at least $p=1/(3\cdot 2^{18})^{25}=\Omega(1)$.

The above fails for $Z$ if we make a Kempe change that is non-improving for $Z$ before $R$ becomes $0$-extendible, but if it fails, we still have a coloring of $R$, and any coloring of $R$ is 25-extendible. Thus, we just try again for the next at most 25 random Kempe changes. The probability that we fail $f$ times is at most
$(1-p)^f<e^{-pf}$, so $Z$ fails $2p^{-1}\ln n$ times with probability
at most $1/n^2$; otherwise
$Z$ succeeded within 
$50p^{-1}\ln n$ iterations. A union bound then implies that 
all configurations succeed
within $50p^{-1}\ln n$ iterations with probability at least $1-1/n$.

Using the method of conditional expectations, we are going to derandomize the above Kempe changes in Section~\ref{sect:D-reducible}. The derandomization always makes the greedy choice maximizing the expected number of configurations that become 0-extendible. In particular, it will always get at least one 0-extendible configurations within 25 Kempe changes.

\paragraph{Why do we insist on D-reducible configurations?}
It is important to note that the above approach would not work for general C-reducible configurations $Z$ with reducer $X$, e.g., the graph in Figure \ref{fig:BirkhoffFranklin} (b).

We could, as above, reduce all of them, removing $Z$ and contracting $X$, and color the graph inductively. However, $X$ is only used to control this initial coloring of the ring $R$ to make sure that the coloring is $i$-extendable, that is, with ring $R$ empty inside, and no contracted $X$, there are $i$ improving Kempe chains, leading to an extendable coloring of the ring. 

However, for most of these C-reducible configurations, the Kempe chains may not be improving, and even worse, they could very easily lead to colorings that are 
not 
extendable for $Z$ (i.e., not possible to color $Z$ with this ring coloring). Then there is no way to complete the coloring process. With ring $R$ empty inside, if there are $i$ improving Kempe chains, then they lead to an extendable coloring of the ring, but we may never obtain such chains because the contracted graph inside the ring may prevent them. 

One could imagine another way of working with C-reducible configurations $Z$, where we kept the contractions of the reducer while doing Kempe changes. However, there are many important cases where there are no improving Kempe changes, e.g., the classic C-reducible Birkhoff diamond in Figure \ref{fig:BirkhoffFranklin} (b).

\drop{
\paragraph{The C$^0$-reducible configuration $F_6$.}
The above counter-argument is problematic only for C$^i$-reducible configurations with $i>0$. If we have $r$ C$^0$-reducible configurations $Z$, such as $F_6$ from Lemma \ref{lem:BirkhoffFranklin}, then things are easy. For
each of them, we assume no obstructing cycle in $G[V(\widehat Z)]$, and further we want the completions $\widehat Z$ to 
satisfy the assumption on
$F_6$ in Theorem \ref{thm:main-linear}.

Since we have no obstructing cycle in $G[\widehat Z]$,  we can contract the reducer without introducing any loops. The contractions only involve $G[V(\widehat Z)]$. Since the subgraphs $G[V(\widehat Z)]$ are non-intersecting, 
the contractions in one cannot introduce an obstructing cycle within another. We can therefore simply do all the contractions of all the C$^0$-reducible configurations, color the reduced graph, uncontract the reducers, and extend the rings coloring to the C$^0$-reducible configurations.
}

\paragraph{Obstructing cycles.}
Finally, we have the case of a linear number of pairwise non-touching obstructing cycles. In this case, we apply the techniques from \cite{RSST-STOC} to handle a single obstructing cycle; these techniques can be extended to a component hierarchy in which components are separated by obstructing cycles. To achieve an efficient reduction using such a hierarchy, it is crucial that the obstructing cycles are non-touching as described in Theorem \ref{thm:main-linear}. With a linear number of obstructing cycles, we can reduce the problem size by a constant factor in linear time. A full description of this case is found in Section~\ref{sect:obstructing-cycles}.

\paragraph{Bottleneck: the number of Kempe chains.}
Let us mention one important challenge. One Kempe change may take linear time, because even a single Kempe chain may be of linear order. One important thing to note is that the previous algorithm \cite{RSST-STOC} requires $O(1)$ Kempe changes to reduce the graph size by a constant size. This explains why the previous algorithm takes $O(n^2)$ time. The correctness of the algorithm requires the proof of the 4CT. 

In contrast, we do only $O(\log n)$ Kempe chains to reduce the graph size by a constant factor. This is our main challenge, because one Kempe chain may take linear time, and in order to achieve our near-linear time algorithm, a linear number of reducible configurations to make reductions by a constant factor is definitely necessary (which is our main technical contribution). 

There is a linear-time algorithm for 5-coloring planar graphs \cite{CHIBA1981317}; however, the method presented there does not involve any Kempe chain/change. In contrast, for 4-coloring, we definitely need Kempe chains/changes, or we would need a completely new proof of 4CT.

\section{Discharging to find a reducible configuration in classic 4CT}
\label{sect:discharging}\showlabel{sect:discharging}

Let $G$ be a triangulation in the plane. We start by assigning to each vertex $v \in V(G)$ the value 
$$
    T_0(v) := 10(6 - d(v)),
$$
which we call the \emph{initial charge at $v$}. It follows from Euler's formula (see Lemma \ref{lem:120} below) that 
$$
    \sum_{v\in V(G)} T_0(v) = 120.
$$
Now we redistribute the charge among the vertices applying \emph{discharging rules}. We use the same set of rules as in \cite{steinberger2010unavoidable}. As mentioned in the introduction, the initial charge $T_0(v)$ can be viewed as a discrete curvature measure. The discharging process is a way to ``smooth out'' the curvature locally around each vertex, and the rules are chosen in such a way that we obtain a reducible configuration or a separating 3- or 4-cycle whenever the final charge is positive. 

Discharging rules are formally defined as follows. 

\begin{dfn}\label{dfn:rule}\showlabel{dfn:rule}
  A \emph{rule} is six-tuple $R = (G(R), \delta^-_R, \delta^+_R, s(R), t(R), r(R))$, so that
  \begin{enumerate}
      \renewcommand{\labelenumi}{(\roman{enumi})}
      \item $G(R) = (V(R), E(R))$ is a near-triangulation without cut-vertices.
      \item $\delta^-_R \colon V(R) \rightarrow \mathbb{N}$, $\delta^+_R \colon V(R) \rightarrow \mathbb{N} \cup \{\infty\}$, such that $5 \leq \delta^-_R(v) \leq \delta^+_R(v)$ for each $v \in V(R)$.
      \item $s(R), t(R) \in V(R)$ are distinct adjacent vertices.
      \item $r(R)$ is a positive integer.
  \end{enumerate}
  A configuration $(Z,\delta)$ \emph{obeys} the rule $R$ if $Z$ is isomorphic to $G(R)$ and every pair $(v_Z, v_R) \in V(Z) \times V(G(R))$ of corresponding vertices under the isomorphism satisfies $\delta^-_R(v_R) \leq d(v_Z) \leq \delta^+_R(v_R)$.
\end{dfn}

The graph $G(R)$ of the rule $R$ is a near-triangulation in which each vertex $v$ is given a range of possible degrees, the interval $[\delta^-_R(v), \delta^+_R(v)]$. The value $r(R)$ represents the amount of charge that is sent from the vertex $s(R)$ to its neighbor $t(R)$. We have $r(R) = 1$ or $r(R)=2$ for all our rules. We describe the degree range $[\delta^-_R, \delta^+_R]$ 
in drawings (for example in Figure \ref{fig:rules}) as follows:
There are three cases. When $\delta^-_R(v) = \delta^+_R(v)$, we 
draw $v$ such that it has the shape for $\delta^-_R(v)$ shown in Figure \ref{fig:shapes}. 
Secondly, when $5 = \delta^-_R(v) < \delta^+_R(v)$, we 
draw $v$ such that it has the shape for the integer $\delta^+_R(v)$ shown in Figure \ref{fig:shapes}. Moreover, we add a minus sign close to $v$. 
Finally, when $\delta^-_R(v) < \delta^+_R(v) = \infty$, we 
draw $v$ such that it has the shape for $\delta^-_R(v)$ and we add a plus sign close to $v$. The edge $s(R)t(R)$ is indicated by an arrow. All our rules in Figure \ref{fig:rules} have one of these three types of degree ranges. 

In this paper, we use the set $\mathcal{R}$ of 84 rules described in Figure \ref{fig:rules}.
These rules are collected in our GitHub repository\footnote{
\url{https://github.com/near-linear-4ct/discharging-rules}}
This set is exactly the same as the set of rules used in \cite{steinberger2010unavoidable}, except that our 12th and 13th rules are considered as a single rule in \cite{steinberger2010unavoidable}. We also note that Steinberger only applied the rules to internally 6-connected graphs, whereas we consider arbitrary triangulations, where they are not necessarily induced. However, we recall that we insist that facial triangles in the rules must be matched to facial triangles in the graph.

\subfile{tikz/rule}

An edge with an arrow represents the direction in which a charge moves. Thus, $s(R)$ (resp., $t(R)$) is the tail (the head) of the edge with an arrow.
The number of arrows in the figure represents $r(R)$.

We calculate the amount of charge sent along an edge $st$ in $G$ by checking which rules apply to $st$.
We \emph{apply} a rule $R$ to $st$ if there is a (unique) homomorphism $\phi$ into $G$ such that the edge $st=\phi(st)$. The homomorphism is understood to be order-preserving, and it should map facial triangles of $Z$ into facial triangles of $G$.

We note that in previous proofs of the 4CT, it is further required that $\phi$ is 1-1, that is, an embedding, but it turns out to both simplify our arguments and yield a stronger main theorem (Theorem \ref{thm:main-linear}) when we allow applications of discharging rules also when $\phi$ is not 1-1. We shall later (in Lemma \ref{lem:1-1}) see that if the graph is internally 6-connected, then the homomorphism $\phi$ from a discharging rule of diameter at most 4  is automatically 1-1.
Removing the 1-1 condition therefore only matters because we want to handle arbitrary triangulations.

\begin{dfn}
\label{dfn:amount}\showlabel{dfn:amount}
  (a) For adjacent vertices $u, v \in V(G)$, we define $\phi(u,v)$ as the sum of the values $r(R)$ over all rules $R\in \mathcal{R}$ that are applied with $s(R) = u$ and $t(R) = v$. This is the \emph{charge sent from $u$ to $v$}.

\label{dfn:finalcharge}\showlabel{dfn:finalcharge}
  (b) For a vertex $u \in V(G)$, we set
    \[
      T(u) := T_0(u) + \sum_{v \sim u} (\phi(v, u) - \phi(u,v)).
    \]
    The value $T(u)$ is called the \emph{final charge} at $u$.
\end{dfn}

We note that it is easy to process all discharging rules and compute the final charge in linear time. The point is that for each edge in the graph, there is only a constant number of ways that we can possibly apply the discharging rules from Figure~\ref{fig:rules} to the edge.

The discharging method uses the following easy observation.

\begin{lem}
\label{lem:120}\showlabel{lem:120}
$\sum_{u \in V(G)} T(u) = 120$.
\end{lem}

\begin{proof}
By applying any discharging rule, the total sum of all charges remains the same. Thus, $\sum_{u \in V(G)} T(u) = \sum_{u \in V(G)} 10(6-d(u))$. If $n$ is the number of vertices of $G$, then Euler's formula for the triangulation $G$ implies that $\tfrac12 \sum_{u\in V(G)}d(u) = |E(G)| = 3n - 6$, and thus we have:
 \[
 \sum_{u \in V(G)} T(u) = \sum_{u \in V(G)} 10(6-d(u)) = 60n - 10\sum_{u \in V(G)}d(u) =
 60n-20(3n-6) = 120.
 \]
\end{proof}

Lemma \ref{lem:120} implies that there is a vertex $u$ with $T(u) > 0$. Previous proofs of 4CT show that for any vertex $u$ with $T(u)>0$, there is a reducible configuration in the vicinity of $u$.\footnote{Each proof of 4CT, by Appel and Haken \cite{AppelHaken89}, by Robertson et al. \cite{RSST} and by Steinberger \cite{steinberger2010unavoidable} has different discharging rules, so the final charge $T(v)$ has a different meaning for each of them, but the stated conclusion is the same.} This yields a contradiction to the assumption that we have a minimum counterexample.
However, we need much more for our main result of having a linear number of reducible configurations (see Theorem \ref{thm:main-linear}).

\begin{dfn}\label{def:B_2}
    Let $v$ be a vertex in $G$ and let the \emph{ball} $B_k(v)$ be the subgraph of $G$ induced by the set of vertices at distance at most $k$ from $v$. We call 
    $B_1(v)$ the \emph{wheel} with \emph{center} $v$, and
    $B_2(v)$ the \emph{cartwheel} with \emph{center}  $v$.
\end{dfn}

As our discharging rules are exactly the same as those in \cite{steinberger2010unavoidable}, we have the following result from \cite{steinberger2010unavoidable}.

\begin{thm}
\label{thm:T(v)> 0}\showlabel{T(v)> 0}
Suppose that $G$ is an internally $6$-connected triangulation that contains a vertex $v$ with final charge $T(v) > 0$. Then one of the D-reducible configurations 
in Steinberger's \cite{steinberger2010unavoidable} set $\mathcal{D}_0\subset \cD$ is contained in the cartwheel $B_2(v)$ as an induced subgraph. 
\end{thm}

This result implies 4CT, since a minimal counterexample is known to be internally 6-connected (see \cite{RSST,RSST-STOC}). However, there may only be a constant number of vertices with
positive charge, so this does not help us find a linear number of reducible configurations.

\section{Extending unavoidability to the flat case}

The classic proof of the 4CT tells us that we can find a reducible configuration in the cartwheel (that is, within the second neighborhood $B_2(v)$) around any positively charged vertex. However, a planar triangulation may have only a few vertices of positive charge. In fact, almost all the vertices may have final charge zero, and we refer to this as the \emph{flat case}. To motivate this terminology, observe that in a large piece of the flat hexagonal tiling, all the vertices would have final charge zero. For example, this occurs with all triangulations that are dual to the class of fullerenes. Those have 12 vertices of degree 5 and all others of degree 6. Positive final charge will only be found in a neighborhood around the degree-5 vertices. In this case, the only other reducible configuration we will find there is a configuration in Figure \ref{fig:flatn}. 

However, the general flat case is much more complicated. We need thousands of additional reducible configurations to cover all cases. Formally, we say that a vertex $v$ is \emph{flat} if its final charge is zero, that is, $T(v)=0$. 
In the flat case, we may need to look at a larger neighborhood. 
We say that the ball $B_k(v)$ is \emph{flat} if all vertices in $B_k(v)$ are flat.

One of our main insights comes from analyzing large, flat balls. We will prove the following nontrivial generalization of Theorem~\ref{thm:T(v)> 0}.

\begin{thm}\label{thm:basic-flat}
    Suppose that $G$ is an internally 6-connected triangulation that contains a vertex $v$ such that 
    $B_{12}(v)$ is flat. Then $B_{14}(v)$ 
    contains a reducible configuration contained in $\mathcal D$ as an induced subgraph.
\end{thm}
Note above that to find the reducible configuration in the flat case, we have to go much beyond the cartwheel $B_2(v)$ considered in the classic Theorem \ref{thm:T(v)> 0}.

Theorem \ref{thm:basic-flat} above is our main structural
graph theoretic insight, but more is needed for an efficient coloring algorithm.

\section{Finding reducible configurations in triangulations}

The input to our coloring algorithm is an arbitrary triangulation that need not be internally 6-connected.
For our results, we want a local view of internal 6-connectivity. Recall that a planar triangulation $G$ is internally 6-connected if it has no obstructing cycles. In this section, we state lemmas and theorems, deferring most proofs to later sections. The main goal here is to communicate our new understanding.

Both reducible configurations and obstructing cycles help us make progress towards a 4-coloring. For the known quadratic 4-coloring algorithm \cite{RSST-STOC}, it suffices to find \emph{one} reducible configuration or \emph{one} obstructing cycle, but here we want to prove Theorem \ref{thm:main-linear} stating that we can find linear number of non-touching either separating 3- or 4-cycles, or reducible configurations from $\cD$. 
We will achieve this locally by finding a linear number of non-touching neighborhoods of constant size, each of which contains one of the desired objects.
First, we need the following property of our discharging rules.

\begin{lem}
\label{maxplus}\showlabel{maxplus} 
For any triangulation $G$ and $v \in V(G)$, we have final charge $T(v) \leq 60 - 2 \cdot d_G(v)$. 
Consequently, $T(v) \leq 54$ and if $d_G(v) \geq 31$, then $T(v) \le -\tfrac{2}{31} d_G(v)$.  
\end{lem}

The proof of this lemma is deferred to Section \ref{sect:edge-charge}.
The final charge of a vertex is thus linearly decreasing with the degree.
The constants (e.g., 54, 31) in Lemma \ref{maxplus} are not important. The existence of an upper bound on the final charge and the degree of a vertex whose final charge is possibly positive or zero is important because we only consider neighborhoods of vertices with non-negative final charge, and now we know they have bounded degree.

\subsection{Local neighborhoods of constant size}

We will carefully define a constant-size local neighborhood around $v$ with $T(v)\geq 0$ that contains a homomorphic image of a reducible configuration and such that each local neighborhood can interfere with only a constant number of other local neighborhoods. By interference between local neighborhoods, we mean that their reducible configurations are touching.

Our local neighborhoods are defined in terms of (extended) degree-bounded balls:

\begin{dfn}\label{def:limitcart}\showlabel{def:limitcart}
 Let $v_0$ be a vertex in $G$. The \emph{$k$-local neighborhood} around $v_0$ has the following elements:
\begin{itemize}
\item 
The \emph{degree-bounded ball} $B_k^8(v_0)$ of the \emph{center vertex} $v_0$ is the  subgraph of $G$ induced by the set of vertices $V_1 \subseteq V(G)$ such that 
for every $v_1 \in V_1$, there is a path $P$ between $v_0$ and $v_1$ of length at most $k$, and the degrees of the vertices of $P$ except for the center $v_0$ are at most 8. 
\item 

The set of \emph{boundary vertices} $V_2$ consists of the vertices $v\ne v_0$ of degree at least 9 such that for every $v \in V_2$, there is a path $P$ between $v_0$ and $v$ of length at most $k$ and all the internal vertices of $P$ are of degree at most 8. We define $\bar B_k^8(v_0) = G[V_1 \cup V_2]$ calling it the \emph{extended degree-bounded ball} around $v_0$.
\end{itemize}
\end{dfn}
The definition of the different cartwheels is illustrated in Figure \ref{fig:cartwheels}.
\begin{figure}
\centerline{\includegraphics[width=0.8\textwidth]{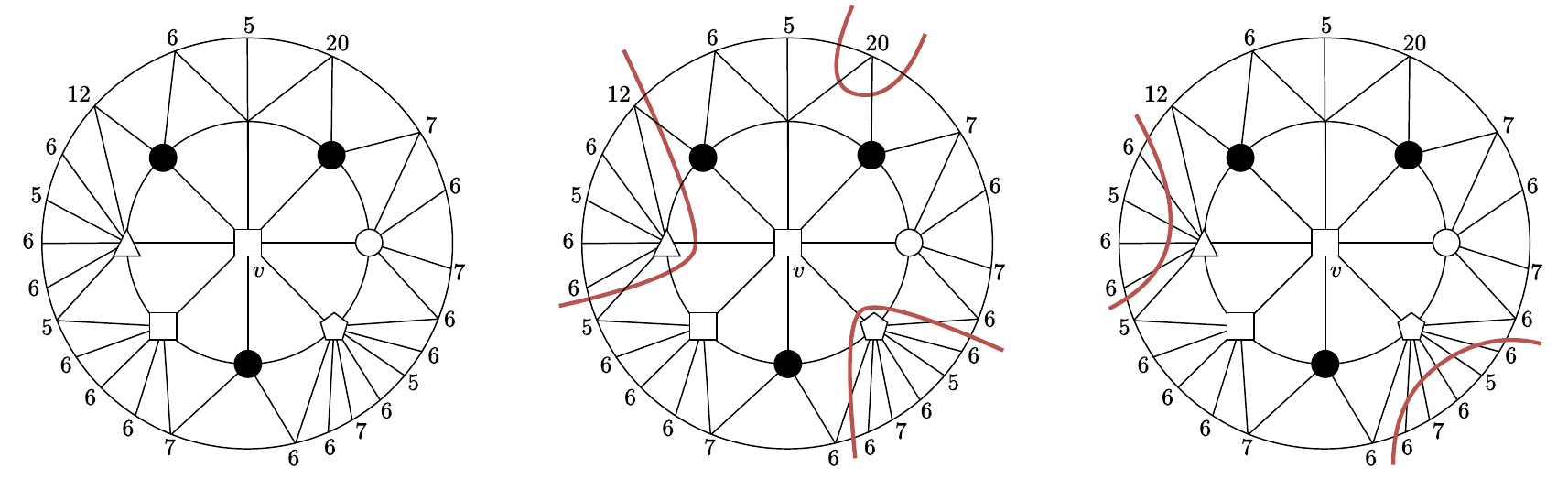}}
  \rule{21ex}{0ex}$B_2(v)$\rule{23ex}{0ex}$B^8_2(v)$\rule{23ex}{0ex}$\bar B^8_2(v)$
  \caption{Cartwheels. 
  Some vertices in a cartwheel could be identified in the triangulation, but then we would have separating 3- or 4-cycles.}
    \label{fig:cartwheels}
\end{figure}
If $d(v_0)$ and $k$ are bounded above by a constant, then this leads to local neighborhoods of constant size.

\begin{lem}\label{lem:ball-size}\showlabel{lem:ball-size} 
For every $k\ge0$, we have $|\bar B^8_k(v_0)| < 5^k\, d(v_0)$.  
\end{lem}

\begin{proof}
  For $j\ge1$, let $S_j = V(B_j^8)\setminus V(B_{j-1}^8)$ and $\bar S_j = V(\bar B_j^8)\setminus V(\bar B_{j-1}^8)$, where all sets are considered around $v_0$. We claim that for every $j\ge2$, $|S_j|\le |\bar S_j| \le 5|S_{j-1}|$. For each $v\in \bar S_j$, pick one of its neighbors in $S_{j-1}$, say $u$. Note that $u$ has a neighbor $w\in S_{j-2}$ and has two more neighbors $u_1,u_2$ that are in the two triangles containing the edge $uw$. Clearly, $u_1,u_2$ belong to $\bar B_{j-1}$ and hence are not in $\bar S_j$. Consequently, since $d(u)\le8$, $u$ has at most 5 neighbors in $\bar S_j$. This easily implies that $|\bar S_j| \le 5|S_{j-1}|$ and hence $|\bar S_j| \le 5^{j-1} d(v_0)$.
  
  Let us now consider the cardinality of $V(\bar B_k^8)$. The claim of the lemma is easy for $k\le1$, so we may assume that $k\ge2$ and proceed with induction. Since $V(\bar B_k^8) = \bar S_k \cup V(\bar B_{k-1}^8)$, the induction hypothesis and the claim derived above give that $|\bar B_k^8| \le |\bar S_k| + |\bar B_{k-1}^8| \le 5^{k-1}d(v_0) + 5^{k-1}d(v_0) < 5^kd(v_0)$.
\end{proof}


By inspection of the discharging rules in Figure \ref{fig:rules}, we will prove:

\begin{lem}\label{limitcartwheel}\showlabel{limitcartwheel}
Let $v$ be a vertex in $G$ and $vu$ any incident edge. The amount of charge, $\phi(v,u)$ and $\phi(u, v)$, sent along this edge is completely determined by the graph structure of the extended degree-bounded ball $\bar B_2^8(v)$. Since all discharging over edges incident to $v$ is determined by $\bar B_2^8(v)$, it follows that $\bar B_2^8(v)$ also determines the final charge $T(v)$.
\end{lem}

Our discharging rules thus act within the extended degree-bounded ball $\bar B_2^8(v)$, and many of the rules are affected by high-degree boundary vertices, so these are important for the final charge of $v$.

However, if we wanted local neighborhoods to have disjoint extended degree-bounded balls, we would face the issue that a linear number of them could intersect at a single high-degree vertex.
We therefore want to claim that it suffices that the non-extended degree-bounded balls are non-touching.

\subsection{Main technical result}

With the above understanding of the local neighborhood, we can now state the generalization of
Theorems \ref{thm:T(v)> 0} and \ref{thm:basic-flat}, which we want to prove in the rest of the paper:

\begin{thm}\label{thm:main-theorem}\showlabel{thm:main-theorem}
Let $G$ be a triangulation and $v\in V(G)$. 
\begin{itemize}
    \item[(i)] If $v$ has positive final charge, $T(v) > 0$ then $B^8_2(v)$ contains either a separating 3- or 4-cycle, or a reducible configuration in ${\mathcal D}$.\footnote{In fact,
    the reducible configuration is either a vertex of degree at most 4 or it is in one of Steinberger's \cite{steinberger2010unavoidable}  $\mathcal{D}_0$.}.
    \item[(ii)] If $d(v)\ge9$ and $T(v) = 0$ then $B^8_2(v)$ contains either a separating 3- or 4-cycle, or  a reducible configuration in $\mathcal{D}$.
    \item[(iii)] If $d(v)\le8$, and all vertices in 
    $B_{12}(v)$
    have degree at most 8 and have final charge $0$, then $B_{12}(v)$ contains either a separating 3- or 4-cycle, or  a  reducible configuration in $\mathcal{D}$. 
    \end{itemize}
\end{thm}
Note that even though Theorem \ref{thm:main-theorem} only considers degree-bounded balls, it also talks about final charges, and to determine the final charge of a vertex $u$, then we actually need to consider the extended degree-bounded ball $\bar B_2^8(u)$.

Recall that vertices of degree less than 5 are reducible configurations in $\cD$, so if any of these are present (in $B^8_2(v)$ in (i) or (ii) or in $B_{12}(v)$ in (iii)), then there is nothing to prove above. 

Theorem \ref{thm:main-theorem} (i) together with Lemma \ref{lem:induced-config} is a strengthening of the result of Steinberger \cite{steinberger2010unavoidable} who assumes that the graph is internally 6-connected, meaning no obstructing cycles. In addition, the reducible configuration in Theorem \ref{thm:main-theorem} (i) is in $B^8_2(v)$ while Steinberger \cite{steinberger2010unavoidable} only placed it in $B_2(v)$.

Theorem \ref{thm:main-theorem} (ii) and (iii) are two of the main new contributions of this paper, with (iii) being by far the hardest. 
We claim they imply Theorem \ref{thm:basic-flat}, which assumes no obstructing cycles, hence by Lemma \ref{lem:induced-config}, that any reducible configuration is induced. To see this, consider a vertex $v$ such that all vertices in $B_{12}(v)$
have final charge zero. If $v$ has degree at least 9, then the statement follows immediately from Theorem \ref{thm:main-theorem} (ii). Suppose instead that $v$ has degree at most $8$. If all vertices in $B_{12}(v)$ 
also have degree at most $8$, then Theorem \ref{thm:main-theorem} (iii) applies. Otherwise, we have some zero-charge vertex $v'$ in $B_{12}(v)$ 
having degree at least $9$, and then Theorem \ref{thm:main-theorem} (ii) applied at the vertex $v'$ is used. Note that the resulting configuration is contained in $B_{14}(v)$.

We note that for case (iii), the number of possible configurations for $\bar B^8_{12}$ is enormous.
We have at most $8\cdot 5^{12}$ vertices, including up to $8\cdot 5^{11}$ boundary vertices, and for each of these we must choose a degree between 5 and 9, for a total of $5^{11\cdot 5^{12}}$ options. This is far beyond what a computer can handle, so our proof requires a careful combination of human reasoning and computational power.

\subsection{Finding a linear number of reducible configurations}\label{sec:accounting}
We want to find a linear number of either separating 3- or 4-cycles or reducible configurations from $\cD$, so as to prove
Theorem \ref{thm:main-linear} assuming all the results claimed above.

\paragraph{Accounting.}
We now show that Theorem \ref{thm:main-theorem} and Lemma \ref{maxplus} imply a
linear number of non-touching either 3- or 4-cycles, or reducible configurations, as stated in Theorem \ref{thm:main-linear}. More precisely, we want to find a linear number of vertices satisfying the conditions of Theorem \ref{thm:main-theorem} and such that the mentioned 
(non-extended) degree-bounded balls are non-touching; for this implies that the separating 3- or 4-cycle, or reducible configurations (pick one from each ball) are non-touching.

First, we consider the case where we have a linear
number of vertices $v$ satisfying 
the final-charge condition in (i) or (ii) of Theorem~\ref{thm:main-theorem}, that is, either $v$ has positive final charge or $v$ has final charge zero and degree at least $9$. 

Let $U$ be the set of vertices satisfying the final-charge condition in
(i) or (ii). Iteratively, we will construct a subset $U^*$ of vertices $v$ with non-touching (non-extended) $B_2^8(v)$.  Note here that
$B_2^8(v)$ and $B_2^8(w)$ are non-touching if $w$ is not in the extended degree-bounded radius-5 ball $\bar B_5^8(v)$.
The construction of $U^*$ is
as follows: If the current 
$U$ is non-empty, we move an arbitrary vertex $v$ from $U$ to $U^*$ and remove all vertices in $\bar B^8_5(v)\cap U$ from $U$. 

By Lemma \ref{maxplus}, the maximum degree $d(u)$ among all vertices $u\in U$ is at most 30, and this implies that the maximum number of vertices in $\bar B^8_5(v)\cap U$ is at most $2\cdot 30\cdot 5^5=O(1)$, so if $U$ is of linear size, then so is $U^*$. 

We may thus assume that there is only a sublinear number of vertices with positive final charge or with final charge zero and degree greater than 8. We still let $U$ denote the set of these vertices, so $|U|=o(n)$. We claim that the
set $W$ of vertices satisfying the final-charge condition in
(iii) in Theorem \ref{thm:main-theorem} must be linear in size.

For our analysis, it is useful to define
$V^- = \{v\in V(G)\mid T(v)<0\}$ and $V^+ = \{v\in V(G)\mid T(v)>0\}$. We note that
$U=V^+\cup \{v\in V(G)\mid T(v)=0 \textrm{ and } d(v)\ge9\}$.
The set $W$ is the set of vertices $w$ that satisfy 
the final-charge condition in (iii) of Theorem \ref{thm:main-theorem}, that is, $W$ is the set of vertices such that all vertices in $B_{12}(w)$
have degree at most 8 and final charge 0.
This means that $W$ must be the set of vertices that
are not in $B^8_{12}(v)$ for any $v\in V^-\cup U$.
By Lemma \ref{lem:ball-size}, $|B^8_{12}(v)| < d(v)\cdot 5^{12}=O(d(v))$, so we conclude
that $|W|=n-O(\sum_{v\in V^-\cup V^+} d(v))$. 

By Lemma \ref{maxplus}, the maximum degree in $V^+$ is at most 30, so trivially $\sum_{v\in V^+} d(v)=O(|V^+|)$. The challenge is to limit the total degree of the negatively charged vertices.

\begin{lem}
    $\sum_{v\in V^-} d(v)=O(|V^+|)$.
\end{lem}

\begin{proof}
By Lemma \ref{maxplus}, the maximum charge is at most $54$, so 
$\sum_{v\in V^+} T(v)\leq 54|V^+|$. By Lemma \ref{lem:120}, the total charge is 120, so the positive charges dominate the negative charges, that is, $-\sum_{v\in V^-}T(v)<\sum_{v\in V^+} T(v)$. 
Since $T(v) \le -1$ for $v \in V^-$, we know $-T(v) \ge 1$. For $d(v) \le 30$, we trivially have $d(v) \le 30 \le -31 T(v)$.
Finally, by Lemma \ref{maxplus} (with $d(v) \ge 31$), we have
$T(v)\leq 60-2d(v)$ for all vertices. Thus, if $v$ has a negative charge, then this implies that $T(v)\leq -2d(v)/31$. We conclude that
\[
   \sum_{v\in V^{-}}d(v)\le-\sum_{v\in V^{-}} 31 T(v) < 31 \sum_{v\in V^{+}}T(v)\le 31\cdot54\cdot|V^{+}|.\]
\end{proof}

The above calculation shows that $\sum_{v\in V^-\cup V^+} d(v) = O(|V^+|)$. 
Since $V^+\subseteq U$ and $|U|=o(n)$, it follows that $|W|=n-o(n)$. Now, for each $v\in W$, since it satisfies the final-charge condition in 
Theorem \ref{thm:main-theorem} (iii), we know that all vertices in $B_{12}(v)$
have degree at most $8$. 

Like we did for $U$, we now want to identify a subset $W^*\subseteq W$ such that the non-extended $B_{12}^8(v)$, for $v \in W^*$, are non-touching. 
This is done as follows. While $W$ is
non-empty, we move an arbitrary vertex $w$ from $W$ to $W^*$ and remove $\bar B^8_{25}(w)$
from $W$.
Each time, by Lemma \ref{lem:ball-size}, we remove at most $8\cdot 5^{25}$ vertices from $W$,
so we end up with $|W^*|=\Omega(|W|)=\Omega(n)$, and the vertices $w$ in $W^*$ have non-touching $B^8_{12}(v)$.

\paragraph{Linear time implementation.}
First we note that we can implement all the discharging in linear time. This follows because we only have a constant number of discharging rules in $\cR$ from Figure \ref{fig:rules}.

Next we have to do the above cleanings, reducing $U$ to $U^*$ or $W$ to $W^*$. For each vertex added, we just remove the constantly many vertices in $\bar B_5^8(v)$ or $\bar B_{25}^8(w)$.

Finally, we have to search
the degree-bounded balls
$\bar B_2^8(u)$, $u\in U^*$ or $\bar B_{12}^8(w)$, $w\in W^*$, for separating 3- or 4-cycles or reducible configurations from $\cD$. This takes constant time for each ball since it is of constant size and the set $\cD$ is finite.

This completes the proof of Theorem \ref{thm:main-linear} assuming our Theorem \ref{thm:main-theorem} and Lemma \ref{maxplus}, both of which are yet to be proved.

\paragraph{Organization of the rest of the paper.}
In Sections \ref{sect:edge-charge}--\ref{sect:lowdegrees}, we prove our main technical result, 
Theorem \ref{thm:main-theorem}.
We will make it clear that some parts of the proof require a computer to check a finite number of cases. In the appendix, we provide details of the pseudocode, so it can be verified that it does the claimed test. 

In Sections \ref{sec:combine}--\ref{sect:free-cartwheel}, we discuss how to combine intersecting near-triangulations and corresponding configurations in order to make effective computational tasks when dealing with multiple configurations containing the same edge in the unavoidability part of the proof. Finally, Sections \ref{sect:D-reducible} and \ref{sect:obstructing-cycles} give details about derandomization of our main computational tasks leading to the near-linear 4-coloring algorithm.

\section{Discharging}\label{sect:edge-charge}\showlabel{sect:edge-charge}

We will now consider discharging over edges and their impact on the final charge. First, we consider an arbitrary triangulation with no restrictions on obstructing cycles or reducible configurations.

\tikzset{deg5/.style={thick, circle, draw, fill=black, inner sep=1.5pt,}}
\tikzset{deg6/.style={thick, circle, draw, fill=black, inner sep=0pt,}}
\tikzset{deg7/.style={thick, circle, draw, fill=white, inner sep=2pt,}}
\tikzset{deg8/.style={thick, rectangle, draw, fill=white, inner sep=2pt,}}
\tikzset{deg9/.style={thick, regular polygon, regular polygon sides=3, rotate=180, draw, fill=white, inner sep=1pt,}}
\tikzset{deg10/.style={thick, regular polygon, draw, fill=white, inner sep=2pt,}}
\tikzset{->-/.style={decoration={
    markings,
    mark=at position .6 with {\arrow{>}}}, postaction={decorate}}}
\tikzset{->>-/.style={decoration={
    markings, 
    mark=at position .5 with {\arrow{>}};, 
    mark=at position .6 with {\arrow{>}};}, postaction={decorate}}}

\begin{figure}[htbp]
    \centering
    \includegraphics[width=2.4cm]{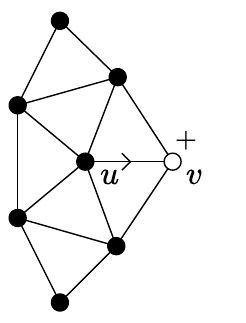}
\caption{A vertex $u$ sends charge 8 to $v$ using the first four rules, each with two orientations except the first one which by itself sends charge 2. The Birkhoff diamond configuration is easily spotted as reducible configuration in this case.}
\label{fig:any_send8}
\end{figure}


\begin{lem}
\label{lem:vsends-noconf}
\showlabel{lem:vsends-noconf}
Consider an arbitrary triangulation $G$ and let $uv$ be any oriented edge. Then
$\phi(u,v) \leq 8$. Moreover, if $\phi(u,v)=8$, then $\bar B_2^8(u) \cap \bar B_2^8(v)$ contains the graph in Figure \ref{fig:any_send8}, where $u$ and $v$ are the marked vertices.
\end{lem}

\begin{proof}[Sketch of computer-assisted proof] 
The proof uses a computer to check how the discharging rules from 
Figure \ref{fig:rules} can be combined so as to send maximal charge from $u$ to $v$. For an upper-bound, we only need to consider
rules sending charge from $u$ to $v$ (ignoring that there could also be charge sent from $v$ to $u$). 
More details on how we can combine the rules will be provided in Section \ref{sec:combine}. 
The maximum discharge is 8 and this can only be achieved as in Figure \ref{fig:any_send8} which combines the first 4 rules: the first rule gives discharge 2. The others give discharge 1 but are applied in two  symmetrical ways.
\end{proof}

We can now prove the statement of Lemma \ref{maxplus}. \smallskip \\
{\bf Lemma \ref{maxplus}} {\em
For any triangulation $G$ and $v \in V(G)$, we have $T(v) \leq 60 - 2 \cdot d_G(v)$. 
Consequently, $T(v) \leq 54$ and if $d_G(v) \geq 31$, then $T(v) \le -\tfrac{2}{31} d_G(v)$.}

\begin{proof}
Let $u_i$ $(0 \leq i < d_G(v))$ be the neighbors of $v$ listed in the clockwise order of the embedding of a given triangulation $G$. By Lemma \ref{lem:vsends-noconf}, $\sum \phi(u_i,v) \leq 8 \cdot d_G(v)$.
As we start with $T_0(v) = 10(6 - d_G(v))$, 
it follows that $T(v) \leq 10 \cdot (6-d_G(v))+8 \cdot d_G(v) = 60-2d_G(v) \leq 54$ since $d_G(v) \geq 3$ for all $v \in V(G)$. It is also easy to conclude that if $d_G(v) \geq 31$, then $T(v) \le -\tfrac{2}{31} d_G(v)$.
\end{proof}

Next, we prove the statement of Lemma \ref{limitcartwheel}. \smallskip \\
{\bf Lemma \ref{limitcartwheel}}
{\em Let $v$ be a vertex in $G$ and $vu$ any incident edge. The amount of charge $\phi(v,u)$ and $\phi(u,v)$ sent along this edge is completely determined by the graph structure in the extended degree-bounded ball $\bar B_2^8(v)$. Since all discharging rules over edges incident to $v$ are determined, this also determines the final charge $T(v)$.}

\begin{proof} 
We consider rules sending
charge from $u$ to $v$. Inspecting every rule from Figure \ref{fig:rules}, we see
that every vertex can be reached by a path in the near-triangulations of length at most two from $u$ ($v$, resp.) and where the internal vertex, if any, is of degree at most 8. Those rules are thus
contained in $\bar B_2^8(v)\cap\bar B_2^8(u)$.  
\end{proof}

We can obtain a stronger bound on the discharge over an edge if we rule out
separating 3- and 4-cycles and reducible configurations.

\begin{lem}\label{lem:vsends}\showlabel{lem:vsends}
 Let $uv$ be an oriented edge in $G$. 
Suppose that $B^8_2(v)$ has no separating 3- or 4-cycle and
no reducible  configurations from $\mathcal{D}$. Then $\phi(u,v)\leq 5$.
Moreover, if $\phi(u,v)=5$ then $\bar B^8_2(v)$
contains one of the graphs shown in Figure \ref{fig:send5} (where $uv$ is marked with an arrow pointing from $u$ toward $v$).
\end{lem}

\tikzset{deg5/.style={thick, circle, draw, fill=black, inner sep=1.5pt,}}
\tikzset{deg6/.style={thick, circle, draw, fill=black, inner sep=0pt,}}
\tikzset{deg7/.style={thick, circle, draw, fill=white, inner sep=2pt,}}
\tikzset{deg8/.style={thick, rectangle, draw, fill=white, inner sep=2pt,}}
\tikzset{deg9/.style={thick, regular polygon, regular polygon sides=3, rotate=180, draw, fill=white, inner sep=1pt,}}
\tikzset{deg10/.style={thick, regular polygon, draw, fill=white, inner sep=2pt,}}
\tikzset{->-/.style={decoration={
    markings,
    mark=at position .6 with {\arrow{>}}}, postaction={decorate}}}
\tikzset{->>-/.style={decoration={
    markings, 
    mark=at position .5 with {\arrow{>}};, 
    mark=at position .6 with {\arrow{>}};}, postaction={decorate}}}

\begin{figure}[htbp]
\begin{tabular}{ccccc}

\begin{minipage}[t]{0.16\hsize}
\centering
\begin{tikzpicture} []
    \node [deg5] at (1.109, 1.251) (v0) {};
    \node [deg7] at (2.109, 1.251) (v1) {};
    \node [above = 0.15 cm of v1, anchor=center] (v1+) { $+$ };
    \node [deg5] at (1.418, 2.202) (v2) {};
    \node [deg5] at (0.3, 1.839) (v3) {};
    \node [deg5] at (1.418, 0.3) (v4) {};
    \draw [->-] (v0) -- (v1);
    \foreach \u / \v in {v0/v1, v0/v2, v0/v3, v0/v4, v1/v2, v1/v4, v2/v3}
        \draw (\u) -- (\v);
    \foreach \u / \v in {}
        \draw[ultra thick, blue] (\u) -- (\v);
\end{tikzpicture}
\end{minipage}
&
\begin{minipage}[t]{0.16\hsize}
\centering
\begin{tikzpicture} []
    \node [deg6] at (0.8, 1.166) (v0) {};
    \node [deg7] at (1.8, 1.166) (v1) {};
    \node [above = 0.15 cm of v1, anchor=center] (v1+) { $+$ };
    \node [deg5] at (1.3, 2.032) (v2) {};
    \node [deg5] at (0.3, 2.032) (v3) {};
    \node [deg5] at (0.3, 0.3) (v4) {};
    \node [deg5] at (1.3, 0.3) (v5) {};
    \node [deg5] at (1.126, 3.017) (v6) {};
    \draw [->-] (v0) -- (v1);
    \foreach \u / \v in {v0/v1, v0/v2, v0/v3, v0/v4, v0/v5, v1/v2, v1/v5, v2/v3, v2/v6, v3/v6, v4/v5}
        \draw (\u) -- (\v);
    \foreach \u / \v in {}
        \draw[ultra thick, blue] (\u) -- (\v);
\end{tikzpicture}
\end{minipage}
&
\begin{minipage}[t]{0.16\hsize}
\centering
\begin{tikzpicture} []
    \node [deg6] at (0.8, 1.166) (v0) {};
    \node [deg7] at (1.8, 1.166) (v1) {};
    \node [above = 0.15 cm of v1, anchor=center] (v1+) { $+$ };
    \node [deg5] at (1.3, 0.3) (v2) {};
    \node [deg5] at (0.3, 0.3) (v3) {};
    \node [deg6] at (0.3, 2.032) (v4) {};
    \node [deg5] at (1.3, 2.032) (v5) {};
    \node [deg5] at (1.126, 3.017) (v6) {};
    \node [deg5] at (2.24, 2.374) (v7) {};
    \draw [->-] (v0) -- (v1);
    \foreach \u / \v in {v0/v1, v0/v2, v0/v3, v0/v4, v0/v5, v1/v2, v1/v5, v1/v7, v2/v3, v4/v5, v4/v6, v5/v6, v5/v7, v6/v7}
        \draw (\u) -- (\v);
    \foreach \u / \v in {}
        \draw[ultra thick, blue] (\u) -- (\v);
\end{tikzpicture}
\end{minipage}
&
\begin{minipage}[t]{0.16\hsize}
\centering
\begin{tikzpicture} []
    \node [deg7] at (1.201, 1.827) (v0) {};
    \node [deg7] at (2.201, 1.827) (v1) {};
    \node [deg5] at (1.824, 2.609) (v2) {};
    \node [deg5] at (0.978, 2.802) (v3) {};
    \node [deg5] at (0.3, 2.261) (v4) {};
    \node [deg5] at (0.978, 0.852) (v5) {};
    \node [deg7] at (1.824, 1.045) (v6) {};
    \node [deg6] at (2.67, 2.802) (v7) {};
    \node [deg7] at (1.824, 3.476) (v8) {};
    \node [above = 0.15 cm of v8, anchor=center] (v8+) { $+$ };
    \node [deg5] at (1.379, 0.3) (v9) {};
    \node [deg6] at (2.65, 1.313) (v10) {};
    \node [above = 0.15 cm of v10, anchor=center] (v10+) { $+$ };
    \draw [->-] (v0) -- (v1);
    \foreach \u / \v in {v0/v1, v0/v2, v0/v3, v0/v4, v0/v5, v0/v6, v1/v2, v1/v6, v1/v7, v1/v10, v2/v3, v2/v7, v2/v8, v3/v4, v3/v8, v5/v6, v5/v9, v6/v9, v6/v10, v7/v8}
        \draw (\u) -- (\v);
    \foreach \u / \v in {}
        \draw[ultra thick, blue] (\u) -- (\v);
\end{tikzpicture}
\end{minipage}
&
\begin{minipage}[t]{0.16\hsize}
\centering
\begin{tikzpicture} []
    \node [deg7] at (1.201, 1.275) (v0) {};
    \node [deg8] at (2.201, 1.275) (v1) {};
    \node [above = 0.15 cm of v1, anchor=center] (v1+) { $+$ };
    \node [deg5] at (1.824, 0.493) (v2) {};
    \node [deg5] at (0.978, 0.3) (v3) {};
    \node [deg5] at (0.3, 1.709) (v4) {};
    \node [deg6] at (0.978, 2.25) (v5) {};
    \node [deg5] at (1.824, 2.057) (v6) {};
    \node [deg5] at (1.824, 2.925) (v7) {};
    \node [deg5] at (2.67, 2.25) (v8) {};
    \draw [->-] (v0) -- (v1);
    \foreach \u / \v in {v0/v1, v0/v2, v0/v3, v0/v4, v0/v5, v0/v6, v1/v2, v1/v6, v1/v8, v2/v3, v4/v5, v5/v6, v5/v7, v6/v7, v6/v8, v7/v8}
        \draw (\u) -- (\v);
    \foreach \u / \v in {}
        \draw[ultra thick, blue] (\u) -- (\v);
\end{tikzpicture}
\end{minipage}
\\
\end{tabular}
\caption{The cases where a vertex sends charge 5 with no separating 4- or 5-cycles and no reducible configurations in $\bar B_2^8(u) \cap \bar B_2^8(v)$.}
\label{fig:send5}
\end{figure}
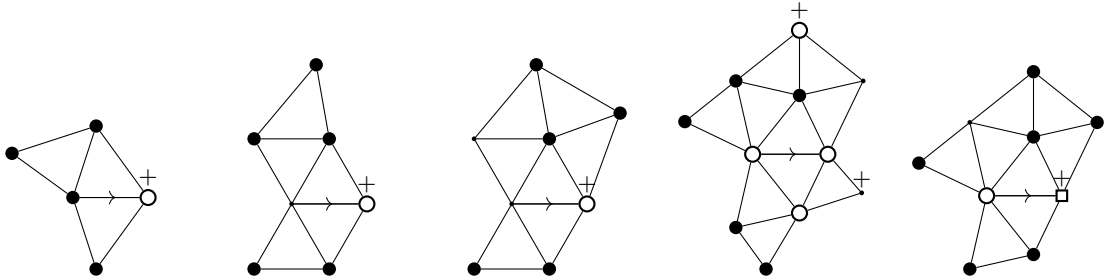

\begin{proof}[Sketch of computer-assisted proof]
First, we make the simple observation that every rule $R$ in $\mathcal{R}$ satisfies $\delta^+_R(s(R)) \leq 8$, so a vertex of degree at least nine sends no charge. Thus, we can assume $d(u)\leq 8$ and $u\in B_2^8(v)$.

Applying Lemma \ref{limitcartwheel} to both $u$ and $v$,
we get that the charge sent over $uv$ in both directions is determined by $\bar B^8_2(v)\cap \bar B^8_2(u)$.  

Now, as in the proof of Lemma \ref{lem:vsends-noconf}, we get a computer to check how we can combine the discharging rules from Figure \ref{fig:rules} using a simple branching program. 
However, this time we give up a combination if it implies that $B_2^8(v)$ contains either a separating 3- or 4-cycle or a reducible configuration from $\mathcal D$. This is seen in Figure \ref{fig:any_send8} where the combined rules sending charge 8 contain the Birkhoff diamond, which is a reducible configuration in $\cD$.

The computer found that the maximum discharge is 5. The maximizing configurations with no reducible configurations from $\cD$ are collected in Figure \ref{fig:send5}.
\end{proof}

\begin{figure}
    \centering
    \includegraphics[width=0.17\linewidth]{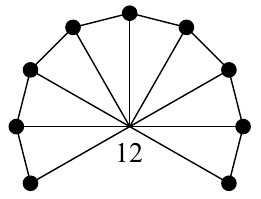}
    \caption{The D-reducible configuration from $\cD$ used in the proof of Lemma \ref{lem:12+}.}
    \label{fig:deg12-conf}
\end{figure}

Using our stronger discharging bound, we can now handle vertices of degree down to 12.
\begin{lem}
\label{lem:12+}
\showlabel{lem:12+}
Suppose that $v$ is a vertex in $G$ of degree at least 12 and with final charge $T(v) \geq 0$. Then $B^8_2(v)$ contains a separating 3- or 4-cycle or a reducible configuration from $\mathcal{D}$.
\end{lem}
\begin{proof}
    For the proof by contradiction, suppose that there are no separating 3- or 4-cycles and no reducible configurations as claimed.
    By Lemma \ref{lem:vsends}, 
    $\phi(u_i,v)\leq 5$ for every neighbor \(u_i\) of \(v\). Hence \(T(v)\le10(6-d_G(v))+5d_G(v)=60-5d_G(v)\).
    The only possibility is a vertex $v$ such that $d_G(v) = 12$ and $\phi(u_i,v) = 5$ for all its neighbors $u_i$ $(0 \leq i < 12)$, where $u_0,\dots,u_{11}$ are listed in clockwise order around $v$. From Figure \ref{fig:send5}, we immediately get that if $u_i$ sends charge $5$ to $v$, then one of its neighbors, say, $u_{i+1}$, has degree 5. But if $u_{i+1}$ has degree $5$ and sends charge 5 to $v$, then, again by  Figure \ref{fig:send5},  its neighbor $u_i$ has degree $5$. Thus we conclude that all neighbors $u_0,\ldots,u_{11}$ of $v$ have degree 5, but then we see that the configuration in $\cD$ shown in Figure \ref{fig:deg12-conf} is contained in $B_1^8(v)\subseteq B_2^8(v)$, a contradiction.
\end{proof}

\section{Reducible configurations around low-degree vertices} \label{sect:lowdegrees}

\paragraph{Cartwheel enumeration.}
Our basic use of computers is for the following \emph{cartwheel enumeration}: for a center vertex $v$ of prescribed bounded degree $d(v)$ and final charge $T(v)\geq 0$, assuming there are no separating 3- or 4-cycles, enumerate all possible configurations of $\bar B^8_2(v)$ where $v$ gets the prescribed final charge and where no reducible configuration from $\cD$ is found in $B_2^8(v)$. We note
here that by Lemma \ref{limitcartwheel}, $\bar B^8_2(v)$ determines the final charge of $v$.  For the center $v$, the maximal degree we consider here is 11 (for larger degrees we have Lemma \ref{lem:12+}), and for other vertices in $\bar B^8_2(v)$ it does not matter if the degree is 9 or larger, so it suffices to consider degrees up to 9. Thus we only have to consider a finite number of extended degree bounded cartwheels, but the number is large: even if
we exploit that the minimum degree is 5, the number of such cartwheels is bigger than $2^{100}$. For each of these cartwheels, we have more than 8200 reducible configurations in $\cD$ to look after, so the number of combinations is far too large for any computer to handle. We therefore have to design an analysis that identifies large groups of cartwheels that can be efficiently checked together. The
cartwheel enumeration algorithm is 
described in Section \ref{sect:free-cartwheel}.

For the classic 4CT proof, it suffices to consider 2-local neighborhoods (c.f. Theorem \ref{thm:T(v)> 0}), but here we need to consider much larger neighborhoods (up to 12-local in Theorem \ref{thm:main-theorem} (iii)), and to do so, we rely on human reasoning.

\begin{lem}
\label{lem:positive-comp}
\showlabel{lem:positive-comp}
    Let $v$ be a vertex in $G$ such that there are no separating 3- or 4-cycles in $\bar B_2^8(v)$. Suppose that one of the following holds:
    \begin{enumerate}
        \item $9 \leq d_G(v) \leq 11$ and $T(v) \geq 0$,
        \item $7 \leq d_G(v) \leq 8$ and $T(v) > 0$, or
        \item $7 \leq d_G(v) \leq 8$ and $T(v) = 0$ and all neighbors of $v$ have degree at most $6$.
    \end{enumerate}
    Then the degree-bounded ball $B_2^8(v)$
    contains a D-reducible configuration from $\mathcal{D}$. For the cases where $T(v)>0$, we
    only need configurations from $\cD_0$.
\end{lem}

\begin{proof}[Sketch of computer-assisted proof] The proof just does the above cartwheel enumeration to check that there are no counterexamples.
\end{proof}

We are now ready to prove the first two cases of Theorem \ref{thm:main-theorem}.

\paragraph{Theorem \ref{thm:main-theorem}.}
\emph{Let $G$ be a triangulation and $v\in V(G)$. 
\begin{itemize}
    \item[(i)] If $v$ has positive final charge, $T(v) > 0$, and ${\bar B}^8_2(v)$ contains no  separating 3- or 4-cycles, then $B^8_2(v)$ contains a reducible configuration in $\mathcal{D}_0\subset {\mathcal D}$. 
    \item[(ii)] If $d(v)\ge9$ and $T(v) = 0$, and ${\bar B}^8_2(v)$ contains no separating 3- or 4-cycles, then $B^8_2(v)$ contains a reducible configuration in $\mathcal{D}$.
\end{itemize}
}
\begin{proof}
   First, we note that (ii) follows directly from
   Lemma \ref{lem:12+} and Case 1 of Lemma \ref{lem:positive-comp}.

   To get (i), we also need
   Case 2 of Lemma \ref{lem:positive-comp}, but we are still missing the case where $v$ has degree at most $6$ and positive final charge. However, this cannot happen, for if $v$ has degree 5 or 6, then it has final charge 0 as proved in \cite{steinberger2010unavoidable}, where exactly the same discharging rules as here are used. 
\end{proof}

We now begin examining vertices, whose final charge is 0, starting from a simple low-degree case that we can handle by hand.
 
\begin{lem}
\label{lem:positive-hand}
\showlabel{lem:positive-hand}
Suppose that $v$ is a vertex in $G$ such that all vertices at distance at most 2 from $v$ are of degree at most 6. 
Then $B_2(v)$ has a separating 3- or 4-cycle or a reducible configuration from $\mathcal{D}$.
\end{lem}

\begin{proof}
   A vertex of degree at most $4$ is by itself a D-reducible configuration in $\mathcal D$, so we can assume that all vertices in $B_2(v)$ are of degree five and six. Then 
    a configuration in Figure \ref{fig:confs-positive-hand} is contained in the degree-bounded ball $B_2^8(v)$, except if $v$ has degree five and is surrounded by neighbors of degree exactly six. If this is the case, then any one of the neighbors of $v$ and all its neighbors constitutes one of the configurations in Figure \ref{fig:confs-positive-hand}, which is in $\cD$.
\end{proof} 

\begin{figure}
    \centering
    \includegraphics[width=0.54\linewidth]{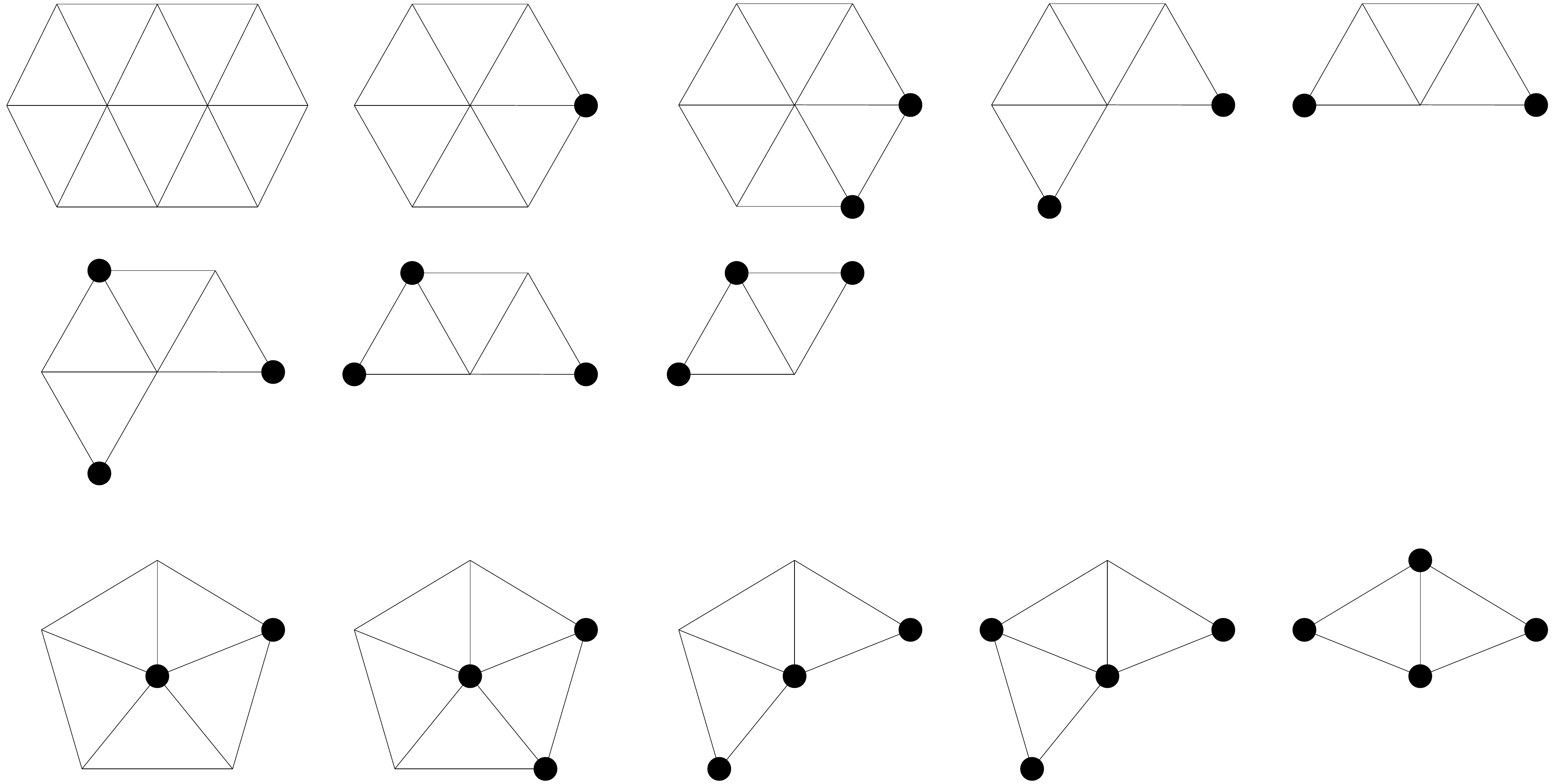}
    \caption{D-reducible configurations from $\cD$ used in the proof of Lemma \ref{lem:positive-hand}.}
    \label{fig:confs-positive-hand}
\end{figure}

The goal of the rest of this section is to prove Theorem \ref{thm:main-theorem} (iii), which assumes that all vertices in $B_{12}(v)$ have degree at most 8 and that $B_{12}(v)$ contains no separating 3- or 4-cycles. This implies that for each vertex $u \in B_{10}(v)$, we have $B_2^8(u)=\bar B_2^8(u)=B_2(u)$.

When dealing with a vertex $v$ of degree 7 or 8 and final charge 0, we will not always be able to find separating 3- or 4-cycles or reducible configurations in its cartwheel $B_2(v)$. Instead we have to consider combinations of neighboring cartwheels as illustrated in Figure \ref{fig:deg8-combine-example}. 

\begin{figure}[htbp]
    \centering
    \includegraphics[width=0.68\linewidth]{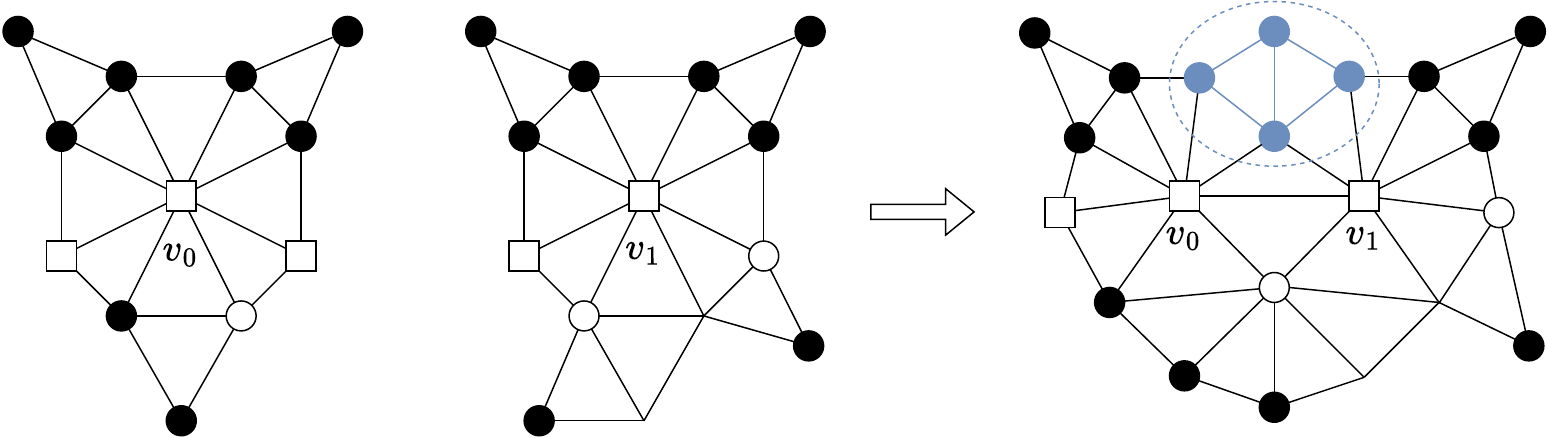}
    \caption{An example combining two cartwheels, neither of which contains a reducible configuration, but where the combination contains one; namely Birkhoff's diamond.}
    \label{fig:deg8-combine-example} 
\end{figure}

\paragraph{A vertex of degree 8.}
We now require some key technical lemmas, whose proofs are again computer-assisted. The first lemma applies whenever we have a vertex $v$ of degree 8, and all vertices in $B_1(v)$ have final charge 0 and all vertices in $B_3(v)$ have degree at most 8. 

\begin{lem}
\label{lem:zero-deg7,8}
\showlabel{lem:zero-deg7,8}  
Let $v$ be a degree-8 vertex and suppose that all vertices
in $B_3(v)$ have degree at most 8.
Let $V_0$ be a vertex-set of $G$
consisting of $v$ together with one or two neighbors as defined in one of the following cases:
\begin{itemize} 
    \item[(i)]  
    $V_0$ consists of two adjacent vertices, both of degree 8.   
\end{itemize}
In the two cases below, we assume that $v$ has no degree-8 neighbor. 
\begin{itemize}
    \item[(ii)] 
    $V_0$ consists of two adjacent vertices $v$ and $v'$, where $d_G(v)=8$, $d_G(v')=7$, and all other neighbors of $v$ have degree at most $6$.
    \item[(iii)] 
    $V_0$ consists of a path of length exactly two with vertex degrees $7, 8, 7$, where the two degree-7 neighbors of the degree-8 vertex $v$ are as close as possible in the successor order around $v$.
\end{itemize}
\noindent   
   Suppose that each vertex $u \in V_0$ has $T(u) = 0$.
   Then $G$ has a separating 3- or 4-cycle or a reducible configuration from $\cD$ in the union of the cartwheels $B = \bigcup_{u\in V_0}B_2(u)$, except for a single special case within (iii), where $v$ is the middle vertex of a 7-8-7 path constituting $V_0$ and $v$ is also the center in the configuration $X$ in $G$ depicted in Figure \ref{fig:exception-zero-concentrate}.
\end{lem}

\begin{figure}[htbp]
    \centering
\includegraphics[width=0.2\linewidth]{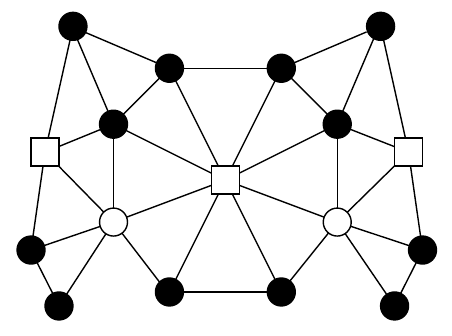}
    \caption{The exceptional configuration $X$ in Lemma \ref{lem:zero-deg7,8}.  This configuration is not reducible, but it is so special that it will help us in other ways.}
\label{fig:exception-zero-concentrate}
\end{figure}

\begin{proof}[Sketch of computer-assisted proof] 
We explain briefly how to proceed. 
First, we use our cartwheel enumeration for a vertex $u$ of prescribed degree $7$ or $8$ and final charge 0, and we use only vertices of degree at most 8. This yields all possible configurations of such cartwheels $B_2(u)$ with no separating 3- or 4-cycles or reducible configurations from $\cD$, and such that $u$ has final charge 0.

Now we systematically try to combine these
configurations into $\bigcup_{u\in V_0}B_2(u)$ for the different possible cases for $V_0$ described in the lemma. If the configurations for the different $B_2(u)$ are generated independently, then typically it will not be possible to combine them. However, 
those that correspond to the balls in a concrete triangulation $G$ must be combinable. Details on how we can combine configurations are found in Section \ref{sec:combine}.
Next, we check if any configuration in $\mathcal{D}$ is contained in the combined configuration.
Figure \ref{fig:deg8-combine-example} shows an example where we combine degree-bounded balls around two neighboring degree 8 vertices
and find a configuration from $\cD$. 

The computer did not always find a configuration from $\cD$. However, in all cases where this was not the case, there is the exceptional configuration $X$ shown in Figure \ref{fig:exception-zero-concentrate}. 
\end{proof}

Lemma \ref{lem:zero-deg7,8} has the following consequence. 
\begin{lem}
\label{lem:zero-8}
\showlabel{lem:zero-8}
    Let $v$ be a degree-8 vertex in $G$.
    If every vertex in $B_3(v)$ has final charge $0$ and all vertices in $B_5(v)$ have degree at most 8, then $B_5(v)$ contains a separating 3- or 4-cycle or a reducible configuration in $\cD$.
\end{lem}
\begin{proof}
If $B_2(v)$ does not contain $X$, we can apply Lemma \ref{lem:zero-deg7,8} to $v$ because every vertex in $B_1(v)$ has final charge $0$ and all vertices in $B_3(v)$ have degree at  most $8$.

Let $v'$ be one of the degree-$8$ vertices in $X$ at distance 2 from  $v$. We can apply Lemma \ref{lem:zero-deg7,8} to $v'$ because every vertex in $B_1(v') \subseteq B_3(v)$ has final charge 0, and all vertices in $B_3(v') \subseteq B_5(v)$ have degree at  most 8. Therefore, if $B_3(v')$ does not contain a separating 3- or 4-cycle or a reducible configuration from $\cD$,
then it contains $X$, that is, we have a copy of $X$ centered both at $v$ and at $v'$, and we call the later copy $X'$. Having a copy $X'$ of $X$ centered at $v'$, there has to be at least two degree-5 neighbors of $v'$ between degree-7 neighbors of $v'$, so we must have the degree-5 vertex $w$ added on one side of $X$, as shown in the middle of Figure
\ref{fig:explain-lem:zero-(i)}. This means that we have the D-reducible configuration depicted on the right side of Figure
\ref{fig:explain-lem:zero-(i)}.
\end{proof}

\begin{figure}[htbp]
    \centering
    \includegraphics[width=0.28\linewidth]{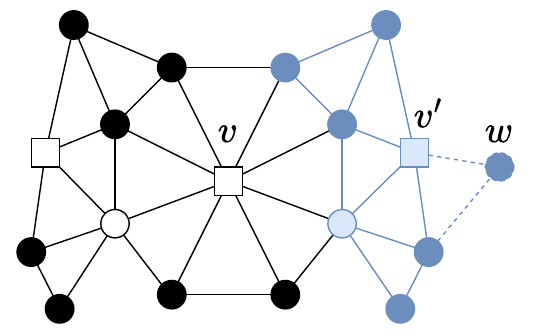}
    \qquad
    \includegraphics[width=0.12\linewidth]{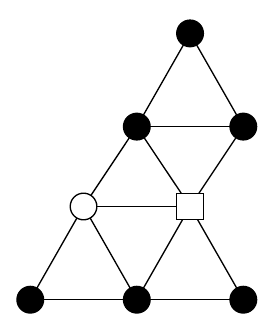}
    \caption{The exceptional configuration $X$ with an added external vertex of degree $5$ that is adjacent to a vertex of degree $8$ contains a blue-colored reducible configuration from $\mathcal{D}$, as illustrated in the right figure.}
    \label{fig:explain-lem:zero-(i)}
\end{figure}

\paragraph{Maximum degree 7.}
We will now focus on the cases where the maximum degree is 7.
We will have a special interest
in the configuration $\Tseven$ defined as the single facial triangle, all of whose vertices have degree 7. Using computers, we will prove the following two lemmas.

\begin{lem}
\label{lem:777}
\showlabel{lem:777}  
Suppose $G$ has a
$\Tseven$ with vertices $v_1v_2v_3$, all of which have final charge 0. Let $B=\bigcup_{i=1}^3 B_2(v_i)$. If all vertices in $B$ have degree at most 7, then $B$ has a separating 3- or 4-cycle or a reducible configuration in $\cD$ (this is where we need the configuration in Figure~\ref{fig:radius3}, which is not contained in any single $B_2(v_i)$).
\end{lem}

\begin{figure}
    \centering
    \includegraphics[width=0.78\linewidth]{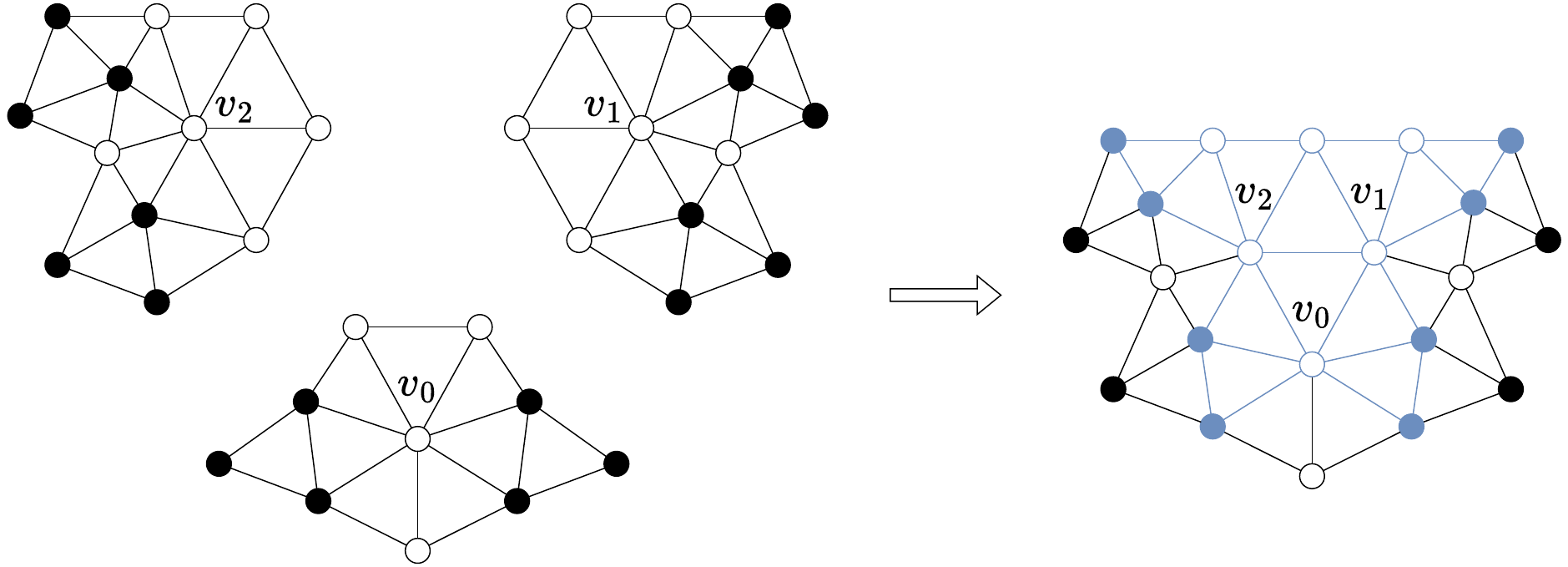}
    \caption{An example where combining three degree-bounded balls around $T_{7^3}$ needs our D-reducible configuration of radius 3.}
\label{fig:777combineRadius3}
\end{figure}

\begin{lem}
\label{lem:zero-deg7,7}
\showlabel{lem:zero-deg7,7}  
Let $v$ be a degree-7 vertex and suppose that all vertices
in $B_3(v)$ have degree at most 7.
Let $V_0$ be a vertex-set of $G$
consisting of a vertex $v$ of degree 7 together with one or two degree-7 neighbors as defined in one of the following cases:
\begin{itemize} 
    \item[(i)] 
    If $v$ has only one degree 7 neighbor $v'$ and all other neighbors have lower degrees, then $V_0=\{v,v'\}$.
    \item[(ii)] Otherwise, $V_0$ consists of three degree-7 vertices on a path of length exactly two. 
\end{itemize}
\noindent   
   Suppose that each vertex $u\in V_0$ has final charge $0$.
   Then $B=\bigcup_{u\in V_0}B_2(u)$
   contains a separating 3- or 4-cycle or a reducible configuration from $\cD$ or a $\Tseven$.
\end{lem}

\begin{proof}[Sketch of computer-assisted proof of Lemma \ref{lem:777} and \ref{lem:zero-deg7,7}] 
The process of checking this proof is similar to Lemma \ref{lem:zero-deg7,8} handling the degree-$8$ cases.
\end{proof}

Note that in one case checked in the proof of Lemma \ref{lem:777}, we find the radius-3 configuration from Figure \ref{fig:radius3}; see Figure \ref{fig:777combineRadius3}.

By using Lemmas \ref{lem:777} and \ref{lem:zero-deg7,7}, we can prove the following result. 
\begin{lem}
    \label{lem:zero-7}
    Let $v$ be a degree-7 vertex in $G$.
    If every vertex in $B_3(v)$ has final charge $0$ and all vertices in $B_5(v)$ have degrees of at most 7, then $B_5(v)$ contains a separating 3- or 4-cycle or a reducible configuration in $\cD$.
\end{lem}
\begin{proof}
    If $B_3(v)$ does not contain $\Tseven$, we can apply Lemma \ref{lem:zero-deg7,7} to find a separating 3- or 4-cycle or a reducible configuration from $\mathcal{D}$ because every vertex in $B_1(v)$ has the final charge $0$ and all vertices in $B_3(v)$ have degrees of at most $7$.

    Thus we may assume that $B_3(v)$ contains $\Tseven$ with vertices $v_1v_2v_3$.
    Then we can apply Lemma \ref{lem:777} to $v_1v_2v_3$ because every vertex in $v_1v_2v_3 \subseteq B_3(v)$ has final charge $0$ and the vertices in $B = \bigcup_{i=1}^3 B_2(v_i) \subseteq B_5(v)$ have degrees of at most 7.
\end{proof}

\drop{
\yuta{Old claims and proofs are below:}
\begin{lem}
\label{lem:zero}
\showlabel{lem:zero}
    
    Let $v$ be a vertex in $G$ with $d_G(v) \le 7$. 
    Let $k=5$ if $d_G(v) = 7$, and $k=7$ if $5 \leq d_G(v) \leq 6$.
    If every vertex $u \in V(B_k(v))$ satisfies $T(u) = 0$ and all vertices in $B_{k+2}(v)$ have degree at most 8, then $B_{k+2}(v)$ contains a separating 3- or 4-cycle or a reducible configuration in $\cD$.
\end{lem}

\begin{proof}
As in our earlier proofs, we may assume that $B_{k+2}(v)$ does not contain vertices of degree less than 5.

\yuta{erased proof of degree 8.}
    
Next, we suppose that $d_G(v)=7$ and $k=5$. By the part of this lemma for degree-8 vertices, we may assume that all neighbors of $v$ are of degree at most 7. Also, by Lemma \ref{lem:positive-comp}, we are done if all neighbors of $v$ are of degree less than 7, so $v$ has at least one neighbor of degree 7. If there is exactly one vertex of degree 7, we are done by Lemma \ref{lem:zero-deg7,8} (i), and if there are more than two vertices of degree 7, we are done by Lemma \ref{lem:zero-deg7,8} (ii) with the 7-7-7 path, which has no exceptional case.

Finally, suppose that $5\le d_G(v)\le6$ and $k=7$. If every vertex whose distance from $v$ is at most two is of degree five or six, the result follows from Lemma \ref{lem:positive-hand}. Otherwise, a vertex of degree seven or eight exists in the second neighborhood of $v$. By the previously proved cases of this lemma, we easily get the desired conclusion.
\end{proof}
}

Combining Lemma \ref{lem:positive-hand}, \ref{lem:zero-8}, and \ref{lem:zero-7} implies the remaining case of Theorem \ref{thm:main-theorem}.

\paragraph{Theorem \ref{thm:main-theorem}.}
\emph{Let $G$ be a triangulation and $v\in V(G)$. 
\begin{itemize}
    \item[(iii)] If $d(v)\le8$ and all vertices in 
    $B_{12}(v)$ 
    have degree at most 8, and have final charge $0$, then $B_{12}(v)$ contains either a separating 3- or 4-cycle, or
    a reducible configuration in $\mathcal{D}$. 
\end{itemize}
}
\begin{proof}
    If $d(v)=8$, we can directly apply Lemma \ref{lem:zero-8} to $v$.
    We consider the case that $d(v)=7$.
    If $B_5(v)$ contains no vertex of degree $8$, we can apply Lemma \ref{lem:zero-7} to $v$.
    Otherwise, let $v'$ be a vertex of degree $8$ in $B_5(v)$.
    We can apply Lemma \ref{lem:zero-8} to $v'$ because every vertex $u \in B_3(v') \subseteq B_8(v)$ has final charge 0 and all vertices in $B_5(v') \subseteq B_{10}(v)$ have degree at most 8.
    It remains to consider the case where $d(v)= 5$ or $6$.
    If all vertices in $B_2(v)$ have degree at most $6$, we can apply Lemma \ref{lem:positive-hand}.
    Otherwise, let $v'$ be a vertex of degree $7$ or $8$ in $B_2(v)$.
    From the discussion for the case where $d(v)=7$ or $8$, if every vertex $u \in B_8(v')$ has final charge $0$ and all vertices in $B_{10}(v')$ have degree at most 8, then the claim holds.
    This is exactly the case because $B_8(v') \subseteq B_{10}(v)$ and $B_{10}(v') \subseteq B_{12}(v)$.
\end{proof}

\section{Homomorphisms, free combinations, and the maximum discharge along an edge}
\label{sec:combine}

In this section, we will
discuss in more detail how we use computer to model and bound the discharge over an edge in an arbitrary triangulation as stated in Lemma \ref{lem:vsends-noconf}. We shall develop a formalism that shall be reused in other parts of the computer proof in the appendix.

In our analysis, we will often need to combine near-triangulations and configurations. 
One important example is when we want to check if several of our discharging rules in Figure \ref{fig:rules} could
possibly be combined so as to send charge along the same edge $uv$. We want this
to be equivalent to asking if the near-triangulations of these rules can be combined around $uv$, including restrictions on the degrees on the outer cycles. We will consider the near-triangulations as embedded with a given rotation (that is, clockwise orientation) around the vertices.  As is the case for our discharging rules, we assume that the near-triangulations
discussed below are all without cut-vertices.
(Conveniently, our analysis will never combine near-triangulations having a cut-vertex, even though they do appear in some of our reducible configurations.)
This implies that the embedding is unique.

At the end of this section, we want to prove Lemma \ref{lem:vsends-noconf} stating that the maximal discharge over any edge is $8$. The lemma discusses what can be done when considering how the discharging rules can be embedded and combined in a triangulation $G$, but, with a computer, we cannot consider all possible triangulations. Instead, we will consider the so-called ``free combinations'' that we can model on a computer with no reference to $G$, and such that if rules can be combined in $G$, then they can also be combined in their free combination, and we will say that the \emph{free combination succeeds}.
Therefore, restrictions on which discharging rules can be freely combined impose upper bounds on what can be combined in any triangulation.

Finally, we lower the upper bound 8 in Lemma \ref{lem:vsends-noconf} to 5 in Lemma \ref{lem:vsends} by limiting the combinations of discharging rules by allowing only those that do not give rise to separating 3- or 4-cycles or reducible configurations. Again, the computer will only look at free combinations.

\subsection{Homomorphic images of near-triangulations and configurations}
\label{sect:embedding}

For now, we assume that all configurations have a single specified degree $\delta(v)$ for each vertex. When we describe reducible configurations and configurations used in discharging rules, we often specify degree ranges on the boundary, but this is only a shorthand for concrete degrees, and we shall return to this issue in Section \ref{sect:ranges}. For simplicity, we assume that our near-triangulations and configurations have no cut-vertices. We shall return to cut-vertices in Section \ref{sect:config-in-combine}.

An \emph{orientation-preserving homomorphism} or just
\emph{homomorphism} of a near-triangulation $Z$ without cut-vertices into a triangulation $G$ is a map $\phi: V(Z) \to V(G)$ such that 
adjacent vertices in $Z$ are mapped to adjacent vertices in $G$, 
and the inner facial triangles in $Z$ are mapped to facial triangles of $G$ preserving their orientation. In other words, we require that $\phi$ preserves clockwise orientation around any vertex.
Thus, $\phi$ is defined on $V(Z)$, but it can be extended in an obvious way so that it is also defined on $E(Z)$.
If the near-triangulation $Z$ is in a configuration $(Z,\delta)$, then for any vertex $v$ in $Z$, we further require $d_G(\phi(v))=\delta(v)$. 
If $\phi$ is 1-1, then we call $\phi$ an \emph{embedding}. (Note that we require that embeddings preserve clockwise orientation.) If a configuration $(Z,\delta)$ is embedded in $G$, we also say that $(Z,\delta)$ is contained in $G$.

\begin{lem}\label{lem:locally 1-1}
Let $Z$ be a near-triangulation without cut-vertices. If 
$\phi$ is a homomorphism from $Z$ to $G$, then $\phi$ is locally 1-1 around each vertex, that is, for any vertex $v\in V(Z)$, the mapping restricted to $v$ and all its neighbors is 1-1. 
\end{lem}

\begin{proof}
Since $\phi$ preserves orientation of adjacent triangles around $v$ and $d_Z(v)\le \delta(v)=d_G(\phi(v))$, it cannot happen that two neighbors of $v$ would be mapped to the same neighbor of $\phi(v)$ in $G$. Therefore, $\phi$ is locally 1-1 around $v$.
\end{proof}

We now show that, if $Z$ is of small diameter and $G$ has no separating 3- or 4-cycle, all of whose vertices are in $\phi(Z)$, then every homomorphism from $Z$ to $G$ is an embedding.

\begin{lem}\label{lem:1-1} \showlabel{lem:1-1}
Let $Z$ be a near-triangulation 
of diameter at most $4$ and with no cut-vertex. If 
$\phi$ is a homomorphism from $Z$ to $G$, and $G$ has no separating 3- or 4-cycle
whose vertices are in $\phi(Z)$, then $\phi$ is 1-1, hence an embedding.
\end{lem}

\begin{proof}
  If $\phi$ is not 1-1, then $Z$ has two distinct vertices $x,y$ mapped to the same vertex by $\phi$. By assumption, $Z$ has a path $P$ of length at most 4 joining $x,y$. Lemma \ref{lem:locally 1-1} implies that $\phi(P)$ is a cycle of length at most 4. The same lemma implies that $P$ cannot have length 2 and assuming $P$ is shortest possible, the cycle $\phi(P)$ has at least one vertex in its interior. Hence $\phi(P)$ is a separating 3- or 4-cycle, a contradiction.
\end{proof}

Even though our graphs are undirected, it is often convenient to think of an edge $xy$ as \emph{oriented} in the sense of having a \emph{tail} $x$ and a \emph{head} $y$. If a homomorphism $\phi$ maps $xy$ to $uv$, then we write $\phi(xy)=uv$ and we assume that $\phi(x)=u$ and $\phi(y)=v$. 

It is important to notice that for a near-triangulation $Z$ without cut-vertices, a homomorphism $\phi$ from $Z$ into a triangulation $G$ is unique as soon as we have decided the image of an oriented edge; for using the incidences between facial triangles and their sides, this recursively determines the images of all facial triangles. A detailed   construction of $\phi$ is given below in Section \ref{sect:grow}.

We note that this also implies that when we are looking for a possible homomorphic image of a configuration $(Z,\delta)$ in $G$, we just have to decide the image $\phi(xy)$ of a single oriented edge $xy$ in $Z$. Via the above incidences, this determines where the rest should go, so either we complete the image of $(Z,\delta)$ with the given image of $\phi(xy)$, or we discover that there is no such homomorphism $\phi$ from $(Z,\delta)$ to $G$. 

Since embeddings are simply 1-1 homomorphisms, to check for an embedding, we need to verify whether there is a homomorphism and, if so, whether it is 1-1. This is how our 4-coloring algorithm can look for reducible configurations in $G$. We only consider the reducible configurations in our fixed set $\cD$, 
including their mirror images. For each, we can decide on an oriented edge, and try to match it to each of the two orientations of the $m$ edges in $G$.

\paragraph{Configurations contained in configurations.}
We define a homomorphism $\phi$ from a configuration $(Z,\delta)$ 
to a configuration $(Z^*,\delta^*)$ just like we did it to a triangulation $G$ but with the additional requirement that for all vertices $v$ in $Z$, $\delta^*(\phi(v))=\delta(v)$.  If the homomorphism is 1-1, then it is an embedding, and then $(Z,\delta)$ is contained in $(Z^*,\delta^*)$.  

A typical application scenario is as follows.
Suppose a configuration $(Z,\delta)$ of diameter at most 4 (e.g., one of our reducible configurations) is contained via $\phi$ in $(Z^*,\delta^*)$. Suppose further that we have a homomorphism $\phi^*$ from $(Z^*,\delta^*)$ to a triangulation $G$. 
Now, $\phi^*\circ \phi$ is a homomorphism from $(Z,\delta)$ to $G$.
Since $(Z,\delta)$ is of diameter at most 4, we conclude by Lemma \ref{lem:1-1} that either $G$ has a separating 3- or 4-cycle with vertices in 
$V(\phi^*(\phi(Z)))\subseteq V(\phi^*(Z^*))$, 
or $\phi^*\circ\phi$ embeds $(Z,\delta)$ in $G$.

\subsection{Growing near-triangulations, configurations and homomorphisms}\label{sect:grow}

In this section we want to understand
when we do or, perhaps more important, when we do not have a homomorphism $\phi$ from a configuration $(Z,\delta)$ into another configuration $(Z^*,\delta^*)$ mapping a given edge $uv$ from $Z$ into a given 
edge $u^*v^*$ in $Z^*$.
We assume that $Z$ has no cut-vertices.  The details will be important in
Section \ref{sect:free-cartwheel}.

First we describe a way to grow a sub-near-triangulation $Z'$ of $Z$ starting from an arbitrary edge $uv\in E(Z)$. Here {sub-near-triangulation} means that $Z'$ is a near-triangulation that is a subgraph of $Z$ and where all inner facial triangles of $Z'$ are also inner facial triangles of $Z$.
We will maintain that 
$Z'$ has no cut-vertices. We say that the $Z'$ is \emph{edge-closed} if it satisfies the following: for each vertex $v$ in $Z'$ with $d_{Z'}(v)=d_Z(v)$
and where $v$ is internal in $Z$, we have that $v$ is internal in $Z'$.
\begin{lem}\label{lem:neighor-xyz}
    If $Z'\neq Z$ is edge-closed, then there is a 
    triangle $xyz$ in $Z$ with $xy$ a boundary edge of $Z'$ and $z$ not in $Z'$.
  
\end{lem}
\begin{proof}

For any edge $x'y'$ on the outer cycle of $Z'$, there is a unique facial triangle $x'y'z'$ which is in $Z$ but not $Z'$. If $z'$ is not in $Z'$ we are done. Otherwise, adding this facial triangle to $Z'$ gives a plane graph with precisely one \emph{bounded} face $F$ which contains vertices in $Z$ but not $Z'$ (because $Z'$ is edge-closed).
We choose $x'y'$ such that we minimize the number of vertices in $F$. If $xy$ is any edge in $Z'$ on the boundary of $F$, then the unique triangle in $Z$ but not $Z'$ satisfies the conclusion of Lemma \ref{lem:neighor-xyz}. 

\end{proof}
We are now ready to provide the specifics of the growth of $Z'$ starting the edge $uv\in E(Z)$, that is, from $Z'=(\{u,v\},\{uv\})$ which is trivially edge-closed.

When $Z'\neq Z$ and $Z'$ is edge-closed, we add a neighboring triangle $xyz$ from Lemma \ref{lem:neighor-xyz}.

Next we do the following edge-closure process. While there is a vertex $w$ that in $Z$ is internal but in $Z'$ is missing only one edge from its neighbor cycle, 
we add this missing edge, thus completing the wheel around $w$ in $Z'$. When done with edge-closure, $Z'$ is edge-closed.

We repeat adding neighboring triangles followed by edge closures as above until $Z'=Z$. Clearly,
\begin{fact}
    If $Z$ is in a configuration $(Z,\delta)$ 
    and $Z'$ is an edge-closed, sub-near-triangulation of $Z$,
then $(Z',\delta)$ is also a configuration, that is, every boundary vertex $v$ of $Z'$ has $d_{Z'}(v)<\delta(v)$.
\end{fact}
We now have a simple way of growing the unique 
homomorphism $\phi$ from $(Z,\delta)$ to $(Z^*,\delta^*)$ if possible.

First we set $\phi(u)=u^*$ and $\phi(v)=v^*$. Next we grow $(Z,\delta)$ from $uv$ as described above. When we add the outside triangle $xyz$ to the current boundary edge $xy$ of $Z'$, we have already fixed 
$x^*y^*=\phi(xy)$ in $Z^*$ and then there should be a corresponding triangle $x^*y^*z^*$ on the same side of $x^*y^*$ in $Z^*$. Otherwise, we conclude that $z$ is "out of bounds" in the sense that no homomorphism
mapping $uv$ to $u^*v^*$ can map $z$ into $Z^*$.

Next we do the edge-closures on $Z'$. Every time the edge-closure adds an edge $xy$ to $Z'$, we check if $\phi(x)\phi(y)$ is an edge in $Z^*$. Otherwise we conclude that we have a "homomorphic edge error" implying that the desired homomorphism does not exist.

Now for configurations, to check if we have a homomorphism $\phi$ from  $(Z,\delta)$ to  $(Z^*,\delta^*)$, mapping $uv$ to $u^*v^*$, we can first check if there is such a homomorphism from $Z$ to $Z^*$, and next, for each vertex $w$ in $Z$ check if $\delta(w)=\delta^*(\phi(w))$.

However, we can also use the degree functions to discover errors earlier. First we check that $\delta(u)=\delta^*(u^*)$ and  $\delta(v)=\delta^*(v^*)$. Later, every time we add the neighboring triangle 
$xyz$ with $z$ outside to $Z'$, when identifying the corresponding $x^*y^*z^*$ in $Z^*$, we check that $\delta(z)=\delta^*(z^*)$. In any of these degree checks fail, we have a "degree error". Note that with degree checks, we do not have to check for "homomorphic edge errors", for when we added
the last edge in the wheel around a vertex $w$, since it has the same degree as its image $w^*$, we have a corresponding edge in the wheel around $w^*$ in $Z^*$.

The strength of the above degree check is that if it fails on $z$ when we added $xyz$ to $Z'$ (but before edge-closures), then although we do have a homomorphism $\phi$ from $Z'$ to $Z^*$ showing that $z$ has to map to $z^*$, 
we can nevertheless stop the growth because we have found a degree error from $Z'$ to $Z^*$. This concrete witness that there is no homomorphism from $(Z,\delta)$ to $(Z^*,\delta^*)$ will be important 
in Section \ref{sect:free-cartwheel}.

\subsection{Pseudo-triangulations and pseudo-configurations}
\label{sect:pseodo-config}

In this section, we introduce pseudo-triangulations and pseudo-configurations. 

In our computer-assisted proofs, we use the computer to enumerate models
of all relevant local neighborhoods in a 
triangulation, e.g., all possible combinations of discharging rules sending charge along the same edge in the same direction. 
For these computer models, it turns out to be much simpler
if we allow the computer to operate with some more general classes of graphs. We will
define pseudo-triangulations 
generalizing triangulations and near-triangulations, and pseudo-configurations generalizing configurations.
These graphs are multigraphs, allowing loops and multiple edges. The pseudo-triangulations 
will locally look like triangulations or near-triangulations, but
they do not have to be connected or planar (indeed, the computer will never check whether they are connected or planar). Pseudo-configurations are pseudo-triangulations with a degree function on boundary vertices in the same
way that configurations are near-triangulations with a degree function on the boundary vertices.

\paragraph{Dart representations and pseudo-triangulations.}
To define pseudo-triangulations, we will use dart representations, whose definition is inspired by the classic
representation of maps.
A \emph{map} is a 2-dimensional cell complex that represents a 2-cell embedding of a graph into a surface (possibly with one or more boundary components). Our representation is analogous to the concept of a map representation developed by Klein \cite{Klein1879} and formalized by Tutte \cite{Tutte1971}; for a more recent work, we refer to \cite{KMNZ1,KMNZ2}.
However, contrary to this work, we will allow a dart to be its own reverse. In particular, we get a one-dart loop that may be thought of as a terminal object in category theory (in formal program specification \cite{BarrWels95}), and which will play an important role in our constructions.

Formally, our \emph{dart representation} has a set $V$ of vertices and a set $D$ of oriented edges or \emph{darts}. Each dart $e\in D$ has
four pointers: $\head(e)$ to a vertex $v$ which we call its \emph{head vertex}, $\reverse(e)\in D$ to a reverse dart which we say has the \emph{opposite direction} unless $\reverse(e)=e$, $\successor(e)$ 
to the succeeding dart with the same head $v$, 
and $\predecessor(e)$ to the preceding dart with the same head $v$. 
Both $\successor(e)$ and $\predecessor(e)$ can be $\nil$, which means that $e$ has no successor or predecessor.
These $\nil$-pointers are essential for our interpretation of a boundary.
The \emph{tail} of $e$ is defined as $\tail(e)=\head(\reverse(e))$. 

If $\tail(e)=\head(e)$, then $e$ is a \emph{loop}. If $e=\reverse(e)$, then clearly $e$ is a ``one-dart" loop. But, we also get a ``two-dart" loop if 
$e \neq f=\reverse(e)$ and $\head(e)=\head(f)$.

From the perspective of the
computer program, the above dart representation should be taken formally as having
two types of objects, vertices and darts, where the darts
have four associated pointers, one to a vertex and three to darts. This format is what defines a dart representation.

Below, we define the requirements we aim to satisfy; however, while the computer manipulates the dart representations, these requirements will be temporarily suspended.

First, we have some \emph{basic requirements} (M1)--(M5):
\begin{itemize}
\drop{     \item[(M0)] \mikkel{No loops if we put it in: for every
edge $e$, $\head(e)\neq\tail(e)=\head(\reverse(e))$. }
}
     \item[(M1)] There are no isolated vertices, i.e., for each vertex $v\in V$, there is a dart $e$ with $\head(e)=v$.  
     \item[(M2)] $\reverse: D\to D$ is an involution, that is, $e=\reverse(\reverse(e))$.
    \item[(M3)] $\successor$ and $\predecessor$ map $D$ to the set $D\cup\{\nil\,\}$. They are inverse to each other when non-$\nil$, that is, if we have two (non-$\nil$\,) darts $e$ and $f$, then
    $e=\predecessor(f)$ if and only if $f=\successor(e)$.
    \item[(M4)] Heads are consistent in the sense that if $e$ has successor $f\ne\nil$, then $\head(e)=\head(f)$.
    \item[(M5)] If $\successor(e) \ne \nil$, then there are darts $f$ and $g$ such that $f=\reverse(\successor(e))$, $g=\reverse(\successor(f))$, and $e=\reverse(\successor(g))$. 
\end{itemize}
If $e,f,g$ in (M5) are distinct and their three heads $a,b,c$ are also distinct, then we call $e,f,g$ an \emph{inner facial triangle} with vertices $a,b$, and $c$.

The requirement (M5) 
implies that if $e$ has a successor, then
$\reverse(e)$ has a predecessor; namely the dart $g$ above. If a dart $e$ has no successor or has no predecessor, we call it a \emph{boundary dart} and its head a \emph{boundary vertex}.
Then its reverse dart is also a boundary dart, and hence the tail of $e$ is also a boundary vertex. All other vertices and darts are called \emph{inner}. 

By (M3), the $\successor$ and $\predecessor$ pointers partition the darts into doubly-linked
lists, each of which may or may not be cyclic. If a list is acyclic, we view it as ordered following successors, starting from the dart $e$ with $\predecessor(e)=\nil$.
By the head consistency requirement (M4), the darts in each list all have the same head. 
In addition to the basic requirements (M1)--(M5), we have one last requirement for a dart representation:
\begin{itemize}
    \item[(M6)] \emph{Single-list condition:}
    For each vertex $v$, there is exactly one list with $v$ as its head. This is called the \emph{incidence list} of $v$. The number of darts in the incidence list is the \emph{degree} $d(v)$ of $v$.
    The list is cyclic if $v$ is an inner vertex. If $v$ is a boundary vertex, then its incidence list is acyclic with  $\nil$ at both ends.
\end{itemize}
When (M1)--(M6) are satisfied, it follows that a vertex is inner if and only if its incidence list is cyclic. If our dart representation satisfies all the above requirements, then we call it a \emph{pseudo-triangulation}
(although ``pseudo-near-triangulation" might have been a better word).

We note that any triangulation or near-triangulation without cut-vertices has a unique representation as above, where each edge gives two darts that are reverse to each other, each dart representing one orientation of the edge. Moreover, $\successor$ corresponds to the clockwise orientation of the darts around the vertices. Also note that (M6) implies that we cannot represent a near-triangulation with a cut-vertex.

We now extend the definition of orientation-preserving homomorphisms to dart representations. 
An \emph{orientation-preserving homomorphism} or just \emph{homomorphism} from a 
dart representation $Z$ to another dart representation $Z'$ is a mapping $\phi$ from vertices and darts of $Z$ to vertices and darts in $Z'$ such that for any dart $e$ in $Z$:
\begin{itemize}
\item $\head(\phi(e))=\phi(\head(e))$;
\item $\reverse(\phi(e))=\phi(\reverse(e))$;
\item if $\successor(e)\neq\nil$, then $\successor(\phi(e))=\phi(\successor(e))$;
\item if $\predecessor(e)\neq\nil$, then $\predecessor(\phi(e))=\phi(\predecessor(e))$.
\end{itemize}
Note that, if we have a \emph{connected} dart representation $Z$ satisfying (M6), then its homomorphism $\phi$ into a dart representation $Z'$ is uniquely determined as soon as we have decided the image $e'\in E(Z')$ of a dart $e\in E(Z)$. The reason is that, starting with $e$, we can trace all of $Z$ just using pointers to heads, reverse darts, successors, and predecessors, and a homomorphism has to follow the corresponding pointers in the image. A homomorphism with $\phi(e)=e'$ may not be possible, so at the end, we do have to check that the constructed mapping really is a homomorphism.

We say that $Z'$ is a \emph{minimal image of $Z$} if
all vertices and darts in $Z'$ are images of vertices and darts in $Z$, and moreover, for any dart $e'$ in $Z'$, the dart $e'$ has a non-$\nil$ successor/predecessor if and only if there exists a dart $e\in E(Z)$ with $\phi(e)=e'$, where 
$e$ has non-$\nil$ successor/predecessor.

The following lemma is now straightforward:

\begin{lem}\label{lem:hom-basics}
    Let $Z,Z'$ be dart representations such that $Z'$ is a minimal image of $Z$. If $Z$ satisfies the basic requirements (M1)--(M5) for pseudo-triangulations, then so does $Z'$.
\end{lem}


Note that the definition of a homomorphism implies that inner triangles of the pseudo-triangulation $Z$ are mapped to inner triangles of a pseudo-triangulation $H$. Note also that our general definition of homomorphisms of pseudo-triangulations coincides with the original one if $Z$ is a near-triangulation and $H$ is a triangulation.

\paragraph{The terminal pseudo-triangulation: a special one-dart loop.}
Recall that a {loop} is a dart $e$ with $\head(e)=\tail(e)=\head(\reverse(e))$ where $e$ may or may not be its own reverse. 
We note that we have a special pseudo-triangulation consisting
of a single vertex $v$ and a one-dart loop $e$ with
$v=\head(e)$ and $e=\reverse(e)=\successor(e)=\predecessor(e)$.
It trivially
satisfies (M1)--(M6) (with (M5) as a degenerate case of a facial triangle where all sides are the same), and it is the image of every pseudo-triangulation using the unique
homomorphism that maps all darts to the single dart and all vertices to the single vertex. As such, it plays the role of a terminal object in category theory, and we refer to it as the \emph{terminal pseudo-triangulation} (see \cite{BarrWels95}). While this one-dart loop has the terminal  role in category theory, we note that in the theory of combinatorial embeddings \cite{MT}, they typically only allow two-dart loops because no dart is allowed to be its own reverse, but then there is no terminal pseudo-triangulation.

An example showing that the above terminal loop is not always a minimal image 
(although it is always an image) is if we consider a near-triangulation with a single boundary edge between two vertices $u$ and $v$. 
Its dart representation has two darts $e_1$ and $e_2$ that are reverse to each other, and whose heads are $\head(e_1)=y$ and $\head(e_2)=x$. 
Both are on the boundary, so all successors and predecessors are $\nil$, but
then a minimal image must
also have $\nil$-pointers for all successors and predecessors.
Consequently, an image need not contain a minimal image.

We finally make two simple observations concerning loops.
\begin{itemize}
    \item If there
    are no loops, and 
    if there is a dart $e$ with non-$\nil$ successor, then in (M5), the edges $e$, $f$, $g$ are all distinct, forming an inner facial triangle.
    \item If there is a loop in a pseudo-triangulation, then it is mapped to a loop in any homomorphic image. 
\end{itemize}

\paragraph{Pseudo-configurations.}
A \emph{pseudo-configuration} $(Z,\delta)$ is a pseudo-triangulation $Z$ without loops together with a degree function $\delta$ such that for each $v\in V(Z)$, $\delta(v)=d_Z(v)$ if $v$ is an inner vertex, while 
$\delta(v)>d_Z(v)$ if $v$ is a boundary vertex.  

A \emph{homomorphism} from one pseudo-configuration $(Z,\delta)$ to another 
$(Z',\delta')$
is a homomorphism $\phi$ from
$Z$ to $Z'$ such that for every vertex $v\in Z$, $\delta'(\phi(v))=\delta(v)$.
The requirement that degrees have to be preserved makes homomorphisms between pseudo-configurations far more restrictive than homomorphisms between pseudo-triangulations.

One of the important properties of homomorphisms between pseudo-configurations is that they can be composed to new homomorphisms, that is, if
$\phi$ is a homomorphism from $(Z,\delta)$ to $(Z',\delta')$ and 
$\phi'$ is a homomorphism from $(Z',\delta')$ to $(Z'',\delta'')$, then $\phi'\circ\phi$ is a homomorphism from 
$(Z,\delta)$ to $(Z'',\delta'')$.

We shall, for example, use this when $(Z,\delta)$ is a reducible configuration,
$(Z',\delta')$ is a pseudo-configuration that is the free combination of discharging rules, 
and $Z''$ is an arbitrary triangulation $G$, and $\delta''=d_G$.
Then $\phi''=\phi'\circ\phi$ is a homomorphism from $(Z,\delta)$ to $G$, and if 
$(Z,\delta)$ is of diameter at most 4, then $\phi''$ is 1-1 by Lemma \ref{lem:1-1}, implying that $\phi''$ is a planar embedding of $(Z,\delta)$ in~$G$.
Crucially, this holds even if 
$Z'$ itself is not planar.
(We believe that $Z'$ must be planar in this case, but we do not need that, and we do not check it.)

\subsection{Requesting identification of darts}

We will now present the basic subroutine that shall be used for the free combination of configurations.
We are given a pseudo-configuration $(Z,\delta)$ and a \emph{request to identify a set of pairs of darts} 
\[S=\{(e_1,f_1),\ldots, (e_k,f_k)\}.\]
We shall refer to each individual pair $(e_1,f_1)$ as an \emph{identification request}. Consider a pseudo-configuration $(Z',\delta')$ that is
the homomorphic image of $(Z,\delta)$ via some homomorphism $\phi'$.
We say that this homomorphic image \emph{respects} $S$ if for $i=1,\ldots,k$, we have
$\phi'(e_i)=\phi'(f_i)$.
Whenever a homomorphic image respecting $S$ exists,
we can construct one, $(Z^*,\delta^*)$, together with a homomorphism 
$\phi^*: (Z,\delta) \to (Z^*,\delta^*)$ (see Lemma \ref{lem:free-homo-trian} below) such that the following holds. 
If there exists any other $(Z',\delta')$ that is a homomorphic image of $(Z,\delta)$
via some homomorphism $\phi'$ that respects $S$, then there is a homomorphism ${\phi^*}' : (Z^*,\delta^*) \to (Z',\delta')$ such that ${\phi^*}' \circ \phi^*=\phi'$.
We call $(Z^*,\delta^*)$ the \emph{free homomorphic image} respecting $S$.

\paragraph{Free combinations of discharging rules.}
As mentioned earlier, we shall define a \emph{free combination} of discharging rules so that if the rules can be combined in some triangulation to 
send charge over an oriented edge $st$, then we can also form their free combination. Let $(Z_1,\delta_1),\ldots,(Z_k,\delta_k)$ be the configurations of these discharging rules with $s_it_i$ the discharging edge of
$(Z_i,\delta_i)$. When we say they 
\emph{can be combined in a triangulation $G$ to send a charge over the dart $st$}, 
we mean that for each $i=1,2, \ldots ,k$, we have an embedding
$\phi_i$ from $(Z_i,\delta_i)$ to $G$ with $\phi_i(s_it_i)=st$. 

Translating into the setup of pseudo-configurations with identifications, 
we consider the pseudo-configuration $(Z^\sqcup,\delta^\sqcup)$ which is the disjoint union of $(st,d_G)$ and all the $(Z_i,\delta_i)$ ($1\le i\le k$). 
Then our identification requests are
$(st,s_it_i)$ for $i=1,\ldots,k$.

The union of the embeddings $\phi_i$ shows us that we have a homomorphism $\phi^\sqcup$ from $(Z^\sqcup,\delta^\sqcup)$ to $G$ that respects all our identification requests, so we know that we will successfully construct a free homomorphic image $(Z^*,\delta^*)$ of $(Z^\sqcup,\delta^\sqcup)$ respecting the identifications. Then
$(Z^*,\delta^*)$ is the previously claimed free combination of $(Z_1,\delta_1),\ldots,(Z_k,\delta_k)$ identifying the darts $s_it_i$.

We note that our discharging rules actually have specified degree ranges, not just individual degrees. We shall return to this issue in Section \ref{sect:ranges}.

\subsubsection{The free homomorphism from a pseudo-triangulation respecting identification requests.}
\label{subsub:homomorphism-tri}


Given a pseudo-triangulation $Z$ (without degree constraints) and some identification requests $(e_1,f_1),\ldots,(e_k,f_k)$, we will  construct a special homomorphism $\phi^*$ defined on $Z$ into a pseudo-triangulation $Z^*$ 
respecting the identification requests, that is, $\phi^*(e_i)=\phi(f_i)$ for $i=1,\ldots,k$ 
We note that there always is such a $Z^*$; namely, the terminal pseudo-triangulation which identifies all darts in a single one-dart loop. However, the homomorphism $\phi^*$ that we construct will respect the identification requests, but otherwise identify as few vertices and darts as possible. 
At the end, our image $Z^*$ will be unique up to isomorphism, and we will call $(Z^*,\delta^*)$ is the free homomorphic image respecting our identification requests. 

We start the construction by placing all given identification requests in a \emph{queue of unprocessed identification requests}.

The construction of $\phi^*$ is
iterative, identifying more
and more vertices and darts. We will maintain the identifications in a union-find data structure, specifically a forest of rooted trees. The nodes of the trees are the vertices and darts in $Z$, and we
let $\phi^*$ of a vertex/dart be its root in the forest. 
When we are done constructing $Z^*$, the roots will be the vertices and darts of $Z^*$.

As we construct $Z^*$, we may add new
dart identification requests $(e,f)$ to our queue of unprocessed identification requests. The guarantee for any $(e,f)$ in the queue is that when we are done constructing $\phi^*$, we will have $\phi^*(e)=\phi^*(f)$.

\drop{As we start to identify vertices and darts, it may for a while be  unclear which dart we are talking about if we just write $xy$. Instead, as in our dart representation, we think a bit more formally
of a dart $e$ as an object with the above mentioned pointers $\head(e)$, $\reverse(e)$, $\successor(e)$, and $\predecessor(e)$.}

As we construct $Z^*$, we will also have to reassign some of the pointers $P\in\{\head,\reverse,\successor,\predecessor\}$.
For every dart $e$ in $Z$, we will use $P(e)$ to denote the pointer value in $Z$,
while $P^*(e)$ denotes the current value in the construction of $Z^*$. Initially, $P^*(e)=P(e)$ for all darts in $e$. 
The $\head$ and $\reverse$ pointers will only be changed in the very end, but for $P\in\{\successor,\predecessor\}$,  we shall maintain the invariant that if $e^*$ is a current root, then $P^*(e^*)\neq\nil$ if and only if $P^*(e^*)=P(e)\neq\nil$
for some $e$ in $Z$
with root $e^*$.

When we start, each vertex and dart forms its own tree, so $\phi^*$ is just the identity function, which is a trivial homomorphism. However, the identity does not respect any of our identification requests.

We now take an arbitrary identification request $(e,f)$ from our queue of unprocessed identification requests.
If we already have $\phi^*(e)=\phi^*(f)$, then no further processing is needed for this request. 

Below, we describe the processing of an identification request $(e,f)$ where we do not already have $\phi^*(e)=\phi^*(f)$. Let $e^*=\phi^*(e)$ and $f^*=\phi^*(f)$.
Our first step is to
do the identification
itself by making $e^*$ a child of $f^*$ 
in the forest of rooted trees.
Thereby $f^*$ becomes the root representative of both $e$ and $f$ with $\phi^*(e)=\phi^*(f)
=f^*$. However, for $\phi^*$ to become a homomorphism, we need to make several other identifications.
First of all, the heads should also be identified, so if $\phi^*(\head(e)) \neq \phi^*(\head(f))$, we make $\phi^*(\head(e))$ a child of $\phi^*(\head(f))$, so that we also get
$\phi^*(\head(e))=\phi^*(\head(f))$.

If we did not have $e^*=f^*$, we would also generate up to three new identification requests, called \emph{homomorphic} identification requests. We always generate the identification request
$(\reverse(e^*),\reverse(f^*))$.
For $P\in\{\successor,\predecessor\}$, if neither $P^*(e^*)$ nor $P^*(f^*)$ is $\nil$, we generate the homomorphic identification request $(P^*(e^*),P^*(f^*))$; otherwise we make no identification request on the successors. However, if $P^*(f^*)=\nil$ and $P^*(e^*)\neq\nil$, we set $P^*(f^*)=P^*(e^*)$. 

The homomorphic identification requests are all added to our queue of identification requests to be processed. When we process a homomorphic identification request, we refer to it as a \emph{homomorphic identification}.

We continue processing identification requests from our queue until it is empty.  This must happen eventually, since we generate new identification requests only when we perform a new identification, thereby reducing the total number of representative roots by one.

We are now done with the construction of $\phi^*$.
The vertices/darts in $Z$ are mapped by $\phi^*$ to their representative roots, and these roots
form the set of vertices/darts in $Z^*$. We now need to assign final pointers to each of these darts $e^*$ in $Z^*$. Recall that we use a superscript $^*$ to indicate pointers in $Z^*$. 
For $P\in\{\head,\reverse\}$,
we simply set $P^*(e^*)=\phi^*(P^*(e^*))=\phi^*(P(e^*))$.
For $P\in\{\successor,\predecessor\}$,
we set $P^*(e^*)=\phi^*(P^*(e^*))$ if $P^*(e^*)\neq\nil$; otherwise we leave
$P^*(e^*)=\nil$.

We will refer to $\phi^*$ and $Z^*$ as the \emph{free homomorphism} and \emph{free homomorphic image} of $Z$ respecting $S$. So far, we have not even proved that $\phi^*$ is a homomorphism. We shall prove that $\phi^*$ is not only a homomorphism, but it is the uniquely most general one in the sense that any other homomorphism $\phi'$ from $Z$ respecting $S$ can be obtained from $\phi^*$ by composition with another homomorphism $\phi^{*\prime}$, that is,  $\phi'=\phi^{*\prime}\circ \phi^*$. 
It is the latter aspect that makes us call $\phi^*$ and $Z^*$ ``free".

The pseudocode for constructing $\phi^*$ and $Z^*$ respecting $S$ is found in Algorithm \ref{alg:homomorphism-tri}.

\paragraph{Analysis.}
We shall now analyze the process outlined above. At the moment, it is not even clear if
$Z^*$ is a pseudo-triangulation, but it is a dart representation.

\begin{lem}\label{lem:Z*}
The constructed map $\phi^*$ is a homomorphism from $Z$ to the constructed $Z^*$
which is the minimal image 
(as defined before Lemma  \ref{lem:hom-basics})
of $Z$ under $\phi^*$.
\end{lem}

\begin{proof}
    For $\phi^*$ to be a homomorphism to
    the final $Z^*$, we need to prove for each dart $e$ in $Z$, and for each
    pointer $P$ that if $P(e)\neq\nil$, then $\phi^*(P(e))=P^*(\phi^*(e))$. The minimality also requires for  $e^*\in Z^*$ that $P^*(e^*)=\nil$ if 
    $P(e)=\nil$ for every dart $e$ in $Z$ with $\phi^*(e)=e$.

    The simplest case is for $P\in\{\head,\reverse\}$ where we do not have issues with $\nil$. Suppose $e^*=\phi^*(e)$.
    This implies that $e$ has been transitively identified with $e^*$ (resulting in the tree containing $e$ having root $e^*$), and for each of these identifications $(e_1,e_2)$, we ensured that
    we would end up with $\phi^*(P(e_1))=\phi^*(P(e_2))$. We conclude that we end up with $\phi^*(P(e))=\phi^*(P(e^*))$. After the construction of $\phi^*$, 
    for each root $e^*$, we set 
    $P^*(e^*)=\phi^*(P(e^*))$, so we
    conclude that 
    $\phi^*(P(e))=P^*(\phi^*(e))$, as desired.

    We now consider the more complicated case where 
    $P\in\{\successor,\predecessor\}$. 
    Our construction maintained the invariant that     if $e^*$ is a current root then $P^*(e^*)\neq\nil$ if and only if $P^*(e^*)=P(f)\neq\nil$
    for some $f$ in $Z$
    with root $e^*$. To prove that we get a homomorphism, we need the following additional invariant:
    \begin{quote}
        Consider any dart $e\in Z$ with $P(e)\neq \nil$ and root $e^*$.
       Then $P^*(e^*)=P(f)\neq\nil$ for
    some $f$ with root $e^*$, and then we will end up with $\phi^*(P(e))=\phi^*(P(f))$.
    \end{quote}
    To see this, let $e,e^*,f$ be as above, and suppose we  identify $(e_1,e_2)$ with $\phi^*(e_1)=e^*$ and $\phi^*(e_2)=g^*\neq e^*$. 
    This is the case where we will make $e^*$ a child of $g^*$.

    Suppose we had $P^*(g^*)=P(h)\neq\nil$. This value will not change, but we generate the identification request $(P^*(e^*),P^*(g^*))=
    (P(f),P(h))$, so we will get $\phi^*(P(f))=\phi^*(P(h))$. The invariant on $e,e^*,f$ implied that we would get $\phi^*(P(e))=\phi^*(P(f))$, so we conclude that we end up with $\phi^*(P(e))=\phi^*(P(h))$, recovering the invariant for $e,g^*,h$.

    Suppose instead that $P^*(g^*)=\nil$. This is the case where we set $P^*(g^*)=P^*(e^*)=P(f)$. This immediately recovers the invariant for $e,g^*,f$.

    We could also have identified
    $(e_2,e_1)$, making $g^*$ a child of $e^*$. Since we had $P^*(e^*)\neq\nil$, 
    we would not change $P^*(e^*)$, so the
    invariant on $e,e^*,f$ is preserved. Since the invariant is maintained in all cases, it remains true when we are done constructing $\phi^*$. 

    Let $e,e^*,f$ be as described in the invariant when we
    are done constructing $\phi^*$. Since
    $P^*(e^*)\neq \nil$, the final step is to set 
    $P^*(e^*)=\phi^*(P^*(e^*))=\phi^*(P(f))=\phi^*(P(e))$, as required for $\phi^*$ to be a homomorphism from $Z$ to $Z^*$.

    Finally we note that $Z^*$ is a minimal image since we maintain that  $P^*(e^*)=\nil$ 
    if $P^*(f)=\nil$ for every $f$ with $\phi^*(f)=e^*$.
\end{proof}

\begin{lem}\label{lem:Z*'}
The constructed $\phi^*$ and $Z^*$ are unique up to an isomorphism.
Moreover, suppose there is a homomorphism $\phi'$ from $Z$ to some pseudo-triangulation $Z'$ and that $\phi'$ respects the identifications in $S$.
Then there also exists a homomorphism $\phi^{*\prime}$ from $Z^*$ to $Z'$ such that $\phi' = \phi^{*\prime}\circ\phi^*$. 
\end{lem}

\begin{proof}
Let $\phi'$ be an arbitrary homomorphism from $Z$ respecting $S$, as stated in the lemma.
We claim that in the construction of $\phi^*$, there can never 
be a first homomorphic identification request $(e,f)$ such that 
$\phi'(e)\neq \phi'(f)$. This follows because if we, for all previously processed identification requests, had
$\phi'(e)=\phi'(f)$, then the definition of a homomorphism ensures that all the homomorphic identification requests
$(e',f')$ generated from 
$(e,f)$ also have $\phi'(e')=\phi'(f')$.
We thus have $\phi^*(e)=\phi^*(f)$ if and only if $\phi'(e)=\phi'(f)$ for every possible homomorphism. The identifications under
$\phi^*$ are thus unique, but then the minimal homomorphic image $Z^*$ must be unique up to an isomorphism.

The map $\phi^{*\prime}$ is defined
by setting $\phi^{*\prime}(\phi^*(x))=\phi'(x)$ for every vertex or dart $x$ in $Z$.
We need to argue that $\phi^{*\prime}$ is a homomorphism, but this follows because $Z^*$ was a minimal image. The argument holds for all pointers; therefore, we focus only on the successor. If $e^*$ has
successor $f^*$ in $Z^*$, then we
know that there exists an $e$ with successor $f$ in $Z$ with $e^*=\phi^*(e)$ and $f^*=\phi^*(f)$. However,
$\phi'$ is a homomorphism, so we
also have $\phi'(f)=\successor'(\phi'(e))$, so
indeed $\phi^{*\prime}(f^*)=\successor'(\phi^{*\prime}(e^*))$.
\end{proof}
From Lemma \ref{lem:Z*}, we have that $Z^*$ is a minimal homomorphic image, and since $Z$ is a pseudo-triangulation, we get from
Lemma \ref{lem:hom-basics} that $Z^*$ satisfies the basic requirements for being a pseudo-triangulation. It remains to prove that $Z^*$ 
satisfies the single-list condition.
\begin{lem}\label{le:Z*-pseudo-triangulation}
 $Z^*$ is a pseudo-triangulation. 
\end{lem}
\begin{proof}
It remains to prove the single-list condition (M6) that all darts with head $y$ are in a single, possibly cyclic, incidence
list.

To prove this, we 
partition darts and vertices from $Z^*$ into
equivalence classes.
Vertices are in the same equivalence class if they have been identified. Darts
are in the same equivalence class if they are connected via identifications, successor, or predecessor pointers. Note that we do \emph{not} consider reverse pointers.

When we start from $Z$, we have
no identifications, so each vertex has its own equivalence class. The darts are partitioned into incidence lists, connected via successor and predecessor pointers.

We claim that the construction of $Z^*$ maintains the invariant that each vertex equivalence class
is the head set of
exactly one dart equivalence class. This is trivially the case when
we start from $Z^*=Z$ since $Z$ satisfies the single-list condition (M6).

In the construction of $Z^*$, we now perform a series of dart identifications. Note that these do not immediately change successor or predecessor pointers. Identifying darts could unite dart equivalence classes, but every time we identify darts, we also identify their heads, so the
invariant is maintained.

In the final cleaning, we keep only a representative for each set of identified darts or vertices, redirecting successor and predecessor pointers to go between representatives, but it is done so that the invariant is maintained.

Having only representatives,  each vertex class consists of a single representative vertex. Each dart equivalence
class is representatives connected via successor and predecessor pointers, that is, an incidence list. The invariant now implies that no two incidence lists
can have the same head.
\end{proof}

\drop{\carsten{The proof of the previous lemma is problematic because there may be cut-vertices. Also it is dangerous to work on an "intermediate" representation since then a dart may have multiple successors. I suggest to work on the "final" representation like this:}\mikkel{I tried to clarify the proof I made.}
\begin{proof}
It remains to prove that all darts with head $y$ are in a single, possibly cyclic, doubly-linked list, 
\carsten{where "doubly-linked" is via successor and predecessor pointers.
To prove this, the identifications partition the darts (respectively the vertices) into equivalence classes, where an equivalence class consists of all those darts (respectively the vertices) that have been identified. Consider two darts with the same head in $Z^*$ represented by $e_1,e_2$ in $Z$ where $head(e_i)=v_i$, $i=1,2$. Since $\phi^*(e_1), \phi^*(e_2)$ have the same head in $Z^*$, there must be a dart $e_2'$ with head $v_1$ such that $e_2',e_2$ are identified. \mikkel{Not so clear to me... there could be several levels of identification via $v_1'$, $v_1''$. I think my proof is safer}
Assume without loss of generality that $e_2'$ can be obtained from $e_1$ by applying $successor$  $k$ times. Then $\phi^*(e_2')=\phi^*(e_2)$ can be obtained from $\phi^*(e_1)$ by applying $successor$ at most $k$ times.}
\end{proof}
}

Summing up, we have proved the following result.

\begin{lem}\label{lem:free-homo-trian}
Suppose that we are given a pseudo-triangulation $Z$ and a set $S$ of dart identification requests. Our construction yields
a homomorphism $\phi^*$ from $Z$ to a a minimal image $Z^*$
that is a pseudo-triangulation. Both $\phi^*$ and $Z^*$ are unique up to an isomorphism. 
Moreover, consider any other homomorphism $\phi'$ from $Z$ to a 
pseudo-triangulation $Z'$ that respects $S$. Then there
exists a homomorphism $\phi^{*\prime}$ from $Z^*$ to $Z'$ such that 
$\phi^{*\prime}\circ\phi^*=\phi'$.
\end{lem}

Note that if $Z'$ has no loops, then $Z^*$ cannot have a loop since any loop in $Z^*$ would be mapped to loop in $Z'$.

\subsubsection{The free homomorphism from a pseudo-configuration respecting identification requests}
\label{subsub:homomorphism-conf}

We are now given a pseudo-configuration $(Z,\delta)$ from
which we want to construct a free homomorphism $\phi^*$ into a pseudo-configuration $(Z^*,\delta^*)$ such that $\phi^*$ 
respects a given set $S=\{(e_1,f_1),\ldots, (e_k,f_k)\}$ of input identification requests. Contrary to 
the case with pseudo-triangulations, there will be many cases where the identifications fail to produce a pseudo-configuration.
However, we will prove that we never fail if there exists a pseudo-configuration $(Z',\delta')$ that is a homomorphic image of $(Z,\delta)$ respecting $S$. In that case, we will even show that there is a homomorphism $\phi^{*\prime}$ from
our $(Z^*,\delta^*)$ to $(Z',\delta')$.
Below we first present the construction of 
$(Z^*,\delta^*)$. Later we present the analysis showing that it has the desired properties.

For ease of reading, we will use $uv$ for a dart with head $v$ and tail $u$. Since we may have parallel edges, $u$ and $v$ may not determine the arc uniquely, but the identity will be clear from the context, e.g., if we say that $uv$ is the first arc in the incidence list of $v$.

Our first step is to apply Lemma \ref{lem:free-homo-trian}. This yields 
the pseudo-triangulation $Z^*$ and a homomorphism $\phi^*$ from $Z$ to $Z^*$ respecting $S$, but we do not expect to $Z^*$ to the final answer. Rather we expect to manipulate $Z^*$ over several rounds, possibly reporting an error.
Recall that pseudo-configurations have no loops, so we fail with a \emph{loop error} if $Z^*$ has a loop, that is, a dart $e$ with $\head(e)=\tail(e)=\head(\reverse(e))$.
We now check if the degree
constraints are satisfied. First of all, for all vertices $u$ and $v$ with
$\phi^*(u)=\phi^*(v)$, we need
$\delta(u)=\delta(v)$; otherwise we report a \emph{degree mismatch} error.

If there is no
degree mismatch error, then we define $\delta^*(\phi^*(x))=\delta(x)$.
We note that it may be more natural to check degrees during the construction of $Z^*$ in Lemma \ref{lem:free-homo-trian} which repeatedly identify darts and vertices: whenever we identify two vertices, we just check that they have the same degree. However, the check for loops has to await the completion of $Z^*$ when the reverse pointers have attained their final values.

We want $(Z^*,\delta^*)$ to be a degree respecting
configuration in the sense that for all inner vertices $v$, $d_{Z^*}(v)=\delta^*(v)$ and for all boundary vertices 
$d_{Z^*}(v)<\delta^*(v)$. We are going to look iteratively for vertices violating of these constraints, either fixing them, 
as described below,
or reporting an error. The simplest case is if there is an inner vertex $v$ with $d_{Z^*}(v)<\delta^*(v)$. If so, we report a \emph{sub-degree} error.

Suppose there is a vertex
$v$ with $d_{Z^*}(v)>\delta^*(v)$.
Then there must
be distinct 
incident darts $e$ and $f$ of successor
distance exactly $\delta^*(v)$, that is, in the incidence list of $v$, we get from $e$ to $f$ following successor pointers $\delta^*(v)$ times. If so, we make a \emph{degree} identification request on $(e,f)$. We process this degree identification request as an input identification request in Lemma \ref{lem:free-homo-trian} where it  triggers new homomorphic identification requests in $Z^*$ until we have a new homomorphic image. When done we have to check again both for degree mismatches and loops, reporting an error if any of these are found. 

The final issue is if
there is a boundary vertex $v$ with $\delta^*(v)=d_{Z^*}(v)$. We will fix this as follows. Let $uv$ and $wv$ be the first and last dart in the incidence lists of $v$. We want
to make $v$ an inner vertex with incident inner facial triangle $uvw$. In particular,
this implies that $u$, $v$, and $w$ must be distinct; otherwise, we report a \emph{boundary} error. We now \emph{add new boundary darts $uw$ and $wu$}.
This completes the neighbor cycle around $v$ and adds the last inner triangle $uvw$. We have to update the successor/predecessor pointers accordingly.  First, we make $uv$ the successor of $wv$ and $wv$ the predecessor of $uv$, so now the
list of incident darts is cyclic and $v$ has become internal with the correct degree $\delta(v)$. We take the new dart $uw$ and make it
the new predecessor of $vw$ around $w$, and $vw$ the successor of $uw$. Symmetrically, we make its reversal
$wu$ the successor of $vu$
and $vu$ the predecessor  $uw$. Finally, $uw$ gets the predecessor set to nil and $wu$ gets the successor set to $\nil$.
Note that this way, $u,w$ are still boundary vertices while $v$ is not.
The corresponding pseudocode is presented in Algorithm \ref{alg:add_boundary_darts}.

We repeat applying the above cases as long
as there are vertices $v$ in $Z^*$ violating the degree constraints. 
If no error is reported,
then $\phi^*$ is our \emph{free homomorphism} respecting $S$, and $(Z^*,\delta^*)$ its \emph{free homomorphic image}.  We will say that we \emph{successfully constructed a free homomorphism $\phi^*$ with  image $(Z^*,\delta^*)$}. 

We shall argue that $(Z^*,\delta^*)$ is unique. To appreciate this, let us consider a simple example. Suppose
we have a boundary face $uvw$. More precisely, $uv$ and $wv$ is the first and last dart around $v$, 
$vw$ and $uv$ is the first and last dart around $v$,
and  $wu$ and $vu$ is the first and last dart around $u$. Suppose, moreover, that each of the vertices $x\in\{u,v,w\}$ has $d_{Z^*}(x)=\delta^*(x)$. What will
happen is that $uvw$ becomes
an inner triangle where
for each $x\in\{u,v,w\}$,
what was the first incident dart becomes successor of what was the last incident dart. To get to this state  will, however, take multiple iterations that can happen in different ways. First, pick
a vertex, say $v$, and make it
internal by adding a new
parallel boundary edge between $u$ and $w$. Now,
$d_{Z^*}(u)=\delta^*(u)+1$
and $d_{Z^*}(w)=\delta^*(w)+1$. We can now do a degree identification of the two parallel boundary darts incident to $w$. The resulting homomorphic identifications will also identify the two parallel boundary darts incident to $u$, and now $uvw$ is the claimed inner triangle.

The pseudocode for this is found in Algorithm \ref{alg:homomorphism-conf}. Note that Algorithm \ref{alg:homomorphism-conf} is designed for the more general case
where we instead of specifying a single degree specify a range of
possible degrees. 
However, by treating a single value $\delta(v)$ as the range $[\delta(v), \delta(v)]$ and calling the function with this input, we obtain the same procedure described in this section.

 \begin{lem}\label{lem:free-homo-conf}
Suppose that we are given a pseudo-configuration $(Z,\delta)$ and a set $S$ of dart identification requests. Suppose there is a homomorphism $\phi'$ from $(Z,\delta)$ to a pseudo-configuration $(Z',\delta')$ such that $\phi'$ respects $S$.
Then we can 
successfully construct a free homomorphism $\phi^*$ with image $(Z^*,\delta^*)$ respecting $S$. Up to an isomorphism,
the image $(Z^*,\delta^*)$ 
is unique, independent of  $\phi'$ and $(Z',\delta')$, and it is guaranteed that there
exists a homomorphism $\phi^{*\prime}$ from $(Z^*,\delta^*)$ to $(Z',\delta')$ so that 
$\phi^{*\prime}\circ\phi^*=\phi'$.
\end{lem}

\begin{proof}
The first part of our construction is to ignore the degrees and just apply Lemma \ref{lem:free-homo-trian} to $Z$ and $Z'$. This yield a homomorphism $\phi^*$ with  image $Z^*$ respecting $S$ and with a homomorphism $\phi^{*\prime}$ from 
$Z^*$ to $Z'$ such that
$\phi^{*\prime}\circ\phi^*=\phi'$.

Now that $S$ is respected, the rest of the construction
aims to manipulate $Z^*$ to become a degree-respecting configuration 
$(Z^{**},\delta^{**})$ (which will play the role of $(Z^*,\delta^*)$ in the statement of Lemma \ref{lem:free-homo-conf}), 
or produce an error if this is not possible. However, here $(Z',\delta')$ serves as a witness that the construction is possible, hence that an error should not occur.

As we do the modifications to $Z^*$
we will maintain that $Z^*$
is a pseudo-triangulation with homomorphisms $\phi^*$ from $Z$ to $Z^*$ and $\phi^{*\prime}$ from $Z^*$ to $Z'$ such that
$\phi^{*\prime}\circ \phi^*=\phi'$.

Assuming the above homomorphisms, we claim that  $Z^*$ cannot have a loop; for since $Z'$ is in a configuration, it has no loops, but then $Z^*$ cannot have a loop
since $\phi^{*\prime}$ would map such a loop in $Z'$.

Likewise we argue that there cannot be a degree mismatch; for consider
any two vertices $u$ and $v$ with $\phi^*(u)=\phi^*(v)$. Then we must
have $\phi'(u)=
\phi^{*\prime}(\phi^*(u))=
\phi^{*\prime}(\phi^*(v))=\phi'(v)$, and the definition of homomorphisms between configurations imply that 
$\delta(u)=\delta'(\phi'(u))=\delta'(\phi'(v))=\delta(v)$.

Let us now consider an inner vertex $v$ in 
$Z^*$. It has a cyclic list $u_1v,\ldots,u_dv$ of darts around $v$ with
$d=d_{Z^*}(v)$. Because $\phi^{*\prime}$ is a homomorphism, $\phi^{*\prime}(u_1v),\ldots,\phi^{*\prime}(u_dv)$ must be a cyclic walk around $\phi^{*\prime}(v)$ following successors in $Z'$. Therefore $\phi^{*\prime}(v)$ is an inner vertex in $Z'$ with $d_{Z'}(\phi^{*\prime}(v))\leq d_{Z^*}(v)$. But since $Z'$ is a configuration, we have $d_{Z'}(\phi^{*\prime}(v))=\delta'(\phi^{*\prime}(v))=\delta^*(v)$. Thus, we have a contradiction 
if we report an
error saying $d_{Z^*}(v)<\delta^*(v)$.

Now let $v$ be any vertex in 
$Z^*$ with $d_{Z^*}(v)>\delta^*(v)$. Let $uv$ and $wv$ be the incident darts from our construction that are exactly $\delta^*(v)$ apart in the successor order around $v$.
Then in $Z'$ we must get from $\phi^{*\prime}(uv)$ to $\phi^{*\prime}(wv)$
following successors $\delta^*(v)$ times around $\phi^{*\prime}(v)$. However, $\delta^*(v)=\delta'(\phi^{*\prime}(v))$, and since $Z'$ is a configuration, this implies  $\phi^{*\prime}(uv)=\phi^{*\prime}(wv)$.

We know that $\phi^{*\prime}(uv)=\phi^{*\prime}(wv)$, so as in our construction, we can apply Lemma \ref{lem:free-homo-trian} with $Z=Z^*$ and $\phi'=\phi^{*\prime}$ and an input identification request $(uv,wv)$.
The result is a homomorphic image $Z^{**}$ with homomorphisms $\phi^{**}$ from $Z^*$ to $Z^{**}$ and $\phi^{**'}$ from
$Z^{**}$ to $Z'$. We can now let
$Z^{**}$ play the role of $Z^*$,
with $\phi^{**}\circ \phi^*$ the new $\phi^*$ and $\phi^{**\prime}$ the new $\phi^{*\prime}$.

Finally, we have the last case in our construction with a boundary vertex  $v$ of $Z^*$
where $uv$ and $wv$ are the first and the last in the successor list of darts with head $v$ and the length of the list is exactly $\delta^*(v)$. The image in $Z'$ with the homomorphy $\phi^{*\prime}$ has the same successor sequence
from $\phi^{*\prime}(uv)$ to $\phi^{*\prime}(wv)$ of length $\delta^*(v)$ but since $Z'$ is a configuration and $\delta'(\phi^{*\prime}(v))=\delta'(v)$, we conclude that the list is cyclic in $Z'$ with $\phi^{*\prime}(uv)$ the successor of $\phi^{*\prime}(wv)$.

Since $Z'$ is loop free, we know it has an inner
facial triangle with 
distinct vertices $u'=\phi^{*\prime}(u)$, $v'=\phi^{*\prime}(v)$, and $w'=\phi^{*\prime}(w)$ including an
edge (consisting of two darts) between $u'$ and $w'$. 
This implies that $u$, $v$, and $w$ are distinct,
so we have a contradiction if we report a boundary error because they are not distinct.

In our construction, we  added the boundary dart  $uw$ with reversal $wu$ to $Z^*$, assigning all the relevant predecessors and successors so that the homomorphism $\phi^{*\prime}$ can be extended to map $uw$ and
$wu$ to the darts in $u'w'$ and $w'u'$. The homomorism $\phi^*$ from
$Z$ to $Z^*$ is not affected by this addition,
for in the image of $Z$,
we only changed boundary  $\nil$-pointers from $uv$
and $wv$.

We have thus proved that all our modifications
of $Z^*$ maintain the homomorphisms $\phi^*$ and $\phi^{*\prime}$. It is also clear that if we terminate, it is because we respect the degrees in $\delta^*$ so that
$(Z^*,\delta^*)$ is a pseudo-configuration.

The next question is whether we are sure to terminate. First consider the case where we add a new boundary dart $uw$ with reversal $wu$ around $v$ that thereby becomes inner.
The homomorphic image of an inner vertex cannot be a boundary vertex, so each vertex can only turn from boundary to inner once. The number of boundary darts added is therefore at most twice the number of boundary vertices in $Z$.  All other operations are dart identifications, each of which reduces the number of dart identifiers in $Z^*$. The total number of dart identifications 
is therefore bounded by  twice the number of boundary vertices plus 
the number of dart identifiers in $Z$.

Finally, we want to argue that the outcome $(Z^*,\delta^*)$ is unique up to an isomorphism. 
Suppose we have obtained two outcomes $(Z^*_1,\delta^*_1)$ and
$(Z^*_2,\delta^*_2)$. We can then let $(Z^*_2,\delta^*_2)$ play the role of $(Z',\delta')$ for $(Z^*_1,\delta^*_1)$ and vice versa. We now know that $\phi^{*\prime}_1\circ \phi^*_1=\phi^*_2$, so any vertex from $Z$ identified by $\phi^*_1$ must also be identified by $\phi^*_2$, and vice versa. We also note from
our construction that $V(Z^*_i)=\phi_i^*(V(Z))$ (this does not hold for the darts, as we may add darts). We conclude that $|V(Z^*_1)|=|V(Z^*_2)|$. This further implies that $\phi^{*\prime}_1$ must be 1-1 from $V(Z^*_1)$ to $V(Z^*_2)$, but then 
$\phi^{*\prime}_1$ is also
1-1 from $E(Z^*_1)$
to $E(Z^*_2)$, implying $|E(Z^*_1)|\leq |E(Z^*_2)|$. A symmetric argument implies
$|E(Z^*_2)|\leq |E(Z^*_1)|$, so
$|E(Z^*_1)|=|E(Z^*_2)|$, and therefore $\phi^{*\prime}_1$
is a bijective homomorphism, which is equivalent to saying it is an isomorphism.
\end{proof}

Often, it is important that we can handle dart identification requests
in batches or sequentially, one at a time. To this end, we can use the construction of Lemma \ref{lem:free-homo-conf} to prove:

\begin{lem}\label{lem:free-homo-conf-trans}
Let a pseudo-configuration $(Z,\delta)$ and two sets $S_1$ and $S_2$ of dart identification requests be given. Suppose that there is a homomorphism $\phi'$ from $(Z,\delta)$ to a 
pseudo-configuration $(Z',\delta')$ such that $\phi'$ respects $S_1\cup S_2$.
As in 
Lemma \ref{lem:free-homo-conf}, we can successfully construct a free homomorphism $\phi^*$ from $(Z,\delta)$ to $(Z^*,\delta^*)$ respecting $S_1\cup S_2$
by first constructing a free homomorphism $\phi_1^*$ from $(Z,\delta)$ to $(Z^*_1,\delta^*_1)$ respecting $S_1$, and second
constructing a free homomorphism $\phi_2^*$ from $(Z^*_1,\delta^*_1)$ to $(Z^*_{12},\delta^*_{12})$
respecting 
\[S^*_2=\{(\phi_1^*(e),\phi_1^*(f)\mid (e,f)\in S_2\}.\]
The composition $\phi^*_2\circ \phi^*_1$ is
a free homomorphism from $(Z,\delta)$ to $(Z^*_{12},\delta^*_{12})$
respecting $S_1\cup S_2$, and
therefore $(Z^*_{12},\delta^*_{12})$ is isomorphic to $(Z^*,\delta^*)$.
\end{lem}

\subsection{Degree ranges}\label{sect:ranges}

For our discharging rules, it is important that we can specify, for boundary vertices $v$, not just a degree $\delta(v)$ but a range of degrees $[\delta^-(v),\delta^+(v)]$. Indeed
we may have $\delta^+(v)=\infty$ implying that $v$ can embed to a vertex of arbitrarily high degree.
We shall use  $(Z,\delta^-,\delta^+)$ to denote such a configuration with a range of degrees. For an inner vertex $v$, we have only one degree option $d_Z(v)=\delta^-(v)=\delta^+(v)$. For a boundary vertex $v$, we need
$\delta^-(v)>d_Z(v)$. We can think of $(Z,\delta^-,\delta^+)$ as representing the set of configurations $(Z,\delta)$
such that $\delta(v)\in[\delta^-(v),\delta^+(v)]$ for all vertices $v$ in $Z$.

We will now describe how to handle a set $S$ of dart identification requests for a configuration $(Z,\delta^-,\delta^+)$ with degree ranges.
The desired output is set of 
configurations with degree
ranges $(Z_1^*,\delta^{*-}_1,\delta^{*+}_1),\ldots,(Z_k^*,\delta^{*-}_k,\delta^{*+}_k)$ satisfying that
$(Z^*,\delta^*)$ belongs
to some $(Z_i^*,\delta^{*-}_i,\delta^{*+}_i)$, 
$1\leq i\leq k$, if and only if
there is a $(Z,\delta)\in (Z,\delta^-,\delta^+)$ such that $(Z^*,\delta^*)$ is the free homomorphic image of $(Z,\delta)$ respecting $S$ as described in Lemma \ref{lem:free-homo-conf}. If there is no such $(Z^*,\delta^*)$, then the construction is given up.

We shall basically mimic our previous algorithm to construct such a homomorphic
image for configuration with specified individual degrees (not ranges).
The first step is again to apply Lemma \ref{lem:free-homo-trian}.
This yields the pseudo-triangulation $Z^*$ and a homomorphism $\phi^*$ from $Z$ to $Z^*$ respecting $S$.
Again we check for loops, giving up if any is found.
For the degree ranges, for each $v^*$ in $Z^*$, we let the degree range $[\delta^{*-}(v^*),\delta^{*-}(v^*)]$ be the intersection of the degree ranges $[\delta^{-}(v),\delta^{+}(v)]$ of
all $v$ with $\phi^*(v) = v^*$. If the intersection is empty, we give up. This corresponds to the degree mismatch when we only had a single specified degree for each vertex.
The corresponding pseudocode is presented in Algorithm \ref{alg:dartIdentification}.

We now have a $(Z^*,\delta^{*-},
\delta^{*+})$ that respects $S$, but it is typically not a configuration yet, so we need to modify it as we did in the specified single degree case. However, many of the modifications we perform on  
a vertex $v$ depend on its specified degree, and we may therefore have to \emph{subdivide} the interval $[\delta^{*-}(v), 
\delta^{*+}]$.
More precisely, first
we create a copy $(Z^{*\prime},\delta^{*\prime-}, 
\delta^{*\prime+})$
of $(Z^*,\delta^{*-}, 
\delta^{*+})$.
Next, for some $d\in [\delta^{*-}(v),\delta^{*+}(v))$,  we
set $\delta^{*+}(v)=d$ and $\delta^{*\prime-}(v)=d+1$.
In combination, 
$(Z^*,\delta^{*-},\delta^{*+})$ and $(Z^{*\prime},\delta^{*\prime-}, 
\delta^{*\prime+})$ cover exactly the  same single degree configurations as before, and we can now continue the construction with a branch for each of them. However, we want to minimize the number of subdivisions so that we do not end up with more configurations  $(Z_i^*,\delta^{*-}_i,\delta^{*+}_i)$ than necessary.

As a first simple case, recall from the specified single degree case that we gave up with a sub-degree error if there was an inner vertex $v$ with $d_{Z^*}(v)<\delta^*(v)$. With degree ranges, we can give up with the same error if 
$d_{Z^*}(v)<\delta^{*-}(v)$.
As we modify our configuration, we will always check for such errors, so below we assume for
all inner $v$ that $d_{Z^*}(v)\geq \delta^{*-}(v)$.
The corresponding pseudocode is presented in Algorithm \ref{alg:inner_subdegree_error}.

Since we are trying to avoid subdividing ranges, we prioritize looking for a vertex $v$ with a single
specified degree $\delta^*(v)=\delta^{*-}(v)=\delta^{*+}(v)$ with degree issues, as described below.

One degree issue is if a vertex $v$ has a single specified degree $d_{Z^*}(v) > \delta^*(v)$, then
as in the case without degree ranges, we identify 
two incident darts $e$ and $f$ of successor
distance exactly $\delta^*(v)$. This identification triggers
a new cascade of homomorphic identification requests in $Z^*$ until we have a new homomorphic image. The process involves vertex identifications, and every time we identify two vertices $u$ and $v$, the resulting
range is the intersection of the ranges of $u$ and $v$, and we give up if this intersection is empty. At the end, we also check for loops.
The corresponding pseudocode is presented in the if clause of Algorithm \ref{alg:fix_single_degree_issue}.

Another degree issue is if 
a boundary vertex $v$ has a single specified degree $\delta^{*}(v)=d_{Z^*}(v)$.
If $uv$ and $wv$ are the incident boundary arcs, then this is the case where we added a  
new boundary edge between $u$ and $w$, turning $v$ into an inner vertex.
The corresponding pseudocode is presented in the else-if clause of Algorithm \ref{alg:fix_single_degree_issue}.

If no vertex with a single specified degree has degree-issue, we look for a vertex $v$ with $\delta^{*-}(v)<\delta^{*+}(v)$ and $\delta^{*-}(v)\leq d_{Z^*}(v)$
We subdivide this interval into a single degree 
$\delta^{*}(v)=\delta^{*-}(v)$  
and the remainder $[\delta^{*-}(v)+1,
\delta^{*+}(v)]$. We continue with a branch for each of the resulting two configurations.
The corresponding pseudocode is presented in Algorithm \ref{alg:single_out_lower_deg}.
With each branch we return to see if some of the previous cases apply before we do any additional branching.

Note that when none of the cases apply, then for each inner vertex $v$, we have a single specified degree $\delta^*(v)=\delta^{*-}(v)=\delta^{*+}(v)=d_{Z^*}(v)$ and for
each other vertex $v$, we have $\delta^{*-}(v)>d_{Z^*}(v)$, so indeed this is a pseudo-configuration with specified degree ranges.
The overall algorithm for resolving degree-issues is presented in Algorithm \ref{alg:resolve_degree_issues}.

The above process terminates (in exponential time, but here, for computer checks, we are not concerned about efficiency) even if some vertices $v$ have $\delta^{*+}(v)=\infty$.
To see this, first note that when we subdivide a range, the branches correspond to disjoint degree combinations. Also, when we branch,
we single out a specified single degree $\delta^*(v)\leq d_{Z^*}(v)$. Thus it suffices to bound the maximal possible degree. It is easy to see that the sum of boundary degrees can only decrease, and for inner vertices, the degree can only decrease.
Thus we can never get a degree that is bigger than the sum of the boundary degrees or the maximal in-degree in the input pseudo-triangulation $Z$.

Since we are just mimicking the single degree case, it should be clear that the branching leads to the desired set of 
configurations.  Thus we have

\begin{lem}\label{lem:free-homo-conf-ranges}
Assume that we are given a pseudo-configuration $(Z,\delta^-,\delta^+)$ with degree ranges and a set $S$ of dart identification requests.
We will construct configurations with degree
ranges $(Z_1^*,\delta^{*-}_1,\delta^{*+}_1),\ldots,(Z_k^*,\delta^{*-}_k,\delta^{*+}_k)$ with homomorphisms $\phi^*_i$ from $Z$ to $Z^*_i$ such that
$(Z^*,\delta^*)$ belongs
to some $(Z_i^*,\delta^{*-}_i,\delta^{*+}_i)$, 
$1\leq i\leq k$, if and only if
there is a $(Z,\delta)\in (Z,\delta^-,\delta^+)$ such that $(Z^*,\delta^*)$ is the free homomorphic image by $\phi^*_i$ of $(Z,\delta)$ respecting $S$ (as described in Lemma \ref{lem:free-homo-conf}). If there is no such $(Z^*,\delta^*)$, then the construction will tell that there are no solutions. 
\end{lem}

The pseudocode for the algorithm behind Lemma \ref{lem:free-homo-conf-ranges}
is found in Algorithm \ref{alg:homomorphism-conf}.

While this is not important for correctness, we note the
representation $(Z_1^*,\delta^{*-}_1,\delta^{*+}_1),\ldots,(Z_k^*,\delta^{*-}_k,\delta^{*+}_k)$ that we construct for Lemma \ref
{lem:free-homo-conf-ranges}, is minimal in the following sense.
For any $(Z,\delta)\in (Z,\delta^-,\delta^+)$ with
a homomorphic image respecting $S$, there is exactly
one homomorphism $\phi_i^*$ that is valid in the sense that
for every $v\in Z$, if $v^*=\phi_i^*(v)$ is internal to the image $Z^*_i$, then $d_{Z^*_i}(v^*)=\delta(v)$, and if $v^*$ is on the boundary of $Z_i^*$, then $d_{Z_i^*}(v^*)<\delta(v)$. In fact, $(Z_i^*,\delta^{*-}_i,\delta^{*+}_i)$ is exactly the union of the images of all
$(Z,\delta)\in (Z,\delta^-,\delta^+)$ for which $\phi^*_i$ is valid.

\subsubsection{Free combination of discharging rules with degree ranges}
\label{subsub:combine_rules}

We want to ask if
a given set of discharging rules 
$R_1,\ldots,R_\ell$ can possibly be combined to 
send charge over an oriented edge or dart $e=st$ in some triangulation $G$. 

Recall that a rule $R_i$ consists of a configuration with degree ranges $(Z_i,\delta^-_i,\delta^+_i)$, a charge $r_i$, and an oriented edge $s_it_i$. For the purpose of
combining rules, we will just view $s_it_i$ as a dart $e_i$ with $\tail(e_i)=s_i$ and $\head(e_i)=t_i$. Rule $R_i$ applies to $e=st$ in $G$ if and only if there is a homomorphism
$\phi'_i$ from $Z_i$ to $G$ such that $\phi'_i(e_i)=e$ and such that for each $v_i\in Z_i$,             $d_G(\phi'(v_i))\in [\delta^-(v_i),\delta^+(v_i)]$. If rules $R_1,\ldots,R_\ell$ all apply to $e$, then they
send charge $\sum_{i=1}^\ell r_i$ over $e$. 

To see if the rules $R_1,\ldots,R_\ell$ apply, we can equivalently look at the disjoint union
\[(Z^\sqcup,\delta^{\sqcup-},\delta^{\sqcup+})= (Z_1,\delta^-_1,\delta^+_1)\sqcup\cdots\sqcup(Z_\ell,\delta^-_\ell,\delta^+_\ell).\]
We then ask if it contains
a pseudo-configuration 
$(Z,\delta)\in (Z^\sqcup,\delta^{\sqcup-},\delta^{\sqcup+})$
with a homomorphism
$\phi'$ to $G$ mapping each $e_i$ to $e$. 

We now apply Lemma \ref{lem:free-homo-conf-ranges} to $(Z^\sqcup,\delta^{\sqcup-},\delta^{\sqcup+})$ with identification requests 
\[S=\{(e_i,e_{i+1})\mid i=1,\dots,\ell-1\}.\]
Since $\phi'$ respects 
$S$, the construction will be successful, resulting in free homomorphic images  
$(Z_1^*,\delta^{*-}_1,\delta^{*+}_1),\ldots,(Z_k^*,\delta^{*-}_k,\delta^{*+}_k)$, one of which contains the homomorphic image $(Z^*,\delta^*)$ of $(Z,\delta)$. Note that 
Lemma \ref{lem:free-homo-conf-ranges} also provides the free homomorphisms $\phi^*_i$ from $Z$ to $Z^*_i$.

We shall view the free combination of the discharging rules $R_1,\ldots,R_\ell$
as a set of 
combined discharging rules $R^*_i$ $(1\le i \le k)$ where
\[R^*_i = ((Z_i^*,\delta^{*-}_i,\delta^{*+}_i),e^*_i,r^*_i)\textnormal{ with } e^*_i = \phi_i^*(e_1)\textnormal{ and }r^*_i=\sum_{i=1}^\ell r_i.\]
Each of these $R^*_i$ is called a \emph{free combination} of $R_1,\ldots,R_\ell$.
If the above free combination of $R_1,\ldots,R_\ell$ fails in the sense of producing an empty set of combined discharging rules, then we know 
that these rules cannot be combined to send charge over any oriented edge $st$ in any triangulation $G$.
The converse is not true, for we only know that each $Z^*_i$ is a pseudo-triangulation. In principle, it may not even be planar.

Instead of combining many discharging rules directly as suggested above, we will add discharging rules to the combination, one at a time,
as justified by Lemma \ref{lem:free-homo-conf-trans}. For example, suppose that we have computed the combined discharging rules $R^*_1\ldots,R^*_k$ for $R_1\ldots,R_\ell$. To add a new rule $R_{\ell+1}$ to the combinations, for $i=1,\cdots,k$, we
compute the set of combined discharging rules for the two rules $R^*_i$ and $R_{\ell+1}$, and then we take the union of all these sets.

\drop{To upper-bound the charge along any oriented edge, we look for the free combination yielding the maximal charge (charge two for the first rule and charge one for every other rule). 
With 43 rules, 84 including two orientations of all but two of them, there could potentially be $2^{84}$ combinations to consider. What saves
us is that most discharging rules cannot be combined, and if some rules cannot be combined, then neither can any superset of these rules.}

Returning to the problem of computing the maximal discharge over any edge, 
let $R_1,\ldots,R_s$ be the complete set of discharging rules from Figure \ref{fig:rules}. For $i=0,\ldots,s$, we want to
construct $\cR_i$ consisting of all the free combinations of all the subsets of rules from $R_1,\ldots,R_i$. At the end, we will return the
maximal charge of any rule in $\cR=\cR_s$.

We define a ``neutral" discharging rule $R_0 = ((Z_0,\delta^-_0,\delta^+_0),e_0,0)$,
where $Z_0$ has two vertices $s_0$ and $t_0$ and two darts; namely $e_0=s_0t_0$ and its reverse $t_0s_0$.
Moreover, $\delta_0^-(s_0)=\delta^-_0(t_0)=1$ and $\delta^+_0(s_0)=\delta_0^+(t_0)=\infty$. 
Note that the free combination of $R_0$ and any rule $R_i$ is a set consisting of a rule isomorphic to $R_i$. 

We now initialize with 
$\cR^*_0=\{R_0\}$. Then, for $i=1,\ldots,s$,
we start by setting $\cR^*_i=\cR^*_{i-1}$. Next,
for every combined rule
$R^*_{i-1}\in \cR^*_{i-1}$, we
construct and add all free combinations of $R^*_{i-1}$ and $R_i$ to $\cR^*_i$. Finally, we set $\cR^*=\cR_s$.

In summary, we have proved the following.

\begin{lem}\label{lem:free-comb-rules}
Based on a set of discharging rules 
$\cR$, we can construct a
complete set $\cR^*$ of free combinations of subsets of discharging rules from $\cR$ such that if there are
some discharging rules $R_1,\ldots,R_\ell\in\cR$ that can be combined
to send charge $r$ over some dart $e$ in some triangulation $G$, then $R_1,\ldots,R_\ell$ has a free
combination $R^*\in\cR^*$
that sends this charge over $e$. More precisely, with
\[R^*=((Z^*,\delta^{*-},\delta^{*+}),e^*,r^*)\textnormal,\]
we have $r^*=r$ and a homomorphism $\phi^{*\prime}$ from $Z^*$ to $G$ mapping
$e^*$ to $e$ and each $v^*\in Z^*$ to 
a vertex $v$ in $G$ with $\delta^{*-}(v^*)\leq d_G(v)\leq \delta^{*+}(v)$. 
\end{lem}

The corresponding pseudocode is found in Algorithm \ref{alg:combine_rules}.

Together with the pseudocode we have pointers to the $\Cpp$ implementations. As stated in Lemma \ref{comp:vsends-noconf}, when this $\Cpp$ program was executed on our discharging rules depicted in Figure \ref{fig:rules}, the maximal edge charge found was 8. This completes the proof of 
{\bf Lemma \ref{lem:vsends-noconf}}.

\medskip

Later we shall need the following lemma on the possible ranges for our discharging rules from Figure \ref{fig:rules}.

\begin{lem}\label{lem:ranges}
For every free combination
\[R^*=((Z^*,\delta^{*-},\delta^{*+}),e^*,r^*)\]
of our discharging rules from Figure \ref{fig:rules}, for every vertex 
$v^*\in Z^*$, the degree range is one of the following three types:
\begin{itemize}
    \item a single degree, that is, 
    $\delta^{*-}(v^*)=\delta^{*+}(v^*)$,
    \item the range $[5,6]$, or
    \item an infinite range with $\delta^{*-}(v^*)<\delta^{*+}(v^*)=\infty$.
\end{itemize}
\end{lem}
\begin{proof}
The statement is true for all the original rules from Figure \ref{fig:rules}, and is also trivially closed under the intersection of intervals (when different vertices are identified). The only other way we get new intervals is by subdividing an interval, splitting off the lowest degree, but we still remain within the three types above.
\end{proof}

\section{Reducible configurations in pseudo-configurations}\label{sect:config-in-combine}

We shall often need to check if a reducible configuration $(Z,\delta)$ embeds into a given pseudo-configuration $(Z^*,\delta^*)$. 

\paragraph{Cut-vertices.}
So far, we have only defined homomorphisms for configurations without cut-vertices. However, a
configuration $(Z,\delta)$
may have a cut-vertex $v$ (in principle, there could be more cut-vertices, but we never have more than one in our reducible configurations in $\cD$). By (Z3) such a cut-vertex should
have exactly two outer neighbors, that is, $\delta(v)=d_Z(v)+2$. By (Z5), when we embed $(Z,\delta)$ into a triangulation, then these two outer neighbors should be non-consecutive around $v$ as illustrated in Figure \ref{fig:extended configuration}.

A different way of describing the requirements (Z3) and (Z5)
is that we consider 
extending 
$Z$ with one of the outer neighbors in one of the two possible positions between the two blocks of $Z$ incident to $v$. The
result $Z^u$ has no cut-vertex, and now any 
embedding of $Z^u$ into a triangulation includes
an embedding of $Z$ satisfying (Z3) and (Z5).

Using degree ranges, we 
extend $\delta$ to 
cover $u$. We have no
special degree restrictions on $u$ so we set $\delta^{u-}(u)=d_{Z^u}(u)+1$ and $\delta^{u+}(u)=\infty$. All other
vertices $v\in Z$ have $\delta^{u-}(v)=\delta^{u+}(v)=\delta(v)$. Now $(Z^u,\delta^{u-},\delta^{u+})$ is configuration without cut-vertices, so $(Z^u,\delta^{u-},\delta^{u+})$ is also a pseudo-configuration. 
Recall that we have two possible placements of the outer neighbor, hence two corresponding
extensions 
$(Z^{u_1},\delta^{u_1-},\delta^{u_1+})$ and  
$(Z^{u_2},\delta^{u_2-},\delta^{u_2+})$.

We now define that 
\emph{$\phi^*$ is a homomorphism from the configuration $(Z,\delta)$
with cut-vertex $v$ into a
pseudo-configuration $(Z^*,\delta^*)$}
if and only 
if $\phi^*$ can be extended to a homomorphism from
$(Z^{u_1},\delta^{u_1-},\delta^{u_1+})$ or $(Z^{u_2},\delta^{u_2-},\delta^{u_2+})$ into $(Z^*,\delta^*)$. The reason that we have to consider both of these extensions is that the original cut-vertex $v$ could map to a boundary vertex of $(Z^*,\delta^*)$, and then it would only be one of the outer neighbors of $v$ that would map into
$(Z^*,\delta^*)$. We note that this definition easily extends to cases with multiple cut-vertices, but with $c$ cut-vertices, we would have to consider $2^c$
different extensions. 

Having defined homomorphisms from configurations $(Z,\delta)$, it is easily seen that Lemma \ref{lem:1-1} holds for this case as well (nothing in the proof rules out cut-vertices except that homomorphisms were not defined with cut-vertices). Thus, allowing for cut-vertices, we have
\begin{lem}\label{lem:1-1-cut} \showlabel{lem:1-1-cut}
Let $(Z,\delta)$ be a configuration of diameter at most $4$. If 
$\phi$ is a homomorphism from $(Z,\delta)$ to $G$, and $G$ has no separating 3- or 4-cycle 
whose vertices are in $\phi(Z)$, then $\phi$ is 1-1 
(so $\phi$ is an embedding showing that $(Z,\delta)$ is contained in $G$).
\end{lem}
Often the question is if there is \emph{any} reducible configuration $(Z,\delta)\in \cD$
that has a homomorphism $\phi^*$ into $(Z^*,\delta^*)$. For uniformity of code, it is convenient to 
convert all $(Z,\delta)\in\cD$ to configurations with degree ranges $(Z',\delta^{\prime-},\delta^{\prime+})$. If $Z$
has no cut-vertex, we just
use $(Z',\delta^{\prime-},\delta^{\prime+})=
(Z,\delta,\delta)$. If $Z$ has a cut-vertex, we check both the extended configurations $(Z^{u_1},\delta^{u_1-},\delta^{u_1+})$ and $(Z^{u_2},\delta^{u_2-},\delta^{u_2+})$ described above.
When we have a configuration with degree
$(Z',\delta^{\prime-},\delta^{\prime+})$,
we need to check that the homomorphism $\phi^*$ satisfies $\delta^{\prime-}(v)\leq \delta^*(\phi^*(v))\leq \delta^{\prime+}(v)$ for all $v\in Z'$.

To facilitate the checking for homomorphisms, it is convenient to
select a special dart $f=xy$ for each
reducible configuration $(Z,\delta)$. We select it to maximize $(\delta(y), \delta(x))$ lexicographically.
When we have our reducible configuration with degree ranges $(Z',\delta^{\prime-},\delta^{\prime+})$,
we have $\delta^{\prime-}(x)=\delta^{\prime+}(x)=\delta(x)$ and $\delta^{\prime-}(y)=\delta^{\prime+}(y)=\delta(y)$. Since $Z'$ has no cut-vertex, a homomorphism $\phi^*$ from $Z'$ to $Z^*$ is uniquely determined by the image $f^*$ of $f$ in $Z^*$. The rest of $\phi^*$ is then determined by the homomorphic requirements.
Constructing $\phi^*$ from
$\phi^*(f)=f^*$ fails if it cannot map all of $Z'$ into $Z^*$. Assuming that $\phi^*$ was constructed successfully into $Z^*$, we have to check that $\delta^{\prime-}(v)\leq \delta^*(\phi^*(v))\leq \delta^{\prime+}(v)$ for all $v\in Z'$.

Note that we only have to check for $f^*\in Z^*$ such that $\delta^*(\head(f^*))=\delta(t)$ and $\delta^*(\tail(f^*))=\delta(s)$. This restriction on $f^*$ was quite important for speed.

\paragraph{Center vertices and centered reducible configurations.}\label{para:center-vertices}
The target pseudo-configuration $(Z^*,\delta^*)$ may have a distinguished \emph{center} vertex $c^*\in Z^*$.
Recall from Lemma \ref{lem:reduc} (D1) that
our reducible configurations $(Z,\delta)$ in $\cD$ have at most one vertex of degree above $8$. 
If $Z$ has such a high degree vertex $c$, $\delta(c)>8$, then we
want the homomorphism $\phi^*$ from $(Z,\delta)$ to $(Z^*,\delta^*)$ to map
$c$ to $c^*$. To emphasize
this property, we say that the 
homomorphism $\phi^*$ and the image
of $(Z,\delta)$ are \emph{centered}.

Note that the special edge $f=st$ from $(Z,\delta)$ was chosen above to maximize $\delta(t)$, so $t$ is the high degree vertex of $Z$ if any. If so, we should only consider mapping $f$ to darts $f^*$ with
$\head(f^*)=c^*$.

The centering is important when we have a homomorpism $\phi$ from a reducible configuration $(Z,\delta)$
 to an extended degree-bounded ball
$\bar B_2^8(v)$, for if the image
is centered in $\bar B_2^8(v)$,
then it must be contained in the non-extended $B_2^8(v)$ which is what we normally require.

The pseudocode to check if there is homomorphism $\phi^*$ from a reducible configuration $(Z,\delta)$ to a centered pseudo-configuration $(Z^*,\delta^*)$
is described in Algorithm \ref{alg:contain_conf}. 

\subsection{Configuration with degree ranges blocked by reducible configurations}\label{sect:block}

We say that our set of reducible configurations $\cD$ 
\emph{blocks} a configuration with degree ranges $(Z^*,\delta^{*-},\delta^{*+})$ and center $c^*$ if every $
(Z^*,\delta^*)\in(Z^*,\delta^{*-},\delta^{*+})$ contains the homomorphic image of some reducible configuration from $\cD$. Unfortunately there could be infinitely many $(Z^*,\delta^*)\in (Z^*,\delta^{*-},\delta^{*+})$
because degree ranges could be infinite if $\delta^{*-}(v^*)<
\delta^{*+}(v^*)=\infty$ for some $v^*\in Z^*$. However, we claim that for each $v^*\in Z^*$, we only need to consider
a bounded range, as described in the lemma below.

\begin{lem}\label{lem:bounded-ranges}
Consider a configuration  $(Z^*,\delta^{*-},\delta^{*+})$
with degree ranges and center $c^*$. Define
$\delta^{*\pm}$ such
that $\delta^{*\pm}(c^*)=\delta^{*-}(c^*)$ if $\delta^{*+}(c^*)\leq 12$; otherwise $\delta^{*\pm}(c^*)=\delta^{*+}(c^*)$. For vertices $v^*\neq c^*$,
$\delta^{*\pm}(v^*)=\delta^{*-}(v^*)$ if $\delta^{*+}(v^*)\leq 8$; otherwise $\delta^{*\pm}(v^*)=\delta^{*+}(v^*)$. Then
$(Z^*,\delta^{*-},\delta^{*+})$
is
blocked by $\cD$ if $(Z^*,\delta^{*\pm},\delta^{*+})$
is blocked by $\cD$.
\end{lem}
\begin{proof}
Suppose that $(Z^*,\delta^{*\pm},\delta^{*+})$
is blocked by $\cD$, and consider any $(Z^*,\delta^*)\in (Z^*,\delta^{*-},\delta^{*+})$.
Define $\delta^{*\prime}$ such that
$\delta^{*\prime}(v)=\delta^*(v)$ if $\delta^*(v)\geq \delta^{*\pm}$; otherwise,
$\delta^{*\prime}(v)=\delta^{*\pm}(v)$. Then $(Z^*,\delta^{*\prime})\in (Z^*,\delta^{*\pm},\delta^{*+})$. Therefore, by assumption,
we have a homomorphism $\phi^*$
from some $(Z,\delta)\in \cD$
to $(Z^*,\delta^{*\prime})$.
However, if for any $v^*\in Z^*$,
we have $\delta^{*\prime}(v^*)>\delta^*$, then 
$\delta^{*\prime}(v^*)>12$ if
$v^*=c^*$; otherwise $\delta^{*\prime}(v^*)>8$. In either case, we conclude that
$v^*$ is not the image of any $v\in Z$. This is clear in the
case where $v^*\neq c^*$, and for the center, it follows because reducible configurations from $\cD$ have no vertices of degree above 12 (c.f. Lemma \ref{lem:reduc} (D1)). We conclude that $\phi^*$ is also
a homomorphism from $(Z,\delta)$ to $(Z^*,\delta^*)$.
The same consideration holds for the center $c^*$.
\end{proof}

The pseudocode to check if a configuration with degree ranges is blocked by $\cD$ 
is described in Algorithm \ref{alg:block_by_reduc}.

\subsection{The maximal charge over an edge without reducible configurations}
\label{sub:max_without_reduc}

We want to bound the maximal possible discharge
over an oriented edge $uv$ assuming that $B^8_2(v)$ has no separating 3- or 4-cycle and that no reducible configurations
from $\cD$. 
The discharge $r$ over $uv$ is determined by the concrete discharging rules $R_1,\ldots,R_\ell$
embedding into $G$ sending charge over $uv$. Recall from Lemma \ref{lem:free-comb-rules} that we constructed a
complete set $\cR^*$ of free combinations of subsets of discharging rules from $\cR$. Then $R_1,\ldots,R_\ell$
has a free combination $R^*\in\cR^*$
that sends the charge $r^*=r$ over $uv$. More precisely, with
\[R^*=((Z^*,\delta^{*-},\delta^{*+}),e^*,r^*)\textnormal,\]
we have $r^*=r$ and a homomorphism $\phi^{*\prime}$ from $Z^*$ to $G$ mapping
$e^*$ to $uv$ and each $w^*\in Z^*$ to 
a vertex $w$ in $G$ with $\delta^{*-}(w^*)\leq d_G(w)\leq \delta^{*+}(w^*)$. 

We shall always use the head of a  (combined) discharging rule as the center 
of its configuration, that is, above
the center of $Z^*$ is $c^*=\head(e^*)$ the
which is mapped to $v$. Then $Z^*$ maps into
$\bar B^8_2(v)$ because every rule discharging
over $uv$ is mapped into $\bar B^8_2(v)$ (c.f. proof of Lemma \ref{limitcartwheel}).

We now now define $(Z^*,\delta^*)\in (Z^*,\delta^{*-},\delta^{*+})$ such that 
$\delta^*(w^*)=d_G(\phi^{*\prime}(w^*))$
for every $w^*\in Z^*$.
Suppose there is a reducible configuration
$(Z,\delta)\in \cD$ with a homomorphism $\phi^*$ into $(Z^*,\delta^*)$ mapping a vertex with degree above 8, if any, to the center $c^*=\head(e^*)$. Then we have the situation discussed above where $\phi^{*\prime}\circ \phi^*$ maps $(Z,\delta)$ to $B^8_2(v)$.
By assumption $B^8_2(v)$
has no separating 3- or 4-cycles, so $(Z,\delta)$ is embedded into $B^8_2(v)$ contradicting
that $B^8_2(v)$ has no reducible configuration. We conclude conversely that if $R^*=((Z^*,\delta^{*-},\delta^{*+}),e^*,r^*)$ is the free combination of the discharging rules sending charge over $uv$, then its configuration  $(Z^*,\delta^{*-},\delta^{*+})$ with center $\head(e^*)$ is not blocked by $\cD$. 

For brevity, we say that a \emph{(combined)
discharging rule $R^*=((Z^*,\delta^{*-},\delta^{*+}),e^*,r^*)$ is blocked by $\cD$} if
its configuration  $(Z^*,\delta^{*-},\delta^{*+})$ with center $\head(e^*)$ is blocked by $\cD$. 
We conclude that
the maximal possible discharge
over an oriented edge $uv$ assuming that $B^8_2(v)$ has no separating 3- or 4-cycle and no reducible configurations
from $\cD$ is bounded by the maximal
discharge $r^*$ in a freely combined discharging rule $((Z^*,\delta^{*-},\delta^{*+}),e^*,r^*)\in \cR^*$ 
that is not blocked by $\cD$.

\paragraph{Checking if a combined discharging rule is blocked.}
Recall from Lemma \ref{lem:ranges} that for any free combination  $((Z^*,\delta^{*-},\delta^{*+}),e^*,r^*)$
of our concrete discharging rules $\cR$ from Figure \ref{fig:rules}, for every vertex 
$v^*\in Z^*$, the degree range is one of the following three types:
\begin{itemize}
    \item a single degree, that is, 
    $\delta^{*-}(v^*)=\delta^{*+}(v^*)$,
    \item the range $[5,6]$, or
    \item an infinite range with $\delta^{*-}(v^*)<\delta^{*+}(v^*)=\infty$.
\end{itemize}
Recall from Lemma \ref {lem:bounded-ranges} that to check if 
$(Z^*,\delta^{*-},\delta^{*+})$ is blocked, it suffices
to check if $(Z^*,\delta^{*\pm},\delta^{*+}))$is blocked. In the above case with  $\delta^{*+}=\infty$, we get $\delta^{*\pm}=\delta^{*+}(v^*)$. It follows that the maximal degree range in 
$(Z^*,\delta^{*\pm},\delta^{*+}))$
is 2. It is thus a finite problem to check if a combined discharging
rule $((Z^*,\delta^{*-},\delta^{*+}),e^*,r^*)$ is blocked by $\cD$.

\paragraph{Generating all non-blocked freely combined discharging rules.}
Finally, we note that if the discharging rule $R^*=((Z^*,\delta^{*-},\delta^{*+}),e^*,r^*)$
is blocked by $\cD$, then so is any free
combination of $R^*$ and other discharging rules. Instead of first generating the complete set of discharging rules $\cR^*$ from Lemma \ref{lem:free-comb-rules}, and then dropping those that are blocked by $\cD$, it is more efficient to only
generate those that are not blocked as follows.

Let $R_1,\ldots,R_s$ be all discharging rules, and let $R_0$ be the neutral discharging rule. 

For each $i=0,\ldots,s$, we want 
to construct $\cR^{*-\cD}_i$ consisting of all free combinations not blocked by $\cD$ for all subsets of rules from $R_1,\ldots,R_i$. At the end, we just return the maximal charge found in 
$\cR^{*-\cD}_s$

First, we set $\cR^{*-\cD}_0=\{R_0\}$.
Then, for $i=1,\ldots,s$,
we start by setting $\cR_i=\cR_{i-1}$. Next,
for every combined rule
$R^*_{i-1}\in \cR^{*-\cD}_{i-1}$, 
we construct all the free combinations $R^*=((Z^*,\delta^{*-},\delta^{*+}),e^*,r^*)$
of $R^*_{i-1}$ and $R_i$.
If $(Z^*,\delta^{*-},\delta^{*+})$ is not blocked by $\cD$,
then we add $R^*$ to
$\cR^{*-\cD}_i$. Finally, we set $\cR^{*-\cD}=\cR^{*-\cD}_s$.
In summary, we have proved:
\begin{lem}\label{lem:free-comb-rules-noD}
Based on a set of discharging rules 
$\cR$ from Figure \ref{fig:rules}, we can construct a
complete set $\cR^{*-\cD}$ of free combinations not blocked by $\cD$ of subsets of discharging rules from $\cR$.
\end{lem}
The pseudocode for this is found in Algorithm \ref{alg:combine_rules}. Together with the pseudocode, we have pointers to the $\Cpp$ implementations. As stated in Lemma \ref{comp:vsends}, 
when executing this $\Cpp$ program on our discharging rules depicted in Figure \ref{fig:rules} together with the reducible configurations from $\cD$, then the maximal discharge found was 5. 
This completes the proof of {\bf Lemma \ref{lem:vsends}.}

\section{Free cartwheels}
\label{sect:free-cartwheel}
\showlabel{sect:free-cartwheel}
We will now make a generic description of what we call \emph{free cartwheels} and \emph{free extended degree-bounded cartwheels}. They are basically like the cartwheels in Figure \ref{fig:cartwheels} where no vertices have been identified.
These are configurations that will be used in our computer checks to model the cartwheels $B_2(v)$ and the extended degree-bounded cartwheels $\bar B^8_2(v)$ in any near-triangulation $G$, assuming that there are no vertices of degree less than 5 (lower degree vertices are themselves reducible configurations by Lemma \ref{lem:3-4}).

We start from a single \emph{center} vertex $v$
with some degree $\delta(v)\geq 5$.

Next, we do the free completion to get
the \emph{wheel} around $v$. It is the near-triangulation $W(v)$ consisting of $v$ in the center with its neighbors $v_1,\ldots,v_{\delta(v)}$ around
the boundary in this clockwise order.
To get a configuration, we need to select degrees $\delta(v_i)$ for the neighbors.

Next, we do the free completion again, to construct a generic \emph{free cartwheel $C(v)$ around $v$}, which is also a near-triangulation. We
start from the above wheel $W(v)$. Next, for each neighbor $v_i$, we pick
a degree $\delta(v_i)\geq 5$.
The free cartwheel $C(v)$ is the union $\bigcup_{u\in V(W(v))}W(u)$ of the wheels around these vertices. This includes degrees for all the new vertices outside $W(v)$. The wheels only intersect minimally
as indicated in Figure \ref{fig:cartwheels}.
In particular, if two vertices on the neighboring cycle of $v$ have distance (on that cycle) more than 2, their wheels intersect only in $v$. If their distance is precisely 2, their wheels intersect in precisely an edge including its ends. If their distance is 1, their wheels intersect in a diamond. The two latter cases are true because $\delta(v)\geq 5$.

For a \emph{free extended degree-bounded cartwheel} $\bar C^8(v)$, we use the same generation
as above, but we only add the 
wheel around a $v_i$ if its degree is at most 8, that is, $\bar C^8(v)=W(v)\cup \bigcup_{u\in \{v_1,\ldots,v_{\delta(v)}\}, \delta(v_i)\leq 8}W(v_i)\subseteq C(v)$.   Clearly,
\begin{fact}\label{lem:nice-cartwheels}
    Free cartwheels and free extended degree-bounded cartwheels are configurations and they have no separating 3- or 4-cycles.
\end{fact}
The following lemma essentially states that the free extended degree-bounded cartwheels $\bar C^8(v)$ can be used to model all the extended degree bounded cartwheels $\bar B_2^8(v)$ that we may find in any triangulation.

\begin{lem}\label{lem:cartwheel}\showlabel{lem:cartwheel}
For any vertex $v$ in a triangulation $G$, suppose that the  degree-bounded cartwheel $B^8_2(v)$ has no vertices of degree less than 5. Then there is a free extended degree-bounded cartwheel $\bar C^8(v')$ with a homomorphism $\phi$ to the extended degree-bounded cartwheel $\bar B^8_2(v)$. The homomorphism is centered in the sense of mapping $v'$ to $v$.  Moreover,
the final charge of $v'$ in $\bar C^8(v')$
is the same as the final charge of $v$ in $\bar B^8_2(v)$.
\end{lem}

\begin{proof}
   In this proof, we shall repeatedly  use the growth procedure is Section \ref{sect:grow} and we shall use the concepts and terminology from Section \ref{sect:grow}.

   We will grow $\bar C^8(v')$, fixing the degrees $\delta_C$ in $\bar C^8(v)$ as we go along so that we maintain a homomorphism $\phi$ into $G$. Specifically, we    fix an arbitrary edge $vw$ out of $v$ in $G$. We make a corresponding edge $v'w'$ in $\bar C^8(v')$, setting $\phi(v'w')=vw$, $\delta_C(v')=d_G(v)$ and $\delta_C(w')=d_G(w)$.
   We now grow $\bar C^8(v')$ 
   by first growing the wheel around $v'$, second growing the remaining parts of the wheels around neighbors of degree at most $8$. The crucial point is that every time we add a new neighboring triangle  $x'y'z'$ corresponding to triangle $xyz$ in $G$, we set $\delta_C(z')=d_G(z)$. The edge-closures can never go wrong because the degrees are correct. 

   Next we claim that $\phi$ maps $\bar C^8(v')$ into $\bar B^8_2(v)$. This follows because any path from $v'$ in $\bar C^8(v')$
   is mapped to a path on vertices with exactly the same degrees in $G$. Since every vertex in $\bar C^8(v')$ can be reached by a path of length at most two where all interior vertices have degree at most 8, implying that the path stays in $\bar B^8_2(v)$.

    Below we identify $v'$ in $\bar C^8(v')$ with $v$. We will now prove that  the final charge of $v$ in $\bar C^8(v)$ is the same as of $v$ in $G$. Since $v$ is interior in $\bar C^8(v)$ and has the same degree as in $G$, $\phi$ is a 1-1 map between neighbors in $\bar C^8(v)$
   and in $G$ (see also Lemma \ref{lem:locally 1-1}). We can therefore view them as identified 
   with $v_1,\ldots v_{d_G(v)}$, respectively. We now want to argue that each edge $vv_i$ gets the same discharge in both directions in $\bar C^8(v)$
   and in $G$.

    Suppose a discharging rule $R$ from Figure \ref{fig:rules} applies to $vv_i$ (or $v_iv$) in $\bar C^8(v)$ via some homomorphism $\psi_C$ which maps the discharging edge $st=s(R)t(R)$ of $R$ into $vv_i$ (or $v_iv$). 
    Then it clearly also applies in $G$ using the composed homomorphism $\phi\circ \psi_C$. 

    For the other direction, suppose have a homomorphism $\psi_G$ applying $R$ to $vv_i$ (or $v_iv$) in $G$. We need to prove that $R$ also applies to $vv_i$ (or $v_iv$) in $\bar C^8(v)$.
    For each vertex $x\in V(R)$, we know that $d_G(\psi_C(x))\in [\delta_R^-(x),\delta_R^+(x)]$. Let $Q=G(R)$ be the near-triangulation of $R$,
    and define $\delta_Q(x)=d_G(\psi_C(x))$.  Then $R$ applies to $vv_i$ in 
    $\bar C^8(v)$ if the configuration $(Q,\delta_Q)$ has a homomorphism $\psi_C$ into 
    $\bar C^8(v)$ mapping $st$ into $vv_i$. 

    We will now try to grow the homomorphism $\psi_C$.   Starting from the oriented edge $st$, we try to grow a sub-near-triangulation $Q'$ of $Q$ together with
    the homomorphism $\psi_C$ in to $\bar C^8(v)$. If we fail, 
    it is because something went wrong when we added  a neighboring triangle $xyz$ to $Q'$ where $xy$ was the boundary edge and $z$ the new in $Q'$. This would be either an out-of-bounds error or a degree error.

    Let us first assume that it is a degree error, that is, with $z_C=\psi_C(z)$, we have  $\delta_C(z_C)\neq \delta_Q(z)$.

    Now $\psi_C$ is a homomorphism from $Q'$ to $\bar C^8(v)$ which by Fact \ref{lem:nice-cartwheels} has no separating 3- or 4-cycles, so by Lemma \ref{lem:1-1}, $\psi_C$ is 1-1. It follows that $\psi_C$ is an isomorphism between $Q'$ and its image $C'=\psi_C(Q')$, but then the composition 
    $\psi_G\circ\psi_C^{-1}$ is a homomorphism from $C'$ to $G$ mapping $vv_i$ to itself. However, there is only one such homomorphism, so     $\psi_G\circ\psi_C^{-1}$ must coincide with $\phi$.
    With $z_G=\phi(z_C)=\psi_G(z)$, we now have a contradiction, for $\phi$ is a homomorphism implying that $d_G(z_G)=\delta_C(z_C)$ contracting that $\delta_C(z_C)\neq \delta_Q(z_C)=d_G(z_G)$. Thus we conclude that our first error cannot be a degree error.

    To rule out an out-of-bound error, let us be a bit more careful with ordering in which we grow $Q'$ from $st$ inside $Q$. We will basically follow the same order in which we grow $\bar C^8(v)$.
    More precisely, when applying $R$ to $G$, we set
    $\psi_G(st)$ to either $v_iv$ or $vv_i$, and when
    growing $Q'$ we start with $st$ setting $\psi_C(st)=\psi_G(st)$. Thus $\psi_C^{-1}(v)$ is either $s$ or $t$. 
    
    Now, inspecting our discharging rules in Figure \ref{fig:rules}, we note that we can always grow $Q'$ to $Q$ by first adding all triangles incident to $\psi_C^{-1}(v)$ in $Q$, next adding all triangles in $Q$ incident to degree-8 neighbors of $\psi_C^{-1}(v)$. In Figure \ref{fig:rules}, we see that this eventually will cover all of $Q$. We grow the homomorphism $\psi_C$ into $\bar C^8(v)$ accordingly. We already proved that our first failure cannot be a degree error, so if we fail, it should be an out-of-bound error for the latest vertex $z$ added to $Q'$. However, with this growth order, we must have a path $P$ from $\psi_C^{-1}(v)$ to $z$ in $Q'$ of length at most two, where the interior vertex, if any, is of degree at most $8$, and $\psi_C$ maps this path, except $z$, into a path in $\bar C^8(v)$ with the same degrees. By definition of 
    the extended degree-bounded cartwheel, all neighbors of the predecessor of $z$ in $P$ should be in $\bar C^8(v)$, contradicting that $z$ is out of bounds.
\end{proof}

\subsection{Free cartwheels with  limited degrees}
\label{sect:free-cw-enum}
We will now describe a finite computer check to prove statements such as Lemma \ref{lem:positive-comp}. The lemma concerns the extended degree-bounded ball $\bar B^8_2(v)$. It states that if certain degree conditions are satisfied for the vertices in $\bar B_2^8(v)$ (degrees are measured in the underlying graph $G$), and $v$ gets a certain final charge (at least 0), then either $\bar B^8_2(v)$ contains a separating 3- or 4-cycle or a centered reducible configuration from $\cD$.
Recall from Section \ref{sect:config-in-combine} that centered implies that the reducible configuration is contained in $B^8_2(v)$.
To prove such a statement, we claim that it suffices to
check all free extended degree-bounded cartwheels $\bar C^8(v)$ satisfying the same constraints: check if $\bar C^8(v)$ has the right final charge for $v$, and if so, check that $\bar C^8(v)$ contains a reducible configuration from $\cD$.
However, the boundary degrees in $\bar C^8(v)$ are unbounded.

By the following lemma, it suffices to consider free cartwheels with non-center degrees at most $9$. 
\begin{lem}\label{lem:max9}
Consider a free extended degree-bounded cartwheel $C$ centered at $v$. Suppose $C$ has a vertex $u$ $(u\neq v)$ of degree $\delta(u)>9$ and let $C'$ be the result of setting  $\delta'(u)=9$. 
Then $v$ has exactly the same charge in $C$ and $C'$ and $C$ and $C'$ contain exactly the same centered reducible configurations from $\cD$.    
\end{lem}
\begin{proof} First, we note that by the construction of extended degree-bounded cartwheels, $u$ has to be a boundary vertex, and it must have $d_C(u)\leq 5$, so decreasing $\delta(u)$ to 9 cannot
violate the constraint that $d_C(u)<\delta(u)$.

Finally, we note that no discharging rules distinguish between degrees above 9, and neither do any reducible configurations, since $u$ is not the center.
\end{proof}

For every extended degree-degree bounded ball $\bar B^8_2(v)$ in a triangulation $G$, there is a unique free extended degree-bounded ball $C$ centered at $v$ and with fixed degrees that homomorphically maps to $\bar B^8_2(v)$. More precisely, there is a homomorphism $\phi$ such that the degree of $w\in C$ is $d_G(\phi(w))$. However, we define the free extended degree-bounded ball \emph{corresponding} to $\bar B^8_2(v)$ as the one $C'$ where we take any non-center degree $\delta(u)$ bigger than 9 and reduce it to $\delta'(u)=9$ (for all other vertices, the degrees are unchanged so $\delta'(u)=\delta(u)$).
\begin{lem}\label{lem:cartwheel-x}\showlabel{lem:cartwheel-x}
Consider the extended degree-bounded cartwheel $\bar B^8_2(v)$ of a vertex $v$ in a triangulation $G$. Suppose $\bar B^8_2(v)$ has no degrees less than 5. Let $\bar C^{\prime 8}(v)$ be the corresponding extended degree-bounded cartwheel
(with non-center degrees at most 9). Then
$v$ has the same final charge in $\bar C^{\prime 8}(v)$ as in $\bar B^8_2(v)$, and 
if $B^8_2(v)$ has no separating 3- or 4-cycle and no reducible configuration from $\cD$, then $\bar C^{\prime 8}(v)$ has no centered reducible
configuration from $\cD$.
\end{lem}
\begin{proof}
The proof is a straightforward combination of Lemma \ref{lem:cartwheel} and \ref{lem:max9}. First we apply 
Lemma \ref{lem:cartwheel} and get a free extended degree-bounded cartwheel $\bar C^8(v)$ with a homomorphism $\phi$ to the extended degree-bounded cartwheel $\bar B^8_2(v)$. The final charge of $v$ in $\bar C^8(v)$
is the same as the final charge of $v$ in $\bar B^8_2(v)$. Moreover, $\bar C^{\prime 8}(v)$ has a centered reducible
configuration $(Z,\delta)$ from $\cD$, then $\phi$ maps to $(Z,\delta)$ to $B^8_2(v)$.
By Lemma \ref{lem:1-1}, if $B^8_2(v)$ has no separating 3- or 4-cycle, then $\phi$ is 1-1 implying that $(Z,\delta)$ is contained $B^8_2(v)$. 

By Lemma \ref{lem:max9}, these relationships between $\bar C^8(v)$ and $\bar B^8_2(v)$ are preserved for $\bar C^{\prime 8}(v)$ where
we take all boundary degrees above 9 in $\bar C^8(v)$ and
reduce them to $9$.
\end{proof}
With Lemma \ref{lem:cartwheel-x}, it is clear that if we want to show for all  extended degree-bounded balls $\bar B^8_2(v)$ satisfying certain degree conditions (not distinguishing between degrees above 9)
where $v$ gets a certain final charge (at least 0) that $\bar B^8_2(v)$ contains a separating 3- or 4-cycle or a reducible configuration from $\cD$,
then it suffices to prove that all free extended degree-bounded cartwheels $\bar C^8(v)$ with non-center degrees at most 9 satisfying the same degree constraints and getting the same final charge have no
reducible configuration from $\cD$.
We can skip vertices of degree 3 and 4 
since they are themselves reducible in $\cD$ so all non-center vertices have degrees between 5 and 9. Moreover, to prove Lemma \ref{lem:positive-comp}, we only need to consider centers with degrees between 7 and 11. For brevity, we shall refer to these as \emph{free cartwheels with limited degrees.}

Note that there are only finitely many free cartwheels with limited degrees, but still there are far too many for a computer to consider. For example, if the center has degree 11 and degrees in the first neighborhood are 8, we have 44 vertices in the second neighborhood, each of which having 5 degrees to chose from, leaving us $5^{44}>2^{100}$ options which is way too many for any computer today.

\subsection{Enumerating bad cartwheels with tail ranges}
\label{sub:overcharged_cartwheel_enumeration}

A \emph{bad cartwheel} is a free cartwheel with limited degrees that has no centered reducible configuration in $\cD$ and where the final center charge is non-negative. Any counterexample to Lemma \ref{lem:positive-comp} would correspond to a bad cartwheel. Other bad cartwheels are important to the proof of Lemma~\ref{lem:zero-deg7,8}, and we would like the computer to enumerate all bad cartwheels\footnote{The program to enumerate bad cartwheels was also used to identify reducible configurations for our set $\cD$. More precisely, we iteratively asked the computer to find bad cartwheels, looking for reducible configurations within them to add to the set $\cD$. For this reason, the code has been tested and tried many times.}.

Again, we have an issue with the numbers being too large. To handle this, we will use degree
ranges. The ranges will only be used for vertices $w$ in the second neighborhood of the center $v$, and we only use degree ranges
of the form $[d,9]$ for some $d \in \{5,6,7,8\}$.
We call them \emph{tail ranges}.
These can be represented as configurations $(Z_C,\delta_C^-,\delta_C^+)$ with degree ranges.
We refer to them as \emph{free cartwheels with limited degrees and tail ranges}, noting that every $(Z_C,\delta_C) \in (Z_C,\delta_C^-,\delta_C^+)$ is a free cartwheel with fixed limited degrees.

The output of our enumeration will not be individual bad cartwheels, but a set of 
free cartwheels with limited degrees and tail ranges that is guaranteed to include all bad cartwheels. If, for example, none of these have a center degree below 9, then the first case of Lemma \ref{lem:positive-comp} follows.

Our algorithm iteratively refines cartwheels with tail ranges, considering different fixings of the degrees within the ranges. In general, we want to refine as little as possible to avoid generating too many combinations.
What drives the refinements is the need to identify reducible configurations or determine which discharging rules could apply. We will prune, removing configurations 
that cannot contain bad cartwheels if either we conclude that the final charge must be negative or if they are blocked by reducible configurations in $\mathcal{D}$.

We note here that what is special about only using tail ranges is that when we
want to check if  $(Z_C,\delta_C^-,\delta_C^+)$ is blocked by $\cD$, then by Lemma \ref{lem:bounded-ranges}, we can replace each tail range by a single degree 9, and then all degrees are fixed when we look for reducible configurations.
While this keeps the complexity down, we also note that a non-center vertex of degree 9 is not in any reducible configuration, so we cannot have too many of them.

\drop{Note that we can cut 
$(Z_C,\delta_C^-,\delta_C^+)$ from our output if we can guarantee that
no $(Z_C,\delta_C) \in (Z_C,\delta_C^-,\delta_C^+)$ is bad in the sense
of having non-negative final center charge and no centered reducible configuration from $\cD$.
Such a guarantee  typically requires us 
to limit the degree ranges to fixed degrees so that we rule out reducible configurations and better see which discharging rules apply.}

\newcommand{\cRnoD}{\mathcal{R}^{*-\mathcal{D}}\xspace}
\newcommand{\cartwheelRange}{(Z_C, \delta_C^-, \delta_C^+)\xspace}

\subsubsection{Charges along an edge and a bound on the final charge}
\label{sect:hom-to-ranges}

Suppose we have discharging rule $R = ((Z,\delta^-,\delta^+),e,r)$ and free cartwheel with limited degrees and tail ranges $(Z_C,\delta_C^-,\delta_C^+)$ with a specified dart $e_C$. Considering all
the fixed degree versions
$(Z_C,\delta_C) \in (Z_C,\delta_C^-,\delta_C^+)$, we will want to know if $R$ never applies, sometimes applies, or always applies to $e_C$.

We say that a rule $R$ \emph{always applies} to a dart $e_C$ in $\cartwheelRange$ if there exists a homomorphism $\phi^*$ from $Z$ to $Z_C$ satisfying $\phi^*(e)=e_C$ and for every $v$ in $Z$,
\[
[\delta^-(v), \delta^+(v)] \supseteq [\delta_C^{-}(\phi^*(v)), \delta_C^{+}(\phi^*(v))]
\]
holds.
This means that $R$ applies to $e_C$ for every $(Z_C,\delta_C) \in (Z_C, \delta_C^-, \delta_C^+)$.
We say that a rule $R$ \emph{sometimes applies} to a dart $e_C$ in $\cartwheelRange$ if there exists a homomorphism $\phi^*$ from $Z$ to $Z_C$ satisfying $\phi^*(e)=e_C$ and for every $v$ in $Z$,
\[
[\delta^-(v), \delta^+(v)] \cap [\delta_C^{-}(\phi^*(v)), \delta_C^{+}(\phi^*(v))] \neq \emptyset
\]
holds.
This means that $R$ applies to $e_C$ for some $(Z_C,\delta_C) \in (Z_C, \delta_C^-, \delta_C^+)$.
Using the above definitions, we also say that a rule $R$ \emph{never applies} to a dart $e_C$ in $\cartwheelRange$ if $R$ does not sometimes apply to a dart $e_C$.
This means that $R$ does not apply to $e_C$ for any $(Z_C,\delta_C) \in (Z_C, \delta_C^-, \delta_C^+)$.
The pseudocode for checking a rule $R$ always/never applies to a specified dart in a given pseudo configuration is presented in Algorithms \ref{alg:ineivtably_apply} and \ref{alg:never_apply}, respectively.

\paragraph{Pruning by an upper bound of the final charge.}
As an application of the above definitions, we explain how to calculate the upper bound of the final charge for all $(Z_C, \delta_C) \in (Z_C, \delta_C^-, \delta_C^+)$, given a free cartwheel with limited degrees and tail ranges $(Z_C, \delta_C^-, \delta_C^+)$.
Let $v$ be the center of $Z_C$ and $v_1, \ldots, v_d$ be the neighbors of the center.

For each $1 \leq j \leq d$, we calculate the upper bound of the amount of charge sent from the neighbor $v_j$ to the center $v$ by taking the maximum of $r^*$ for all combined rules $((Z^*, \delta^{*-}, \delta^{*+}), e^*, r^*) \in \cRnoD$ that sometimes apply to $v_jv$ in $\cartwheelRange$.
We denote this value by $r_j$.
Note that we do not use all combined rules $\mathcal{R}^*$ in Lemma \ref{lem:vsends-noconf} but use only those in $\cRnoD$ in Lemma \ref{lem:vsends} because we only consider free cartwheels with limited degrees and tail ranges that contain no centered reducible configuration from $\mathcal{D}$.
The pseudocode corresponding to the algorithm is presented in Algorithm \ref{alg:amount_of_possible_charge_send}.

For each $1 \leq j \leq d$, we also calculate the lower bound of the amount of charge sent from the center $v$ to the neighbor $v_j$ by taking the sum of $r$ for all rules $((Z, \delta^-, \delta^+), e, r) \in \mathcal{R}$ that always apply to $vv_j$.
We denote this value by $s_j$.
The pseudocode corresponding to the algorithm is presented in Algorithm \ref{alg:amount_of_charge_send}.

Then, combining the initial charge $10 \cdot (6-d)$, we calculate $10 \cdot (6-d) + \sum_{1 \leq j \leq d} r_j - \sum_{1 \leq j \leq d} s_j$ as an upper bound of the final charge.
During the process of enumerating bad cartwheels, we can cut the current free cartwheel with limited degrees and tail ranges $\cartwheelRange$ if this value is negative.
The following lemma shows that this pruning method is valid.

\begin{lem}
\label{lem:estimate_charge_easy}
    Let $\cartwheelRange$ be a free cartwheel with limited degrees and tail ranges.
    Let $(Z_C, \delta_C) \in \cartwheelRange$ be a free cartwheel with limited degrees that contains no centered reducible configurations in $\mathcal{D}$.
    Then, the final charge of the center $v$ in $(Z_C, \delta_C)$ is at most $10 \cdot (6-d) + \sum_{1 \leq j \leq d} r_j - \sum_{1 \leq j \leq d} s_j$ calculated as above.
\end{lem}
\begin{proof}
    For $1 \leq j \leq d$, let $\mathcal{R}_j$ be the set of rules applied to $v_jv$ in $(Z_C, \delta_C)$.
    Let $R^*_j = (Z^*_j, \delta_j^{*-}, \delta_j^{*+}, e_j^*, r_j^*) \in \mathcal{R}^*$ be the combined rule of $\mathcal{R}_j$.
    Then, some $(Z^*_j, \delta_j^*) \in (Z^*_j, \delta_j^{*-}, \delta_j^{*+})$ has a homomorphic image into $(Z_C, \delta_C)$ with $e_j^*$ mapping to $v_jv$.
    Since $(Z_C, \delta_C)$ contains no centered reducible configuration in $\mathcal{D}$, the same holds for $(Z^*_j, \delta_j^*)$ with center $\head(e_j^*)$.
    Thus, $R^*_j \in \cRnoD$ holds and $R^*_j$ sometimes applies to $v_jv$ in $\cartwheelRange$, so $r_j \geq r^*_j$.
    Since the set of rules applied to $vv_j$ is a superset of the rules always applied to $vv_j$ in $\cartwheelRange$, $s_j$ is the lower bound of the amount of charge sent from $v$ to $v_j$ in $C$.
    Then, the claim holds.
\end{proof}

\subsubsection{Enumerating degrees of neighbors}
\label{subsub:enum-degree-neighbor}
We describe the procedure to enumerate free cartwheels with limited degrees with tail ranges that include all bad cartwheels.
First, we iterate through the possible degrees of center $v$.
Specifically, we choose $d \in \{7,8,9,10,11\}$ and set $\delta_C^-(v)=d$ and $\delta_C^+(v)=d$.
Second, we enumerate all possible assignments of degree for neighbors of $v$.
Since there are five choices $5,6,7,8$, or $9$ for each neighbor, there are $5^{d}$ possible combinations in total.
However, we only consider the assignments unique up to rotation, reducing the number of combinations\footnote{The number is calculated by $\frac{1}{d}\sum_{i=1}^d 5^{\gcd(i,d)}$ by Burnside's lemma \cite{burnside1909theory}, also called Cauchy-Frobenius lemma \cite{frobenius}.}.
We will use $v_1, \ldots, v_{d}$ to denote neighbors of $v$.
Fixing the degrees of $v_i$ determines the shape of $C$, as described in Section \ref{sect:free-cartwheel}.
Recall that we only add second neighbors of $v$, that are also adjacent to $v_i$ satisfying $\delta_C^-(v_i)=\delta_C^+(v_i) < 9$.
However, the degrees of the second neighbors of $v$ remain undetermined.
Thus, we assign the range $[5, 9]$ to $[\delta_C^-(u),\delta_C^+(u)]$ for all second neighbors $u$ of $v$.
The resulting free cartwheel with limited degrees and tail ranges $(Z_C, \delta_C^-, \delta_C^+)$ represents the set of all free cartwheels with limited degrees $(Z_C, \delta_C) \in (Z_C, \delta_C^-, \delta_C^+)$.
The pseudocode corresponding to the algorithm is presented in Algorithm \ref{alg:enum_wheel}.

Here, we apply the pruning method by an upper bound of the final charge, described in the previous section.
This pruning method has a significant impact, especially when the center degree $d$ is higher (e.g., $d=9,10,11$).
One typical example of this pruning is the proof of Lemma \ref{lem:12+}.
For a vertex of degree at least 12, we only have to consider the degree of neighbors of $v$
to find a positively charged vertex.
Since all degrees are $5$, we can find a reducible configuration in $\mathcal{D}$.
We also cut $\cartwheelRange$ if it is blocked by $\mathcal{D}$ because being blocked implies that every $(Z_C, \delta_C) \in (Z_C, \delta_C^-, \delta_C^+)$ contains some centered reducible configuration from $\mathcal{D}$.
The pseudocode corresponding to the algorithm, including pruning, is presented in Algorithm \ref{alg:enum_possible_bad_wheels}.

\subsubsection{Fixing rules applied from neighbors to the center}
\label{subsub:fix_in_rules}

In the subsequent steps, we describe the operation for one fixed $(Z_{C_0}, \delta_{C_0}^-, \delta_{C_0}^+)$ generated in the previous section because we can handle each in completely parallel.
Let $\mathcal{C}_0$ be a singleton set of $(Z_{C_0}, \delta_{C_0}^-, \delta_{C_0}^+)$.
In the following, we generate sets of free cartwheels with limited degrees and tail ranges $\mathcal{C}_i$ for $i=1,\ldots,d$ in this order.
Note that $\mathcal{C}_d$ is an intermediate set; the final output of the algorithm will be derived from $\mathcal{C}_d$ in a subsequent step.

For the initial $(Z_{C_0}, \delta_{C_0}^-, \delta_{C_0}^+) \in \mathcal{C}_0$, we do not know which rules apply to $v_iv$ for any neighbor $v_i$.
When constructing $\mathcal{C}_i$ from $\mathcal{C}_{i-1}$, we will branch by deciding the exact set of rules $\mathcal{R}_i$ that apply to $v_iv$. The set $\cR_i$ applies to a cartwheel $(Z_{C_0}, \delta_{C_0})\in
(Z_{C_0}, \delta_{C_0}^-, \delta_{C_0}^+)$ if and only if one of the free combinations $R^*_i$ of $\cR_i$ applies to
$(Z_{C_0}, \delta_{C_0})$. However, if $R^*_i$ is blocked by $\cD$, then so is $(Z_{C_0}, \delta_{C_0})$. Therefore, we only have to consider free combinations
$R^*_i \in \cRnoD$ obtained in Lemma \ref{lem:free-comb-rules-noD}.

The procedure for constructing $\mathcal{C}_i$ from $\mathcal{C}_{i-1}$ is as follows.
For each $\cartwheelRange \in \mathcal{C}_{i-1}$ and for each $R_i^*=((Z_R^*, \delta_R^{*-}, \delta_R^{*+}), e_R^*, r_R^*) \in \cRnoD$, we make the branching so that $\mathcal{R}_i$ is the set of rules associated with $R_i^*$ unless $R_i^*$ never applies to $v_iv$.
Let $\phi^*$ be the homomorphism from $Z^*$ to $Z_C$ with $\phi^*(e_R^*)=v_iv$.
We create a copy $(Z_{C'}, \delta_{C'}^-, \delta_{C'}^+)$ of $\cartwheelRange$.
For every vertex $v_R \in Z_R^*$, we assign $[\delta_{C'}^-(v_C), \delta_{C'}^+(v_C)] = [\delta_{C}^-(v_C), \delta_{C}^+(v_C)] \cap [\delta_R^{*-}(v_R), \delta_R^{*+}(v_R)]$, where $v_C=\phi^*(v_R)$.
If some vertices in $(Z_{C'}, \delta_{C'}^-, \delta_{C'}^+)$ have degree-ranges except for the tail range, we refine them to fixed degrees and enumerate all combinations.
Here, we say these pseudo-configurations are obtained from $(Z_{C'}, \delta_{C'}^-, \delta_{C'}^+)$ by \emph{enumerating concrete degrees except for tail ranges}.
We will also apply this operation later.
Finally, the pseudocode for enumerating concrete degrees, excluding tail ranges, and for calculating degree intersections according to a combined rule is presented in Algorithms \ref{alg:concrete_degree} and \ref{alg:update_degree_by_rule}, respectively.

\paragraph{Pruning.}

We apply three pruning methods when adding $\cartwheelRange$ to $\mathcal{C}_i$ for $i=1,\ldots,d$.

Recall that if $\cartwheelRange \in \mathcal{C}_{i}$, then for every $j\leq i$, we have decided the exact set $\cR_j$ of rules applying to $v_jv$. We constructed 
$\cartwheelRange$ so that every
rule in $\cR_j$ applies
to every 
$(Z_C,\delta_C)\in\cartwheelRange$, but the restriction
implies that we can ignore any
such $(Z_C,\delta_C)$ if a rule from $\cR\setminus\cR_j$ applies
to $v_jv$.

This has two implications. One is that we have decided the exact discharge $r^*_j$ over $v_jv$. The other more interesting one is for pruning; namely, that we can 
cut $\cartwheelRange$ if there exists a rule in $\mathcal{R} \setminus \mathcal{R}_j$ that always applies to $v_jv$.
The pseudocode for this pruning is presented in Algorithm \ref{alg:prune_non_assoc_rule}.

We also apply a similar pruning method for an upper bound of the final charge in Section \ref{sect:hom-to-ranges}.
The first implication above fixes the discharge $r^*_j$ over $v_jv$ for $j\leq i$, so this tightens the calculation of an upper bound on the final charge to $10 \cdot (6-d) + \sum_{1 \leq j \leq i} r_j^* + \sum_{i < j \leq d} r_j - \sum_{1 \leq j \leq d} s_j$, where again $r_j$ is the sum of dischargings from rules that \emph{sometimes} apply to $v_jv$ while $s_j$ is the sum of dischargings from rules that \emph{always} apply to $vv_j$.
If this value is less than $0$, we cut $\cartwheelRange$.
The pseudocode for calculating this value is presented in Algorithm \ref{alg:upperbound_of_charge}.
Corresponding to Lemma
\ref{lem:estimate_charge_easy}, we have the following lemma stating the validity of this pruning.

\begin{lem}
\label{lem:estimate_charge}
    Let $\cartwheelRange$, $\mathcal{R}_j$ and $r^*_j$ ($1 \leq j \leq i$) be defined as above.
    Let $(Z_C, \delta_C) \in \cartwheelRange$ be a free cartwheel with limited degrees such that the set of rules applying to $v_jv$ equals $\mathcal{R}_j$ for $1 \leq j \leq i$, and contains no centered reducible configurations in $\mathcal{D}$.
    Then, the final charge of the center $v$ in $(Z_C, \delta_C)$ is at most $10 \cdot (6-d) + \sum_{1 \leq j \leq i} r_j^* + \sum_{i < j \leq d} r_j - \sum_{1 \leq j \leq d} s_j$ calculated as above.
\end{lem}

Finally, we cut $\cartwheelRange$ if it is blocked by $\mathcal{D}$.
The pseudocode for all pruning methods and constructing $\mathcal{C}_d$ is presented in Algorithms \ref{alg:prune} and \ref{alg:fix_in_rules}, respectively. We note that the same code can be used for the initial pruning corresponding to the special case where $i=0$.
In the following, we apply these pruning methods at each branch.

\subsubsection{Refinement}
\label{subsub:refinement}
\newcommand{\cCalways}{\mathcal{C}_{\textsf{always}}\xspace}
\newcommand{\cCnever}{\mathcal{C}_{\textsf{never}}\xspace}
Up to this point, we have determined the set of rules applied to $v_iv$ for every neighbor $v_i$ to obtain $\mathcal{C}_d$.
The free cartwheels with limited degrees and tail ranges in $\mathcal{C}_d$ may have second neighbors whose degrees have not been determined yet (i.e., their degree-ranges remain a tail range $[d, 9]$ for some $d$).
We construct the final output $\mathcal{C}$ of the algorithm from $\mathcal{C}_d$ by selectively refining some of these degrees, but not all of them.

We first outline the overall flow of the algorithm to construct the final set $\mathcal{C}$ from $\mathcal{C}_d$.
Initially, we add every free cartwheel with limited degree and tail ranges in $\mathcal{C}_d$ into a queue.
While the queue is not empty, we pop one free cartwheel with limited degrees and tail ranges $\cartwheelRange$ from the queue.
Then, we iterate over all pairs of a rule $R \in \mathcal{R}$ and a neighbor $v_i$.
For each pair $(R, v_i)$, we determine whether to refine $\cartwheelRange$ into two sets of pseudo-configurations $\cCalways$: where the rule $R$ always applies to $vv_i$, and $\cCnever$: where the rule $R$ never applies to $vv_i$.
If we determine not to refine for a given pair, we proceed to the next pair.
Otherwise, we perform the refinement, and apply the three previously described pruning methods to the free cartwheels with limited degrees and tail ranges in $\cCalways$ and $\cCnever$.
We push the surviving free cartwheels with limited degrees and tail ranges to the queue.
If $\cartwheelRange$ is not refined for any pair, we add it to the final set $\mathcal{C}$.
We repeat until the queue becomes empty.
The pseudocode corresponding to the overall algorithm for constructing $\mathcal{C}$ from $\mathcal{C}_d$ is presented in Algorithm \ref{alg:fix_out_rules}.

It remains to describe the criteria for refinement and the refinement procedure itself.
The fundamental idea is that if $R$ always applies, then it is sure to decrease
the final charge at $v$, increasing the chance that we conclude that the final charge is negative. However, for the cases where $R$ never applies, we also want an advantage; namely, a vertex in the second neighborhood with fixed degree below $9$; for this increases the chance that a cartwheel is blocked by a reducible configuration (reducible configurations can not use non-center vertices with degree 9). To ensure we get this advantage when $R$ never applies, we will use only $R$ that "dominantly" applies, as defined below.

Let $R=((Z_R, \delta_R^-, \delta_R^+), e_R, r_R) \in \mathcal{R}$ be a rule and $v_i$ be a neighbor.
We check whether $R$ always applies to $vv_i$ in $(Z_C, \delta_C^-, \delta_C^+)$.
If so, no refinement is performed.
Next, we check whether $R$ sometimes applies to $vv_i$. If so, we consider the homomorphism $\phi$ from $(Z_R, \delta_R^-, \delta_R^+)$ to $(Z_C, \delta_C^-, \delta_C^+)$ satisfying $\phi(e_R)=vv_i$. We say that $R$ \emph{dominantly applies} if every vertex $u_R$ in $Z_R$ with $\delta^+(u_R) < \infty$ is mapped to a vertex $u$ with $\delta_C^+(u) < 9$ by $\phi$. It is only if $R$ dominantly applies that we perform the refinement.
The pseudocode for checking whether a rule applies dominantly and for determining whether a refinement occurs is presented in Algorithms \ref{alg:dominantly_apply} and \ref{alg:should_refine}, respectively.
Note that every $(Z_R, \delta_R) \in (Z_R, \delta_R^-, \delta_R^+)$ has a diameter of at most 4, so a homomorphism $\phi$ here is injective by Lemma \ref{lem:1-1} since a free cartwheel has no separating 3- or 4-cycle.

We now explain the refinement procedure.
Let $\phi$ be the homomorphism from above.
For a vertex $u_R$ in $Z_R$, let $u=\phi(u_R)$.
Let $U_R$ be the set of vertices $u_R$ in $Z_R$ such that $\delta_C^+(u)=9$ and $\delta_C^-(u) < \delta_R^-(u_R)$.
Furthermore, let $U=\phi(U_R)$.

We now show for every vertex $u_R\in Z_R\setminus U_R$ and $u=\phi(u_R)$, the inclusion $[\delta_C^-(u), \delta_C^+(u)] \subseteq [\delta_R^-(u_R), \delta_R^+(u_R)]$ holds.
For every vertex $u \not \in U$ with $\delta_C^+(u) < 9$, $u$ has a fixed degree $\delta_C^-(u)=\delta_C^+(u) < 9$ because the only degree-ranges we have are tail ranges.
Since $R$ sometimes applies, the inclusion holds.
For the remaining vertices $u \not \in U$, $\delta_C^+(u)=9$ trivially holds.
Since $u \not \in U$, $\delta_R^-(u_R) \leq \delta_C^-(u)$ holds.
Since $R$ dominantly applies, $\delta_R^+(u_R)=\infty$ holds.
This implies that $\delta_C^+(u) \leq \delta_R^+(u_R)$.
The inclusion thus holds for every $u_R\in Z_R\setminus U_R$ and $u=\phi(u_R)$.

We can also show that every 
$u_R\in U_R$ and $u=\phi(u_R)$ satisfies
$\delta_C^+(u) \leq \delta_R^+(u_R)$.
This is because $\delta_R^+(u_R)=\infty$ holds since $\delta_C^+(u)=9$ and $R$ dominantly applies.

To ensure that $R$ always applies, we create a copy $(Z_{C'}, \delta_{C'}^-, \delta_{C'}^+)$ of $(Z_C, \delta_C^-, \delta_C^+)$ and assign $\delta_{C'}^-(u)=\delta_R^-(u_R)$ for every vertex $u_R\in U_R$ and $u=\phi(u_R)$.
Now $R$ always applies to
$(Z_{C'}, \delta_{C'}^-, \delta_{C'}^+)$, and we can finally set
$\cCalways=\{(Z_{C'}, \delta_{C'}^-, \delta_{C'}^+)\}$.

We note that $(Z_{C'}, \delta_{C'}^-, \delta_{C'}^+)$ is exactly the set of cartwheels from
$(Z_C, \delta_C^-, \delta_C^+)$ for which $R$ applies to $vv_i$. Conversely, this means that $R$ never applies to any 
$(Z_C, \delta_C^-, \delta_C^+)\setminus (Z_{C'}, \delta_{C'}^-, \delta_{C'}^+)$. To cover this complement set with cartwheels of limited degrees and tail ranges, we do as follows.
For each $u \in U$, we create a copy $(Z_{C}^u, \delta_{C}^{u-}, \delta_{C}^{u+})$ of $(Z_C, \delta_C^-, \delta_C^+)$, and assign $\delta_{C}^{u+}(u)=\delta_R^-(u_R)-1$, which as desired is strictly less than 9.
For each such copy, we generate the set of pseudo-configurations obtained from $(Z_{C}^u, \delta_{C}^{u-}, \delta_{C}^{u+})$ by enumerating concrete degrees except for tail ranges.
We define $\cCnever$ as the union of them.
The pseudocode for the refinement process, obtaining $\cCalways$, and obtaining $\cCnever$ is presented in Algorithms \ref{alg:refinement}, \ref{alg:refine_always}, and \ref{alg:refine_never}, respectively.

To see that this algorithm terminates, note that if a free cartwheel with limited degrees and tail ranges is refined into $\cCalways$ and $\cCnever$ by a pair $(R, v_i)$, no resulting free cartwheel with limited degrees and tail ranges in $\cCalways$ or $\cCnever$ will be refined by the same pair again.
This is because the application of $R$ to $vv_i$ is definitely resolved (i.e., it either always applies or never applies) for all free cartwheels with limited degrees and tail ranges in $\cCalways$ and $\cCnever$.
Thus, the refinement process occurs at most $|\mathcal{R}| \cdot d$ times for a single free cartwheel with limited degrees and tail ranges in $\mathcal{C}_d$.
Therefore, the queue eventually becomes empty, ensuring that the algorithm terminates.

From the initial singleton set $\mathcal{C}_0$, we obtain the final output set $\mathcal{C}$.
The pseudocode for constructing $\mathcal{C}$ is presented in Algorithm \ref{alg:enum_bad_cartwheel}.
Let $\mathcal{C}_{\textsf{all}}$ be the union of all such output sets $\mathcal{C}$ obtained from every possible initial set $\mathcal{C}_0=\{(Z_{C_0}, \delta_{C_0}^-, \delta_{C_0}^+)\}$.
Note that the initial sets are enumerated by all combinations of the center degree $d$ and degrees of its neighbors, as described in Section \ref{subsub:enum-degree-neighbor}.
The pseudocode for constructing $\mathcal{C}_{\textsf{all}}$ is presented in Algorithm \ref{alg:enum_all_bad_cartwheel}.

Note that some free cartwheel with limited degrees $C \in (Z_C, \delta_C^-, \delta_C^+) \in \mathcal{C}_{\textsf{all}}$ may contain centered reducible configurations from $\mathcal{D}$.

\begin{lem}
    Let $C$ be a bad free cartwheel with limited degrees such that the center degree $d$ is in $\{7,8,9,10,11\}$.
    Then, the final resulting set $\mathcal{C}_{\textsf{all}}$ contains some free cartwheel with limited degrees and tail ranges $(Z_C, \delta_C^-, \delta_C^+)$ where $C \in (Z_C, \delta_C^-, \delta_C^+)$.
    Moreover, if some $(Z_C, \delta_C^-, \delta_C^+) \in \mathcal{C}_{\textsf{all}}$, then some free cartwheel with limited degrees $C \in (Z_C, \delta_C^-, \delta_C^+)$ is bad.
\end{lem}
\begin{proof}
    Let $(Z_C, \delta_C)$ be a free cartwheel with limited degrees such that the center degree $d$ is in $\{7,8,9,10,11\}$, it contains no centered reducible configuration from $\mathcal{D}$ and, the final charge of the center $v$ in $(Z_C, \delta_C)$ is at least 0, and let $v_1,\ldots,v_{d}$ be neighbors of $v$.
    In the initial step of the algorithm, we enumerated all combinations of the degrees of the center and its neighbors, so we have the corresponding degrees of $(Z_C, \delta_C)$.

    First, we prove the correctness of the algorithm without pruning.
    Subsequently, we will show that the three pruning methods are valid.
    
    Let $\mathcal{R}_i$ be the set of rules applied to $v_iv$ in $(Z_C, \delta_C)$.
    Let $R_i^* = (Z_i^*, \delta_i^{*-}, \delta_i^{*+}, e_i^*, r_i^*)$ be one free combination of $\mathcal{R}_i$ such that
    some $(Z_i^*, \delta_i^*) \in (Z_i^*, \delta_i^{*-}, \delta_i^{*+})$ has a homomorphism into $(Z_C, \delta_C)$ with $e_i^*$ mapping to $v_iv$.
    Since $(Z_C, \delta_C)$ contains no centered reducible configuration from $\mathcal{D}$, the same holds for $(Z_i^*, \delta_i^*)$ with center $\head(e_i^*)$.
    
    It implies that $\mathcal{D}$ does not block $(Z_i^*, \delta_i^{*-}, \delta_i^{*+})$.
    Then, $R^*_i$ belongs to $\cRnoD$ obtained in Lemma \ref{lem:free-comb-rules-noD}.
    For each $i=1,\ldots,d$, we branch by deciding the exact set of rules that apply to $v_iv$, and update the degree-ranges by one free combination of them in $\cRnoD$.
    As a result, we obtain $(Z_C, \delta_C^-, \delta_C^+) \in \mathcal{C}_d$ such that $(Z_C, \delta_C) \in (Z_C, \delta_C^-, \delta_C^+)$.

    When we generate the final set $\mathcal{C}$ from $\mathcal{C}_d$, we only refine the degree-ranges of some vertices.
    It does not change the set of free cartwheels with limited degrees represented before and after the refinement.
    Thus, $(Z_C, \delta_C^-, \delta_C^+) \in \mathcal{C}$ such that $(Z_C, \delta_C) \in (Z_C, \delta_C^-, \delta_C^+)$.

    We will show that three pruning methods are valid (i.e., these methods do not remove $(Z_C, \delta_C^-, \delta_C^+)$ such that $(Z_C, \delta_C) \in (Z_C, \delta_C^-, \delta_C^+)$.).
    Since Lemmas \ref{lem:estimate_charge_easy}, \ref{lem:estimate_charge} hold, and the final charge of $v$ in $(Z_C, \delta_C)$ is at least 0, the pruning by an upper bound of the final charge is valid.
    Since $(Z_C, \delta_C)$ contains no centered reducible configuration from $\mathcal{D}$, the pruning about blocking by $\mathcal{D}$ is also valid.
    The pruning method, when rule $R \in \mathcal{R} \setminus \mathcal{R}_j$ always applies, is also valid since $\mathcal{R}_j$ is the exact set of rules applying to $v_jv$.

    Finally, we show that some free cartwheel with limited degrees $(Z_C, \delta_C) \in (Z_C, \delta_C^-, \delta_C^+) \in \mathcal{C}$ is bad.
    Let $(Z_C, \delta_C^-, \delta_C^+) \in \mathcal{C}$.
    For every $u$ in $Z_C$ has a fixed degree or a tail range.
    We construct $(Z_C, \delta_C)$ by setting $\delta_C(u)=9$ if $\delta_C^+(u)=9$, and $\delta_C(u)=\delta_C^-(u)$ otherwise.
    By the pruning by blocking, $(Z_C, \delta_C^-, \delta_C^+)$ contains no centered reducible configuration from $\mathcal{D}$.
    Then, the same holds for $(Z_C, \delta_C)$.
    Let $(Z_i^*, \delta_i^{*-}, \delta_i^{*+}, e_i^*, r_i^*) \in \cRnoD$ be a free combination of $\mathcal{R}_i$ for $(Z_C, \delta_C^-, \delta_C^+)$.
    By combining the pruning about an upper bound of the final charge, $10 \cdot (6 - \delta(v)) + \sum_{i=1}^d r_i^* - \sum_{i=1}^d s_i \geq 0$, where $s_i$ is the sum of charges for all rules $R \in \mathcal{R}$ always apply to $vv_i$ in $(Z_C, \delta_C^-, \delta_C^+)$.
    In every free cartwheel with limited degree in $(Z_C, \delta_C^-, \delta_C^+)$, the amount of charge $v_i$ sent to $v$ is at least $r_i^*$.
    All we have to prove is that the sum of charges for all rules $R \in \mathcal{R}$ applied to $vv_i$ equals $s_i$ also in $(Z_C, \delta_C)$.

    Let $R=(Z_R, \delta_R^-, \delta_R^+, e_R, r_R) \in \mathcal{R}$ be a rule that sometimes applies but not always applies to $vv_i$ in $(Z_C, \delta_C^-, \delta_C^+)$ for some neighbor $v_i$.
    Showing that $R$ does not apply to $vv_i$ in $(Z_C, \delta_C)$ leads to the above claim.
    Since $(Z_C, \delta_C^-, \delta_C^+)$ is in $\mathcal{C}$, the refinement for all pairs of a rule in $\mathcal{R}$ and a neighbor $v_i$ does not occur anymore.
    This implies that $R$ does not dominantly apply to $vv_i$.
    Then, there is a vertex $u_R$ with $\delta_R^+(u_R) < \infty$ that is mapped to a vertex $u$ with $\delta_C^+(u)=9$ by the homomorphism from $Z_R$ to $Z_C$ with $e_R$ mapping to $vv_i$.
    As we said before, if $\delta_R^+(u_R) < \infty$, then $\delta_R^+(u_R) < 9$ holds.
    Then, $\delta_C(u)=9$ by definition of $\delta_C$, so $R$ does not apply to $vv_i$ in $(Z_C, \delta_C)$.
\end{proof}

For each free cartwheel with limited degrees and tail ranges $C$ in the final output $\mathcal{C}_{\textsf{all}}$, we calculate the upper bound of the final charge of $v$ by Lemma \ref{lem:estimate_charge}.
We checked that all of them are 0, not positive.
Moreover, we checked that the center degree $d$ of each $C$ in $\mathcal{C}_{\textsf{all}}$ is $7$ or $8$.
Moreover, we checked that the center of $C$ in $\mathcal{C}_{\textsf{all}}$ is adjacent to at least one vertex of degree more than 6.
These checks are verified by computer in Algorithm \ref{alg:enum_bad_cartwheel}, described as Lemma \ref{lem:for-lem:positive-comp}.
This completes the proof of \textbf{Lemma \ref{lem:positive-comp}}.

\subsection{Free combinations of free cartwheels}
\label{subsect:combine-cartwheel}
\showlabel{subsect:combine-cartwheel}
In this section, we describe the algorithms used to prove Lemma \ref{lem:zero-deg7,8}, \ref{lem:777}, and \ref{lem:zero-deg7,7}.
In these lemmas, we want to combine cartwheels in order to find reducible configurations. We only need to consider cartwheels that do not already have reducible configurations, so our starting point is the set $\cC=\mathcal{C}_{\textsf{all}}$ of cartwheels that survived from the last section.
Note that the number of free cartwheels in $\mathcal{C}_{\textsf{all}}$ is 728 for center degree 8, and $9366$ for center degree $7$.

In the previous section, we considered degrees up to 9, but in this section, we only consider vertices of degree up to 8. In particular, this means that we can eliminate from $\cC$ any configuration $(Z,\delta^-,\delta^+)$
which for some vertex $v\in Z$ has fixed degree
$\delta^-(v)=\delta^+(v)=9$. The only ranges we
have for cartwheels in $\cC$ are tail ranges $[d,9]$. We could reduce those to $[d,8]$, but this would not be a good idea unless $d=8$. The point is that to see if
$(Z,\delta^-,\delta^+)$ is blocked by the reducible configurations in $\cD$, by Lemma \ref{lem:bounded-ranges}, we can reduce a tail range to the single value 9 while we have to do all possible combinations over all other degree ranges. The role of a tail range is simply to say that a vertex cannot be used in any of the reducible configurations.
After this pruning, the vertices in configurations $\cC$ 
have either fixed degrees at most 8 or tail ranges.
The pseudocode corresponding to this process is presented in Algorithm \ref{alg:delete_degree_k}.

For combining cartwheels, we shall use a general routine that takes a configuration $(Z,\delta^-,\delta^+)$ with a specified dart $e$
and look for all possible 
free combinations with 
a cartwheel $(Z',\delta^{\prime -},\delta^{\prime+})\in \cC$ identifying $e$ with some dart $e'$ incident with the center of $(Z',\delta^{\prime -},\delta^{\prime+})$.
We eliminate all free combinations blocked by a reducible configuration from $\cD$, returning only those combinations that are not blocked.
The pseudocode for this is found in \ref{alg:combine_each_cartwheel}.

\paragraph{A vertex of degree 8.} 
We shall now describe the algorithmic check needed to prove 
prove Lemma \ref{lem:zero-deg7,8}.

Consider every cartwheel
$(Z,\delta^-,\delta^+)\in \cC$ with center $v$ of degree 8. If $v$ has a neighbor of degree 8, 
then for all such degree-8 neighbors $v'$, we consider all the above mentioned free combinations with cartwheels from $\cC$, identifying $v,v'$ with a dart to the center.

A computer check found that all combinations were blocked, and from this, we want to conclude that Lemma \ref{lem:zero-deg7,8} (i) is true. However, Lemma \ref{lem:zero-deg7,8} (i) is about cartwheels in 
$G$, not free cartwheels.

In more detail, for Lemma \ref{lem:zero-deg7,8} (i), we consider the situation in $G$ where all vertices
in $B_3(v)$ have degree at most 8 and $v'$ is a degree-8 neighbor of $v$,
$v$ and $v'$ have final charge $0$. With
$B=B_2(v)\cup B_2(v')$, the lemma claims that $B$ has a separating 3- or 4-cycle or a reducible configuration in $\cD$.
Suppose for a contradiction that $B$ has no separating 3- or 4-cycle or reducible configuration in $\cD$.
Then, by Lemma \ref{lem:cartwheel}, $\bar B_2^8(v)=B_2(v)$ is isomorphic to a free cartwheel $(Z,\delta)$ with center $v$. The free cartwheel inherits from 
$B_2(v)$ that $v$ has charge $0$ and no reducible configuration in $\cD$, so $(Z,\delta)\in (Z,\delta^-,\delta^+)\in \cC$. Likewise, we get
a $(Z',\delta')\in
(Z',\delta^{\prime -},\delta^{\prime+})\in\cC$ isomorphic to $B_2(v')$.
By Lemma \ref{lem:free-homo-conf}, this implies that we from the free combination $(Z^*,\delta^*)$ of $(Z,\delta)$ and $(Z',\delta')$ get a homomorphism into $B$. But since the free combinations were blocked by $\cD$, we know that there is a reducible configuration from $\cD$ with a homomorphism into $(Z^*,\delta^*)$ and hence further on into $B$ where we know it is embedded because $B$ has no separating 3- or 4-cycles. This completes the formal
proof of 
{\bf Lemma \ref{lem:zero-deg7,8} (i)} based on the computer verified check in Algorithm \ref{alg:for-lem:zero-deg7,8(i)} that the free combinations were all blocked by reducible configurations in $\cD$.
In the rest of this section, we will just assume this understanding of the relation between the cartwheels in $G$ and our free cartwheels from $\cC$.

The proof of Lemma \ref{lem:zero-deg7,8} (ii) is very similar. Here, $v$ has no neighbor of degree 8 and only one neighbor, $v'$, of degree 7. In this case, again we look at all free combinations with cartwheels in $\cC$, identifying $vv'$ with a dart to the center, checking that they are all blocked. We checked this with a computer in Algorithm \ref{alg:for-lem:zero-deg7,8(ii)}, completing the proof of {\bf Lemma \ref{lem:zero-deg7,8} (ii)}.

Finally, we have the case of Lemma \ref{lem:zero-deg7,8} (iii). Here $v$ has no neighbors of degree 8 and at least two neighbors of degree $7$. We consider any two such neighbors $v'$ and $v''$ minimizing the successor distance from $v'v$ to $v''v$ around $v$. In this case, 
we want to look at all free combinations with 
$(Z',\delta^{\prime -},\delta^{\prime+}),(Z'',\delta^{\prime\prime -},\delta^{\prime\prime+})\in\cC$, identifying $vv'$ with a center dart in $Z'$ and $vv''$ with a center dart in $Z''$. However,
recall from Lemma \ref{lem:free-homo-conf-trans} that we can take the free combinations, one at a time. For efficiency, it is better to first do the
free combinations $(Z^*,\delta^{*-},\delta^{*+})$ of $(Z,\delta^- ,\delta^+)$ with $(Z',\delta^{\prime -},\delta^{\prime+})$, removing those blocked by $\cD$, and then only combine the survivors with $(Z'',\delta^{\prime\prime-},\delta^{\prime\prime+})$.
In the full combinations, we again remove everyone that is blocked by $\cD$.
This routine for combining cartwheels from $\mathcal{C}$ twice is presented in Algorithm \ref{alg:combine_each_cartwheel_twice}.
A computer check in Algorithm \ref{alg:for-lem:zero-deg7,8(iii)} guarantees that the remaining configurations all contain
the exceptional configuration $X$ from Figure \ref{fig:exception-zero-concentrate} with $v$ in the center. This completes the proof of
{\bf Lemma \ref{lem:zero-deg7,8} (iii)}.

\paragraph{Maximum degree 7.}
We shall now show the algorithmic check needed to prove Lemma \ref{lem:777} and \ref{lem:zero-deg7,7}.
In this section, we consider only degrees up to $7$. In particular, all center degrees are 7.
Specifically, we eliminate from $\mathcal{C}$ any configuration $(Z, \delta^-, \delta^+)$ that has some vertex $v$ in $Z$ with fixed degree $\delta^-(v)=\delta^+(v)=8$.
The only degree-ranges we have for cartwheels in $\mathcal{C}$ are tail ranges $[d,9]$.
Similarly to the case with degree $8$,
we reduce the degree-range 
$[d,9]$ to $[d,7]$ only if $d=7$.

For the proof of Lemma \ref{lem:777}, $v$ has two consecutive neighbors $v', v''$ of degree $7$.
Note that $v,v'$ and $v''$ form $\Tseven$.
Let $(Z', \delta^{'-}, \delta^{'+})$ and $(Z'', \delta^{''-}, \delta^{''+})$ be cartwheels in $\mathcal{C}$ respectively.
Then, we want to look at all free combinations with
$(Z', \delta^{\prime-}, \delta^{\prime+}), (Z'', \delta^{\prime\prime-}, \delta^{\prime\prime+}) \in \mathcal{C}$, identifying $vv'$ with a center dart in $Z'$ and $vv''$ with a center dart in $Z''$.
For efficiency, we take the free combinations, one at a time, similarly to the proof of Lemma \ref{lem:zero-deg7,8} (iii).
We verified that all combinations are blocked by $\mathcal{D}$ by executing Algorithm \ref{alg:for-lem:777} on a computer. 
This completes the proof of \textbf{Lemma \ref{lem:777}}.

We now delete all the cartwheels that contain $\Tseven$ from $\mathcal{C}$, that is, 
 we remove $(Z_C, \delta_C^-, \delta_C^+) \in \mathcal{C}$ if there is a triangle $v_1v_2v_3$ with $\delta_C^-(v_i)=\delta_C^+(v_i)=7$ for all $i=1,2,3$.

To prove  
Lemma \ref{lem:zero-deg7,7}, we consider an arbitrary cartwheel
$(Z_C, \delta_C^-, \delta_C^+)$ from the reduced set $\mathcal{C}$.

For the proof of Lemma \ref{lem:zero-deg7,7} (i), $v$ has exactly one neighbor $v'$ of degree $7$.
In this case, we look at free combinations with all cartwheels $(Z', \delta^{\prime-}, \delta^{\prime+})$ from $\mathcal{C}$, identifying $vv'$ with a center dart in $Z'$. We remove all combinations that are blocked by $\mathcal{D} \cup \{\Tseven\}$ by executing Algorithm \ref{alg:for-lem:zero-deg7,7(i)}, and check that none remain.
This completes the proof of \textbf{Lemma \ref{lem:zero-deg7,7} (i)}. 

For the proof of Lemma \ref{lem:zero-deg7,7} (ii), $v$ has at least two neighbors of degree $7$.
We consider any two such neighbors $v'$ and $v''$.
We want to look at all free combinations with $(Z', \delta^{\prime-}, \delta^{\prime+}), (Z'', \delta^{\prime\prime-}, \delta^{\prime\prime+}) \in \mathcal{C}$, identifying $vv'$ with a center dart in $Z'$ and $vv''$ with a center dart in $Z''$. As in the proof of Lemma \ref{lem:zero-deg7,8} (iii), we do the combinations, one at a time, each time removing the combination if it is blocked by $\mathcal{D} \cup \{\Tseven\}$. This is done by executing Algorithm \ref{alg:for-lem:zero-deg7,7(ii)} on a computer.
This completes the proof of \textbf{Lemma \ref{lem:zero-deg7,7} (ii)}.

This was the last lemma for which a computer check was needed, so this completes the proof of {\bf Theorem~\ref{thm:main-theorem}}, our main technical result on reducible configurations.

\section{Recursing deterministically with D-reducible configurations}\label{sect:D-reducible}

We will now show that we can derandomize the simple randomized algorithm from Section \ref{sec:algorithm-sketch}. 
Recall the situation. We have a linear number of non-touching D-reducible configurations. Using induction, we have 4-colored
all vertices outside these reducible configurations.

We will now start doing Kempe changes on the colored vertices. If the ring $R$ of a non-colored configuration $Z$ gets a 0-extendible coloring, then we extend this 4-coloring of $R$ to $Z$, and then $Z$ is no longer among the non-colored configurations (strictly speaking, it is $(Z,d_G)$ that is the reducible configuration).

In general, if the  coloring of a ring is  $i$-extendible, then we may have to
do $i$ improving Kempe changes to make it $0$-extendible. We know from Lemma \ref{lem:reduc} that all colorings of the rings of our reducible configurations are 25-extendible. 

In Section \ref{sec:algorithm-sketch}, we proved
that $O(\log n)$ random Kempe changes sufficed to get all reducible configurations colored with high probability. 

We will derandomize the construction using the method of conditional expectations.
The main step is to show that we can deterministically find a Kempe change that improves the level of coloring for a constant fraction of the rings.
The key is to prove the following lemma:

\begin{lem}\label{lem:det-Kempe-change}
    Let $Z_1,\ldots, Z_r$ be some (not all) non-colored reducible configurations in $G$. The colorings of their rings  $R_1,\ldots,R_r$ are not yet $0$-extendible.
    
    We can then, deterministically in linear time, find and execute a Kempe change that is improving for at least a fraction $p=\frac1{3\cdot 2^{18}}$ of the rings, and also tell which rings got  improved.
\end{lem}
\begin{proof}
We assume that for every one of our reducible configurations $D \in \cD$ (c.f. Lemma \ref{lem:reduc}), we store the Kempe level of each coloring of its ring. This is a huge constant amount of information, but we need it to determine if a given Kempe change will improve.

For each of the three color combinations, that is, for $i\in\{1,2,3\}$, we use the color pairs $\{0,i\}$ and $\{1,2,3\}\setminus \{i\}$, and we identify all the bi-colored Kempe chains. Each vertex now knows which Kempe chain it belongs to.

We now consider the perspective of a given ring $R_i$. Its coloring is only affected by the swap or no-swap of the Kempe chains intersecting it, and there are at most $|R_i|\le18$ such Kempe chains. We now consider all combinations of swap and no-swap for each
intersecting Kempe chain. We know which vertices are involved, hence the effect on the coloring, so we can 
check if a Kempe change following a given combination of intersecting swap/no-swap is improving.

We repeat this for all three color combinations, and we know that for each ring there is at least one color combination and a combination of intersecting swap/no-swap combinations that leads to improvement, and we pick just one such combination for each ring $R_i$.

For the next Kempe change, if we picked both the color combination and the swap/no-swap for each Kempe chain uniformly at random, then the probability that the change was improving for $R_i$ would be at
least $1/(3\cdot 2^{|R_i|})\geq p=1/(3\cdot 2^{18})$, and hence we would expect a fraction $p$ of all rings to get improved. 

We can get an improvement for at least a fraction $p$ of the rings using the method of conditional expectation as described below.

At the beginning, all our rings are ``live". For each of them, we know the color pair combination for the improving Kempe change, and we pick the color pair that works for most of the rings. Only the rings that get their color combination stay alive; the rest will ``die". We know that at least 1/3 stay alive.

We are now going to go through each Kempe chain using the selected color pairs and decide whether to swap the colors. A ring $R_i$ only stays live as long as all decided swaps were as it wanted them. If the number of undecided Kempe chains intersecting $R_i$ is $u_i$, and if all remaining choices were made at random, then the probability that $R_i$ would survive till the end is $2^{-u_i}$. Thus, conditioned on decisions made, the expected number of survivors is
\[\sum_{R_i\textnormal{ is still live}}2^{-u_i}, \textnormal{ where $u_i$ is
the number of  undecided chains intersecting } R_i.\]
We can now pick any undecided Kempe chain $X$. We know that one of the two choices — swap or no-swap — will not decrease the expectation, and we can determine which one it is by considering only the live rings that intersect $X$. 
Because the active reducible configurations are non-touching, their rings are strictly vertex-disjoint. Therefore, each vertex in the graph belongs to at most one active ring. By maintaining a simple global array mapping each vertex to its host ring (or null), iterating over the vertices of $X$ allows us to identify all intersecting rings in strictly $\mathcal{O}(|X|)$ time. 
Thus, the decision for $X$ is done in time $O(|X|)$, and since the Kempe chains are disjoint, we conclude that the total time is linear.

Since the conditional expectation is non-decreasing as we make our choices, and since we end up with all Kempe chains decided, we conclude that the final set of live rings is at least as big as the original expectation.
\end{proof}

We will now use the above lemma to decide on 25 Kempe changes. We say a ring $R_i$ is \emph{active} if all previous Kempe changes have been improving, and it is not yet 0-extendible. We apply Lemma \ref{lem:det-Kempe-change} to all active rings. We know that a fraction $p$ gets an improving move. For the rings $R_i$ whose coloring becomes 0-extendible, we extend the coloring to $Z_i$. The active rings that got an improving move but are still not 0-extendible remain active. We repeatedly apply Lemma \ref{lem:det-Kempe-change} to the active rings, finding a new Kempe change that improves by a factor of $p$ for the remaining active rings. After 25 repetitions, we know that there are no active rings left, and that at least a fraction $p^{25}$ of the original rings have been extended. Conversely, this means that the number of empty rings 
has been reduced by at least a factor $1-p^{25}$.

We now restart with all the empty rings, all of which are 25-extendible\footnote{This is guaranteed by D-reducibility. C-reducibility is not enough.}. If we started with $r$ reducible configurations, then after at most $\log_{1/(1-p^{25})} r = O(\log r) = O(\log n)$ rounds of 25 Kempe changes, we know that no empty rings remain, hence that we have colored the whole original configuration $G$.

The recursive step, reducing the problem size by a constant factor, is thus completed in $O(n\log n)$ time. This suffices for our $O(n\log n)$ time 4-coloring algorithm.

\section{Recursing with non-touching obstructing cycles}\label{sect:obstructing-cycles}

We now assume that we are in the case of Theorem \ref{thm:main-linear} where we get a linear number of non-touching obstructing cycles, which we can assume are chordless. These obstructing cycles will play a role similar to that of 0-extendible reducible configurations. 

It is now convenient to view our triangulation $G$ as embedded in the plane. To do so, take an arbitrary triangle and make it the outer face. Now, each obstructing cycle has the inside (interior) and outside (exterior).

Any family of non-touching obstructing cycles in a near-triangulation can be represented by a rooted tree structure, whose nodes correspond to the obstructing cycles $R_1,\dots,R_r$ and whose root $R_0$ corresponds to the outer face cycle. In the tree, $R_j$ is a descendant of $R_i$ if $R_j$ is inside $R_i$. An example is shown in Figure \ref{fig:obstructing cycles}.

\begin{figure}[htb]
    \centering
    \includegraphics[width=0.65\textwidth]{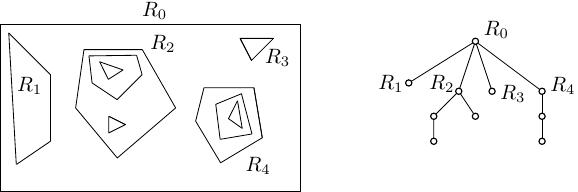}
    \caption{A family of non-touching obstructing cycles can be represented by a rooted tree structure. }
    \label{fig:obstructing cycles}
\end{figure}

We will now find all obstructing cycles that have no other obstructing cycle from our family in their interior. They correspond to the leaves of the mentioned rooted tree. The idea is to recurse on their interiors, and replace the interiors with a constant size near-triangulation. If $H$ was the original near-triangulation with outer face $R$ and its constant-size replacement is $T$, we want that any coloring of $R$ that extends to a coloring of $T$ can also be extended to $H$. For an obstructing cycle $R$ of length 3, this is trivial -- we let $T$ be just the triangle $R$. In the algorithm, we first color the interior $H$ of $R$, including $R$. Next, we remove the interior and color the rest. This yields a coloring of $R$. The two colorings of the three vertices in $R$ are
equivalent up to a color permutation. Using this permutation, we can combine the exterior coloring with the coloring we found recursively for the interior.
For longer obstructing cycles, things get more complicated. For 4-cycles, there will be four possibilities for $T$, while for the 5-cycle, there are 11 different cases. This is described in the next subsection.

\subsection{Obstructing 4-cycles and 5-cycles}

It has been known since the work of Franklin \cite{Franklin1922} that obstructing cycles give a coloring reduction procedure. 
Robertson, Sanders, Seymour, and Thomas \cite{RSST-STOC} showed how to perform such reductions efficiently. They described things via consistent 3-edge-colorings. Since this is critical for our algorithm, we provide a self-contained description, which also shows how we can use many non-touching obstructing cycles so that the overall time is still near-linear.

Let us start by observing that a cycle $R=v_1v_2v_3v_4$ of length 4 has four 4-colorings (up to a permutation of colors). Two of them have $v_1,v_3$ colored the same, two have $v_2$ and $v_4$ colored the same (and the coloring using only two colors appears in both cases), and there is a coloring where all vertices are colored differently.

\begin{lem} \label{lem:X4}
Consider a 4-colored near-triangulation $H$ with outer facial 4-cycle $R=v_1 v_2 v_3 v_4$. 
By using a single Kempe change in $H$, we will either:
\begin{itemize}
\item[(i)] specify a value of $j \in \{1,2\}$ and produce both 4-colorings of $R$ where $v_j$ and $v_{j+2}$ have different colors, or
\item[(ii)] specify a value $j \in \{1,2\}$ and produce both 4-colorings of $R$ where $v_j$ and $v_{j+2}$ have the same color.
\end{itemize}
The above colorings (together with the specified $j$) are found in time that is linear in the size of $H$.
\end{lem}

\begin{proof}
The first of the two colorings will always be the one we already have (since $H$ is colored). For the second one,
we first consider the case where $R$ has all four colors 1,2,3,4, that is, $v_i$ has color $i$ for $i=1,2,3,4$. Either $H$ has a Kempe chain containing $v_1,v_3$ or a Kempe chain containing $v_2,v_4$, say the former.\footnote{Because the graph is planar, we cannot have both, and it is a well-known fact that one of them must occur.} Then the Kempe chain $C$ with colors 2,4 containing the vertex $v_4$ does not contain $v_2$. Consequently, the Kempe change on $C$ results in a coloring of $R$ where $v_2$ and $v_4$ have the same color. So, we have (i) with $j=1$.

Consider the second case where $R$ has three distinct colors. We may assume that $v_i$ has color $i$ for $i=1,2,4$ whereas $v_3$ has color 1. If the Kempe chain with colors 1,3 containing $v_3$ does not contain $v_1$, then the Kempe change on this Kempe chain gives (i) with $j=2$. So assume that the Kempe chain contains $v_1$. Then the Kempe change with colors 2,4 containing $v_4$ will change the color of $v_4$ but not $v_2$. So, we have (ii) with $j=1$.

In the third and final case, $v_1,v_2,v_3,v_4$ have colors 1,2,1,2. Then an appropriate Kempe change 
will change just one of the colors and give us outcome (ii).
\end{proof}

We can now describe our recurrence over an obstructing 4-cycle $R$. As in Lemma \ref{lem:X4}, let $H$ be the interior of $R$, including $R$. First, we color $H$ recursively and then perform the Kempe change described in the lemma in time $O(|H|)$. This will generate another coloring of $H$ corresponding to case (i) or (ii).

We now remove the interior of $R$. If we are in case (i), we add the edge from $v_j$ to $v_{j+2}$. This triangulates the empty face inside $R$. After this reduction, we color the graph. No matter how this colors $R$, we know that the coloring is extendable to $H$ using one of the colorings we computed recursively.

If we are in case (ii), we identify $v_j$ and
$v_{j+2}$. This creates parallel edges to the other vertices, and we only keep one edge from each parallel pair. After this reduction, we color the rest of the graph. No matter how this colors $R$, we again know that the coloring can be extended to $H$ using one of the colorings we computed recursively.

Let us now look at 4-colorings of a 5-cycle $R=v_1 v_2 v_3 v_4v_5$. Again, we are only interested in how the color classes partition $V(R)$. Note that there are five colorings that use only three colors. The following result is analogous to Lemma \ref{lem:X4}. 

\begin{lem}\label{lem:X5} 
Consider a 4-colored near-triangulation $H$ with the outer facial 5-cycle $R=v_1 v_2 v_3 v_4v_5$, where precisely three colors appear on $R$. 
Using at most six single Kempe changes in $H$, we will either
\begin{itemize}
\item[(i)] specify a value $j \in \{1,2,3,4,5\}$ and produce all four 4-colorings of $R$ where $v_j$ is the only vertex of color 1, say, or
\item[(ii)] specify a value of $j \in \{1,2,3,4,5\}$ and
produce all three 4-colorings of $R$ where $v_j$ and $v_{j+2}$ (index $j+2$ taken modulo $5$) have the same color, or
\item[(iii)] produce all five 3-colorings of $R$.
\end{itemize}
The above colorings (together with the specified $j$) can be found in $O(|H|)$ time.
\end{lem}

\begin{proof}
Assume without loss of generality that $v_1,v_2,v_3,v_4,v_5$ are colored 1,2,3,2,3, respectively. If we can
rotate this coloring one step clockwise (that is, we get the coloring 3,1,2,3,2 up to permutation of colors) by one Kempe change, and then again rotate one step and repeat five times, we have case (iii). So, we may assume that the coloring 3,1,2,3,2 (of $v_1,v_2,v_3,v_4,v_5$, respectively) cannot be obtained by one Kempe change. 
So, the Kempe chain with colors 1,2 containing $v_1,v_2$ also contains $v_4$.
This implies that $H$ has no Kempe chain with colors 3,4 containing $v_3,v_5$. So, one Kempe change produces the coloring 1,2,4,2,3. 

Consider again the coloring 1,2,3,2,3. If there is a Kempe chain with colors 2,4 containing $v_2,v_4$, then a Kempe change changes the colors of $v_1,v_5$ without changing the color of $v_3$, that is, we get 3,2,3,2,1. In summary, we now have the colorings 1,2,3,2,3, 1,2,4,2,3, and 3,2,3,2,1. This is case (ii) with $j=2$. So, we may assume that there is no Kempe chain with colors 2,4 containing $v_2,v_4$. This implies that one Kempe change with colors 2,4 results in the coloring 1,4,3,2,3. Using this, we consider the Kempe chain with colors 3 and 4, which contains $v_5$.
If that does not contain $v_2,v_3$, then one Kempe change results in the coloring 1,4,3,2,4, in which case we have (i) with $j=1$. On the other hand, if there is a Kempe chain with colors 3,4 containing $v_2,v_3,v_5$, then the Kempe chain with colors 1,2 containing $v_4$ can be used to produce the coloring 1,4,3,1,3, in which case we have (ii) with $j=3$.
\end{proof} 

Note that the 4-colorings of $R$ in Lemma \ref{lem:X5}(i), respectively (ii), are precisely the 4-colorings of the near-triangulation with $R$ as outer cycle and with the edges $v_jv_{j+2},v_jv_{j+3}$ being added, respectively $v_j,v_{j+2}$ being identified, whereas (iii) describes the colorings of $R$ after a vertex of degree 5 has been added.

We can now describe our recurrence over an obstructing 5-cycle $R$ using Lemma \ref{lem:X5}. We have the plane graph $H$ with boundary $R$. An interesting new aspect is that, before we recurse on $H$, we have to add a new degree-5 vertex $v$ to $H$ outside $R$, whose neighbors are the vertices in $R$. Now we do the recursive coloring of $H$ with $v$ attached, so that only three colors will be used on $R$. Next, we get the colorings of $H$ corresponding to cases (i)--(iii). We now remove the interior of $R$ from the graph, and the case tells us how to replace it before we color the rest of the graph, e.g., in case (iii), we add a new degree-5 vertex $v$ inside $R$ whose neighbors are the vertices in $R$. Now we color the rest of the graph, and no matter how this colors $R$, we know how to extend the coloring to $H$.

\subsection{The recursive algorithm}

We will now describe the recurrence more carefully. For accounting, instead of having an outer triangle, it is easy to just use a single minimal degree vertex $v_0$ as the outer vertex, since it will touch at most a constant number of our non-touching obstructing cycles, that we can then give up. When we start, we have an outer ring $R_0$ that just consists of $v_0$.

Generally, for the recursive call, we have some near-triangulation $H$ with some ring $R$ around the outer boundary. This ring corresponds to one of our chordless obstructing cycles.

In $H$,  we do a search to find all outermost obstructing
cycles $R_i\neq R$ in $H$.\footnote{These are the sons of the root $R$ in our rooted tree mentioned before.}
Note here that since $G$ is triangulated and $R$ is chordless with non-empty inside, all vertices in $R$ have neighbors inside $R$. Since our obstructing cycles are non-touching, any such neighbors must be outside all the $R_i$ in $H$.

Excluding the inside of each $R_i$, we get what we refer to as the \emph{component} of the recursive call. The component $C$ is found in time proportional to its size. 

We want to make a recursive coloring of $H$, but before doing so, we need to make some modifications. If $R$ is longer than 3, we place a vertex outside $R$ whose neighbors are the vertices in $R$. The resulting $H'$ is a triangulation, and when we 4-color $H'$,
we get a 3-coloring of $R$ (as required for Lemma \ref{lem:X5}). However, the output of our recursive call is not just this 3-coloring. Rather, the output is one of the cases in Lemma \ref{lem:X4} or \ref{lem:X5}, including 
a representation of the constantly many colorings mentioned in the case.

Now, for each $R_i$, let $H_i=Int(R_i)$. We make a recursive coloring of $H_i$. This tells us one of the cases
in Lemma \ref{lem:X4} or \ref{lem:X5}. As described under the lemmas, depending on the case, we will make some changes across the empty face inside $R_i$ in the component $C$: in cases (i) we add the specified edges between vertices in $R_i$, in cases (ii) we identify the specified vertices in $R_i$, and in case (iii) we place the star in the middle of $R_i$. Besides these changes, we also add a star outside $R$ to $H$. The resulting graph is a modified component $C'$.

We now color $C'$ inductively (we could also say recursively, but we want to distinguish this from the recursive structure separated by our obstructing cycles).
Afterwards, we do the Kempe changes specified 
in Lemma \ref{lem:X4} or \ref{lem:X5}, producing a case including all the specified colorings of $C$. There are only constantly many, so we can represent all of them in space proportional to the size of $C$.

When we are done with our recursion starting from $G$, our recursion tree will have a component for each recursive call, and the components are connected via our 
obstructing cycles. For each modified component $C'$, we have the specified colorings. We now construct the final coloring of $G$ in a top-down manner. 
At the root, we have the outer ring \(R_0=\{v_0\}\), which we color arbitrarily. Now, in the top down process, we have a component $C$ with outer boundary ring $R$ that is colored in a way we know is consistent with one of the colorings of $C$ that we stored, so now we can extend the coloring of $R$ to $C$ using this coloring, which we know has a valid coloring of each $R_i$. Thus, given all the component colorings produced in the recursive call, we can color all of $G$ in linear time.

\subsection{Runtime analysis}

For our runtime analysis, let $c$ be the number of obstructing cycles, hence the number of recursive calls. We only recurse via obstructing cycles if $c=\Omega(n)$.

As the recursion is currently described, it makes the problem bigger, not smaller; that is, the total number of vertices in the components that we color inductively may be larger than $n$.

To be more precise, in a recursive call with ring $R$ in graph $H$ and component $C$, the ring $R$ is also in the parent call as some ring $R'_i$, so we can view the ring with its up to 5 vertices as a copy. In addition, when 
we modify $C$ to $C'$, adding a vertex outside $R$, we may add a vertex inside each $R_i$. There could be many $R_i$, but we can think of the vertex inside $R_i$ as paid for by the recursive call via $R_i$. In total, this means that we get up to $7$ extra vertices per inductive call.

We can fix this by not using induction for small components. More precisely, we know that the average modified component size is bounded by $(n+7c)/c=O(1)$. We say that a modified component is \emph{big} if it is bigger
than $d=100(n+7c)/c=O(1)$. The number $b$ of big components is at most $c/100$. The modified components that are not big, we 4-color with the classic 
quadratic algorithm from \cite{RSST-STOC}. Each of them is of constant size, hence colored in constant time, so the total time spent on coloring small modified components 
is only linear in the size of the graph.

We claim that the total number of vertices in the big modified components is $n-\Omega(c)$, which implies that our induction has reduced the number of nodes by a constant factor.

At first, this may seem trivial because we only had to pay at most 7 extra per component, but recall that a component paid one for the vertex that the parent call had to place in its center, and the parent of a small component could be big. 

Consider the recursion tree again, and identify all cases
where we have a big component $C$ with a ring $R_i$ such that the recursion of the graph $H_i$ in $R_i$ only uses
small components. Since $R_i$ is obstructing, if $R_i$ is of length 5, we know that $H_i$ has at least two vertices inside $R_i$, but when we modify $C$ to $C'$, in case Lemma \ref{lem:X5} (iii), we have to place a vertex in $R_i$. The gain is
thus at least 1. In all other cases, we do not add any other vertex to $C'$, and we know that we remove at least one vertex from $H$ inside $R_i$. Let
$x$ be the number of obstructing cycles in $H_i$, $R_i$ included. If $x=1$, we just saw that we gain at least $1$ when we remove $H_i$. If $x>1$, we note that all the $x-1$ rings inside $R_i$ are non-touching, and they are all of size at least $3$, so it follows easily that we gain at least $x$.  Therefore, if $c_1$ is the total number of small components with no big descending component, then we get savings of at least $c_1$ from the big components.

Now, for each big component $C$, we let it pay both for its own ring $R$, including putting a vertex inside, but we also go to its nearest big ancestor $C^+$, and let it pay for the ring $R^+_i$ containing $C$, including possibly putting a vertex in $R_i^+$. The total payment is thus at most 12, adding up to $12b$ over all big components.

For all small components with descending big components, we know that the ring has size at least 3,  so if we have $c_2$ of these small components, then we can save at least $3c_2$. We have $c=c_1+c_2+b$, and we save $c_1+2c_2-12 b$. Since $b\leq c/100$, we conclude that we save at least $c/2$, hence at most $n-c/2$ vertices are transferred to the induction, as desired.

The recursive step, reducing the problem size by a constant factor, is thus completed in $O(n)$ time, which is more than sufficient for our $O(n\log n)$ time 4-coloring algorithm. This completes the description of the algorithm, hence the proof of {\bf Theorem \ref{thm:linear-time}}.

\section{The overall algorithmic result}
In this paper we show that we can 4-color a planar triangulation in $O(n\log n)$ time. First, applying Corollary \ref{thm:main-linear}, in linear time, we find either (a) a linear number of pairwise non-touching induced reducible configurations in $\cD$, or (b) a linear number of pairwise non-touching obstructing cycles.

In both cases, we can recursively 
reduce the problem by a constant factor. However, in case (a), the recursive step takes $O(n\log n)$ time, while case (b) takes $O(n)$ time.

We now clarify that any progress on case (a) leads to progress on the overall 4-coloring. First, we formalize case (a) as follows.

\begin{task}[Color Extension] \label{task}
Let be given a planar triangulation $G$ of order $n$ with a family $\cZ$ of $\Omega(n)$ pairwise non-touching induced D-reducible configurations in $\cD$.

Find a fast algorithm that  selects a subfamily  $\cZ'\subseteq\cZ$
of size $\Omega(n)$ such that, for any 4-coloring of $G\setminus \bigcup\cZ'$, we can (again by the fast algorithm) perform $O(\log n)$ Kempe changes that result in a 4-coloring of $G\setminus \bigcup\cZ'$ which extends to all of $G$.
\end{task}

Our overall constructions show:
\begin{thm}
If we can solve the Color Extension
Task \ref{task} of order $n$ in time $T(n)$, then we can 4-color any planar graph of order $n$ in time $T^*(n)\leq O(n+T(n))+T^*(n/2)$.
\end{thm}
In this paper we do the color extension implementing each of the $O(\log n)$ Kempe changes in $O(n)$ time. This leads to $O(n\log n)$ time both for color extension and 4-coloring. However, in a forthcoming paper, we shall get down to linear time, as discussed below.

\section{Further works}
In a forthcoming paper we shall show that 
the $O(\log n)$ Kempe changes for color extension in Task \ref{task} can all be identified and implemented in $O(n)$ time, hence that the whole 4-coloring can be implemented in linear time. The idea is to spend $O(n)$ time to preprocess the graph so that we can identify and simulate each good Kempe change in $O(n/\log n)$ amortized time. Thus we implement all $O(\log n)$ Kempe changes in $O(n)$ total time, implying  that we 4-color in linear time.

We note that creating such a representation for fast identification and simulation of good Kempe changes is a challenge of data structures.
This is a very different type of research than the new understanding of reducible configurations developed in this paper, 
leading to a combinatorial algorithm 4-coloring any planar graph in $O(n \log n)$ time.

In another forthcoming paper based on this work, solving a problem of Albertson \cite{albertsonc},  we expect to prove that every compact 2-dimensional surface $S$, there exists a natural
number $q_S$ such that each graph $G$ 
embedded in $S$ has a set $X$ of at most 
$q_S$ vertices such that $G-X$ is 4-colorable.

\section*{Acknowledgment} Gemini wanted us to write ``The authors gratefully acknowledge the use of advanced large language reasoning models (including Google Gemini and DeepThink) which assisted in the static analysis, semantic auditing, and logical cross-checking of the formal pseudocodes and $\Cpp$ implementation. All mathematical proofs, algorithmic designs, and final conclusions remain the sole intellectual creation and responsibility of the authors.'' It was only a few elementary mistakes that were caught by Gemini, but it feels good to know that the proof and the relation to the $\Cpp$ programs have had this extra
layer of verification which is particularly good at catching typical human mistakes.

\bibliographystyle{alpha}
\bibliography{4CT}

@book {MT,
    AUTHOR = {Mohar, Bojan and Thomassen, Carsten},
     TITLE = {Graphs on {S}urfaces},
    SERIES = {Johns Hopkins Studies in the Mathematical Sciences},
 PUBLISHER = {Johns Hopkins University Press, Baltimore, MD},
      YEAR = {2001},
     PAGES = {xii+291},
   MRCLASS = {05C10 (57M15)},
  MRNUMBER = {1844449},
}

@book {DvorakBook,
    AUTHOR = {Zden\v{e}k Dvo\v{r}\'ak},
     TITLE = {Graph {M}inors, {T}heory and {A}pplications},
    SERIES = {Springer Monographs in Mathematics},
 PUBLISHER = {Springer Cham},
      YEAR = {2025},
     PAGES = {xiv+383},
       DOI = {10.1007/978-3-031-87469-7},
}

@article{reed1997tree,
  title={Tree width and tangles: {A} new connectivity measure and some applications},
  author={Reed, Bruce},
  journal={Surveys in combinatorics},
  volume={241},
  pages={87--162},
  year={1997}
}

@book {Diestel_book,
    AUTHOR = {Diestel, Reinhard},
     TITLE = {Graph {T}heory},
    SERIES = {Graduate Texts in Mathematics},
    VOLUME = {173},
   EDITION = {Fifth},
 PUBLISHER = {Springer, Berlin},
      YEAR = {2018},
     PAGES = {xviii+428},
  MRNUMBER = {3822066},
}

@book {Heesch69,
    AUTHOR = {Heesch, Heinrich},
     TITLE = {Untersuchungen zum {V}ierfarbenproblem},
    SERIES = {Hochschulskriptum},
    VOLUME = {810/a/b},
 PUBLISHER = {Bibliographisches Institut, Mannheim-Vienna-Z\"urich},
      YEAR = {1969},
     PAGES = {290},
   MRCLASS = {05.55},
  MRNUMBER = {0248048},
}

@article {steinberger2010unavoidable,
    AUTHOR = {Steinberger, John P.},
     TITLE = {An unavoidable set of {D}-reducible configurations},
   JOURNAL = {Trans. Amer. Math. Soc.},
  FJOURNAL = {Transactions of the American Mathematical Society},
    VOLUME = {362},
      YEAR = {2010},
    NUMBER = {12},
     PAGES = {6633--6661},
   MRCLASS = {05C15},
  MRNUMBER = {2678989},
       DOI = {10.1090/S0002-9947-2010-05092-5},
       URL = {http://dx.doi.org/10.1090/S0002-9947-2010-05092-5},
}

@article {gonthier08,
    AUTHOR = {Gonthier, Georges},
     TITLE = {Formal Proof—{T}he {F}our-{C}olor {T}heorem},
   JOURNAL = {Notices Amer. Math. Soc.},
  FJOURNAL = {Notices of the American Mathematical Society},
    VOLUME = {55},
      YEAR = {2008},
    NUMBER = {11},
     PAGES = {1382-–1393},
}

@article {kempe1879,
    AUTHOR = {Kempe, Alfred B.},
     TITLE = {On the Geographical Problem of the Four Colours},
   JOURNAL = {Amer. J. Math.},
  FJOURNAL = {American Journal of Mathematics},
    VOLUME = {2},
      YEAR = {1879},
    NUMBER = {3},
     PAGES = {193--220},
}

@article {heawood1890,
    AUTHOR = {Heawood, John P.},
     TITLE = {Map-Colour Theorem},
   JOURNAL = {Quarterly Journal of Pure and Applied Mathematics},
  FJOURNAL = {Quarterly Journal of Pure and Applied Mathematics},
    VOLUME = {24},
      YEAR = {1890},
     PAGES = {332--338},
}

@article {AllaireSwart78,
    AUTHOR = {Allaire, Frank and Swart, Edward Reinier},
     TITLE = {A systematic approach to the determination of reducible
              configurations in the four-color conjecture},
   JOURNAL = {J. Combin. Theory Ser. B},
  FJOURNAL = {Journal of Combinatorial Theory. Series B},
    VOLUME = {25},
      YEAR = {1978},
    NUMBER = {3},
     PAGES = {339--362},
   MRCLASS = {05C15},
  MRNUMBER = {516267},
       DOI = {10.1016/0095-8956(78)90010-2},
       URL = {http://dx.doi.org/10.1016/0095-8956(78)90010-2},
}

@book {AppelHaken89,
    AUTHOR = {Appel, Kenneth and Haken, Wolfgang},
     TITLE = {Every planar map is four colorable},
    SERIES = {Contemporary Mathematics},
    VOLUME = {98},
      NOTE = {With the collaboration of J. Koch},
 PUBLISHER = {American Mathematical Society, Providence, RI},
      YEAR = {1989},
     PAGES = {xvi+741},
      ISBN = {0-8218-5103-9},
   MRCLASS = {05C15},
  MRNUMBER = {1025335},
       DOI = {10.1090/conm/098},
       URL = {http://dx.doi.org/10.1090/conm/098},
}

@article {4ct2,
    AUTHOR = {Appel, K. and Haken, W. and Koch, J.},
     TITLE = {Every planar map is four colorable. {II}. {R}educibility},
   JOURNAL = {Illinois J. Math.},
  FJOURNAL = {Illinois Journal of Mathematics},
    VOLUME = {21},
      YEAR = {1977},
    NUMBER = {3},
     PAGES = {491--567},
   MRCLASS = {05C15},
  MRNUMBER = {0543793},
       URL = {http://projecteuclid.org/euclid.ijm/1256049012},
}

@article {4ct1,
    AUTHOR = {Appel, K. and Haken, W.},
     TITLE = {Every planar map is four colorable. {I}. {D}ischarging},
   JOURNAL = {Illinois J. Math.},
  FJOURNAL = {Illinois Journal of Mathematics},
    VOLUME = {21},
      YEAR = {1977},
    NUMBER = {3},
     PAGES = {429--490},
   MRCLASS = {05C15},
  MRNUMBER = {0543792},
       URL = {http://projecteuclid.org/euclid.ijm/1256049011},
}

@article {birkhoff,
    AUTHOR = {Birkhoff, George D.},
     TITLE = {The Reducibility of Maps},
   JOURNAL = {Amer. J. Math.},
  FJOURNAL = {American Journal of Mathematics},
    VOLUME = {35},
      YEAR = {1913},
    NUMBER = {2},
     PAGES = {115--128},
  MRNUMBER = {1506176},
       URL = {https://doi.org/10.2307/2370276},
}

@article {RSST,
    AUTHOR = {Robertson, Neil and Sanders, Daniel and Seymour, Paul and
              Thomas, Robin},
     TITLE = {The four-colour theorem},
   JOURNAL = {J. Combin. Theory Ser. B},
  FJOURNAL = {Journal of Combinatorial Theory. Series B},
    VOLUME = {70},
      YEAR = {1997},
    NUMBER = {1},
     PAGES = {2--44},
   MRCLASS = {05C15},
  MRNUMBER = {1441258},
       URL = {https://doi.org/10.1006/jctb.1997.1750},
}

@article {Franklin1922,
    AUTHOR = {Franklin, Philip},
     TITLE = {The {F}our {C}olor {P}roblem},
   JOURNAL = {Amer. J. Math.},
  FJOURNAL = {American Journal of Mathematics},
    VOLUME = {44},
      YEAR = {1922},
    NUMBER = {3},
     PAGES = {225--236},
  MRNUMBER = {1506473},
       URL = {https://doi.org/10.2307/2370527},
}

@inproceedings {RSST-STOC,
    AUTHOR = {Robertson, Neil and Sanders, Daniel P. and Seymour, Paul and
              Thomas, Robin},
     TITLE = {Efficiently four-coloring planar graphs},
 BOOKTITLE = {Proceedings of the {T}wenty-eighth {A}nnual {ACM} {S}ymposium
              on the {T}heory of {C}omputing ({P}hiladelphia, {PA}, 1996)},
     PAGES = {571--575},
 PUBLISHER = {ACM, New York},
      YEAR = {1996},
      ISBN = {0-89791-785-5},
   MRCLASS = {05C85 (05C15)},
  MRNUMBER = {1427555},
       DOI = {10.1145/237814.238005},
       URL = {https://doi.org/10.1145/237814.238005},
}

@book{BarrWels95,
  author       = {Michael Barr and
                  Charles Wells},
  title        = {Category theory for computing science {(2.} ed.)},
  series       = {Prentice Hall international series in computer science},
  publisher    = {Prentice Hall},
  year         = {1995},
  isbn         = {978-0-13-323809-9},
  bibsource    = {dblp computer science bibliography, https://dblp.org}
}

@book{Stroustrup86,
  author       = {Bjarne Stroustrup},
  title        = {The {C++} Programming Language, First Edition},
  publisher    = {Addison-Wesley},
  year         = {1986},
  isbn         = {0-201-12078-X},
  bibsource    = {dblp computer science bibliography, https://dblp.org}
}

@inproceedings{CM19,
  author    = {Shiri Chechik and
               Doron Mukhtar},
  editor    = {Timothy M. Chan},
  title     = {Optimal Distributed Coloring Algorithms for Planar Graphs in the {LOCAL}
               model},
  booktitle = {Proceedings of the Thirtieth Annual {ACM-SIAM} Symposium on Discrete
               Algorithms, {SODA} 2019, San Diego, California, USA, January 6-9,
               2019},
  pages     = {787--804},
  publisher = {{SIAM}},
  year      = {2019},
  url       = {https://doi.org/10.1137/1.9781611975482.49},
  doi       = {10.1137/1.9781611975482.49},
  bibsource = {dblp computer science bibliography, https://dblp.org}
}

@article{HopcroftT74,
  author       = {John E. Hopcroft and
                  Robert Endre Tarjan},
  title        = {Efficient Planarity Testing},
  journal      = {J. {ACM}},
  volume       = {21},
  number       = {4},
  pages        = {549--568},
  year         = {1974},
  url          = {https://doi.org/10.1145/321850.321852},
  doi          = {10.1145/321850.321852},
  bibsource    = {dblp computer science bibliography, https://dblp.org}
}

@article{CHIBA1981317,
title = {A linear 5-coloring algorithm of planar graphs},
journal = {Journal of Algorithms},
volume = {2},
number = {4},
pages = {317-327},
year = {1981},
issn = {0196-6774},
doi = {https://doi.org/10.1016/0196-6774(81)90031-6},
url = {https://www.sciencedirect.com/science/article/pii/0196677481900316},
author = {Norishige Chiba and Takao Nishizeki and Nobuji Saito}
}

@article{tait1880dual,
  title   = {Note on a theorem in geometry of position},
  author  = {Peter Guthrie Tait},
  journal = {Trans. Roy. Soc. Edinburgh},
  volume  = {29},
  pages   = {657--660},
  year    = {1880}
}

@article{madison1897,
  title   = {Note on the history of the map-coloring problem},
  author  = {Isabel Maddison},
  journal = {Bull. Amer. Math. Soc.},
  volume  = {3},
  pages   = {257},
  year    = {1897}
}

@article {Klein1879,
    AUTHOR = {Klein, Felix},
     TITLE = {Ueber die {A}ufl\"osung gewisser {G}leichungen vom siebenten
              und achten {G}rade},
   JOURNAL = {Math. Ann.},
  FJOURNAL = {Mathematische Annalen},
    VOLUME = {15},
      YEAR = {1879},
    NUMBER = {2},
     PAGES = {251--282},
      ISSN = {0025-5831,1432-1807},
   MRCLASS = {99-04},
  MRNUMBER = {1510011},
       DOI = {10.1007/BF01444143},
       URL = {https://doi-org.proxy.lib.sfu.ca/10.1007/BF01444143},
}

@incollection{Tutte1971,
    AUTHOR = {Tutte, William T.},
     TITLE = {What is a map?},
 BOOKTITLE = {New directions in the theory of graphs ({P}roc. {T}hird {A}nn
              {A}rbor {C}onf., {U}niv. {M}ichigan, {A}nn {A}rbor, {M}ich.,
              1971)},
     PAGES = {309--325},
 PUBLISHER = {Academic Press, New York-London},
      YEAR = {1973},
   MRCLASS = {05C10},
  MRNUMBER = {376413},
MRREVIEWER = {Laszlo\ Lovasz},
}

@article {KMNZ2,
    AUTHOR = {Kawarabayashi, Ken-ichi and Mohar, Bojan and Nedela, Roman and
              Zeman, Peter},
     TITLE = {Automorphisms and isomorphisms of maps in linear time},
   JOURNAL = {ACM Trans. Algorithms},
  FJOURNAL = {ACM Transactions on Algorithms},
    VOLUME = {21},
      YEAR = {2025},
    NUMBER = {1},
     PAGES = {Art. 6, 32},
      ISSN = {1549-6325,1549-6333},
   MRCLASS = {68U05 (68R10)},
  MRNUMBER = {4829829},
}

@incollection {KMNZ1,
    AUTHOR = {Kawarabayashi, Ken-ichi and Klav\'ik, Pavel and Mohar, Bojan
              and Nedela, Roman and Zeman, Peter},
     TITLE = {Isomorphisms of maps on the sphere},
 BOOKTITLE = {Polytopes and discrete geometry},
    SERIES = {Contemp. Math.},
    VOLUME = {764},
     PAGES = {125--147},
 PUBLISHER = {Amer. Math. Soc., Providence, RI},
      YEAR = {2021},
      ISBN = {978-1-4704-4897-4},
   MRCLASS = {05C10 (05C60 05C85 52B10)},
  MRNUMBER = {4239238},
       DOI = {10.1090/conm/764/15358},
       URL = {https://doi-org.proxy.lib.sfu.ca/10.1090/conm/764/15358},
}

@article{burnside1909theory,
  title={Theory of groups of finite order},
  author={Burnside, William},
  journal={Messenger of Mathematics},
  volume={23},
  pages={112},
  year={1909},
  publisher={Cambridge University Press}
}

@article{frobenius,
  title={Über die Congruenz nach einem aus zwei endlichen Gruppen gebildeten Doppelmodul},
  author={Frobenius, Georg},
  journal={Crelle's Journal},
  volume={101},
  pages={273--299},
  year={1887},
}

@article{albertsonc, 
   title={Open problem 2}, 
   author={Albertson, Michael},
   journal={The Theory and Applications of Graphs, ed. G. Chartrand et al.},
   publisher={Wiley},  
   year={1981}, 
   volume={609},
}

%
\appendix
%



\newpage

\renewcommand*\Call[2]{\textproc{#1}(#2)}

\algblock[TryCatchFinally]{try}{endtry}
\algcblock[TryCatchFinally]{TryCatchFinally}{finally}{endtry}
\algcblockdefx[TryCatchFinally]{TryCatchFinally}{catch}{endtry}
	[1]{\textbf{catch} #1}
	{\textbf{end try}}

\algdef{SE}[REPEATN]{RepeatN}{End}[1]{\algorithmicrepeat\ #1 \textbf{times}}{\algorithmicend}

\newcommand{\id}[1]{\operatorname{id}_{#1}}

\section{Appendix: Pseudocodes}
\label{sect:code}\showlabel{sect:code}

Our proof relies heavily on computer checks. We have already presented an algorithmic description of these checks. In this section, we present pseudocode implementations of these algorithms along with pointers to \Cpp-code 
implementing the pseudocode. The pseudocode here serves as the connection between the algorithm and the \Cpp code. It is a precise description of the algorithm's actions, without the reasoning about why and how it is correct. To verify the overall correctness, the reader should first understand the correctness of the algorithms in the body of the paper, next verify that the pseudocode implements the algorithms, then verify that the \Cpp-code implements pseudocode, and finally run it on their own trusted computer. 


\subsection{Reducibility of configurations}
The program corresponding to the pseudocode in this section is presented in this GitHub repository \drop{\url{https://github.com/edge-coloring/reducibility_checker}}. The link is removed in order to anonymize paper, will be included in the camera-ready version. 

In the following pseudocodes, the four colors we use to color a graph are in $\{0,1,2,3\}$.

We first present the CheckDReducibility algorithm, which takes a configuration as input and returns whether it is D-reducible.
The definition of D-reducibility is equivalent to that used in the original 4-coloring theorem proof, and the algorithm follows previous work.

In this algorithm, we use two subroutines, AllRingColorings and AllKempeChains, explained later.

\begin{algorithm}[H]
    \caption{CheckDReducibility($(Z, \delta)$)}
    \label{alg:check-red_a}
    \begin{algorithmic}[1]
        \Require {A configuration $(Z, \delta)$}
        \Ensure {\textsc{True} if the configuration is D-reducible, \textsc{False} if it is not}
        \State $\hat{Z} \gets$ the free completion of $Z$
        \State $R \gets$ the ring of $\hat{Z}$
        \State $\mathcal{C} \gets $ AllRingColorings($\hat{Z}$, $R$)
        \Repeat
            \ForAll{$\phi \notin \mathcal{C}$}
                \ForAll {$\mathrm{colorPair} \in \{\{0,1\}, \{0,2\}, \{0,3\}\}$}
                    \State $R_{\phi,\mathrm{colorPair}}$ a cycle created from $R$ by contracting the edges connecting vertices colored in colors $c_1, c_2$, where $c_1,c_2 \in \mathrm{colorPair}$ or $c_1,c_2 \notin \mathrm{colorPair}$
                    \State allKempeIsFeasible $\gets$ True
                    \ForAll{$P \in$ AllPlanarKempeChains($R_{\phi,\mathrm{colorPair}}$)}
                        \State changeExists $\gets$ False
                        \ForAll{$P' \subseteq P$}
                            \State $\psi(r) := \mathrm{flip}(\phi(r))$ if $r \in R' \in V(R_{\phi, \mathrm{colorPair}})$ such that $R' \in P'$, $\phi(r)$ otherwise
                            \If{$\psi(r) \in \mathcal{C}$}
                                \State changeExists $\gets$ True
                            \EndIf
                        \EndFor
                        \If{changeExists is False}
                            \State allKempeIsFeasible $\gets$ False
                        \EndIf
                    \EndFor
                    \If{allKempeIsFeasible}
                        \State $\mathcal{C} \gets \mathcal{C} \cup \{\phi\}$
                        \Break
                    \EndIf
                \EndFor            
            \EndFor
        \Until{No more updates are made to $\mathcal{C}$}
        \If{$\mathcal{C}$ is the set of all (proper) $4$-colorings}
            \Return{True}
        \EndIf
        \State{\Return{False}}
        
    \end{algorithmic}
\end{algorithm}

The AllRingColorings algorithm receives a free completion $\hat{Z}$ and its ring $R$ as input and returns the set of all $4$-colorings of $R$ such that it can be properly extended to $\hat{Z}$.

\begin{algorithm}[H]
    \caption{AllRingColorings($\hat{Z}$, $R$)}
    \label{alg:check-red_ring-coloring}
    \begin{algorithmic}[1]
        \Require {A free completion $\hat{Z}$ of some configuration and its ring $R$}
        \Ensure {The set of all $4$-colorings of $R$ such that $\hat{Z}$ as a whole can be $4$-colored}
        \State $\mathcal{C} \gets \varnothing$
        \State $C \gets$ an empty map
        \State $n \gets$ the number of vertices in $\hat{Z}$
        \State $v_1, \cdots, v_n \gets$ an arbitrary ordering of the vertices in $\hat{Z}$
        \Procedure{DFS}{$i$}
            \If{$i > n$}
                \State $\phi \gets$ the restriction of $C$ to $V(R)$ \Comment{Extract the ring coloring from the coloring of $Z$}
                \State $\mathcal{C} \gets \mathcal{C} \cup \{\phi\}$
                \State \Return
            \EndIf
            \ForAll{$c \in \{0,1,2,3\}$}
                \State isColorProper $\gets$ True
                \ForAll{$j$ such that $1 \le j < i$ and $v_i, v_j$ are adjacent in $\hat{Z}$}
                    \If{$C[j] = c$}
                        \State isColorProper $\gets$ False
                        \Break
                    \EndIf
                \EndFor
                \If{isColorProper is True}
                    \State $C[i] \gets c$
                    \State DFS($i + 1$)
                \EndIf
            \EndFor
        \EndProcedure
        \State \Call{DFS}{1}
        \State \Return{$\mathcal{C}$}
    \end{algorithmic}
\end{algorithm}

The following algorithm generates a list of all planar Kempe chains.
The \emph{matching parentheses} generated in Line 2 are strings that can be constructed recursively by concatenating two matching parentheses or by adding the characters ``(" and ``)" to each side of a matching parenthesis string, starting with an empty string.

\begin{algorithm}[H]
    \caption{AllPlanarKempeChains($C$)}
    \label{alg:check-red_b}
    \begin{algorithmic}[1]
        \Require {A cycle $C$ of even length $n$.}
        \Ensure {A set $P$ of partitions of $V(C)$ resembling the connections via planar Kempe chains}
        \State $n \gets |C|$
        \State $P \gets \emptyset$
        \ForAll {$S$: matching parentheses of length $n$}
            \State $T \gets$ an empty stack
            \State $L \gets$ an empty list
            \State Add $0$ to top of $T$
            \State $c \gets 1$
            \ForAll{$i = 1,\cdots, n$}
                \If{$S_i =$ '('}
                    \State Add $c$ to top of $T$
                    \State Append $c$ to $L$
                    \State $c \gets c+1$
                \EndIf
                \If{$S_i =$ ')'}
                    \State Append uppermost value of $T$ to $L$
                    \State Pop uppermost value from $T$.
                \EndIf
            \EndFor
            \State $P \gets P \cup \{L\}$
        \EndFor
        \State \Return $P$
    \end{algorithmic}
\end{algorithm}

\subsection{Homomorphism}
\label{sect:code-hom}
The program corresponding to the pseudocode in the following sections is presented in this GitHub repository \drop{\url{https://github.com/near-linear-4ct/computer-checks}}. 
Each pseudocode is linked to the corresponding function in our \Cpp program.
The link is removed in order to anonymize paper, will be included in the camera-ready version. 

\newcommand{\githubbase}{https://github.com/near-linear-4ct/computer-checks/blob/49eecf5b9b59e32e2b0ef6fedb9505e6714fcd7d}
\newcommand{\githublink}[3]{\href{\githubbase/#1\##2}{#3}} 

\newcommand{\enumWheelsWithLink}{\githublink{src/cartwheel.cpp}{L92-L115}{enumWheels}}
\newcommand{\generateCartwheelWithLink}{\githublink{src/cartwheel.cpp}{L117-L162}{generateCartwheel}}
\newcommand{\enumPossibleBadWheelsWithLink}{\githublink{src/cartwheel.cpp}{L164-L180}{enumPossibleBadWheels}}
\newcommand{\fixInRulesWithLink}{\githublink{src/cartwheel.cpp}{L183-L211}{fixInRules}}
\newcommand{\updateDegreeByRuleWithLink}{\githublink{src/cartwheel.cpp}{L213-L227}{updateDegreeByRule}}
\newcommand{\concreteDegreeExceptTailWithLink}{\githublink{src/cartwheel.cpp}{L229-L250}{concreteDegreeExceptTail}}
\newcommand{\pruneWithLink}{\githublink{src/cartwheel.cpp}{L252-L267}{prune}}
\newcommand{\pruneByNonAssociatedRuleWithLink}{\githublink{src/cartwheel.cpp}{L269-L284}{pruneByNonAssociatedRule}}
\newcommand{\upperBoundOfChargeWithLink}{\githublink{src/cartwheel.cpp}{L286-L304}{upperBoundOfCharge}}
\newcommand{\fixOutRulesWithLink}{\githublink{src/cartwheel.cpp}{L307-L350}{fixOutRules}}
\newcommand{\shouldRefineWithLink}{\githublink{src/cartwheel.cpp}{L352-L355}{shouldRefine}}
\newcommand{\refinementWithLink}{\githublink{src/cartwheel.cpp}{L357-L375}{refinement}}
\newcommand{\refineAlwaysWithLink}{\githublink{src/cartwheel.cpp}{L377-L387}{refineAlways}}
\newcommand{\refineNeverWithLink}{\githublink{src/cartwheel.cpp}{L389-L402}{refineNever}}
\newcommand{\enumBadCartwheelsWithLink}{\githublink{src/cartwheel.cpp}{L404-L423}{enumBadCartwheels}}
\newcommand{\centerDartsByDegreeWithLink}{\githublink{src/cartwheel.cpp}{L425-L434}{centerDartsByDegree}}

\newcommand{\deleteDegreeFromKToNWithLink}{\githublink{src/combine_cartwheel.cpp}{L7-L25}{deleteDegreeFromKTo9}}
\newcommand{\checkDegEWithLink}{\githublink{src/combine_cartwheel.cpp}{L53-L80}{checkDeg8}}
\newcommand{\checkEEWithLink}{\githublink{src/combine_cartwheel.cpp}{L82-L91}{check88}}
\newcommand{\checkESWithLink}{\githublink{src/combine_cartwheel.cpp}{L93-L102}{check87}}
\newcommand{\checkSESWithLink}{\githublink{src/combine_cartwheel.cpp}{L104-L133}{check787}}
\newcommand{\containXWithLink}{\githublink{src/combine_cartwheel.cpp}{L161-L172}{containX}}
\newcommand{\checkStriangleWithLink}{\githublink{src/combine_cartwheel.cpp}{L182-L211}{check7triangle}}
\newcommand{\checkDegSWithLink}{\githublink{src/combine_cartwheel.cpp}{L221-L248}{checkDeg7}}
\newcommand{\checkSSWithLink}{\githublink{src/combine_cartwheel.cpp}{L250-L259}{check77}}
\newcommand{\checkSSSWithLink}{\githublink{src/combine_cartwheel.cpp}{L261-L275}{check777}}

\newcommand{\mirrorWithLink}{\githublink{src/configuration.cpp}{L83-L89}{mirror}}
\newcommand{\extendFromCutVerticesWithLink}{\githublink{src/configuration.cpp}{L100-L127}{extendFromCutVertices}}
\newcommand{\findCutPairsWithLink}{\githublink{src/configuration.cpp}{L129-L156}{findCutPairs}}
\newcommand{\removeRingWithLink}{\githublink{src/configuration.cpp}{L158-L203}{removeRing}}
\newcommand{\maximumDegreeDartWithLink}{\githublink{src/configuration.cpp}{L205-L222}{maximumDegreeDart}}

\newcommand{\homomorphismWithLink}{\githublink{src/pseudo_configuration.hpp}{L86-L131}{homomorphism}}

\newcommand{\freeHomomorphismConfigurationWithLink}{\githublink{src/pseudo_configuration.cpp}{L97-L108}{freeHomomorphismConfiguration}}
\newcommand{\dartIdentificationWithLink}{\githublink{src/pseudo_configuration.cpp}{L77-L94}{dartIdentification}}
\newcommand{\resolveDegreeIssuesWithLink}{\githublink{src/pseudo_configuration.cpp}{L110-L141}{resolveDegreeIssues}}
\newcommand{\innerSubdegreeErrorWithLink}{\githublink{src/pseudo_configuration.cpp}{L143-L152}{innerSubdegreeError}}
\newcommand{\vertexSingleDegreeIssueWithLink}{\githublink{src/pseudo_configuration.cpp}{L154-L168}{vertexSingleDegreeIssue}}
\newcommand{\fixSingleDegreeIssueWithLink}{\githublink{src/pseudo_configuration.cpp}{L171-L188}{fixSingleDegreeIssue}}
\newcommand{\addBoundaryDartsWithLink}{\githublink{src/pseudo_configuration.cpp}{L190-L210}{addBoundaryDarts}}
\newcommand{\singleOutLowerDegreeWithLink}{\githublink{src/pseudo_configuration.cpp}{L213-L224}{singleOutLowerDegree}}
\newcommand{\containConfWithLink}{\githublink{src/pseudo_configuration.cpp}{L226-L249}{containConf}}
\newcommand{\dartsByDegreeWithLink}{\githublink{src/pseudo_configuration.cpp}{L251-L268}{dartsByDegree}}
\newcommand{\rootedContainConfWithLink}{\githublink{src/pseudo_configuration.cpp}{L270-L272}{rootedContainConf}}
\newcommand{\blockedByReducibleConfigurationWithLink}{\githublink{src/pseudo_configuration.cpp}{L274-L282}{blockedByReducibleConfiguration}}
\newcommand{\representativeDegreeWithLink}{\githublink{src/pseudo_configuration.cpp}{L284-L313}{representativeDegree}}
\newcommand{\alwaysApplyWithLink}{\githublink{src/pseudo_configuration.cpp}{L315-L317}{alwaysApply}}
\newcommand{\neverApplyWithLink}{\githublink{src/pseudo_configuration.cpp}{L319-L321}{neverApply}}
\newcommand{\amountOfChargeSendWithLink}{\githublink{src/pseudo_configuration.cpp}{L323-L331}{amountOfChargeSend}}
\newcommand{\amountOfPossibleChargeSendWithLink}{\githublink{src/pseudo_configuration.cpp}{L333-L343}{amountOfPossibleChargeSend}}
\newcommand{\dominantlyApplyWithLink}{\githublink{src/pseudo_configuration.cpp}{L345-L351}{dominantlyApply}}
\newcommand{\combineEachCartwheelWithLink}{\githublink{src/pseudo_configuration.cpp}{L354-L370}{combineEachCartwheel}}
\newcommand{\combineEachCartwheelTwiceWithLink}{\githublink{src/pseudo_configuration.cpp}{L373-L388}{combineEachCartwheelTwice}}

\newcommand{\fromVRotationsWithLink}{\githublink{src/pseudo_triangulation.cpp}{L40-L75}{fromVRotations}}
\newcommand{\freeHomomorphismTriangulationWithLink}{\githublink{src/pseudo_triangulation.cpp}{L197-L249}{freeHomomorphismTriangulation}}

\newcommand{\addRuleToCombinationWithLink}{\githublink{src/rule.cpp}{L151-L174}{addRuleToCombination}}
\newcommand{\combineRulesWithLink}{\githublink{src/rule.cpp}{L176-L199}{combineRules}}

\newcommand{\homomorphism}{\texttt{homomorphism}\xspace}
\newcommand{\freeHomomorphismTriangualtion}{\texttt{freeHomomorphismTriangulation}\xspace}
\newcommand{\dartIdentification}{\texttt{dartIdentification}\xspace}
\newcommand{\resolveDegreeIssues}{\texttt{resolveDegreeIssues}\xspace}

\newcommand{\alwaysApply}{\texttt{alwaysApply}\xspace}
\newcommand{\neverApply}{\texttt{neverApply}\xspace}
\newcommand{\dominantlyApply}{\texttt{dominantlyApply}\xspace}
\newcommand{\amountOfChargeSend}{\texttt{amountOfChargeSend}\xspace}
\newcommand{\amountOfPossibleChargeSend}{\texttt{amountOfPossibleChargeSend}\xspace}
\newcommand{\enumWheels}{\texttt{enumWheels}\xspace}
\newcommand{\generateCartwheel}{\texttt{generateCartwheel}\xspace}
\newcommand{\enumPossibleBadWheels}{\texttt{enumPossibleBadWheels}\xspace}
\newcommand{\fixInRules}{\texttt{fixInRules}\xspace}
\newcommand{\updateDegreeByRule}{\texttt{updateDegreeByRule}\xspace}
\newcommand{\concreteDegreesExceptTail}{\texttt{concreteDegreesExceptTail}\xspace}
\newcommand{\pruneByNonAssociatedRule}{\texttt{pruneByNonAssociatedRule}\xspace}
\newcommand{\upperBoundOfCharge}{\texttt{upperBoundOfCharge}\xspace}
\newcommand{\prune}{\texttt{prune}\xspace}
\newcommand{\fixOutRules}{\texttt{fixOutRules}\xspace}
\newcommand{\refinement}{\texttt{refinement}\xspace}
\newcommand{\shouldRefine}{\texttt{shouldRefine}\xspace}
\newcommand{\refineAlways}{\texttt{refineAlways}\xspace}
\newcommand{\refineNever}{\texttt{refineNever}\xspace}
\newcommand{\enumBadCartwheels}{\texttt{enumBadCartwheels}\xspace}
\newcommand{\CARTWHEELDEGREES}{\texttt{CARTWHEEL\_DEGREES}\xspace}

\newcommand{\gIntersection}{g_\textsf{intersection}}
\newcommand{\gInclude}{g_\textsf{include}}
The following routine, \homomorphism, finds a homomorphism $\phi^*$ from $Z$ to $Z^*$ that maps $e$ to $e^*$ and satisfies a specified degree constraint for all mapped pairs of vertices $(v, \phi^*(v))$.
If no degree constraint is imposed, it returns a standard homomorphism between pseudo-triangulations.
One of the common constraints we use is that the intersection of the degree ranges to be non-empty (i.e., $[\delta^-(v), \delta^+(v)] \cap [\delta^{*-}(\phi^*(v)), \delta^{*+}(\phi^*(v))] \neq \emptyset$ ), or the degree range $[\delta^-(v), \delta^+(v)]$ includes the other degree range $[\delta^{*-}(\phi^*(v)), \delta^{*+}(\phi^*(v))]$.
We specify these constraints using the boolean functions $\gIntersection$ and $\gInclude$, respectively.
Standard applications of this routine include verifying homomorphic images of reducible configurations and determining whether a rule applies to a specific dart.
For references, see Algorithms \ref{alg:rooted_contain_conf}, \ref{alg:ineivtably_apply}, \ref{alg:never_apply} and \ref{alg:dominantly_apply}.

Note that $V(Z)$ and $D(Z)$ denote the sets of vertices and darts of $Z$, respectively, when $Z$ is a pseudo-triangulation.
We use a queue data structure. The queue supports 
the following operations.

\begin{itemize}
    \item \texttt{push}($x$): Pushes $x$ into the queue.
    \item \texttt{pop}(): Removes and returns the oldest element from the queue.
    \item \texttt{empty}(): Returns \texttt{true} if the queue is empty, and \texttt{false} otherwise.
\end{itemize}

We adopt the dot notation for these operations for data structures (e.g., \texttt{T.root}($x$) for a union-find instance $T$, \texttt{Q.push}($x$) for a queue instance \texttt{Q}).
For this data structure, we use \texttt{std::queue} in our \Cpp program.

\begin{algorithm}[H]\caption{\homomorphismWithLink($(Z, \delta^-, \delta^+), e, (Z^*, \delta^{*-}, \delta^{*+}) ,e^*, g$)}
    \label{alg:homomorphism}
    \begin{algorithmic}[1]
        \Require{A pseudo configuration $(Z,\delta^-,\delta^+)$ with a dart $e$ and a pseudo configuration $(Z^*, \delta^{*-}, \delta^{*+})$ with a dart $e^*$, a function $g$ that takes two degree-ranges and returns the boolean flag.}
        \Ensure{The homomorphism $\phi^*$ from $Z$ to $Z^*$ mapping $e$ to $e^*$ if it exists and further satisfies $g(\delta^-(v), \delta^+(v), \delta^{*-}(v^*), \delta^{*+}(v^*))=\True$ for all vertices $v \in Z$ with  $v^*=\phi^*(v)$. It returns \texttt{null} if no such homomorphism exists.}
        \State $\phi^* \gets$ a map from $V(Z) \cup D(Z)$ to $V(Z^*) \cup D(Z^*) \cup \{\bot\}$, initially $\phi^*(x)=\bot$ for all $x \in V(Z) \cup D(Z)$.
        \State $\texttt{Q} \gets$ an empty queue
        \State \texttt{Q.push}$((e,e^*))$
        \While{not $\texttt{Q.empty()}$}
            \State $(f, f^*) \gets \texttt{Q.pop()}$
            \If {$\phi^*(f) \neq \bot$}
                \If {$\phi^*(f) \neq f^*$}
                    \State \Return \texttt{null}
                \EndIf
                \Continue
            \EndIf
            \State $\phi^*(f) \gets f^*$
            \State $h \gets \head(f)$
            \State $h^* \gets \head(f^*)$
            \If {$\phi^*(h) \neq \bot$ and $\phi^*(h) \neq h^*$}
                \State \Return \texttt{null}
            \EndIf
            \State $\phi^*(h) \gets h^*$
            \If {$g(\delta^-(h), \delta^+(h), \delta^{*-}(h^*), \delta^{*+}(h^*))=\False$}
                 \State \Return \texttt{null} \Comment{A degree constraint is not satisfied.}
            \EndIf
            \State \texttt{Q.push}$(\reverse(f), \reverse(f^*))$
            \If {$\successor(f) \neq \nil$ and $\successor(f^*)=\nil$}
                \State \Return \texttt{null} \Comment{a homomorphism from $Z$ to $Z^*$ does not exist.}
            \ElsIf {$\successor(f) \neq \nil$ and $\successor(f^*) \neq \nil$}
                \State \texttt{Q.push}$(\successor(f), \successor(f^*))$
            \EndIf
            \If {$\predecessor(f) \neq \nil$ and $\predecessor(f^*)=\nil$}
                \State \Return \texttt{null} \Comment{a homomorphism from $Z$ to $Z^*$ does not exist.}
            \ElsIf {$\predecessor(f) \neq \nil$ and $\predecessor(f^*) \neq \nil$}
                \State \texttt{Q.push}$(\predecessor(f), \predecessor(f^*))$
            \EndIf
        \EndWhile
        \State \Return $\phi^*$
    \end{algorithmic}
\end{algorithm}

\subsection{The free homomorphism from a pseudo-triangulation respecting identification requests}
We present the pseudocode for computing the free homomorphism of a pseudo-triangulation respecting identification requests. Although we have already described this algorithm in Section \ref{subsub:homomorphism-tri}, we provide the pseudocode for clarity.

We employ a union-find data structure, implemented as a forest of rooted trees.
We denote by \texttt{uf}($S$) the union-find structure defined over a set $S$, where
initially each element of $S$ is a root.
This data structure supports the following operations.

\begin{itemize}
    \item \texttt{root}($x$): Returns the root of the tree containing $x$.
    \item \texttt{unite}($x$, $y$): Unites the two trees containing $x$ and $y$ by setting the parent of \texttt{root}($x$) to \texttt{root}($y$).
    \item \texttt{same}($x$, $y$): Returns \texttt{true} if the same tree contains both $x$ and $y$, and \texttt{false} otherwise.
\end{itemize}

\begin{algorithm}[H]
    \caption{\freeHomomorphismTriangulationWithLink($Z$, $\{(e_i, f_i)\}_{i=1}^k$)}
    \label{alg:homomorphism-tri}
    \begin{algorithmic}[1]
        \Require{A pseudo-triangulation $Z$,
        a set of pairs of darts $\{(e_i, f_i)\}_{i=1}^k$ in $D(Z)$.}
        \Ensure{A pseudo-triangulation $Z^*$, and a homomorphism $\phi^* \colon V(Z) \cup D(Z) \to V(Z^*) \cup D(Z^*)$ respecting the identification requests $\{(e_i, f_i)\}_{i=1}^k$.}
        \State $\texttt{uf\_V} \gets \texttt{uf}(V(Z))$
        \State $\texttt{uf\_D} \gets$ \texttt{uf}($D(Z)$)
        \State $\texttt{Q} \gets$ an empty queue
        \ForAll{$i=1,\dots,k$}
            \State $\texttt{Q.push}$$((e_i, f_i))$
        \EndFor
        \While{not \texttt{Q.empty}()}
            \State $(e,f) \gets \texttt{Q.pop}()$ 
            \If {\texttt{uf\_D.same}($e,f$)}
                \Continue
            \EndIf
            \If {not \texttt{uf\_V.same}(\head($e$), \head($f$))}
                \State \texttt{uf\_V.unite}(\head($e$), \head($f$))\label{line:unite}
            \EndIf
            \State $e^* \gets \texttt{uf\_D.root}(e)$
            \State $f^* \gets \texttt{uf\_D.root}(f)$
            \State \texttt{uf\_D.unite}($e^*$, $f^*$) \Comment{Now, $f^*$ becomes the root representative.}
            \State \texttt{Q.push}(\reverse($e^*$), \reverse($f^*$))
            \If {\successor($e^*$) $\neq$ \nil\ and \successor($f^*$) $\neq$ \nil}
                \State \texttt{Q.push}(\successor($e^*$), \successor($f^*$))
            \EndIf
             \If {\predecessor($e^*$) $\neq$ \nil\ and \predecessor($f^*$) $\neq$ \nil}
                \State \texttt{Q.push}(\predecessor($e^*$), \predecessor($f^*$))
            \EndIf
            \If {\successor($e^*$) $\neq$ \nil\ and \successor($f^*$) $=$ \nil}
                \State \successor($f^*$) $\gets$ \successor($e^*$)
            \EndIf
            \If {\predecessor($e^*$) $\neq$ \nil\ and \predecessor($f^*$) $=$ \nil}
                \State \predecessor($f^*$) $\gets$ \predecessor($e^*$)
            \EndIf
        \EndWhile
        \State $V^* \gets$ the set of $v \in V(Z)$ with \texttt{uf\_V.root}($v$) $= v$ 
        \State $D^* \gets$ the set of $d \in D(Z)$ with \texttt{uf\_D.root}($d$) $= d$
        \ForAll{$d \in D^*$}
            \State \head($d$) $\gets$ \texttt{uf\_V.root}(\head($d$))
            \State \reverse($d$) $\gets$ \texttt{uf\_D.root}(\reverse($d$))
            \State \successor($d$) $\gets$ \nil\ if \successor($d$) $=$ \nil\ otherwise, \texttt{uf\_D.root}(\successor($d$))
            \State \predecessor($d$) $\gets$ \nil\ if \predecessor($d$) $=$ \nil\ otherwise, \texttt{uf\_D.root}(\predecessor($d$))
        \EndFor
        \State $\phi^* \colon V(Z) \cup D(Z) \to V^* \cup D^*$ is a map defined $\phi^*(x)=\texttt{uf\_V.root}(x)$ if $x \in V(Z)$, and $\texttt{uf\_D.root}(x)$ otherwise.
        \State \Return $Z^*$ defined by $(V^*, D^*)$ and $\phi^*$
    \end{algorithmic}
\end{algorithm}

\subsection{The free homomorphism from a pseudo-configuration respecting identification requests}
\label{code:hom-conf}
In the following, we present the pseudocode for computing the free homomorphism of a pseudo-configuration in Algorithm \ref{alg:homomorphism-conf}.
Here, we focus on the case where vertices are assigned degree ranges, as this includes the case where vertices have single concrete degrees.
Specifically, for a pseudo-configuration with single concrete degrees $(Z, \delta)$, calling the function with input $(Z, \delta, \delta)$ yields the result for $(Z, \delta)$.
The algorithm requires several subroutines.
First, we introduce the subroutine \dartIdentification that identifies the darts of the underlying pseudo-triangulation obtained by forgetting the degree function from the input pseudo-configuration.
Next, it maps the degrees, calculating the intersection of the degree ranges for each set of identified vertices.
This procedure returns \texttt{null} if it attempts to identify vertices whose degree ranges are disjoint (a degree-mismatch error), or the resulting pseudo-configuration has a loop (a loop error).

\begin{algorithm}[H]
    \caption{\dartIdentificationWithLink($(Z, \delta^-, \delta^+)$, $\{(e_i, f_i)\}_{i=1}^k$)}
    \label{alg:dartIdentification}
    \begin{algorithmic}[1]
        \Require{A pseudo-configuration $(Z, \delta^-, \delta^+)$, a set of pairs of darts $\{(e_i, f_i)\}_{i=1}^k$ in $D(Z)$.}
        \Ensure{A pseudo-configuration $(Z^*, \delta^{*-}, \delta^{*+})$, a map $\phi^* \colon V(Z) \cup D(Z) \to V(Z^*) \cup D(Z^*)$ obtained from $Z$ by identifying all pairs of darts $\{(e_i,f_i)\}_{i=1}^k$ or \texttt{null} if the identification causes a loop, degree-mismatch error.}
        \State $Z^*, \phi^* \gets $ \Call{freeHomomorphismTriangulation}{$Z, \{(e_i, f_i)\}_{i=1}^k$}
        \If {$Z^*$ has a loop (i.e., a dart $d$ with $\head(d)=\head(\reverse(d))$) exists.}
            \State \Return \texttt{null} \Comment{a loop error}
        \EndIf
        \State $\delta^{*-}, \delta^{*+} \colon V(Z^*) \to \mathbb{N} \cup \{\infty\}$, initially defined by $\delta^{*-}(v)=1, \delta^{*+}(v)=\infty$ for all $v \in V(Z^*)$.
        \ForAll {$v \in V(Z)$}
            \State $v^* \gets \phi^*(v)$
            \If {$[\delta^{*-}(v^*), \delta^{*+}(v^*)] \cap [\delta^-(v), \delta^+(v)] = \emptyset$} \Comment{Two ranges are disjoint.}
                \State \Return \texttt{null} \Comment{a degree-mismatch error}
            \EndIf
            \State $[\delta^{*-}(v^*), \delta^{*+}(v^*)]=[\delta^{*-}(v^*), \delta^{*+}(v^*)] \cap [\delta^{-}(v), \delta^{+}(v)]$ \Comment{Calculate the intersection.}
        \EndFor
        \State \Return $((Z^*, \delta^{*-}, \delta^{*+}), \phi^*)$
    \end{algorithmic}
\end{algorithm}

We can optimize the \dartIdentification subroutine by using a modified version of Algorithm \ref{alg:homomorphism-tri}.
Specifically, each time we identify two vertices in Line 13 of Algorithm \ref{alg:homomorphism-tri}, we verify that the intersection of the degree-ranges associated with these vertices is non-empty.
To achieve this efficiently, the root representative of each set in $\texttt{uf\_V}$ maintains the intersection of the degree-ranges for all vertices in its tree.
Consequently, at the end of the execution, the degree-ranges of the root representatives correspond exactly to the resulting degree-ranges $\delta^{*-}$ and $\delta^{*+}$ calculated in \dartIdentification subroutine.
This modified procedure is presented in Algorithm \ref{alg:homomorphism-tri-degree}.
The optimized \dartIdentification subroutine calls this modified algorithm, checks for the existence of any loops, and then directly returns the resulting homomorphic image along with its degree-ranges.
However, we do not implement this since the unoptimized algorithm is fast enough.

\begin{algorithm}[H]
\caption{freeHomomorphismTriangulationWithDegreeIntersection($(Z,\delta^-,\delta^+)$, $\{(e_i, f_i)\}_{i=1}^k$)}
    \label{alg:homomorphism-tri-degree}
    \begin{algorithmic}[1]
        \State \texttt{The algorithm in Line 1-11 in Algorithm \ref{alg:homomorphism-tri}}
        \If {not \texttt{uf\_V.same}(\head($e$), \head($f$))}
            \State $h^*_e \gets \texttt{uf\_V.root}(\head(e))$
            \State $h^*_f \gets \texttt{uf\_V.root}(\head(f))$
            \State \texttt{uf\_V.unite}($h^*_e$, $h^*_f$) \Comment{Now, $h^*_f$ becomes the representative root.}
            \If {$[\delta^-(h^*_e), \delta^+(h^*_e)] \cap [\delta^-(h^*_f), \delta^+(h^*_f)] = \emptyset$} \Comment{Two ranges are disjoint.}
                \State \Return \texttt{null} \Comment{a degree-mismatch error}
            \EndIf
            \State $[\delta^-(h^*_f), \delta^+(h^*_f)] \gets [\delta^-(h^*_f), \delta^+(h^*_f)] \cap [\delta^-(h^*_e), \delta^+(h^*_e)]$ \Comment{Calculate the intersection.}
        \EndIf
        \State \texttt{The algorithm in Line 15-41 in Algorithm \ref{alg:homomorphism-tri}}
        \State $\delta^{*-}(v) \gets \delta^-(v)$ and $\delta^{*+}(v) \gets \delta^+(v)$ for every $v \in V^*$
        \State \Return $(Z^*, \delta^{*-}, \delta^{*+})$ and $\phi^*$.
  \end{algorithmic}
\end{algorithm}

Algorithm \ref{alg:homomorphism-conf} is the main routine that creates the free homomorphism respecting identification requests, described in Section \ref{subsub:homomorphism-conf} and Section \ref{sect:ranges}.
We first identify the given pairs of darts using Algorithm \ref{alg:dartIdentification}.
Then, if the identification is successful, we proceed to resolve potential \emph{degree-issues} by calling the \resolveDegreeIssues subroutine.
For a pseudo-configuration with single concrete degrees, we require that $d_Z(v) < \delta(v)$ for any boundary vertex, and $d_Z(v)=\delta(v)$ for any inner vertex $v$.
Generalizing this into degree ranges, a vertex is said to have \emph{degree-issues} if it violates the following condition.
\begin{itemize}
    \item $d_Z(v) < \delta^-(v)$ if $v$ is a boundary vertex and
    \item $d_Z(v) = \delta^-(v) = \delta^+(v)$ if $v$ is an inner vertex.
\end{itemize}

\begin{algorithm}[H]
    \caption{\freeHomomorphismConfigurationWithLink($(Z, \delta^-, \delta^+)$, $\{(e_i, f_i)\}_{i=1}^k$)}
    \label{alg:homomorphism-conf}
    \begin{algorithmic}[1]
        \Require{A pseudo-configuration $(Z, \delta^-, \delta^+)$, a set of pairs of darts $\{(e_i, f_i)\}_{i=1}^k$ in $D(Z)$.}
        \Ensure{A set $\mathcal{Z}$ of a pair of pseudo-configurations $(Z_j^*, \delta_j^{*-}, \delta_j^{*+})$ and a homomorphism $\phi_j^* \colon V(Z) \cup D(Z) \to V(Z_j^*) \cup D(Z_j^*)$ respecting the identification $\{(e_i,f_i)\}_{i=1}^k$.}
        \State $A \gets $\Call{dartIdentification}{$(Z, \delta^-, \delta^+)$, $\{(e_i, f_i)\}_{i=1}^k$}
        \If {$A = \texttt{null}$}
            \State \Return $\emptyset$
        \EndIf
        \State $((Z^* \delta^{*-}, \delta^{*+}), \phi^*) \gets A$
        \State $\tilde{\mathcal{Z}} \gets$ \Call{resolveDegreeIssues}{$Z^*, \delta^{*-}, \delta^{*+}$}
        \State \Return $\{((\tilde{Z}, \tilde{\delta}^{-}, \tilde{\delta}^{+}), \tilde{\phi} \circ \phi^*) \mid ((\tilde{Z}, \tilde{\delta}^{-}, \tilde{\delta}^{+}), \tilde{\phi}) \in \tilde{\mathcal{Z}} \}$
    \end{algorithmic}
\end{algorithm}

The \resolveDegreeIssues routine in Algorithm \ref{alg:resolve_degree_issues} is designed to resolve degree-issues by iteratively identifying or adding darts via several subroutines.
$\id{X}$ denotes the identity map on a set $X$.
We use the four subroutines.

First, the \texttt{innerSubdegreeError} subroutine detects any inner vertex $v$ with $d_Z(v) < \delta^-(v)$, which implies a subdegree error.
If such an error is detected, we immediately give up on resolving degree issues.
We present the \texttt{innerSubdegreeError} subroutine in Algorithm \ref{alg:inner_subdegree_error}.
Note that we decide that $v$ is a boundary vertex if some dart $e$ with $\head(e)=v$ satisfies $\successor(e)=\nil$.
Otherwise, $v$ is considered an inner vertex.

Second, the \texttt{vertexSingleDegreeIssue} subroutine detects a vertex with a single specified degree (i.e., $\delta(v) \coloneq \delta^-(v)=\delta^+(v)$) and degree-issue.
Specifically, it searches for
\begin{itemize}
    \item a vertex $v$ with $\delta(v) < d_Z(v)$, or
    \item a boundary vertex $v$ with $\delta(v) = d_Z(v)$.
\end{itemize}
If no such vertex exists, then this subroutine returns $\texttt{null}$.
We present the \texttt{vertexSingleDegreeIssue} subroutine in Algorithm \ref{alg:vertex_single_degree_issue}.

Third, the \texttt{fixSingleDegreeIssue} subroutine attempts to resolve a degree-issue found by \texttt{vertexSingleDegreeIssue}.
If such a degree-issue cannot be resolved due to errors, such as a loop, subdegree, or degree-mismatch error, this subroutine returns \texttt{null}.
We present the \texttt{fixSingleDegreeIssue} subroutine in Algorithm \ref{alg:fix_single_degree_issue}.
More detailed implementations are given later.

Fourth, the \texttt{singleOutLowerDegree} subroutine is responsible for subdividing degree ranges if it is necessary to resolve degree-issues.
More precisely, this subroutine detects a vertex with $\delta^-(v) < \delta^+(v)$ and $\delta^-(v) \leq d_Z(v)$.
If found, it subdivides degree range $[\delta^-(v), \delta^+(v)]$ into $[\delta^-(v), \delta^-(v)]$ and $[\delta^-(v)+1, \delta^+(v)]$, creating two corresponding copies of the given pseudo-configuration.
This is the only subroutine that performs degree range subdivision.
We present the \texttt{singleOutLowerDegree} subroutine in Algorithm \ref{alg:single_out_lower_deg}.

We iteratively apply these subroutines.
The process terminates successfully when the pseudo-configuration is free of degree-issues:
if \texttt{innerSubDegreeError} returns \False, \texttt{vertexSingleDegreeIssue} returns \texttt{null}, and \texttt{singleOutLowerDegree} returns \texttt{null}.
In this case, the resulting pseudo-configuration is added to the output set.

\begin{algorithm}[H]
    \caption{\resolveDegreeIssuesWithLink($Z, \delta^-, \delta^+$)}
    \label{alg:resolve_degree_issues}
    \begin{algorithmic}[1]
        \Require{A pseudo-configuration $(Z, \delta^-, \delta^+)$ with potential degree-issues.}
        \Ensure{A set of pseudo-configurations $(\tilde{Z}, \tilde{\delta}^-, \tilde{\delta}^+)$ with $\tilde{\phi} \colon V(Z) \cup D(Z) \to V(\tilde{Z}) \cup D(\tilde{Z})$ obtained from $(Z, \delta^-, \delta^+)$ by resolving degree-issues.}
        \State $\mathcal{Z} \gets \emptyset$
        \State $\texttt{Q} \gets$ an empty queue
        \State \texttt{Q.push}($(Z, \delta^-, \delta^+), \id{V(Z) \cup D(Z)}$)
        \While {not \texttt{Q.empty}()}
            \State $((\tilde{Z}, \tilde{\delta}^-, \tilde{\delta}^+), \tilde{\phi}) \gets \texttt{Q.pop}()$
            \If {\Call{innerSubdegreeError}{$\tilde{Z}, \tilde{\delta}^-, \tilde{\delta}^+$} = \True}
                \Continue
            \EndIf
            \State $v \gets $\Call{vertexSingleDegreeIssue}{$\tilde{Z}, \tilde{\delta}^-, \tilde{\delta}^+$}
            \If {$v \neq \texttt{null}$}
                \State $A \gets $ \Call{fixSingleDegreeIssue}{$(\tilde{Z}, \tilde{\delta}^-, \tilde{\delta}^+)$, $v$}
                \If {$A \neq \texttt{null}$}
                    \State $((Z^*, \delta^{*-}, \delta^{*+}), \phi^*) \gets A$ 
                    \State \texttt{Q.push}($(Z^*, \delta^{*-}, \delta^{*+}), \phi^* \circ \tilde{\phi}$)
                \EndIf
                \Continue
            \EndIf
            \State $B \gets $\Call{singleOutLowerDegree}{$\tilde{Z}, \tilde{\delta}^-, \tilde{\delta}^+$}
            \If {$B \neq \texttt{null}$}
                \State $((Z_1, \delta_1^-, \delta_1^+), (Z_2, \delta_2^-, \delta_2^+)) \gets B$
                \State \texttt{Q.push}($(Z_1, \delta_1^-, \delta_1^+), \tilde{\phi}$)
                \State \texttt{Q.push}($(Z_2, \delta_2^-, \delta_2^+), \tilde{\phi}$)
                \Continue
            \EndIf
            \State $\mathcal{Z} \gets \mathcal{Z} \cup \{((\tilde{Z}, \tilde{\delta}^-, \tilde{\delta}^+), \tilde{\phi})\}$ 
        \EndWhile
        \State \Return $\mathcal{Z}$
    \end{algorithmic}
\end{algorithm}

\begin{algorithm}[H]
    \caption{\innerSubdegreeErrorWithLink($Z, \delta^-, \delta^+$)}
    \label{alg:inner_subdegree_error}
    \begin{algorithmic}[1]
        \Require{A pseudo-configuration $(Z, \delta^-, \delta^+)$.}
        \Ensure{\True\ if an inner vertex $v$ with $d_Z(v) < \delta^-(v)$ exists, otherwise \False.}
        \ForAll{$v \in V(Z)$}
            \If {$v$ is an inner vertex and $d_Z(v) < \delta^-(v)$}
                \State \Return \True
            \EndIf
        \EndFor
        \State \Return \False
    \end{algorithmic}
\end{algorithm}

\begin{algorithm}[H]
    \caption{\vertexSingleDegreeIssueWithLink($Z, \delta^-, \delta^+)$}
    \label{alg:vertex_single_degree_issue}
    \begin{algorithmic}[1]
        \Require{A pseudo-configuration $(Z, \delta^-, \delta^+)$.}
        \Ensure{A vertex $v$ with $\delta^-(v)=\delta^+(v)$ and degree-issue if exists, otherwise \texttt{null}.}
        \ForAll{$v \in V(Z)$}
            \If {$\delta^-(v) \neq \delta^+(v)$}
                \Continue
            \EndIf
            \If {$\delta^-(v) < d_Z(v)$}
                \State \Return $v$
            \ElsIf {$v$ is a boundary vertex and $d_Z(v)=\delta^-(v)$}
                \State \Return $v$
            \EndIf
        \EndFor
        \State \Return \texttt{null}
    \end{algorithmic}
\end{algorithm}

The following subroutine \texttt{fixSingleDegreeIssue} resolves a degree-issue, which occurred in a vertex with $\delta(v) \coloneq \delta^-(v)=\delta^+(v)$, as follows.
\begin{enumerate}
    \item When $v$ is a vertex $v$ with $\delta(v) < d_Z(v)$: let $e$ be a dart whose head is $v$. If $v$ is a boundary vertex, we choose the first dart. Let $f$ be the dart reached from $e$ by following \successor\ $\delta(v)$ times. We identify $e$ and $f$ via the \texttt{dartIdentification} subroutine.
    \item When $v$ is a boundary vertex with $\delta(v) = d_Z(v)$: we make $v$ an inner vertex by appropriately adding darts and updating pointers. We employ the \texttt{addBoundaryDarts} subroutine in Algorithm \ref{alg:add_boundary_darts} to handle this case. 
\end{enumerate}
In any case, it may return \texttt{null}, which means that a degree-issue cannot be resolved successfully.

\begin{algorithm}[H]
    \caption{\fixSingleDegreeIssueWithLink($(Z, \delta^-, \delta^+), v$)}
    \label{alg:fix_single_degree_issue}
    \begin{algorithmic}[1]
        \Require{A pseudo-configuration $(Z, \delta^-, \delta^+)$, a vertex $v$ with $\delta^-(v)=\delta^+(v)$ and degree-issue.}
        \Ensure{A pair of pseudo-configuration $(Z^*, \delta^{*-}, \delta^{*+})$ and $\phi^* \colon V(Z) \cup D(Z) \to V(Z^*) \cup D(Z^*)$ obtained from $(Z, \delta^-, \delta^+)$ by resolving degree-issue on $v$ or \texttt{null} if it causes an error.}
        \If {$\delta^-(v) < d_Z(v)$}
            \State $e$ $\gets$ a dart whose head is $v$. If $v$ is a boundary vertex, we choose the first dart (i.e., \predecessor($e$)= \nil)).
            \State $f$ $\gets$ the dart reached from $e$ by following \successor\ pointer $\delta^-(v)$ times.
            \State \Return \Call{dartIdentification}{$(Z, \delta^-, \delta^+), \{(e, f)\}$} \Comment{It may return \texttt{null} due to a loop, degree-mismatch error.}
        \ElsIf {$v$ is a boundary vertex and $\delta^-(v)=d_Z(v)$}
            \State $A \gets$ \Call{addBoundaryDarts}{$(Z, \delta^-, \delta^+), v$}
            \If {$A \neq \texttt{null}$}
                \State $(\tilde{Z}, \tilde{\delta}^-, \tilde{\delta}^+) \gets A$
                \State \Return $((\tilde{Z}, \tilde{\delta}^-, \tilde{\delta}^+)$, $\id{V(Z) \cup D(Z)})$
            \Else
                \State \Return \texttt{null} \Comment{Due to a boundary error.}
            \EndIf
        \EndIf
    \end{algorithmic}
\end{algorithm}

\begin{algorithm}[H]
    \caption{\addBoundaryDartsWithLink($(Z, \delta^-, \delta^+), v$)}
    \label{alg:add_boundary_darts}
    \begin{algorithmic}[1]
        \Require{A pseudo-configuration $(Z, \delta^-, \delta^+)$, a boundary vertex $v$ with $\delta^-(v)=\delta^+(v)=d_Z(v)$.}
        \Ensure{A pseudo-configuration obtained from $(Z, \delta^-, \delta^+)$ by adding darts and making $v$ an inner vertex or \texttt{null} if a boundary error occurs.}
        \State $e_{\text{first}} \gets$ the first dart whose head is $v$ (i.e., \predecessor($e_{\text{first}}$) = \nil).
        \State $e_{\text{last}} \gets$ the last dart whose head is $v$ (i.e., \successor($e_{\text{last}}$) = \nil).
        \State $u \gets $\head(\reverse($e_{\text{first}})$)
        \State $w \gets $\head(\reverse($e_{\text{last}})$)
        \If {$u = w$}
            \State \Return \texttt{null} \Comment{a boundary error}
        \EndIf
        \State $d_{uw} \gets $ a dart whose \head\ is $u$, \reverse\ is $d_{wu}$, \successor\ is \nil, \predecessor\ is \reverse($e_{\text{first}}$).
        \State $d_{wu} \gets $ a dart whose \head\ is $w$, \reverse\ is $d_{uw}$, \successor\ is \reverse($e_{\text{last}}$), \predecessor\ is \nil.
        \State \predecessor($e_{\text{first}}$) $\gets e_{\text{last}}$
        \State \successor($e_{\text{last}}$) $\gets e_{\text{first}}$
        \State \successor(\reverse($e_{\text{first}}$)) $\gets d_{uw}$
        \State \predecessor(\reverse($e_{\text{last}}$)) $\gets d_{wu}$
        \State $\tilde{Z} \gets$ the pseudo-triangulation defined by $V(Z)$ and $D(Z) \cup \{d_{uw}, d_{wu}\}$.
        \State \Return $(\tilde{Z}, \delta^-, \delta^+)$
    \end{algorithmic}
\end{algorithm}

\begin{algorithm}[H]
    \caption{\singleOutLowerDegreeWithLink($Z, \delta^-, \delta^+)$}
    \label{alg:single_out_lower_deg}
    \begin{algorithmic}[1]
        \Require{A pseudo-configuration $(Z, \delta^-, \delta^+)$.}
        \Ensure{Two pseudo-configurations $(Z_1, \delta_1^-, \delta_1^+), (Z_2, \delta^-_2, \delta^+_2)$ if a vertex $v$ with $\delta^-(v) < \delta^+(v)$ and $\delta^-(v) \leq d_Z(v)$ exists, otherwise \texttt{null}}
        \ForAll{$v \in V(Z)$}
            \If {$\delta^-(v) < \delta^+(v)$ and $\delta^-(v) \leq d_Z(v)$}
                \State $(Z_1, \delta^-_1, \delta^+_1) \gets$ a copy of $(Z, \delta^-, \delta^+)$
                \State $\delta^+_1(v) \gets \delta^-(v)$
                \State $(Z_2, \delta^-_2, \delta^+_2) \gets$ a copy of $(Z, \delta^-, \delta^+)$
                \State $\delta^-_2(v) \gets \delta^-(v) + 1$
                \State \Return $((Z_1, \delta_1^-, \delta_1^+), (Z_2, \delta^-_2, \delta^+_2))$
            \EndIf
        \EndFor
        \State \Return \texttt{null}
    \end{algorithmic}
\end{algorithm}

\subsection{Functions for reading configuration, rule, cartwheel files}
\label{sub:IO}
\newcommand{\fromVRotations}{\texttt{fromVRotations}\xspace}
Our \Cpp program reads and writes files representing configurations, rules, and cartwheels in specific formats.
While we omit the detailed format descriptions here, they are provided in the documentation accompanying our program.
In these formats, we do not use the dart representations to express an underlying pseudo-triangulation. 
Instead, we use rotations of neighbors around each vertex.
Additionally, the boundary is explicitly marked.
Since the configurations, rules, and cartwheels we used have no multiple darts between two vertices, this representation contains the same information as the dart representation.
To construct the dart representation from the rotations around each vertex, we use the \fromVRotations subroutine in Algorithm \ref{alg:from_v_rotations}.
The \fromVRotations subroutine takes the input \texttt{rotations}, and constructs a pseudo triangulation represented by \texttt{rotations}.
When the size of \texttt{rotations} is $N$, the vertex set consists of $v_0,\ldots,v_{N-1}$.
When \texttt{rotations[$i$]} is a list of integers $[a_0, \ldots, a_k]$, it means that $v_i$ is adjacent to $v_{a_0}, \ldots, v_{a_k}$ in this rotations.
If $a_i=-1$, it represents the boundary.
It creates all the darts and assigns the appropriate pointers to them.

\newcommand{\concat}{\mathbin{\|}}
In the following, we will use list data structures whose size is changing dynamically.
Let $\texttt{L1}, \texttt{L2}$ be instances of the list.
In adding a new element $x$ to the back of the list $\texttt{L1}$, we write $\texttt{L1.push\_back}(x)$.
In concatenating two lists $\texttt{L1}, \texttt{L2}$, we write $\texttt{L1} \concat \texttt{L2}$.
The size of the list $\texttt{L1}$ is denoted by $|\texttt{L1}|$.
For this data structure, we use \texttt{std::vector} for our \Cpp program.

\begin{algorithm}[H]
    \caption{\fromVRotationsWithLink($N, \texttt{rotations}$)}
    \label{alg:from_v_rotations}
    \begin{algorithmic}[1]
        \Require{We denote the number of vertices by $N$ and a vertex set by $V=\{v_0, \ldots, v_{N-1}\}$. A list \texttt{rotations}, where \texttt{rotations[$i$]} represents the rotation of indices of vertices around $v_i \in V$, where $-1$ marks the boundary.}
        \Ensure{The pseudo triangulation represented by $V$ and $\texttt{rotations}$ or raise an error if $\texttt{rotations}$ has multiple darts between two vertices or the rotations of two vertices have a discrepancy.}
        \State $\texttt{darts} \gets$ an $N \times N$ array of darts, initialized by \nil.
        \ForAll{$a \in \{0, \ldots, N-1\}$}
            \ForAll{$i \in \{0, \ldots, |\texttt{rotations}[a]|-1\}$}
                \State $b \gets \texttt{rotations}[a][i]$
                \If {$b = -1$}
                    \Continue
                \EndIf
                \If {$\texttt{darts}[a][b] \neq \nil$}
                    \State \textbf{raise an error} \Comment{Multiple darts between two vertices.}
                \EndIf
                \State $\texttt{darts}[a][b] \gets$ a fresh dart
            \EndFor
        \EndFor
        \State $D \gets \emptyset$ \Comment{$D$ becomes the dart set.}
        \ForAll{$a \in \{0, \ldots, N-1\}$}
            \State \texttt{size} $\gets$ $|\texttt{rotations}[a]|$
            \ForAll{$i \in \{0, \ldots, \texttt{size}-1\}$}
                \State $b \gets \texttt{rotations}[a][i]$
                \If {$b = -1$}
                    \Continue
                \EndIf
                \State $e \gets \texttt{darts}[a][b]$
                \State $\head(e) \gets v_a$
                \If {$\texttt{darts}[b][a] = \nil$}
                    \State \textbf{raise an error} \Comment{The rotations of two vertices have a discrepancy.}
                \EndIf
                \State $\reverse(e) \gets \texttt{darts}[b][a]$
                \State $s \gets \texttt{rotations}[a][i+1]$ if $i<\texttt{size}-1$ otherwise $\texttt{rotations}[a][0]$
                \State $\successor(e) \gets \texttt{darts}[a][s]$ if $s \neq -1$ otherwise \nil
                \State $p \gets \texttt{rotations}[a][i-1]$ if $i>0$ otherwise $\texttt{rotations}[a][\texttt{size}-1]$.
                \State $\predecessor(e) \gets \texttt{darts}[a][p]$ if $p \neq -1$ otherwise \nil
                \State $D \gets D \cup \{e\}$
            \EndFor
        \EndFor
        \State \Return The pseudo triangulation that consists of $(V, D)$
    \end{algorithmic}
\end{algorithm}

\subsection{Reducible configurations in pseudo-configurations}
\label{code:identify-reducible}
\def\CONFDEGMAX{\texttt{CONF\_DEG\_MAX}\xspace}
\def\blockByReducibleConfiguration{\texttt{blockedByReducibleConfiguration}\xspace}
\def\representativeDegree{\texttt{representativeDegree}\xspace}
\def\containConf{\texttt{containConf}\xspace}
\def\dartsByDegree{\texttt{dartsByDegree}\xspace}
\def\rootedContainConf{\texttt{rootedContainConf}\xspace}
\newcommand{\extendFromCutVertices}{\texttt{extendFromCutVertices}\xspace}
\newcommand{\findCutPairs}{\texttt{findCutPairs}\xspace}
\newcommand{\maximumDartDegree}{\texttt{maximumDartDegree}\xspace}
\newcommand{\mirror}{\texttt{mirror}\xspace}
\newcommand{\removeRing}{\texttt{removeRing}\xspace}

In this section, we present a procedure to determine if a given pseudo-configuration $(Z^*, \delta^*)$ with the center $c^*$ has a homomorphic image of a reducible configuration in $\mathcal{D}$.
The main routine is the \containConf routine in Algorithm \ref{alg:contain_conf}.

\paragraph{Cut-vertices}
First, we describe the procedure to convert a configuration in $\mathcal{D}$ to cut-vertex-free configurations, as described in Section \ref{sect:config-in-combine}.
The \extendFromCutVertices routine in Algorithm \ref{alg:extend_from_cut_vertices} handles this process.
Because this routine is called after reading a configuration file, it takes the rotations of neighbors around each vertex as input, not the dart representation, to construct new cut-vertex-free configurations.
From (Z3), each cut vertex in a configuration has two neighbors in the ring.
The \findCutPairs routine in Algorithm \ref{alg:find_cut_pairs} returns a set of pairs of ring vertices adjacent to each cut-vertex.
The \extendFromCutVertices routine calls the \findCutPairs routine to get such a set of pairs of ring vertices $P$.
Since our reducible configuration in $\mathcal{D}$ has at most one cut-vertex, $|P| \leq 1$.
We consider $2^{|P|}$ different extensions.
For each extension, we remove the vertices in the ring that are not added for the extension.
The \removeRing routine in Algorithm \ref{alg:remove_ring} handles this process.

Also, we choose a special dart $f=xy$ such that both $y$ and $x$ have fixed degrees and maximizing $(\delta^{-}(y), \delta^{-}(x))$ lexicographically, as described in Section \ref{sect:config-in-combine}.
This fixed-degree condition is satisfied when neither $x$ nor $y$ is an auxiliary vertex added for a cut-vertex.
The \maximumDartDegree procedure in Algorithm \ref{alg:maximum_dart_degree} handles this process.

\begin{algorithm}[H]
    \caption{\extendFromCutVerticesWithLink($N$, $R$, $\delta^-$, $\delta^+$, \texttt{rotations})}
    \label{alg:extend_from_cut_vertices}
    \begin{algorithmic}[1]
        \Require{The number of vertices $N$ in the free completion of a configuration and $R$ in its ring; a vertex set $V = \{v_0, \ldots, v_{N-1}\}$ containing the ring vertices $V_R = \{v_0, \ldots, v_{R-1}\}$, degree functions $\delta^-, \delta^+$ on $V \setminus V_R$, and a list \texttt{rotations} representing rotations of neighbors for each vertex in $V$.}
        \Ensure{The set of configurations obtained by adding a neighbor for each cut-vertex, with the special dart.}
        \State $P \gets$ \Call{findCutPairs}{$N, R, \texttt{rotations}$} \Comment{In all configurations from $\mathcal{D}$, the number of cut-vertices are at most 1, so $|P| \leq 1$.}
        \State $\mathcal{K} \gets \emptyset$
        \ForAll{$S \gets \{0, \ldots, 2^{|P|}-1\}$}
            \State $\texttt{remove} \gets$ an array of size $R$, initialized by $\True$
            \ForAll{$i \gets \{0, \ldots, |P|-1\}$}
                \State $v_a, v_b \gets P[i]$
                \If {the $i$-bit of $S$ is 1}
                    \State $\texttt{remove}[a] \gets \False$
                \Else 
                    \State $\texttt{remove}[b] \gets \False$
                \EndIf
            \EndFor
            \State $(Z, \delta^-, \delta^+) \gets$ \Call{removeRing}{$N, R, \delta^-, \delta^+, \texttt{rotations}, \texttt{remove}$}
            \State $f \gets$ \Call{maximumDegreeDart}{$(Z, \delta^-, \delta^+)$}
            \State $\mathcal{K} \gets \mathcal{K} \cup \{((Z, \delta^-, \delta^+), f)\}$
        \EndFor
        \State \Return $\mathcal{K}$
    \end{algorithmic}
\end{algorithm}

In the \findCutPairs routine, for each vertex $v_i$ in a given configuration, we check whether $v_i$ is a cut-vertex.
If $v_i$ is a cut-vertex, we add a pair of neighbors in the ring to the resulting list.
To check whether $v_i$ is a cut-vertex, we calculate the set of vertices $U_R$ in the ring that are adjacent to $v_i$, and the number of times $t$ that $v_i$ touches the consecutive vertices of the ring.
By (Z3), $v_i$ must be a cut-vertex if both $|U_R|$ and $t$ are $2$. 
Also, if $t \geq 2$ and $|U_R| \neq 2$, then it would contradict (Z3).
In this case, we raise an error since an input configuration is invalid, but all configurations in $\mathcal{D}$ satisfy (Z3), so we have no error.

\begin{algorithm}[H]
    \caption{\findCutPairsWithLink($N, R, \texttt{rotations}$)}
    \label{alg:find_cut_pairs}
    \begin{algorithmic}[1]
        \Require{The number of vertices $N$ in the free completion of a configuration and $R$ in its ring; a vertex set $V = \{v_0, \ldots, v_{N-1}\}$ containing the ring vertices $V_R = \{v_0, \ldots, v_{R-1}\}$, and a list \texttt{rotations} representing rotations of neighbors for each vertex in $V$.}
        \Ensure{The set of pairs of vertices in the ring that are adjacent to cut-vertices of the input configuration, or raise an error if the input is an invalid configuration.}
        \State $\texttt{P} \gets$ an empty array
        \ForAll{$v_i \in \{v_R, \ldots, v_{N-1}\}$}
            \State $U_R \gets \emptyset$ \Comment{$U_R$ is the set of vertices in the ring that is adjacent to $v_i$.}
            \State $t \gets 0$ \Comment{$t$ represents the number of times $v_i$ touches the consecutive vertices of the ring.}
            \State $d \gets |\texttt{rotations}[i]|$ 
            \ForAll{$j \gets \{0, \ldots, d-1\}$}
                \State $k_1 \gets \texttt{rotations}[i][j]$
                \If {$k_1 < R$}
                    \State $U_R \gets U_R \cup \{v_{k_1}\}$
                \EndIf
                \State $k_2 \gets \texttt{rotations}[i][(j+1) \bmod d]$
                \If {$k_1 < R$ and $k_2 \geq R$} \Comment{border between rign and internal vertex.}
                    \State $t \gets t + 1$
                \EndIf
            \EndFor
            \If {$t \geq 2$ and $|U_R| \neq 2$}
                \State \textbf{raise an error} \Comment{Invalid configuration}
            \EndIf
            \If {$t = 2$ and $|U_R|=2$} \Comment{$v_i$ is a cut-vertex.}
                \State $v_a, v_b \gets U_R$
                \State $\texttt{P.push\_back}((v_a, v_b))$
            \EndIf
        \EndFor
        \State \Return \texttt{P}
    \end{algorithmic}
\end{algorithm}

The \removeRing routine in Algorithm \ref{alg:remove_ring} removes the vertices $v_i$ in the ring with the input $\texttt{remove}[i]=\True$.
This routine consists of three steps.
First, for each remaining vertex, we assign a new vertex ID.
Second, we construct new rotations of neighbors for each remaining vertex.
Finally, we update the degrees of the remaining vertices.
For each remaining vertex $u$ in the ring such that the number of incident darts is $d(u)$, we set new degrees $\delta^{\prime-}(u)=d(u)+1$, $\delta^{\prime+}(u)=\infty$.
For each vertex not in the ring, the degree is the same as the input degree.
The routine returns the pseudo configuration constructed from new rotations and degrees by the \fromVRotations routine in Algorithm \ref{alg:from_v_rotations}.

\begin{algorithm}
    \caption{\removeRingWithLink$(N, R, \delta^-, \delta^+, \texttt{rotations}, \texttt{remove})$}
    \label{alg:remove_ring}
    \begin{algorithmic}[1]
        \Require{The number of vertices $N$ in the free completion of a configuration and $R$ in its ring; a vertex set $V = \{v_0, \ldots, v_{N-1}\}$ containing the ring vertices $V_R = \{v_0, \ldots, v_{R-1}\}$, degree functions $\delta^-, \delta^+$ on $V \setminus V_R$, and a list \texttt{rotations} representing rotations of neighbors for each vertex in $V$, and a list \texttt{remove}.}
        \Ensure{The pseudo configuration by deleting vertices $v_i$ in $V_R$ with $\texttt{remove}[i]=\True$.}
        
        \Comment{Step 1: Assign new vertex IDs}
        \State \texttt{old2new} $\gets$ an array of length $N$, each element initialized by $-1$
        \State \texttt{new\_id} $\gets 0$
        \ForAll{$i \gets \{0, \ldots, N-1\}$}
            \If {$i < R$ and \texttt{remove}$[i]$ $=$ \True}
                \Continue
            \EndIf
            \State \texttt{old2new}$[i] \gets \texttt{new\_id}$
            \State $\texttt{new\_id} \gets \texttt{new\_id} + 1$
        \EndFor
        \State $N' \gets \texttt{new\_id}$

        \Comment{Step 2: Construct new rotations}
        \State $\texttt{new\_rotations} \gets$ an list of $N'$ lists
        \ForAll{$i \gets \{0, \ldots, N-1\}$}
            \If {$i < R$ and \texttt{remove}$[i]$ $= \True$}
                \Continue
            \EndIf
            \ForAll{$j \gets \texttt{rotations}[i]$}
                \If {$j = -1$}
                    \State \texttt{new\_rotations[old2new$[i]$].push\_back($-1$)}
                \Else
                    \State \texttt{new\_rotations[old2new$[i]$].push\_back(old2new[$j$])}
                \EndIf
            \EndFor
        \EndFor

        \Comment{Step 3: Update degrees}
        \State $U' \gets \{u_0, \ldots, u_{N'-1}\}$
        \State $\delta^{\prime-}, \delta^{\prime+} \colon U' \to \mathbb{N} \cup \{\infty\}$
        \ForAll{$i \gets \{0, \ldots, R-1\}$}
             \If {$\texttt{remove}[i] = \True$}
                 \Continue
             \EndIf
             \State $k \gets \texttt{old2new}[i]$
             \State $d \gets$ the number of elements, which are not $-1$ in $\texttt{new\_rotations}[k]$
             \State $\delta^{\prime-}(u_k) \gets d + 1$, $\delta^{\prime+}(u_k) \gets \infty$
        \EndFor
        \ForAll{$i \gets \{R, \ldots, N-1\}$}
            \State $k \gets \texttt{old2new}[i]$
            \State $\delta^{\prime-}(u_k)=\delta^-(v_i)$, $\delta^{\prime+}(u_k)=\delta^+(v_i)$
        \EndFor

        \State \Return The pseudo configuration obtained by \Call{fromVRotations}{$N'$, \texttt{new\_rotations}} with degree functions $\delta^{\prime-}, \delta^{\prime+}$.
    \end{algorithmic}
\end{algorithm}

\begin{algorithm}
    \caption{\maximumDegreeDartWithLink($(Z, \delta^-, \delta^+)$}
    \label{alg:maximum_dart_degree}
    \begin{algorithmic}[1]
        \Require{A configuration $(Z, \delta^-, \delta^+)$.}
        \Ensure{The dart $f=xy$ such that both $y$ and $x$ have fixed degrees and maximizing $(\delta^-(y), \delta^-(x))$ lexicographically.}
        \State $f \gets \nil$
        \State $d_f \gets (0, 0)$
        \ForAll{$e \gets D(Z)$}
            \State $y \gets \head(e)$
            \State $x \gets \head(\reverse(e))$
            \If {$\delta^-(y) \neq \delta^+(y)$ or $\delta^-(x) \neq \delta^+(x)$}
                \Continue
            \EndIf
            \State $d_e \gets (\delta^-(y), \delta^-(x))$
            \If {$d_e > d_f$} \Comment{Compare them lexicographically}
                \State $f \gets e$
                \State $d_f \gets d_e$
            \EndIf       
        \EndFor
        \State \Return $f$
    \end{algorithmic}
\end{algorithm}

The \mirror routine in Algorithm \ref{alg:mirror} converts a given pseudo configuration into its mirror by, for every dart $e$, swapping \predecessor($e$) and \successor($e$).

\begin{algorithm}[H]
    \caption{\mirrorWithLink($Z, \delta^-, \delta^+$)}
    \label{alg:mirror}
    \begin{algorithmic}[1]
        \Require{A pseudo configuration $(Z, \delta^-, \delta^+)$.}
        \Ensure{The mirror of $(Z, \delta^-, \delta^+)$.}
        \ForAll{$e \gets D(Z)$}
            \State swap $\predecessor(e), \successor(e)$
        \EndFor
        \State \Return $(Z, \delta^-, \delta^+)$
    \end{algorithmic}
\end{algorithm}

By calling the \extendFromCutVertices routine for each configuration in $\mathcal{D}$, we get another set of configurations with the special dart.
Moreover, we update this set by adding all mirrors of them.
Here, the symbol $\bar{\mathcal{D}}$ denotes this set of configurations obtained from $\mathcal{D}$.
In the program, we use $\bar{\mathcal{D}}$.
Note that every configuration in $\bar{\mathcal{D}}$ has the special dart chosen by the \maximumDartDegree subroutine.
Note also that every vertex of a configuration in $\bar{\mathcal{D}}$ basically has a degree range, since auxiliary neighbors of a cut-vertex have a degree range.

\paragraph{Checking configurations}
We use the constant \CONFDEGMAX that depends on the configuration set $\mathcal{D}$.
Specifically, \CONFDEGMAX is 12, representing the maximum value of $\delta$ for any configuration $(Z, \delta) \in \mathcal{D}$.

The \containConf subroutine takes an input set of configurations $\mathcal{K}$.
For the purpose of this algorithm, $\mathcal{K}$ is limited to either $\emptyset$, $\bar{\mathcal{D}}$, or $\bar{\mathcal{D}} \cup \{T_{7^3}\}$.
The final one is necessary to handle Lemma \ref{lem:zero-deg7,7}.
For $T_{7^3}$, we choose the special dart in the same way as $\bar{\mathcal{D}}$, so they all have each special dart.

Every configuration $(Z', \delta^{\prime-}, \delta^{\prime+})$ from $\bar{\mathcal{D}}$ has at most one vertex $c$ with $\delta^{\prime-}(c)=\delta^{\prime+}(c) > 8$.
We choose the special dart, say $f=xy$, of $Z'$ by maximizing the degrees of endpoints lexicographically, so the head $y$ is $c$, if such a vertex $c$ exists.
This is checked by looking $\delta^{\prime-}(y) > 8$.
In the \containConf routine, we only consider mapping of $c$ to the center $c^*$ in $Z^*$.
Thus, if $\delta^{\prime-}(y) > 8$, then $f$ must be mapped to $f^*$ with $\head(f^*)=c^*$.

Furthermore, we only need to consider mapping of $f$ to some dart $f^*$ in $Z^*$ whose endpoints have the same degree values.
To facilitate this, we classify the darts $e^*$ of $(Z^*, \delta^*)$ according to the values of $\delta^*$ of $\head(e^*)$, $\tail(e^*)=\head(\reverse(e^*))$.
This procedure is in the \dartsByDegree subroutine in Algorithm \ref{alg:darts_by_degree}.
Since $\mathcal{K}$ is either $\emptyset$, $\bar{\mathcal{D}}$, or $\bar{\mathcal{D}} \cup \{T_{7^3}\}$, we only need to classify the darts whose endpoints have degrees of at most \CONFDEGMAX.

The procedure to determine the existence of a homomorphism from $Z'$ to $Z^*$ such that $f$ maps to a specified dart $f^* \in D(Z^*)$ is in the \rootedContainConf subroutine in Algorithm \ref{alg:rooted_contain_conf}.

\begin{algorithm}[H]
    \caption{\containConfWithLink($(Z^*, \delta^*)$, $c^*$, $\mathcal{K}$)}
    \label{alg:contain_conf}
    \begin{algorithmic}[1]
        \Require{A pseudo-configuration $(Z^*, \delta^*)$ with a center vertex $c^*$, and a set of configurations $\mathcal{K}$.}
        \Ensure{\True\ if there exists a configuration $((Z', \delta^{\prime-}, \delta^{\prime+}), f) \in \mathcal{K}$ and a homomorphism $\phi$ from $(Z', \delta^{\prime-}, \delta^{\prime+})$ to $(Z^*, \delta^*)$ such that all $v \in V(Z')$ with $\delta^{\prime-}(v)=\delta^{\prime+}(v) > 8$ satisfy $\phi(v)=c^*$ , otherwise \False.}
        \State \texttt{darts\_by\_degree} $\gets$ \Call{dartsByDegree}{$(Z^*, \delta^*)$}
        \ForAll{$((Z', \delta^{\prime-}, \delta^{\prime+}), f) \in \mathcal{K}$} \Comment{$f$ is chosen by the \maximumDartDegree routine.}
            \State $y \gets \head(f)$
            \State $x \gets \head(\reverse(f))$
            \State $d_y \gets \delta^{\prime-}(y)$
            \State $d_x \gets \delta^{\prime-}(x)$
            \ForAll{$f^* \in \texttt{darts\_by\_degree}[d_y][d_x]$}
                 \If {$d_y > 8$ and $\head(f^*) \neq c^*$}
                     \Continue
                 \EndIf
                 \If {\Call{rootedContainConf}{$(Z^*, \delta^*)$, $f^*$, $(Z', \delta^{\prime-}, \delta^{\prime+})$, $f$}}
                     \State \Return \True
                 \EndIf
            \EndFor 
        \EndFor
        \State \Return \False
    \end{algorithmic}
\end{algorithm}

\begin{algorithm}[H]
    \caption{\dartsByDegreeWithLink($Z^*, \delta^*$)}
    \label{alg:darts_by_degree}
    \begin{algorithmic}[1]
        \Require{A pseudo-configuration $(Z^*, \delta^*)$.}
        \Ensure{A $(\CONFDEGMAX + 1) \times (\CONFDEGMAX + 1)$ array of dart sets \texttt{darts\_by\_degree} such that \texttt{darts\_by\_degree}[$d_y$][$d_x$] is the set of darts $xy$ where $\delta^*(x)=d_x, \delta^*(y)=d_y$.}
        \State \texttt{darts\_by\_degree} $\gets$ a $(\CONFDEGMAX + 1) \times (\CONFDEGMAX + 1)$ array of dart sets
        \ForAll{$e \in D(Z^*)$}
            \State $y \gets \head(e)$
            \State $x \gets \head(\reverse(e))$
            \State $d_y \gets \delta^*(y)$
            \State $d_x \gets \delta^*(x)$
            \If {$d_y > \CONFDEGMAX$ or $d_x > \CONFDEGMAX$}
                \Continue
            \EndIf
            \State $\texttt{darts\_by\_degree}[d_y][d_x] = \texttt{darts\_by\_degree}[d_y][d_x] \cup \{e\}$
        \EndFor
        \State \Return \texttt{darts\_by\_degree}
    \end{algorithmic}
\end{algorithm}

In the \rootedContainConf subroutine,
we check whether the homomorphism $\phi^*$ from $(Z, \delta^-, \delta^+)$ to $(Z^*, \delta^*)$ with $e$ mapping to $e^*$ satisfy, for every $v$ in $Z$ with $v^*=\phi^*(v)$, $\delta^-(v) \leq \delta^*(v^*) \leq \delta^+(v)$ holds by calling Algorithm \ref{alg:homomorphism} with the $\gInclude$ function.
If so, we return true; otherwise return false.

\begin{algorithm}[H]
    \caption{\rootedContainConfWithLink($(Z^*, \delta^*)$ $e^*$, $(Z, \delta^-, \delta^+)$, $e$)}
    \label{alg:rooted_contain_conf}
    \begin{algorithmic}[1]
        \Require{Two pseudo-configurations $(Z^*, \delta^*)$, $(Z, \delta^-, \delta^+)$ with darts $e^* \in D(Z^*), e \in D(Z)$.}
        \Ensure{\True\ if a homomorphism $\phi$ from $(Z, \delta^-, \delta^+)$ to $(Z^*, \delta^*)$ with $\phi(e)=e^*$ exists, otherwise \False.}
        \State \Return \Call{homomorphism}{$(Z, \delta^-, \delta^+), e, (Z^*, \delta^*, \delta^*), e^*, \gInclude$} $\neq \texttt{null}$
    \end{algorithmic}
\end{algorithm}

\subsection{Configuration with degree ranges blocked by reducible configurations}
\label{code:blocking-by-reduc}

In this section, we present the routine that determines whether a pseudo-configuration with degree ranges $(Z^*, \delta^{*-}, \delta^{*+})$ with the center $c^*$ is blocked by a given set of reducible configurations $\mathcal{K}$, as described in Section \ref{sect:block}.
The main routine is \blockByReducibleConfiguration in Algorithm \ref{alg:block_by_reduc}.

By Lemma \ref{lem:bounded-ranges}, we first select a finite set of degrees from $[\delta^{*-}(v), \delta^{*+}(v)]$, for each vertex $v \in Z$, that are representative to check blocking by reducible configurations in $\bar{\mathcal{D}}$.
We then generate all combinations of these degrees across all vertices.
Specifically, we select 
\begin{itemize}
    \item all degrees from $[\delta^{*-}(c^*), \delta^{*+}(c^*)]$ if $\delta^{*+}(c^*) \leq \CONFDEGMAX$, otherwise use only $\delta^{*+}(c^*)$, or
    \item all degrees from $[\delta^{*-}(v^*), \delta^{*+}(v^*)]$ if $\delta^{*+}(v^*) \leq 8$, otherwise use only $\delta^{*+}(v^*)$, if $v^* \neq c^*$.
\end{itemize}
This procedure is in the \representativeDegree subroutine in Algorithm \ref{alg:split}.

Next, for each pseudo-configuration with a single specified degree obtained by the \representativeDegree subroutine, we check whether some configuration from $\mathcal{K}$ has a homomorphic image into it satisfying the mapping condition about $c^*$ using the \containConf subroutine in Algorithm \ref{alg:contain_conf}.

\begin{algorithm}[H]
    \caption{\blockedByReducibleConfigurationWithLink($(Z^*, \delta^{*-}, \delta^{*+})$, $c^*$, $\mathcal{K}$)}
    \label{alg:block_by_reduc}
    \begin{algorithmic}[1]
        \Require{A pseudo configuration $(Z^*, \delta^{*-}, \delta^{*+})$, a center vertex $c^*$, a set of configurations $\mathcal{K}$.}
        \Ensure{\True\ if $(Z^*, \delta^{*-}, \delta^{*+})$ is blocked by $\mathcal{K}$, otherwise \False.}
        \ForAll{$(\tilde{Z}, \tilde{\delta}) \in$ \Call{representativeDegree}{($Z^*, \delta^{*-}, \delta^{*+}$), $c^*$}}
            \If {\Call{containConf}{$(\tilde{Z}, \tilde{\delta})$, $c^*$, $\mathcal{K}$} = \False}
                \State \Return \False
            \EndIf
        \EndFor
        \State \Return \True
    \end{algorithmic}
\end{algorithm}

\begin{algorithm}[H]
    \caption{\representativeDegreeWithLink($(Z^*, \delta^{*-}, \delta^{*+})$, $c^*$)}
    \label{alg:split}
    \begin{algorithmic}[1]
        \Require{A pseudo-configuration $(Z^*, \delta^{*-}, \delta^{*+})$ with a center vertex $c^*$.}
        \Ensure{A set of pseudo-configurations $\{(Z^*, \delta')\}$ obtained from $(Z^{*}, \delta^{*-}, \delta^{*+})$ by
        generating all combinations of single representative degrees for each vertex.}
        \State $\mathcal{T}' \gets \{\delta\}$,
        where $\delta$ is any map from $V(Z^*)$ to $\mathbb{N}$.
        \ForAll {$v^* \in V(Z^*)$ \bf{one at the time}}
            \If {$v^* = c^*$ and $\delta^{*+}(v^*) > \CONFDEGMAX$}
                \State $L \gets \{\delta^{*+}(v^*)\}$
            \ElsIf {$v^* \neq c^*$ and $\delta^{*+}(v^*) > 8$}
                \State $L \gets \{\delta^{*+}(v^*)\}$
            \Else
                \State $L \gets \{d \mid \delta^{*-}(v^*) \leq d \leq \delta^{*+}(v^*)\}$
            \EndIf
            \State $\tilde{\mathcal{T}} \gets \emptyset$
            \ForAll{$\delta' \in \mathcal{T}'$}
                \ForAll {$d \in L$}
                    \State $\tilde{\delta} \gets$ a copy of $\delta'$.
                    \State $\tilde{\delta}(v^*) \gets d$
                    \State $\tilde{\mathcal{T}} \gets \tilde{\mathcal{T}} \cup \{ \tilde{\delta}\}$
                \EndFor
            \EndFor
            \State $\mathcal{T}' \gets \tilde{\mathcal{T}}$
        \EndFor
        \State \Return $\{(Z^*, \delta') \mid \delta' \in \mathcal{T}'\}$
    \end{algorithmic}
\end{algorithm}

\subsection{Free combination of discharging rules with degree ranges}
\newcommand{\addRuleToCombination}{\texttt{addRuleToCombination}\xspace}
\newcommand{\combineRules}{\texttt{combineRules}\xspace}
In this subsection, we present the pseudocode that combines discharging rules to prove Lemma \ref{lem:vsends}, \ref{lem:vsends-noconf}.
Although we have already described the algorithm in Section \ref{subsub:combine_rules}, we provide the pseudocode for clarity.
In our \Cpp program, a free combination $R^*$ of discharging rules has the flag representing the subset of rules $\tilde{R} \subseteq \mathcal{R}$ that lead to $R^*$.
The flag is represented by $f \colon \mathcal{R} \to \{\True, \False\}$ satisfying $f(R)=\True$ if and only if $R \in \tilde{R}$.
This flag is used in the \pruneByNonAssociatedRule subroutine in Algorithm \ref{alg:prune_non_assoc_rule}.

\begin{algorithm}[H]
    \caption{\addRuleToCombinationWithLink($R^*, R, \mathcal{K}$)}
    \label{alg:add_rule}
    \begin{algorithmic}[1]
        \Require{A combined rule $R^*$, a rule $R$, a configuration set $\mathcal{K}$
        is the set of configurations to be avoided.}

        \Ensure{Finds the set of rules combined from $R^*$ and $R$ that are not blocked by  $\mathcal{K}$.}
        \State $(Z^*, \delta^{*-}, \delta^{*+}, e^*, r^*, f^*) \gets R^*$
        \State $(Z, \delta^-, \delta^+, e, r) \gets R$
        \State $\tilde{\mathcal{Z}} \gets$ \Call{freeHomomorphismConfiguration}{$(Z^*, \delta^{*-}, \delta^{*+}) \sqcup (Z, \delta^-, \delta^+), \{(e^*,e)\}$}
        \State $f^*(R) \gets \True$
        \State $\tilde{\mathcal{R}} \gets \{(\tilde{Z}, \tilde{\delta}^{-}, \tilde{\delta}^{+}, \tilde{\phi}(e^*), r^* + r, f^*) \mid ((\tilde{Z}, \tilde{\delta}^{-}, \tilde{\delta}^{+}), \tilde{\phi}) \in \tilde{\mathcal{Z}}\}$
        \If {$\mathcal K=\emptyset$}
            \State \Return $\tilde{\mathcal{R}}$
        \EndIf
        \State $\mathcal{R}^* \gets \emptyset$
        \ForAll{$(\tilde{Z}, \tilde{\delta}^{-}, \tilde{\delta}^{+}, \tilde{e}, \tilde{r}, \tilde{f}) \in \tilde{\mathcal{R}}$}
            \If {\Call{blockedByReducibleConfiguration}{$(\tilde{Z}, \tilde{\delta}^{-}, \tilde{\delta}^{+}), \head(\tilde{e}), \mathcal{K}$} $=$ \True}
                \Continue
            \EndIf
            \State $\mathcal{R}^* \gets \mathcal{R}^* \cup \{(\tilde{Z}, \tilde{\delta}^{-}, \tilde{\delta}^{+}, \tilde{e}, \tilde{r}, \tilde{f})\}$
        \EndFor
        \State \Return $\mathcal{R}^*$
    \end{algorithmic}
\end{algorithm}

\begin{algorithm}[H]
    \caption{\combineRulesWithLink($\mathcal{R}$, $\mathcal{K}$) - Lemma \ref{lem:vsends}, \ref{lem:vsends-noconf}}
    \label{alg:combine_rules}
    \begin{algorithmic}[1]
        \Require{A set of rules $\mathcal{R}$, a set of configurations $\mathcal{K}$}
        \Ensure{The set $\cR^*$ of combined rules, each is a free combination of some subset $ \tilde{\mathcal{R}} \subseteq \mathcal{R}$ but only combinations not blocked by $\mathcal{K}$ with the map $f$ such that $f(R) = \True$ if and only if $R \in \tilde{\mathcal{R}}$}
        \State $f_0 \colon \mathcal{R} \to \{\True, \False\}$, initialized by $f_0(R)=\False$ for all $R \in \mathcal{R}$
        \State $\mathcal{R}^* \gets \{(Z_0, \delta_0^-, \delta_0^+, e_0, 0, f_0)\}$, where $(Z_0, \delta_0^-, \delta_0^+)$ is a pseudo-configuration with $V(Z_0) = \{s_0,t_0\}$, $D(Z_0) = \{s_0t_0, t_0s_0\}$, $\delta_0^-(s_0)=\delta_0^-(t_0)=1$, $\delta_0^+(s_0)=\delta_0^+(t_0)=\infty$, and $e_0=s_0t_0$.
        \ForAll{$R \in \mathcal{R}$}
            \State $\mathcal{R}^{*'} \gets \mathcal{R}^*$
            \ForAll{$R^* \in \mathcal{R}^*$}
                \State $\mathcal{R}^{*'} \gets \mathcal{R}^{*'} \cup$ \Call{addRuleToCombination}{$R^*, R, \mathcal{K}$}
            \EndFor
            \State $\mathcal{R}^* \gets \mathcal{R}^{*'}$
        \EndFor
        \State \Return $\mathcal{R}^*$
    \end{algorithmic}
\end{algorithm}

By calling the \combineRules routine with input $\mathcal{R}$ being the set of rules depicted in Figure \ref{fig:rules} and $\mathcal{K}=\emptyset$, we can obtain the set $\mathcal{R}^*$ in Lemma \ref{lem:free-comb-rules}.
We check the following Lemma.

\begin{lem}
\label{comp:vsends-noconf}
\showlabel{comp:vsends-noconf}
    Let $\mathcal{R}^*$ be the output for Algorithm \ref{alg:combine_rules} with inputs $\mathcal{R}$ being the set of rules depicted in Figure \ref{fig:rules} and $\mathcal{K}=\emptyset$.
    The maximum value of $r^*$ for all $(Z^*, \delta^{*-}, \delta^{*+}, e^*, r^*) \in \mathcal{R}^*$ is 8.
    Moreover, if it is exactly 8, then $(Z^*, \delta^{*-}, \delta^{*+})$ is isomorphic to one of pseudo-configurations depicted in Figure \ref{fig:any_send8} with $e^*$ mapping to the edge with the arrow.
\end{lem}

Similarly, by calling the \combineRules routine with input $\mathcal{R}$ being the set of rules depicted in Figure \ref{fig:rules} and $\mathcal{K}=\bar{\mathcal{D}}$, we can obtain the set $\mathcal{R}^{*-\mathcal{D}}$ in Lemma \ref{lem:free-comb-rules}.
We check the following Lemma.

\begin{lem}
\label{comp:vsends}
\showlabel{comp:vsends}
    Let $\mathcal{R}^{*-\mathcal{D}}$ be the output for Algorithm \ref{alg:combine_rules} with inputs $\mathcal{R}$ being the set of rules depicted in Figure \ref{fig:rules} and $\mathcal{K}=\bar{\mathcal{D}}$.
    The maximum value of $r^*$ for all $(Z^*, \delta^{*-}, \delta^{*+}, e^*, r^*) \in \mathcal{R}^*$ is 5.
    Moreover, if it is exactly 5, then $(Z^*, \delta^{*-}, \delta^{*+})$ is isomorphic to one of pseudo-configurations depicted in Figure \ref{fig:send5} with $e^*$ mapping to the edge with the arrow.
\end{lem}

\subsection{Enumerating bad cartwheels with tail ranges}
In this section, we present the pseudocodes corresponding to the algorithm for enumerating bad cartwheels, as described in Section \ref{sub:overcharged_cartwheel_enumeration}.

\subsubsection{Charges along an edge and a bound on the final charge}
In this section, we introduce several algorithms for rule application, as described in Section \ref{sect:hom-to-ranges}.
We use these routines to calculate the lower and upper bounds on the amount of charge transferred by the rules.

Recall that $\gInclude, \gIntersection$ are defined in Section \ref{sect:code-hom}.
The \alwaysApply subroutine checks whether a rule $R$ {\bf always} applies to $e^*$ in $(Z^*, \delta^{*-}, \delta^{*+})$.
This routine calls Algorithm \ref{alg:homomorphism} with the $\gInclude$ function, and returns \True\ if such a homomorphism $\phi^*$ from $R$ to $Z^*$ exists.
Similarly, the \neverApply subroutine checks whether a rule $R$ {\bf never} applies to $e^*$ in $(Z^*, \delta^{*-}, \delta^{*+})$.
This routine calls Algorithm \ref{alg:homomorphism} with the $\gIntersection$ function, and returns \True\ if no such homomorphism exists.

The \amountOfChargeSend subroutine calculates the amount of charge transferred across $e$ by the rule set $\mathcal{R}$ for every $(Z_C, \delta_C) \in (Z_C, \delta_C^-, \delta_C^+)$.
This value is computed as the sum of the charges associated with the rules that always apply to $e$ in $(Z_C, \delta_C^-, \delta_C^+)$.
The \amountOfPossibleChargeSend subroutine calculates the maximum possible charge transferred across $e$ by the rule set $\mathcal{R}$ for any $(Z_C, \delta_C) \in (Z_C, \delta_C^-, \delta_C^+)$ that contains no centered reducible configuration from $\mathcal{D}$.
This value is computed as the maximum charge associated with the combined rules in $\mathcal{R}^{*-\mathcal{D}}$, excluding the rules that never apply to $e$ in $(Z_C, \delta_C^-, \delta_C^+)$.

\begin{algorithm}[H]
    \caption{\alwaysApplyWithLink($(Z^*, \delta^{*-}, \delta^{*+}), e^*, R$)}
    \label{alg:ineivtably_apply}
    \begin{algorithmic}[1]
        \Require{A pseudo-configuration $(Z^*, \delta^{*-}, \delta^{*+})$ with a dart $e^*$ and a rule $R$.}
        \Ensure{\True\ if $R$ applies to $e^*$ for every $(Z^*, \delta^*) \in (Z^*, \delta^{*-}, \delta^{*+})$, otherwise \False.}
        \State $(Z, \delta^-, \delta^+, e, r) \gets R$
        \State \Return \Call{homomorphism}{$(Z, \delta^-, \delta^+), e, (Z^*, \delta^{*-}, \delta^{*+}), e^*, \gInclude$} $\neq$ \texttt{null}
    \end{algorithmic}
\end{algorithm}

\begin{algorithm}[H]
    \caption{\neverApplyWithLink($(Z^*, \delta^{*-}, \delta^{*+}), e^*, R$)}
    \label{alg:never_apply}
    \begin{algorithmic}[1]
        \Require{A pseudo-configuration $(Z^*, \delta^{*-}, \delta^{*+})$ with a dart $e^*$ and a rule $R$.}
        \Ensure{\True\ if $R$ does not apply to $e^*$ for any $(Z^*, \delta^*) \in (Z^*, \delta^{*-}, \delta^{*+})$, otherwise \False.}
        \State $(Z, \delta^-, \delta^+, e, r) \gets R$
        \State \Return \Call{homomorphism}{$(Z, \delta^-, \delta^+), e, (Z^*, \delta^{*-}, \delta^{*+}), e^*, \gIntersection$} $=$ \texttt{null}
    \end{algorithmic}
\end{algorithm}

\begin{algorithm}[H]
    \caption{\amountOfChargeSendWithLink($(Z_C, \delta_C^-, \delta_C^+), e, \mathcal{R}$)}
    \label{alg:amount_of_charge_send}
    \begin{algorithmic}[1]
        \Require{A pseudo-configuration $(Z_C, \delta_C^-, \delta_C^+)$ with a dart $e$, a set of rules $\mathcal{R}$.}
        \Ensure{The sum of charges $r$ associated with $R \in \mathcal{R}$ that always applies to $e$ in $(Z_C, \delta_C^-, \delta_C^+)$.}
        \State $a \gets 0$
        \ForAll{$R \in \mathcal{R}$}
            \If {\Call{alwaysApply}{$(Z_C, \delta_C^-, \delta_C^+), e, R$}}
                 \State $a \gets a + r(R)$
            \EndIf
        \EndFor
        \State \Return $a$
    \end{algorithmic}
\end{algorithm}

\begin{algorithm}[H]
    \caption{\amountOfPossibleChargeSendWithLink($(Z_C, \delta_C^-, \delta_C^+), e, \mathcal{R}^{*-\mathcal{D}}$)}
    \label{alg:amount_of_possible_charge_send}
    \begin{algorithmic}[1]
        \Require{A pseudo-configuration $(Z_C, \delta_C^-, \delta_C^+)$ with a dart $e$, a set of rules $\mathcal{R}$.}
        \Ensure{The maximum of charge $r^*$ associated with $R^* \in \mathcal{R}^{*-\mathcal{D}}$ except when $R^*$ never applies to $e$ in $(Z_C, \delta_C^-, \delta_C^+)$.}
        \State $a \gets 0$
        \ForAll{$R^* \in \mathcal{R}^{*-\mathcal{D}}$}
            \If {\Call{neverApply}{$(Z_C, \delta_C^-, \delta_C^+), e, R^*$}}
                \Continue
            \EndIf
            \State $a \gets \max(a, r(R^*))$
        \EndFor
        \State \Return $a$
    \end{algorithmic}
\end{algorithm}

\subsubsection{Enumerating degrees of neighbors}

In this section, we describe the procedure to enumerate degrees of neighbors of the center, described in Section \ref{subsub:enum-degree-neighbor}.
The main routine is the \enumPossibleBadWheels routine in Algorithm \ref{alg:enum_possible_bad_wheels}.
We use the array $\CARTWHEELDEGREES \texttt{ = [5,6,7,8,9]}$, which represents the possible degrees assigned to vertices except the center for free cartwheels with limited degrees.
In the \enumWheels routine in Algorithm \ref{alg:enum_wheel}, given a center degree $d$, we enumerate all degree assignments for neighbors up to rotational equivalence.
To achieve this, we generate all arrays of length $d$ whose elements are chosen from $\CARTWHEELDEGREES$, and retain only those for which every rotation is lexicographically greater than or equal to the array itself.
Considering this condition, we slightly prune the search as follows: for each $d_0$ in \CARTWHEELDEGREES, we first assign $d_0$ to $\texttt{degrees}[0]$ and do not assign any element smaller than it to the remaining positions in the array.
This is justified because no array containing an element strictly smaller than its first element satisfies the above condition.

Then, we call the \generateCartwheel subroutine in Algorithm \ref{alg:generate_cartwheel} to generate a free cartwheel with limited degrees with all the second neighbors having tail ranges $[5,9]$.

\begin{algorithm}[H]
    \caption{\enumWheelsWithLink($d$)}
    \label{alg:enum_wheel}
    \begin{algorithmic}[1]
        \Require{A degree of the center $d$.}
        \Ensure{A set of free cartwheels with limited degrees and tail ranges $(Z, \delta^-, \delta^+)$ such that the center has degree $d$, the degrees of its neighbors are chosen from \CARTWHEELDEGREES, and all the degree-ranges of its second neighbors are $[5,9]$.}
        \State $\mathcal{W} \gets \emptyset$
        \Procedure{enumDegree}{$\texttt{degrees}, i, i_{\mathrm{lowest}}$}
            \If {$i=d$}
                \If {some rotation of \texttt{degrees} is lexicographically strictly smaller than \texttt{degrees}}
                    \State \Return
                \EndIf
                \State $(Z, \delta^-, \delta^+) \gets$ \Call{generateCartWheel}{$d$, \texttt{degrees}}
                \State $\mathcal{W} \gets \mathcal{W} \cup \{(Z, \delta^-, \delta^+)\}$
                \State \Return
            \EndIf
            \ForAll{$j \in \{i_{\mathrm{lowest}}, \ldots, 4\}$}
                \State $\texttt{degrees}[i] \gets \CARTWHEELDEGREES[j]$
                \State \Call{enumDegree}{$\texttt{degrees}, i+1, i_{\mathrm{lowest}}$}
            \EndFor
            \State \Return
        \EndProcedure
        \State $\texttt{degrees} \gets$ an array of size $d$.
        \ForAll{$j \in \{0, \ldots, 4\}$}
            \State $\texttt{degrees[0]} \gets \CARTWHEELDEGREES[j]$
            \State \Call{enumDegree}{$\texttt{degrees}, 1, j$}
        \EndFor
        \State \Return $\mathcal{W}$
    \end{algorithmic}
\end{algorithm}

In the \generateCartwheel subroutine in Algorithm \ref{alg:generate_cartwheel},
we construct $\texttt{rotations}$ representing the rotation of vertices around each vertex.
When $\texttt{rotations}[i]$ is a list of integers $[a_0, \ldots, a_k]$, it means that $v_i$ is adjacent to $v_{a_0}, \ldots, v_{a_k}$ in this rotation.
If $a_i=-1$, it represents the boundary.

First, we compute the rotations for the vertices within the neighbors of the center.
Second, we add the second neighbors of the center, if they are adjacent to the neighbor of the center of degree not $9$.
More precisely, for each neighbor $v_i$ of the center $v$, we add neighbors of it so that $v_i$ has a specified degree, if this degree is not $9$.
Every time we add new vertices, we extend $\texttt{rotations}$.
Finally, we add $-1$ to all second neighbors and to neighbors of degree $9$ to represent the boundary.
Then, we call the \fromVRotations subroutine in Algorithm \ref{alg:from_v_rotations} to get the pseudo triangulation represented by \texttt{rotations}.

\begin{algorithm}[H]
    \caption{\textsc{\generateCartwheelWithLink}($d, \texttt{degrees}$)}
    \label{alg:generate_cartwheel}
    \begin{algorithmic}[1]
        \Require{The degree of the center $d$, and an array $\texttt{degrees}$ specifying the degrees of its neighbors.}
        \Ensure{A free cartwheel with limited degrees and tail ranges $C$ such that the center has degree $d$, the degrees of its neighbors are specified by \texttt{degrees}, and all the degree-ranges of its second neighbors are $[5,9]$.}
        
        \State $\texttt{rotations} \gets \text{ a list of $d+1$ lists}$.
        \State $\texttt{rotations}[0] \gets [1, \ldots, d]$
        
        \ForAll{$i \in \{1, \ldots, d\}$}
            \State $i' \gets i+1 \text{ if } i < d \text{ otherwise } 1$
            \State $i'' \gets i-1 \text{ if } i > 1 \text{ otherwise } d$
            \State $\texttt{rotations}[i] \gets [i', 0, i'']$
        \EndFor
        
        \State $k \gets d + 1$ \Comment{$k$ is the next vertex id to assign.}
        \ForAll{$i \in \{1, \dots, d\}$}
            \If{$\texttt{degrees}[i-1] = 9$} \Comment{Do not add second neighbors if the degree is $9$.}
                \State \textbf{continue}
            \EndIf
            \State $A \gets \texttt{degrees}[i-1] - |\texttt{rotations}[i]|$ \Comment{Number of new neighbors to add to $v_i$}
            \RepeatN{$A$}
                \State $i_\text{last} \gets \text{the last element of } \texttt{rotations}[i]$
                \State $\texttt{rotations.push\_back}([])$
                \State $\texttt{rotations}[k] \gets [i, i_\text{last}]$
                \State $\texttt{rotations}[i]\texttt{.push\_back}(k)$
                \State $\texttt{rotations}[i_\text{last}] \gets [k] \concat \texttt{rotations}[i_\text{last}]$
                \State $k \gets k + 1$
            \End
            \State $i_\text{first} \gets \text{the first element of } \texttt{rotations}[i]$
            \State $i_\text{last} \gets \text{the last element of } \texttt{rotations}[i]$
            \State $\texttt{rotations}[i_\text{first}]\texttt{.push\_back}(i_\text{last})$
            \State $\texttt{rotations}[i_\text{last}] \gets [i_\text{first}] \concat \texttt{rotations}[i_\text{last}]$
        \EndFor
        
        \ForAll{$i \in \{1, \dots, k-1\}$}
            \If{$i > d \text{ \textbf{or} } \texttt{degrees}[i-1] = 9$}
                \State $\texttt{rotations}[i]\texttt{.push\_back}(-1)$ \Comment{$-1$ represents the boundary}
            \EndIf
        \EndFor
        
        \State $Z \gets$ \Call{fromVRotations}{$k, \texttt{rotations}$} \Comment{$V(Z)=\{v_0, \ldots, v_{k-1}\}$}
        \State $\delta^-, \delta^+ \colon V(Z) \to \mathbb{N}$ such that
        $\delta^-(v_0)=\delta^+(v_0)=d$,
        $\delta^-(v_i)=\delta^+(v_i)=\texttt{degrees}[i-1]$ if $0 < i \leq d$, and
        $\delta^-(v_i)=5$, $\delta^+(v_i)=9$ if $d < i < k$.
        \State \Return $(Z, \delta^-, \delta^+)$
    \end{algorithmic}
\end{algorithm}

In the \enumPossibleBadWheels routine, we call the \enumWheels subroutine to obtain all combinations of degrees for neighbors of the center.
Then, we remove free cartwheels with limited degrees and tail ranges if (i) an upper bound of the final charge is less than 0, or (ii) blocked by reducible configurations from $\bar{\mathcal{D}}$ by calling the \prune routine in Algorithm \ref{alg:prune}.

\begin{algorithm}[H]
    \caption{\enumPossibleBadWheelsWithLink($d, \mathcal{R}, \mathcal{R}^{*-\mathcal{D}}, \bar{\mathcal{D}}$)}
    \label{alg:enum_possible_bad_wheels}
    \begin{algorithmic}[1]
        \Require{The center degree $d$, the set of rules $\mathcal{R}$, the set of combined rules $\mathcal{R}^{*-\mathcal{D}}$, the set of configurations $\bar{\mathcal{D}}$.}
        \Ensure{The set of free cartwheels with limited degrees and tail ranges, where the center degree is $d$, all the second neighbors of the center have degree-ranges $[5,9]$, and it can contain possibly bad cartwheels.}
        \State $\mathcal{C}^d_{0} \gets \emptyset$
        \ForAll{$(Z, \delta^-, \delta^+) \in$ \Call{enumWheels}{$d$}}
             \If {\Call{prune}{$(Z, \delta^-, \delta^+), \emptyset, \mathcal{R}, \mathcal{R}^{*-\mathcal{D}}, \bar{\mathcal{D}}$}}
             \Continue
             \EndIf
        \State $\mathcal{C}^d_0 \gets \mathcal{C}^d_0 \cup \{(Z, \delta^-, \delta^+)\}$
        \EndFor
        \State \Return $\mathcal{C}^d_0$
    \end{algorithmic}
\end{algorithm}

\subsubsection{Fixing rules applied from neighbors to the center.}
In this section, we present the pseudocode for fixing the set of rules applied from neighbors to the center, described in Section \ref{subsub:fix_in_rules}.
The main routine is Algorithm \ref{alg:fix_in_rules}.
As described in Section \ref{subsub:fix_in_rules}, we decide the set of rules $\mathcal{R}_i$ applied to $v_iv$.
We take a free combination $R_i^* \in \mathcal{R}^{*-\mathcal{D}}$ and update the degree-ranges of the current free cartwheel so that $\mathcal{R}_i$ is the set of rules that lead to the free combination $\mathcal{R}_i$.

The \updateDegreeByRule subroutine in Algorithm \ref{alg:update_degree_by_rule} updates the degree-ranges as above.
Every time we decide $\mathcal{R}_i$ by updating degree-ranges, we check whether the updated free cartwheel with limited degrees and tail ranges is removed by the pruning method by the \prune routine in Algorithm \ref{alg:prune}.

\begin{algorithm}[H]
    \caption{\fixInRulesWithLink($(Z_{C_0}, \delta_{C_0}^-, \delta_{C_0}^+), \mathcal{R}, \mathcal{R}^{*-\mathcal{D}}, \bar{\mathcal{D}}$)}
    \label{alg:fix_in_rules}
    \begin{algorithmic}[1]
        \Require{A free cartwheel with limited degrees and tail ranges $(Z_{C_0}, \delta_{C_0}^-, \delta_{C_0}^+)$ such that all degree-ranges of second neighbors of the center are $[5,9]$, a set of rules $\mathcal{R}$, a set of combined rules $\mathcal{R}^{*-\mathcal{D}}$, a set of configurations $\bar{\mathcal{D}}$.}
        \Ensure{The set of free cartwheels with limited degrees and tail ranges obtained from $(Z_{C_0}, \delta_{C_0}^-, \delta_{C_0}^+)$ by determining degrees by fixing a set of rules $\mathcal{R}_i$ applied from each neighbor $v_i$ to the center $v$.}
        \State $\mathcal{C}_0 \gets \{((Z_{C_0}, \delta_{C_0}^-, \delta_{C_0}^+), \emptyset)\}$
        \ForAll{$i \in \{1, \ldots, d\}$}
            \State $\mathcal{C}_i \gets \emptyset$
            \ForAll{$((Z_C, \delta_C^-, \delta_C^+), \{R^*_j\}_{j=1}^{i-1}) \in \mathcal{C}_{i-1}$}
                \ForAll{$R_i^* \in \mathcal{R}^{*-\mathcal{D}}$}
                    \State $\mathcal{C}' \gets$ \Call{updateDegreeByRule}{$(Z_C, \delta_C^-, \delta_C^+), v_iv, R_i^*$}
                    \ForAll{$(Z_{C'}, \delta_{C'}^-, \delta_{C'}^+) \in \mathcal{C}'$}
                        \If {\Call{prune}{$(Z_{C'}, \delta_{C'}^-, \delta_{C'}^+), \{R_j^*\}_{j=1}^i, \mathcal{R}, \mathcal{R}^{*-\mathcal{D}}, \bar{\mathcal{D}}$}}
                            \Continue
                        \EndIf
                        \State $\mathcal{C}_i \gets \mathcal{C}_i \cup \{((Z_{C'}, \delta_{C'}^-, \delta_{C'}^+), \{R_j^*\}_{j=1}^i)\}$
                    \EndFor
                \EndFor
            \EndFor
        \EndFor
        \State \Return $\mathcal{C}_d$
    \end{algorithmic}
\end{algorithm}

In the \updateDegreeByRule subroutine in Algorithm \ref{alg:update_degree_by_rule}, we get a homomorphism $\phi^*$ from $R^*$ to $(Z_C, \delta_C^-, \delta_C^+)$ with $e_R^*$ mapping to $e$ in $Z_C$.
We update the degree-ranges of each vertex $v$ in the image of $\phi^*$ in $Z_C$ by calculating the intersection of the degree-ranges of $v_R$ with $v=\phi^*(v_R)$.
After that, we enumerate concrete degrees except for tail ranges of updated free cartwheels with limited degrees and tail ranges by the \concreteDegreesExceptTail subroutine.

\begin{algorithm}
    \caption{\updateDegreeByRuleWithLink($(Z_C, \delta_C^-, \delta_C^+)$, $e$, $R^*$)}
    \label{alg:update_degree_by_rule}
    \begin{algorithmic}
        \Require{A free cartwheel with limited degrees and tail ranges $(Z_C, \delta_C^-, \delta_C^+)$, a dart $e$ in $Z_C$, a combined rule $R^*$.}
        \Ensure{The set of free cartwheels with limited degrees and tail ranges by updating degree-ranges of $(Z_C, \delta_C^-, \delta_C^+)$ so that all rules associated with $R^*$ apply to $e$ in $(Z_C, \delta_C^-, \delta_C^+)$.}
        \State $(Z_R^*, \delta_R^{*-}, \delta_R^{*+}, e_R^*, r_R^*, f_R^*) \gets R^*$
        \State $\phi^* \gets$ \Call{homomorphism}{$(Z_R^*, \delta_R^{*-}, \delta_R^{*+}), e_R^*, (Z_C, \delta_C^-, \delta_C^+), e, \gIntersection$}
        \If {$\phi^* = \texttt{null}$}
            \State \Return $\emptyset$
        \EndIf
        \State $(Z_{C'}, \delta_{C'}^-, \delta_{C'}^+) \gets$ a copy of $(Z_C, \delta_C^-, \delta_C^+)$
        \ForAll{$v_R \in V(Z_R^*)$}
            \State $v_C \gets \phi^*(v_R)$
            \State $[\delta_{C'}^-(v_C), \delta_{C'}^+(v_C)] \gets [\delta_C^-(v_C), \delta_C^+(v_C)] \cap [\delta_R^{*-}(v_R), \delta_R^{*+}(v_R)]$
        \EndFor
        \State \Return \Call{concreteDegreeExceptTail}{$(Z_{C'}, \delta_{C'}^-, \delta_{C'}^+)$}
    \end{algorithmic}
\end{algorithm}

The \concreteDegreesExceptTail subroutine in Algorithm \ref{alg:concrete_degree} enumerates free cartwheels with limited degrees and tail ranges by taking combinations of concrete degrees except for tail ranges.

\begin{algorithm}[H]
    \caption{\concreteDegreeExceptTailWithLink($(Z_C, \delta_C^-, \delta_C^+)$)}
    \label{alg:concrete_degree}
    \begin{algorithmic}[1]
        \Require{A free cartwheel with limited degrees $(Z_C, \delta_C^-, \delta_C^+)$ that may have degree-ranges other than tail range.}
        \Ensure{The set of free cartwheels with limited degrees and tail ranges by taking combinations of concrete degrees within their degree-ranges if it is not tail range.}
        \State $\tilde{\mathcal{C}} \gets \{(Z_C, \delta_C^-, \delta_C^+)\}$
        \ForAll{$v \in V(Z_C)$}
            \If {$\delta_C^-(v)=\delta_C^+(v)$ or $\delta_C^+(v)=9$}
                \Continue
            \EndIf
            \State $\mathcal{C}' \gets \emptyset$
            \ForAll{$i \in \{0, \ldots, 3\}$}
                \State $d \gets \CARTWHEELDEGREES[i]$
                \If {$\delta^-(v) \leq d \leq \delta^+(v)$}
                    \ForAll{$(\tilde{Z}_C, \tilde{\delta}_C^-, \tilde{\delta}_C^+) \in \tilde{\mathcal{C}}$}
                         \State $(Z_{C'}, \delta_{C'}^-, \delta_{C'}^+) \gets$ a copy of $(\tilde{Z}_C, \tilde{\delta}_C^-, \tilde{\delta}_C^+)$
                         \State $\delta_{C'}^-(v)=\delta_{C'}^+(v)=d$
                         \State $\mathcal{C}' \gets \mathcal{C}' \cup \{(Z_{C'}, \delta_{C'}^-, \delta_{C'}^+)\}$
                    \EndFor
                \EndIf
            \EndFor
            \State $\tilde{\mathcal{C}} \gets \mathcal{C}'$
        \EndFor
        \State \Return $\tilde{\mathcal{C}}$
    \end{algorithmic}
\end{algorithm}

\subsubsection{Pruning}
In this section, we present the pseudocode of the pruning methods described in Section \ref{subsub:fix_in_rules}.
The \prune routine in Algorithm \ref{alg:prune} is the main routine that checks whether a given free cartwheel with limited degrees and tail ranges is cut.
The \pruneByNonAssociatedRule subroutine in Algorithm \ref{alg:prune_non_assoc_rule} checks whether a rule in $\mathcal{R} \setminus \mathcal{R}_j$ {\bf always} applies to $v_jv$ for each $1 \leq j \leq i$.
Note that for each free combination $R_j^*$, we computed the flag $f_j$ such that $f_j(R)=\True$ if and only if $R \in \mathcal{R}_j$ in the \combineRules routine in Algorithm \ref{alg:combine_rules}.
Thus, we can easily check $R \in \mathcal{R} \setminus \mathcal{R}_j$.
The \upperBoundOfCharge subroutine in Algorithm \ref{alg:upperbound_of_charge} calculates the upper bound of the final charge of the center of a given free cartwheel with limited degrees and tail ranges, guaranteed by Lemma \ref{lem:estimate_charge}.

\begin{algorithm}[H]
    \caption{\pruneWithLink($(Z_C, \delta_C^-, \delta_C^+), \{R^*_j\}_{j=1}^i, \mathcal{R}, \mathcal{R}^{*-\mathcal{D}}, \bar{\mathcal{D}}$)}
    \label{alg:prune}
    \begin{algorithmic}[1]
        \Require{A free cartwheel with limited degrees and tail ranges $(Z_C, \delta_C^-, \delta_C^+)$ with combined rules $\{R^*_j\}_{j=1}^i$, the set of rules $\mathcal{R}$, the set of combined rules $\mathcal{R}^{*-\mathcal{D}}$, the set of configurations $\bar{\mathcal{D}}$.}
        \Ensure{\True\ if every free cartwheel with limited degrees in $(Z_C, \delta_C^-, \delta_C^+)$ such that the set of rules applied to $v_jv$ is exactly $\mathcal{R}_j$ associated with $R^*_j$ for every $1 \leq j \leq i$ either contains centered reducible configurations from $\bar{\mathcal{D}}$ or has the final charge of less than $0$; otherwise, \False.}
        \If {\Call{pruneByNonAssociatedRule}{$(Z_C, \delta_C^-, \delta_C^+), \{R^*_j\}_{j=1}^i, \mathcal{R}$}}
            \State \Return \True
        \EndIf
        \If {\Call{upperBoundOfCharge}{$(Z_C, \delta_C^-, \delta_C^+), \{R^*_j\}_{j=1}^i, \mathcal{R}, \mathcal{R}^{*-\mathcal{D}}$)} $< 0$}
            \State \Return \True
        \EndIf
        \If {\Call{blockedByReducibleConfiguration}{$(Z_C, \delta_C^-, \delta_C^+), v, \bar{\mathcal{D}}$}}
            \State \Return \True
        \EndIf
        \State \Return \False
    \end{algorithmic}
\end{algorithm}

\begin{algorithm}[H]
    \caption{\pruneByNonAssociatedRuleWithLink($(Z_C, \delta_C^-, \delta_C^+), \{R^*_j\}_{j=1}^i, \mathcal{R}$)}
    \label{alg:prune_non_assoc_rule}
    \begin{algorithmic}[1]
        \Require{A free cartwheel with limited degrees and tail ranges $(Z_C, \delta_C^-, \delta_C^+)$ with combined rules $\{R^*_j\}_{j=1}^i$, the set of rules $\mathcal{R}$.}
        \Ensure{\True\ if some $R \in \mathcal{R} \setminus \mathcal{R}_j$ always applies to $v_jv$ for some $1 \leq j \leq i$: otherwise \False.}
        \ForAll{$j \in \{1, \ldots, i\}$}
            \State $f_j \colon \mathcal{R} \to \{\True, \False\}$ such that $f_j(R)=\True$ if and only if $R \in \mathcal{R}_j \subseteq \mathcal{R}$ where $R^*_j$ is a free combination of $\mathcal{R}_j$.
            \ForAll {$R \in \mathcal{R}$}
                \If {$f_j(R) = \False$ and \Call{alwaysApply}{$(Z_C, \delta_C^-, \delta_C^+)$, $v_jv$, $R$}}
                    \State \Return \True
                \EndIf
            \EndFor
        \EndFor
        \State \Return \False
    \end{algorithmic}
\end{algorithm}

\begin{algorithm}[H]
    \caption{\upperBoundOfChargeWithLink($(Z_C, \delta_C^-, \delta_C^+), \{R^*_j\}_{j=1}^i, \mathcal{R}, \mathcal{R}^{*-\mathcal{D}}$)}
    \label{alg:upperbound_of_charge}
    \begin{algorithmic}[1]
        \Require{A free cartwheel with limited degrees and tail ranges $(Z_C, \delta_C^-, \delta_C^+)$ with combined rules $\{R^*_j\}_{j=1}^i$, the set of rules $\mathcal{R}$, the set of combined rules $\mathcal{R}^{*-\mathcal{D}}$.}
        \Ensure{The upper bound of the final charge of the center $v$ for all free cartwheels with limited degrees such that they are refinements of $(Z_C, \delta_C^-, \delta_C^+)$, the set of rules applied to $v_jv$ is exactly $\mathcal{R}_j$ for $1 \leq j \leq i$, and contains no centered configurations from $\mathcal{D}$.}
        \ForAll{$j \in \{1, \ldots, i\}$}
            \State $r_j^* \gets$ the amount of charge associated with $R^*_j$
        \EndFor
        \ForAll{$j \in \{i+1, \ldots, d\}$}
            \State $r'_j \gets$ \Call{amountOfPossibleChargeSend}{$(Z_C, \delta_C^-, \delta_C^+), v_jv, \mathcal{R}^{*-\mathcal{D}}$}
        \EndFor
        \ForAll{$j \in \{1, \ldots, d\}$}
            \State $s_j \gets$ \Call{amountOfChargeSend}{$(Z_C, \delta_C^-, \delta_C^+), vv_j, \mathcal{R}$}
        \EndFor
        \State $T_0 \gets 10 \cdot (6 - d)$
        \State \Return $T_0 + \sum_{j=1}^i r^*_j + \sum_{j=i+1}^d r'_j - \sum_{j=1}^d s_j$
    \end{algorithmic}
\end{algorithm}

\subsubsection{Refinement}
In this section, we present the pseudocode regarding the refinement procedure, described in Section \ref{subsub:refinement}.
The main routine for constructing $\mathcal{C}$ from $\mathcal{C}_d$ is given in Algorithm \ref{alg:fix_out_rules}.
First, this routine pushes each element in $\mathcal{C}_d$ into the queue.
While the queue is not empty, we pop one element from the queue.
Then, we iterate through every rule $R$ in $\mathcal{R}$ and every neighbor $v_i$.
We use the \shouldRefine subroutine in Algorithm \ref{alg:should_refine} to check whether we should refine based on whether $R$ {\bf always} or {\bf never} applies to $vv_i$.
If \shouldRefine subroutine returns \True, we assign \texttt{refined\_flag} to \True, and do the refinement procedure in the \refinement subroutine in Algorithm \ref{alg:refinement}.
For each refined free cartwheel with limited degrees and tail ranges, we call the \prune subroutine to check whether it is cut.
If \texttt{refined\_flag} is $\False$ at the end of iterations of all pairs of rules and neighbors, no refinement is performed.
In this case, we add it to the resulting set $\mathcal{C}$.

\begin{algorithm}[H]
    \caption{\fixOutRulesWithLink($\mathcal{C}_d, \mathcal{R}, \mathcal{R}^{*-\mathcal{D}}, \bar{\mathcal{D}}$)}
    \label{alg:fix_out_rules}
    \begin{algorithmic}[1]
        \Require{A set of free cartwheels with limited degrees and tail ranges $\mathcal{C}_d$, the set of rules $\mathcal{R}$, the set of combined rules $\mathcal{R}^{*-\mathcal{D}}$, and the set of reducible configurations $\bar{\mathcal{D}}$.}
        \Ensure{The set of free cartwheels with limited degrees and tail ranges $\mathcal{C}$, obtained by refining the cartwheels in $\mathcal{C}_d$ based on whether a rule in $\mathcal{R}$ applies from the center.}
        \State \texttt{Q} $\gets$ an empty queue
        \ForAll{$((Z_C, \delta_C^-, \delta_C^+), \{R^*_j\}_{j=1}^d) \in \mathcal{C}_d$}
            \State $\texttt{Q.push}(((Z_C, \delta_C^-, \delta_C^+), \{R^*_j\}_{j=1}^d))$
        \EndFor
        \State $\mathcal{C} \gets \emptyset$
        \While{not \texttt{Q.empty()}}
            \State $((Z_C, \delta_C^-, \delta_C^+), \{R^*_j\}_{j=1}^d) \gets \texttt{Q.pop}()$
            \State \texttt{refined\_flag} $\gets$ \False
            \ForAll{$i \in [1,d]$}
                \ForAll{$R \in \mathcal{R}$}
                    \If {\Call{shouldRefine}{$(Z_C, \delta_C^-, \delta_C^+), i, R$} $=$ \False}
                        \Continue
                    \EndIf
                    \State \texttt{refined\_flag} $\gets$ \True
                    \State $\mathcal{Z}_{C'} \gets$ \Call{refinement}{$(Z_C, \delta_C^-, \delta_C^+), i, R$}
                    \ForAll {$(Z_{C'}, \delta_{C'}^-, \delta_{C'}^+) \in \mathcal{Z}_{C'}$}
                         \If {\Call{prune}{$(Z_{C'}, \delta_{C'}^-, \delta_{C'}^+), \{R_j\}_{j=1}^d, \mathcal{R}, \mathcal{R}^{*-\mathcal{D}}, \bar{\mathcal{D}}$}}
                             \Continue
                         \EndIf
                         \State \texttt{Q.push}(($(Z_{C'}, \delta_{C'}^-, \delta_{C'}^+), \{R^*_j\}_{j=1}^d$))
                    \EndFor
                    \Break
                \EndFor
                \If {\texttt{refined\_flag} = \True}
                    \Break
                \EndIf
            \EndFor
            \If {\texttt{refined\_flag} = \False}
                \State $\mathcal{C} \gets \mathcal{C} \cup \{((Z_C, \delta_C^-, \delta_C^+), \{R^*_j\}_{j=1}^d)\}$
            \EndIf
        \EndWhile
        \State \Return $\mathcal{C}$
    \end{algorithmic}
\end{algorithm}

\begin{algorithm}[H]
    \caption{\dominantlyApplyWithLink($(Z^*, \delta^{*-}, \delta^{*+}), e^*, R$)}
    \label{alg:dominantly_apply}
    \begin{algorithmic}[1]
        \Require{A pseudo-configuration $(Z^*, \delta^{*-}, \delta^{*+})$ with a dart $e^*$ and a rule $R$.}
        \Ensure{\True\ if $R$ dominantly applies to $e^*$ in $(Z^*, \delta^{*-}, \delta^{*+})$; otherwise \False}
        \State $g_\textsf{dominant}$ takes two degree-ranges $[\delta^-(v), \delta^+(v)]$, $[\delta^{*-}(v^*), \delta^{*+}(v^*)]$ and returns \True\ if and only if the intersection of them is non-empty and $\delta^+(v)=\infty$ or $\delta^{*+}(v^*)<9$ holds.
        \State $(Z, \delta^-, \delta^+, e, r) \gets R$
        \State \Return \Call{homomorphism}{$(Z, \delta^-, \delta^+), e, (Z^*, \delta^{*-}, \delta^{*+}), e^*, g_\textsf{dominant}$} $\neq \texttt{null}$
    \end{algorithmic}
\end{algorithm}

\begin{algorithm}[H]
    \caption{\shouldRefineWithLink($(Z_C, \delta_C^-, \delta_C^+), i, R$)}
    \label{alg:should_refine}
    \begin{algorithmic}[1]
        \Require{A free cartwheel with limited degrees and tail ranges $(Z_C, \delta_C^-, \delta_C^+)$, an index $1 \leq i \leq d$, a rule $R$.}
        \Ensure{\True\ if we should refine $(Z_C, \delta_C^-, \delta_C^+)$ for $R$ to always/never applies to $vv_i$; otherwise \False.}
        \State \Return \Call{alwaysApply}{$(Z_C, \delta_C^-, \delta_C^+), vv_i, R$} $=\False$ and \Call{dominantlyApply}{$(Z_C, \delta_C^-, \delta_C^+), vv_i, R$} $=\True$
    \end{algorithmic}
\end{algorithm}

The \refinement routine in Algorithm \ref{alg:refinement} refines $(Z_C, \delta_C^-, \delta_C^+)$ into $\mathcal{C}_{\textsf{always}}$ and $\mathcal{C}_{\textsf{never}}$ depending on whether $R$ always/never applies to $vv_i$.
First, we calculate the homomorphism $\phi$ from $Z_R$ to $Z_C$ with $e_R$ mapping to $vv_i$.
Then, we construct the set of vertices $U_R$ in $Z_R$.
Then, we refine degrees-ranges of $U=\phi(U_R)$ in the \refineAlways and \refineNever subroutines.
The \refineAlways subroutine in Algorithm \ref{alg:refine_always} generates $\mathcal{C}_{\textsf{always}}$, while the \refineNever subroutine in Algorithm \ref{alg:refine_never} generates $\mathcal{C}_{\textsf{never}}$.
The algorithms for generating them are the same as those described in Section \ref{subsub:refinement}.

\begin{algorithm}[H]
    \caption{\refinementWithLink($(Z_C, \delta_C^-, \delta_C^+), i, R$)}
    \label{alg:refinement}
    \begin{algorithmic}[1]
        \Require{A free cartwheel with limited degrees and tail ranges $(Z_C, \delta_C^-, \delta_C^+)$, an index $1 \leq i \leq d$, a rule $R$.}
        \Ensure{The set of free cartwheels with limited degrees and tail ranges by refining $(Z_C, \delta_C^-, \delta_C^+)$ for $R$ to always/never apply to $vv_i$.}
        \State $(Z_R, \delta_R^-, \delta_R^+, e_R, r_R) \gets R$
        \State $\phi \gets$ \Call{homomorphism}{$(Z_R, \delta_R^-, \delta_R^+), e_R, (Z_C, \delta_C^-, \delta_C^+), vv_i, \gIntersection$}
        \State $U_R \gets \emptyset$
        \ForAll{$u_R$ in vertices of $Z_R$}
            \State $u \gets \phi(u_R)$
            \If {$\delta_C^+(u)=9$ and $\delta_C^-(u) < \delta_R^-(u_R)$}
                \State $U_R \gets U_R \cup \{u_R\}$
            \EndIf
        \EndFor
        \State \Return \Call{refineAlways}{$(Z_C, \delta_C^-, \delta_C^+), U_R, \phi, R$} $\cup$ \Call{refineNever}{$(Z_C, \delta_C^-, \delta_C^+), U_R, \phi, R$}
    \end{algorithmic}
\end{algorithm}

\begin{algorithm}[H]
    \caption{\refineAlwaysWithLink($(Z_C, \delta_C^-, \delta_C^+), U_R, \phi, R$)}
    \label{alg:refine_always}
    \begin{algorithmic}[1]
        \Require{A free cartwheel with limited degrees and tail ranges $(Z_C, \delta_C^-, \delta_C^+)$, the set of vertices $U_R$, the homomorphism $\phi$ from $R$ to $(Z_C, \delta_C^-, \delta_C^+)$ with mapping $e_R$ to $vv_i$ for some $i$.}
        \Ensure{The set of free cartwheels with limited degrees and tail ranges by refining $(Z_C, \delta_C^-, \delta_C^+)$ for $R$ to always apply to $vv_i$.}
        \State $(Z_{C'}, \delta_{C'}^-, \delta_{C'}^+) \gets$ a copy of $(Z_C, \delta_C^-, \delta_C^+)$
        \ForAll{$u_R \in U_R$}
             \State $u \gets \phi(u_R)$
             \State $\delta_{C'}^-(u) \gets \delta_R^-(u_R)$
        \EndFor
        \State $\mathcal{C}_\textsf{always} \gets \{(Z_{C'}, \delta_{C'}^-, \delta_{C'}^+)\}$
        \State \Return $\mathcal{C}_\textsf{always}$
    \end{algorithmic}
\end{algorithm}

\begin{algorithm}[H]
    \caption{\refineNeverWithLink($(Z_C, \delta_C^-, \delta_C^+), U_R, \phi, R$)}
    \label{alg:refine_never}
    \begin{algorithmic}[1]
        \Require{A free cartwheel with limited degrees and tail ranges $(Z_C, \delta_C^-, \delta_C^+)$, the set of vertices $U_R$, the homomorphism $\phi$ from $R$ to $(Z_C, \delta_C^-, \delta_C^+)$ with mapping $e_R$ to $vv_i$ for some $i$.}
        \Ensure{The set of free cartwheels with limited degrees and tail ranges by refining $(Z_C, \delta_C^-, \delta_C^+)$ for $R$ to never apply to $vv_i$.}
        \State $\mathcal{C}_\textsf{never} \gets \emptyset$
        \ForAll{$v_R \in U_R$}
            \State $u \gets \phi(v_R)$
            \State $(Z^u_{C}, \delta_{C}^{u-}, \delta_{C}^{u+}) \gets$ a copy of $(Z_C, \delta_C^-, \delta_C^+)$
            \State $\delta_{C}^{u+}(u) \gets \delta_R^-(v_R) - 1$
            \State $\mathcal{C}_\textsf{never} \gets \mathcal{C}_\textsf{never} \cup$ \Call{concreteDegreeExceptTail}{($Z^u_{C}, \delta_{C}^{u-}, \delta_{C}^{u+}$)}
        \EndFor
        \State \Return $\mathcal{C}_\textsf{never}$
    \end{algorithmic}
\end{algorithm}

\subsubsection{Overall algorithm}
\newcommand{\enumAllBadCartwheels}{\texttt{enumAllBadCartwheels}\xspace}
The overall algorithm for constructing $\mathcal{C}_{\textsf{all}}$ to enumerate all bad cartwheels described in Section \ref{sub:overcharged_cartwheel_enumeration} is presented in the \enumAllBadCartwheels routine in Algorithm \ref{alg:enum_all_bad_cartwheel}.
The \enumAllBadCartwheels routine is not directly implemented as a \Cpp function, but is implemented as a shell script.
When generating $\mathcal{C}_0$ and $\mathcal{C}_{\textsf{all}}$ in the \enumAllBadCartwheels routine, we execute the algorithm in parallel.
Note that the parallel updates to $\mathcal{C}_0$ and $\mathcal{C}_{\textsf{all}}$ in the \enumAllBadCartwheels routine do not cause race conditions.
This is because, in practice, $\mathcal{C}_0$ and $\mathcal{C}_{\textsf{all}}$ are handled as directories, and each process simply writes its output cartwheel to a uniquely named file within them.
The \enumAllBadCartwheels routine calls the \enumBadCartwheels subroutine in Algorithm \ref{alg:enum_bad_cartwheel} to enumerate bad cartwheels generated from a given singleton set $\{(Z_{C_0}, \delta_{C_0}^-, \delta_{C_0}^+)\}$.

\begin{algorithm}[H]
    \caption{enumAllBadCartwheels($\mathcal{R}, \mathcal{R}^{*-\mathcal{D}}, \bar{\mathcal{D}}$)}
    \label{alg:enum_all_bad_cartwheel}
    \begin{algorithmic}[1]
        \Require{The set of rules $\mathcal{R}$, the set of combined rules $\mathcal{R}^{*-\mathcal{D}}$, the set of reducible configurations $\bar{\mathcal{D}}$.}
        \Ensure{The set of free cartwheels with limited degrees and tail ranges $(Z_C, \delta_C^-, \delta_C^+)$ such that the center degree $d$ is in $\{7,8,9,10,11\}$ and at least one $(Z_C, \delta_C) \in (Z_C, \delta_C^-, \delta_C^+)$ is bad.}
        \State $\mathcal{C}_0 \gets \emptyset$
        \ForAll{$d \in \{7,8,9,10,11\}$ \textbf{in parallel}}
            \State $\mathcal{C}_0 \gets \mathcal{C}_0 \cup $ \Call{enumPossibleBadWheels}{$d, \mathcal{R}, \mathcal{R}^{*-\mathcal{D}}, \bar{\mathcal{D}}$}
        \EndFor
        \State $\mathcal{C}_{\textsf{all}} \gets \emptyset$
        \ForAll{$(Z_{C_0}, \delta_{C_0}^-, \delta_{C_0}^+) \in \mathcal{C}_0$ \textbf{in parallel}} 
            \State $\mathcal{C}_{\textsf{all}} \gets \mathcal{C}_{\textsf{all}} \cup$ \Call{enumBadCartwheels}{$(Z_{C_0}, \delta_{C_0}^-, \delta_{C_0}^+), \mathcal{R}, \mathcal{R}^{*-\mathcal{D}}, \bar{\mathcal{D}}$}
        \EndFor
        \State \Return $\mathcal{C}_{\textsf{all}}$
    \end{algorithmic}
\end{algorithm}

\begin{algorithm}[H]
    \caption{\enumBadCartwheelsWithLink($(Z_{C_0}, \delta_{C_0}^-, \delta_{C_0}^+), \mathcal{R}, \mathcal{R}^{*-\mathcal{D}}, \bar{\mathcal{D}}$) - Lemma \ref{lem:positive-comp}}
    \label{alg:enum_bad_cartwheel}
    \begin{algorithmic}[1]
        \Require{The initial free cartwheel with limited degrees and tail ranges $(Z_{C_0}, \delta_{C_0}^-, \delta_{C_0}^+)$, the set of rules $\mathcal{R}$, the set of combined rules $\mathcal{R}^{*-\mathcal{D}}$, the set of reducible configurations $\bar{\mathcal{D}}$.}
        \Ensure{The set of free cartwheels with limited degrees and tail ranges $(Z_C, \delta_C^-, \delta_C^+)$ such that the center degree $d$ is in $\{7,8,9,10,11\}$ and at least one $(Z_C, \delta_C) \in (Z_C, \delta_C^-, \delta_C^+)$ is bad.}
        \State $\mathcal{C}_d \gets$ \Call{fixInRules}{$(Z_{C_0}, \delta_{C_0}^-, \delta_{C_0}^+), \mathcal{R}, \mathcal{R}^{*-\mathcal{D}}, \bar{\mathcal{D}}$}
        \State $\mathcal{C} \gets$ \Call{fixOutRules}{$\mathcal{C}_d, \mathcal{R}, \mathcal{R}^{*-\mathcal{D}}, \bar{\mathcal{D}}$}
        \ForAll {$((Z_C, \delta_C^-, \delta_C^+), \{\mathcal{R}_j^*\}_{j=1}^d) \in \mathcal{C}$}
            \State $C \gets$ \Call{upperBoundOfCharge}{$(Z_C, \delta_C^-, \delta_C^+), \{\mathcal{R}_j^*\}_{j=1}^d, \mathcal{R}, \mathcal{R}^{*-\mathcal{D}}$}
            \State $d \gets \delta_C^-(v)$, where $v$ is the center of $Z_C$.
            \State $\texttt{center\_darts} \gets$ \Call{centerDartsByDegree}{$(Z_C, \delta_C^-, \delta_C^+)$}
            \State \textbf{assert} $C=0$
            \State \textbf{assert} $d=7$ or $d=8$
            \State \textbf{assert} $|\texttt{center\_darts}[7]|+|\texttt{center\_darts}[8]|+|\texttt{center\_darts}[9]| > 0$
        \EndFor
        \State $\mathcal{C}' \gets \{(Z_C, \delta_C^-, \delta_C^+) \mid ((Z_C, \delta_C^-, \delta_C^+), \{\mathcal{R}_j^*\}_{j=1}^d) \in \mathcal{C}\}$
        \State \Return $\mathcal{C}'$
    \end{algorithmic}
\end{algorithm}

\begin{algorithm}[H]
    \caption{\centerDartsByDegreeWithLink($(Z_C, \delta_C^-, \delta_C^+)$)}
    \label{alg:center_darts_by_degree}
    \begin{algorithmic}[1]
        \Require{A free cartwheel $(Z_C, \delta_C^-, \delta_C^+)$.}
        \Ensure{The array of darts so that the $d$-th array represents all the center darts whose tail has degree $d$.}
        \State $\texttt{center\_darts} \gets$ the array of size $10$
        \ForAll{$e \gets$ each dart whose head is the center $v$ in cyclic order}
            \State $d \gets \delta_C^-(\head(\reverse(e))$ \Comment{The values for $\delta_C^-, \delta_C^+$ are the same.}
            \State \texttt{center\_darts}[$d$].push\_back($e$)
        \EndFor
        \State \Return \texttt{center\_darts}
    \end{algorithmic}
\end{algorithm}

\begin{lem}
\label{lem:for-lem:positive-comp}
    The assertions in Lines 7, 8, and 9 of Algorithm \ref{alg:enum_bad_cartwheel} always hold for all inputs in Algorithm \ref{alg:enum_all_bad_cartwheel}.
\end{lem}

\subsection{Free combinations of free cartwheels}
\newcommand{\deleteDegreeFromKtoN}{\texttt{deleteDegreeFromKto9}\xspace}
\newcommand{\combineEachCartwheel}{\texttt{combineEachCartwheel}\xspace}
\newcommand{\combineEachCartwheelTwice}{\texttt{combineEachCartwheelTwice}\xspace}
\newcommand{\checkDegE}{\texttt{checkDeg8}\xspace}
\newcommand{\checkEE}{\texttt{check88}\xspace}
\newcommand{\checkES}{\texttt{check87}\xspace}
\newcommand{\checkSES}{\texttt{check787}\xspace}
\newcommand{\containX}{\texttt{containX}\xspace}
\newcommand{\checkStriangle}{\texttt{check7triangle}\xspace}
\newcommand{\checkDegS}{\texttt{checkDeg7}\xspace}
\newcommand{\checkSS}{\texttt{check77}\xspace}
\newcommand{\checkSSS}{\texttt{check777}\xspace}
In this section, we provide the pseudocodes corresponding to the algorithms described in Section \ref{subsect:combine-cartwheel}.

We use the \deleteDegreeFromKtoN routine to convert the set of free cartwheels $\mathcal{C}_{\textsf{all}}$.
To prove Lemma \ref{lem:zero-deg7,8}, we convert $\mathcal{C}_{\textsf{all}}$ by calling \deleteDegreeFromKtoN with $k=9$.
Furthermore, to prove Lemma \ref{lem:777}, we convert the resulting set by calling \deleteDegreeFromKtoN with $k=8$.

Note that, for every free cartwheel $(Z_C, \delta_C^-, \delta_C^+)$ in $\mathcal{C}_{\textsf{all}}$, the center $v$ has a fixed degree (i.e., $\delta_C^-(v)=\delta_C^+(v)$).
Each neighbor of the center also has a fixed degree. 

\begin{algorithm}[H]
    \caption{\deleteDegreeFromKToNWithLink($\mathcal{C}, k$)}
    \label{alg:delete_degree_k}
    \begin{algorithmic}[1]
        \Require{A set of free cartwheels $\mathcal{C}$, an positive integer $k \leq 9$.}
        \Ensure{The set of free cartwheels obtained from $\mathcal{C}$ by deleting free cartwheels containing vertices of fixed degree $k$ and updating degree range $[k-1,9]$ to fixed degree $k-1$.}
        \State $\mathcal{C}' \gets \emptyset$
        \ForAll{$(Z_C, \delta_C^-, \delta_C^+) \in \mathcal{C}$}
            \State $\texttt{remove} \gets \False$
            \ForAll{$v \in V(Z_C)$}
                \If {$\delta_C^-(v)=\delta_C^+(v)=k$}
                    \State $\texttt{remove} \gets \True$
                    \Break
                \ElsIf {$\delta_C^-(v)=k-1$ and $\delta_C^+(v)=9$}
                    \State $\delta_C^+(v) \gets k-1$
                \EndIf
            \EndFor
            \If {$\texttt{remove} = \False$}
                \State $\mathcal{C}' \gets \mathcal{C}' \cup \{(Z_C, \delta_C^-, \delta_C^+)\}$
            \EndIf
        \EndFor
        \State \Return $\mathcal{C}'$
    \end{algorithmic}
\end{algorithm}

The \combineEachCartwheel routine combines a given pseudo-configuration by identifying $e$ with each center dart $e'$ of $(Z', \delta^{\prime-}, \delta^{\prime+}) \in \mathcal{C}$ that are not blocked by an input configuration set $\mathcal{K}$.
Note that the input $\mathcal{K}$ is $\bar{\mathcal{D}}$ or $\bar{\mathcal{D}} \cup \{T_{7^3}\}$.
The latter set is necessary to handle Lemma \ref{lem:zero-deg7,7}.
The \combineEachCartwheelTwice routine combines two free cartwheels in $\mathcal{C}$ by calling the \combineEachCartwheel subroutine twice.

\begin{algorithm}[H]
    \caption{\combineEachCartwheelWithLink($(Z, \delta^-, \delta^+), e, \mathcal{C}, \mathcal{K}$)}
    \label{alg:combine_each_cartwheel}
    \begin{algorithmic}[1]
        \Require{A pseudo-configuration $(Z, \delta^-, \delta^+)$ with a specified dart $e$, a set of cartwheels $\mathcal{C}$, a set of configurations $\mathcal{K}$.}
        \Ensure{The set of all pseudo-configurations obtained by combining $(Z, \delta^-, \delta^+)$ and $(Z', \delta^{\prime-}, \delta^{\prime+}) \in \mathcal{C}$ by identifying $e$ with a center dart $e'$ of $(Z', \delta^{\prime-}, \delta^{\prime+})$, and not blocked by $\mathcal{K}$ with a homomorphism from $Z$ to it.}
        \State $\mathcal{Z} \gets \emptyset$
        \ForAll{$(Z', \delta^{\prime-}, \delta^{\prime+}) \in \mathcal{C}$}
            \ForAll{$e' \gets$ the dart whose head is the center $v$ of $(Z', \delta^{\prime-}, \delta^{\prime+})$}
                \State $\mathcal{Z}^* \gets$ \Call{freeHomomorphismConfiguration}{$(Z, \delta^-, \delta^+) \sqcup (Z', \delta^{\prime-}, \delta^{\prime+}), \{(e, e')\}$}
                \ForAll{$((Z^*, \delta^{*-}, \delta^{*+}), \phi^*) \in \mathcal{Z}^*$}
                    \State $c^* \gets$ an arbitrary vertex in $Z^*$
                    \If {\Call{blockedByReducibleConfiguration}{$(Z^*, \delta^{*-}, \delta^{*+}), c^*, \mathcal{K}$} $=$ \True}
                        \Continue
                    \EndIf
                    \State $\mathcal{Z} \gets \mathcal{Z} \cup \{((Z^*, \delta^{*-}, \delta^{*+}), \phi^*_{|Z}) \}$
                \EndFor
            \EndFor
        \EndFor
        \State \Return $\mathcal{Z}$
    \end{algorithmic}
\end{algorithm}

\begin{algorithm}[H]
    \caption{\combineEachCartwheelTwiceWithLink($(Z, \delta^-, \delta^+), e_1, e_2, \mathcal{C}, \mathcal{K}$)}
    \label{alg:combine_each_cartwheel_twice}
    \begin{algorithmic}[1]
        \Require{A pseudo-configuration $(Z, \delta^-, \delta^+)$ with two specified darts $e_1, e_2$, a set of cartwheels $\mathcal{C}$, a set of configurations $\mathcal{K}$.}
        \Ensure{The set of pseudo-configurations $Z^{**}$ by combining $(Z, \delta^-, \delta^+)$ and $(Z', \delta^{\prime-}, \delta^{\prime+}), (Z'', \delta^{\prime\prime-}, \delta^{\prime\prime+}) \in \mathcal{C}$ by identifying $e_1$ with a center dart of $(Z', \delta^{\prime-}, \delta^{\prime+})$, $e_2$ with a center dart of $(Z'', \delta^{\prime\prime-}, \delta^{\prime\prime+})$, and not blocked by $\mathcal{K}$ with a homomorphism from $Z$ to it.}
        \State $\mathcal{Z}^* \gets$ \Call{combineEachCartwheel}{$(Z, \delta^-, \delta^+), e_1, \mathcal{C}, \bar{\mathcal{D}}$}
        \State $\mathcal{Z}^{**} \gets \emptyset$
            \ForAll{$((Z^*, \delta^{*-}, \delta^{*+}), \phi^*_{|Z}) \in \mathcal{Z}^*$}
                \ForAll {$((Z^{**}, \delta^{**-}, \delta^{**+}), \phi^{**}) \in$ \Call{combineEachCartwheel}{$(Z^*, \delta^{*-}, \delta^{*+}), \phi^*_{|Z}(e_2), \mathcal{C}, \mathcal{K}$}}
                    \State $\mathcal{Z}^{**} \gets \mathcal{Z}^{**} \cup \{((Z^{**}, \delta^{**-}, \delta^{**+}), \phi^{**} \circ \phi^*_{|Z})\}$
                \EndFor
            \EndFor
        \State \Return $\mathcal{Z}^{**}$
    \end{algorithmic}
\end{algorithm}

\subsubsection{A vertex of degree 8}
The main routine to check Lemma \ref{lem:zero-deg7,8} is \checkDegE in Algorithm \ref{alg:for-lem:zero-deg7,8}. As explained in Section \ref{subsect:combine-cartwheel}, we have a lemma below.
Recall that in Lemma \ref{lem:zero-deg7,8} (iii), we require that two neighbors of degree $7$ are as close as possible in the successor order.
Thus, in the \checkSES routine in Algorithm \ref{alg:for-lem:zero-deg7,8(iii)}, we calculate the minimum distance of the successor order, and only check pairs of neighbors satisfying this minimum distance.

\begin{algorithm}[H]
    \caption{\checkDegEWithLink($\mathcal{C}_{\textsf{all}}, \bar{\mathcal{D}}$) - Lemma \ref{lem:zero-deg7,8}}
    \label{alg:for-lem:zero-deg7,8}
    \begin{algorithmic}[1]
        \Require{The set of free cartwheels $\mathcal{C}_{\textsf{all}}$, the set of configurations $\bar{\mathcal{D}}$.}
        \State $\mathcal{C} \gets$ \Call{deleteDegreeFromKto9}{$\mathcal{C}_\textsf{all}, 9$}
        \ForAll{$(Z_C, \delta_C^-, \delta_C^+) \in \mathcal{C}$ \textbf{in parallel}}
            \State $v \gets$ the center of $Z_C$
            \If {$\delta_C^-(v)=8$} \Comment{$\delta_C^-(v)=\delta_C^+(v)$ trivially holds.}
                \State $\texttt{center\_darts} \gets$ \Call{centerDartsByDegree}{$(Z_C, \delta_C^-, \delta_C^+)$}
                \If {$|\texttt{center\_darts}[8]| > 0$}
                    \State \Call{check88}{$(Z_C, \delta_C^-, \delta_C^+), \texttt{center\_darts}[8], \mathcal{C}, \bar{\mathcal{D}}$} \Comment{Lemma \ref{lem:zero-deg7,8}(i)}
                \ElsIf {$|\texttt{center\_darts}[7]|=1$}
                    \State \Call{check87}{$(Z_C, \delta_C^-, \delta_C^+), \texttt{center\_darts}[7], \mathcal{C}, \bar{\mathcal{D}}$} \Comment{Lemma \ref{lem:zero-deg7,8}(ii)}
                \ElsIf {$|\texttt{center\_darts}[7]|>1$}
                    \State \Call{check787}{$(Z_C, \delta_C^-, \delta_C^+), \texttt{center\_darts}[7], \mathcal{C}, \bar{\mathcal{D}}$} \Comment{Lemma \ref{lem:zero-deg7,8}(iii)}
                \EndIf
            \EndIf
        \EndFor
    \end{algorithmic}
\end{algorithm}

\begin{algorithm}[H]
    \caption{\checkEEWithLink($(Z_C, \delta_C^-, \delta_C^+), \texttt{darts8}, \mathcal{C}, \bar{\mathcal{D}}$) - Lemma \ref{lem:zero-deg7,8}(i)}
    \label{alg:for-lem:zero-deg7,8(i)}
    \begin{algorithmic}[1]
        \Require{A free cartwheel $(Z_C, \delta_C^-, \delta_C^+)$, \texttt{darts8} is the set of center darts in $Z_C$ where the tail has degree $8$, a set of free cartwheels $\mathcal{C}$, the set of configurations $\bar{\mathcal{D}}$.}
        \ForAll{$e \gets$ each dart in \texttt{darts8}}
            \State \textbf{assert} \Call{CombineEachCartwheel}{$(Z_C, \delta_C^-, \delta_C^+), \reverse(e), \mathcal{C}, \bar{\mathcal{D}}$} $= \emptyset$
        \EndFor
    \end{algorithmic}
\end{algorithm}

\begin{algorithm}[H]
    \caption{\checkESWithLink($(Z_C, \delta_C^-, \delta_C^+), \texttt{darts7}, \mathcal{C}, \bar{\mathcal{D}}$) - Lemma \ref{lem:zero-deg7,8}(ii)}
    \label{alg:for-lem:zero-deg7,8(ii)}
    \begin{algorithmic}[1]
        \Require{A free cartwheel $(Z_C, \delta_C^-, \delta_C^+)$, \texttt{darts7} is the set of center darts in $Z_C$ where the tail has degree $7$, a set of free cartwheels $\mathcal{C}$, the set of configurations $\bar{\mathcal{D}}$.}
        \State $e \gets$ the single element of \texttt{darts7}.
        \State \textbf{assert} \Call{combineEachCartwheel}{$(Z_C, \delta_C^-, \delta_C^+), \reverse(e), \mathcal{C}, \bar{\mathcal{D}}$} $=\emptyset$
    \end{algorithmic}
\end{algorithm}

\begin{algorithm}[H]
    \caption{\checkSESWithLink($(Z_C, \delta_C^-, \delta_C^+), \texttt{darts7}, \mathcal{C}, \bar{\mathcal{D}}$) - Lemma \ref{lem:zero-deg7,8}(iii)}
    \label{alg:for-lem:zero-deg7,8(iii)}
    \begin{algorithmic}[1]
        \Require{A free cartwheel $(Z_C, \delta_C^-, \delta_C^+)$, \texttt{darts7} is the set of center darts in $Z_C$ where the tail has degree $7$, a set of free cartwheels $\mathcal{C}$, the set of configurations $\bar{\mathcal{D}}$.}
        \State $\texttt{min\_dist} \gets \infty$
        \State $\texttt{dist} \gets$ an array of size $|\texttt{darts7}|$ filled with zeros
        \ForAll{$i \gets \{0, \ldots, |\texttt{darts7}|-1\}$}
            \State $e_1 \gets \texttt{darts7}[i]$ 
            \State $e_2 \gets$  $\texttt{darts7}[0]$ if $i=|\texttt{darts7}|-1$ otherwise $\texttt{darts7}[i+1]$
            \While{$e_1 \neq e_2$}
                \State $e_1 \gets \successor(e_1)$
                \State $\texttt{dist}[i] \gets \texttt{dist}[i] + 1$
            \EndWhile
            \State $\texttt{min\_dist} \gets \min(\texttt{min\_dist}, \texttt{dist}[i])$
        \EndFor
        \algstore{check787}
    \end{algorithmic}
    \hfill\textit{Continued on next page}
\end{algorithm}

\begin{algorithm}[H]
    \ContinuedFloat
    \caption{(continued)}
    \begin{algorithmic}[1]
        \algrestore{check787}
        \ForAll{$i \gets \{0, \ldots, |\texttt{darts7}|-1\}$}
            \If {$\texttt{dist}[i] > \texttt{min\_dist}$}
                \Continue
            \EndIf
            \State $e_1 \gets \texttt{darts7}[i]$ 
            \State $e_2 \gets$  $\texttt{darts7}[0]$ if $i=|\texttt{darts7}|-1$ otherwise $\texttt{darts7}[i+1]$
            \State $\mathcal{Z}^{**} \gets$ \Call{combineEachCartwheelTwice}{$(Z_C, \delta_C^-, \delta_C^+), \reverse(e_1), \reverse(e_2), \mathcal{C}, \bar{\mathcal{D}}$}
            \ForAll{$((Z^{**}, \delta^{**-}, \delta^{**+}), \phi^{**}_{Z_C}) \gets \mathcal{Z}^{**}$}
                \State $c^{**} \gets \phi^{**}_{Z_C}(v)$, where $v$ is the center of $Z_C$ 
                \State \textbf{assert} \Call{containX}{$(Z^{**}, \delta^{**-}, \delta^{**+}), c^{**}$}
            \EndFor
        \EndFor
    \end{algorithmic}
\end{algorithm}

\begin{algorithm}
    \caption{\containXWithLink($(Z, \delta^-, \delta^+), v$)}
    \label{alg:contain_X}
    \begin{algorithmic}[1]
        \Require{A pseudo-configuration $(Z, \delta^-, \delta^+)$ with a distinguished vertex $c$.}
        \Ensure{\True\ if $(Z, \delta^-, \delta^+)$ contains $X$ such that $v$ corresponds to the central vertex of degree $8$ in $X$ (See Figure \ref{fig:exception-zero-concentrate}).}
        \State $(X, \delta_X) \gets$ the configuration $X$
        \State $e_Z \gets$ any dart whose head is $v$ in $Z$.
        \State $e_X \gets$ any dart whose head is the central vertex of degree $8$ in $X$.
        \ForAll{$i \in \{0, \ldots, 7\}$}
            \If {\Call{homomorphism}{$(X, \delta_X, \delta_X), e_X, (Z, \delta^-, \delta^+), e_Z, \gInclude$} $\neq$ \texttt{null}}
                \State \Return \True
            \EndIf
            \State $e_X \gets \successor(e_X)$
        \EndFor
        \State \Return \False
    \end{algorithmic}
\end{algorithm}

\begin{lem}
\label{lem:comp-check-deg8}
    The assertions in Algorithm \ref{alg:for-lem:zero-deg7,8(i)}, \ref{alg:for-lem:zero-deg7,8(ii)}, and \ref{alg:for-lem:zero-deg7,8(iii)} always hold.
\end{lem}

\subsubsection{Maximal degree 7}
The main routine to check Lemma \ref{lem:777} is \checkStriangle in Algorithm \ref{alg:for-lem:777}.  As explained in Section \ref{subsect:combine-cartwheel}, we have a lemma below. 

\begin{algorithm}[H]
    \caption{\checkStriangleWithLink($\mathcal{C}_{\textsf{all}}, \bar{\mathcal{D}}$) - Lemma \ref{lem:777}}
    \label{alg:for-lem:777}
    \begin{algorithmic}[1]
        \Require{The set of free cartwheels $\mathcal{C}_{\textsf{all}}$, the set of configurations $\bar{\mathcal{D}}$.}
        \State $\mathcal{C} \gets$ \Call{deleteDegreeFromKto9}{$\mathcal{C}_{\textsf{all}}, 9$}
        \State $\mathcal{C} \gets$ \Call{deleteDegreeFromKto9}{$\mathcal{C}, 8$}
        \ForAll{$(Z_C, \delta_C^-, \delta_C^+) \in \mathcal{C}$ \textbf{in parallel}}
            \State $v \gets$ the center of $Z_C$
            \ForAll{$e \gets$ each dart whose head is $v$ in $Z_C$ in the cycic order}
                \State $f \gets \successor(e)$
                \State $v_e \gets \head(\reverse(e))$
                \State $v_f \gets \head(\reverse(f))$
                \If {$\delta_C^-(v_e)=7$ and $\delta_C^-(v_f)=7$}
                    \State \textbf{assert} \Call{combineEachCartWheelTwice}{$(Z_C, \delta_C^-, \delta_C^+), \reverse(e), \reverse(f), \mathcal{C}, \bar{\mathcal{D}}$} $= \emptyset$
                \EndIf
            \EndFor
        \EndFor
    \end{algorithmic}
\end{algorithm}

\begin{lem}
\label{lem:comp-check-7triangle}
    The assertions in Algorithm \ref{alg:for-lem:777} always hold.
\end{lem}

The main routine to check Lemma \ref{lem:zero-deg7,7} is \checkDegS in Algorithm \ref{alg:for-lem:zero-deg7,7}.

\begin{algorithm}[H]
    \caption{\checkDegSWithLink($\mathcal{C}_{\textsf{all}}, \bar{\mathcal{D}}$) - Lemma \ref{lem:zero-deg7,7}}
    \label{alg:for-lem:zero-deg7,7}
    \begin{algorithmic}[1]
        \Require{The set of free cartwheels $\mathcal{C}_{\textsf{all}}$, the set of configurations $\bar{\mathcal{D}}$.}
        \State $\mathcal{C} \gets$ \Call{deleteDegreeFromKto9}{$\mathcal{C}_{\textsf{all}}, 9$}
        \State $\mathcal{C} \gets$ \Call{deleteDegreeFromKto9}{$\mathcal{C}, 8$}
        \State $\mathcal{C} \gets$ all the free cartwheels in $\mathcal{C}$ that are not blocked by $T_{7^3}$
        \ForAll{$(Z_C, \delta_C^-, \delta_C^+) \in \mathcal{C}$ \textbf{in parallel}}
            \State $v \gets$ the center of $Z_C$ \Comment{$\delta_C^-(v)=\delta_C^+(v)=7$ holds.}
            \State \texttt{center\_darts} $\gets$ \Call{centerDartsByDegree}{$(Z_C, \delta_C^-, \delta_C^+)$}
            \If {$|\texttt{center\_darts}[7]| = 1$}
                \State \Call{check77}{$(Z_C, \delta_C^-, \delta_C^+), \texttt{center\_darts}[7], \mathcal{C}, \bar{\mathcal{D}} \cup \{T_{7^3}\}$} \Comment{Lemma \ref{lem:zero-deg7,7} (i)}
            \ElsIf {$|\texttt{center\_darts}[7]| > 1$}
                \State \Call{check777}{$(Z_C, \delta_C^-, \delta_C^+), \texttt{center\_darts}[7], \mathcal{C}, \bar{\mathcal{D}} \cup \{T_{7^3}\}$} \Comment{Lemma \ref{lem:zero-deg7,7} (ii)}
            \EndIf
        \EndFor
    \end{algorithmic}
\end{algorithm}

\begin{algorithm}[H]
    \caption{\checkSSWithLink($(Z_C, \delta_C^-, \delta_C^+), \texttt{darts7}, \mathcal{C}, \bar{\mathcal{D}} \cup \{T_{7^3}\}$) - Lemma \ref{lem:zero-deg7,7}(i)}
    \label{alg:for-lem:zero-deg7,7(i)}
    \begin{algorithmic}[1]
        \Require{A free cartwheel $(Z_C, \delta_C^-, \delta_C^+)$, \texttt{darts7} is the set of center darts in $Z_C$ where the tail has degree $7$, a set of free cartwheels $\mathcal{C}$, the set of configurations $\bar{\mathcal{D}} \cup \{T_{7^3}\}$.}
        \State $e \gets$ the single element of \texttt{darts7}.
        \State \textbf{assert} \Call{combineEachCartwheel}{$(Z_C, \delta_C^-, \delta_C^+), \reverse(e), \mathcal{C}, \bar{\mathcal{D}} \cup \{T_{7^3}\}$} $=\emptyset$
    \end{algorithmic}
\end{algorithm}

\begin{algorithm}[H]
    \caption{\checkSSSWithLink($(Z_C, \delta_C^-, \delta_C^+), \texttt{darts7}, \mathcal{C}, \bar{\mathcal{D}} \cup \{T_{7^3}\}$) - Lemma \ref{lem:zero-deg7,7}(ii)}
    \label{alg:for-lem:zero-deg7,7(ii)}
    \begin{algorithmic}[1]
        \Require{A free cartwheel $(Z_C, \delta_C^-, \delta_C^+)$, \texttt{darts7} is the set of center darts in $Z_C$ where the tail has degree $7$, a set of free cartwheels $\mathcal{C}$, the set of configurations $\bar{\mathcal{D}} \cup \{T_{7^3}\}$.}
        \ForAll{$i \in \{0, \ldots, |\texttt{darts7}|-1\}$}
            \State $e_1 \gets \texttt{darts7}[i]$
            \ForAll{$j \in \{0, \ldots, i - 1\}$}
                \State $e_2 \gets \texttt{darts7}[j]$
                \State \textbf{assert} \Call{combineEachCartwheelTwice}{$(Z_C, \delta_C^-, \delta_C^+), \reverse(e_1), \reverse(e_2), \mathcal{C}, \bar{\mathcal{D}} \cup \{T_{7^3}\}$} $= \emptyset$
            \EndFor
        \EndFor
    \end{algorithmic}
\end{algorithm}

\begin{lem}
\label{lem:comp-check-deg7}
    The assertions in Algorithm \ref{alg:for-lem:zero-deg7,7(i)} and \ref{alg:for-lem:zero-deg7,7(ii)} always hold.
\end{lem}

\end{document}